\documentclass[final,leqno]{siamltex}

\usepackage{epsfig,amsmath,amssymb, color}
\usepackage{algorithmic}
\usepackage{algorithm}
\usepackage{graphicx}
\usepackage{subcaption}
\usepackage{caption}
\usepackage{mathrsfs}
\usepackage{float}
\usepackage{hyperref}
\usepackage{hyperref, bm}

\newcommand{\br}{{\mathbb R}}
\newcommand{\bc}{{\mathbb C}}

\newtheorem{conjecture}[theorem]{Conjecture}
\newtheorem{problem}[theorem]{Problem}

\newtheorem{remark}[theorem]{Remark}

\newcommand{\eps}{{\epsilon}}
\newcommand{\vp}{{\varphi}}
\newcommand{\ip}[2]{\left \langle #1, #2 \right\rangle}
\newcommand{\norm}[1]{\left \lVert #1 \right \rVert}
\newcommand{\moment}[1]{{\left\vert\kern-0.25ex\left\vert\kern-0.25ex\left\vert #1 
    \right\vert\kern-0.25ex\right\vert\kern-0.25ex\right\vert}}
\newcount\colveccount
\newcommand*\colvec[1]{
        \global\colveccount#1
        \begin{pmatrix}
        \colvecnext
}
\def\colvecnext#1{
        #1
        \global\advance\colveccount-1
        \ifnum\colveccount>0
                \\
                \expandafter\colvecnext
        \else
                \end{pmatrix}
        \fi
}
\newcommand{\rcf}[1]{\bm{\hat\mathcal #1}}
\newcommand{\scf}[1]{\hat{\mathcal #1}}
\newcommand{\rdf}[1]{\bm{\hat{#1}}}
\newcommand{\sdf}[1]{\hat{#1}}
\newcommand{\rcr}[1]{\bm{\mathcal #1}}
\newcommand{\scr}[1]{\mathcal #1}
\newcommand{\rdr}[1]{\bm{#1}}
\newcommand{\sdr}[1]{#1}
\newcommand{\ssf}[1]{\hat{\mathscr #1}}

\newcommand{\Var}{{\text{Var}}}

\begin{document}

\title{Covariance Matrix Estimation for the \\ Cryo-EM Heterogeneity Problem}

\author{E. Katsevich\thanks{Department of Mathematics, Princeton University, Princeton, NJ 08544 ({\tt ekatsevi@princeton.edu}).}
        \and A. Katsevich\thanks{Department of Mathematics, University of Central Florida, Orlando, FL 32816 ({\tt alexander.katsevich@ucf.edu})}
	\and A. Singer \thanks{Department of Mathematics and PACM, Princeton University, Princeton, NJ 08544-1000 ({\tt amits@math.princeton.edu})}}
\date{}

\maketitle

\begin{abstract}
In cryo-electron microscopy (cryo-EM), a microscope generates a top view of a sample of randomly-oriented copies of a molecule. The problem of single particle reconstruction (SPR) from cryo-EM is to use the resulting set of noisy 2D projection images taken at unknown directions to reconstruct the 3D structure of the molecule. In some situations, the molecule under examination exhibits structural variability, which poses a fundamental challenge in SPR. The heterogeneity problem is the task of mapping the space of conformational states of a molecule. It has been previously suggested that the leading eigenvectors of the covariance matrix of the 3D molecules can be used to solve the heterogeneity problem. Estimating the covariance matrix is challenging, since only projections of the molecules are observed, but not the molecules themselves. In this paper, we formulate a general problem of covariance estimation from noisy projections of samples. This problem has intimate connections with matrix completion problems and high-dimensional principal component analysis. We propose an estimator and prove its consistency. When there are finitely many heterogeneity classes, the spectrum of the estimated covariance matrix reveals the number of classes. The estimator can be found as the solution to a certain linear system. In the cryo-EM case, the linear operator to be inverted, which we term the projection covariance transform, is an important object in covariance estimation for tomographic problems involving structural variation. Inverting it involves applying a filter akin to the ramp filter in tomography. We design a basis in which this linear operator is sparse and thus can be tractably inverted despite its large size.  We demonstrate via numerical experiments on synthetic datasets the robustness of our algorithm to high levels of noise.
\end{abstract}

\begin{keywords}
Cryo-electron microscopy, X-ray transform, inverse problems, structural variability, classification, heterogeneity, covariance matrix estimation, principal component analysis, high-dimensional statistics, Fourier projection slice theorem, spherical harmonics
\end{keywords}

\begin{AMS}
92C55, 44A12, 92E10, 68U10, 33C55, 62H30, 62J10
\end{AMS}

\section{Introduction}

\subsection{Covariance matrix estimation from projected data}

Covariance matrix estimation is a fundamental task in statistics. Statisticians have long grappled with the problem of estimating this statistic when the samples are only partially observed. In this paper, we consider this problem in the general setting where ``partial observations" are arbitrary linear projections of the samples onto a lower-dimensional space. 

\begin{problem} \label{main_problem}
Let $\rdr X$ be a random vector on $\bc^p$, with $\mathbb E[\rdr X] = \mu_0$ and $\Var(\rdr X) = \Sigma_0$ ($\Var[\rdr X]$ denotes the covariance matrix of $\rdr X$). Suppose also that $\rdr P$ is a random $q \times p$ matrix with complex entries, and $\rdr E$ is a random vector in $\bc^q$ with $\mathbb E[\rdr E] = 0$ and $\Var[\rdr E] = \sigma^2 \textnormal{I}_q$. Finally, let $\rdr I$ denote the random vector in $\bc^q$ given by
\begin{equation}
\rdr I = \rdr P \rdr X + \rdr E.
\label{stat_model}
\end{equation}
Assume now that $\rdr X$, $\rdr P$, and $\rdr E$ are independent. Estimate $\mu_0$ and $\Sigma_0$ given observations $\sdr I_1, \dots, \sdr I_n$ and $\sdr P_1, \dots, \sdr P_n$ of $\rdr  I$ and $\rdr P$, respectively.
\end{problem}

Here, and throughout this paper, we write random quantities in boldface to distinguish them from deterministic quantities. We use regular font (e.g., $\sdr X$) for vectors and matrices, calligraphic font (e.g., $\scr X$) for functions, and script font for function spaces (e.g., $\mathscr B$). We denote true parameter values with a subscript of zero (e.g., $\mu_0$), estimated parameter values with a subscript of $n$ (e.g., $\mu_n$), and generic variables with no subscript (e.g., $\mu$).

Problem \ref{main_problem} is quite general, and has many practical applications as special cases. The main application this paper addresses is the heterogeneity problem in single particle reconstruction (SPR) from cryo-electron microscopy (cryo-EM). SPR from cryo-EM is an inverse problem where the goal is to reconstruct a 3D molecular structure from a set of its 2D projections from random directions \cite{ctf}. The heterogeneity problem deals with the situation in which the molecule to be reconstructed can exist in several structural classes. In the language of Problem \ref{main_problem}, $\rdr X$ represents a discretization of the molecule (random due to heterogeneity), $P_s$ the 3D-to-2D projection matrices, and $I_s$ the noisy projection images. The goal of this paper is to estimate the covariance matrix associated with the variability of the molecule. If there is a small, finite number $(C)$ of classes, then $\Sigma_0$ has low rank $(C-1)$. This ties the heterogeneity problem to principal component analysis (PCA) \cite{pca}. If $\Sigma_0$ has eigenvectors $V_1, \dots, V_p$ (called principal components) corresponding to eigenvalues $\lambda_1 \geq \dots \geq \lambda_p$, then PCA states that $V_i$ accounts for a variance of $\lambda_i$ in the data. In modern applications, the dimensionality $p$ is often large, while $\rdr X$ typically has much fewer intrinsic degrees of freedom \cite{donoho_data}. The heterogeneity problem is an example of such a scenario; for this problem, we demonstrate later that the top principal components can be used in conjunction with the images to reconstruct each of the $C$ classes. 


Another class of applications closely related to Problem \ref{main_problem} are missing data problems in statistics. In these problems, $\sdr X_1, \dots, \sdr X_n$ are samples of a random vector $\rdr X$. The statistics of this random vector must be estimated in a situation where certain entries of the samples $\sdr X_s$ are missing \cite{little_rubin}. This amounts to choosing $\sdr P_s$ to be \textit{coordinate selection operators}, operators which output a certain subset of the entries of a vector. An important problem in this category is PCA with missing data, which is the task of finding the top principal components when some data are missing. Closely related to this is the noisy low rank matrix completion problem \cite{matrix_completion_candes}. In this problem, only a subset of the entries of a low rank matrix $A$ are known (possibly with some error), and the task is to fill in the missing entries. If we let $\sdr X_s$ be the columns of $A$, then the observed variables in each column are $\sdr P_s \sdr X_s + \sdr \eps_s$, where $\sdr P_s$ acts on $\sdr X_s$ by selecting a subset of its coordinates, and $\sdr \eps_s$ is noise. Note that the matrix completion problem involves filling in the missing entries of $\sdr X_s$, while Problem \ref{main_problem} requires us only to find the covariance matrix of these columns. However, the two problems are closely related. For example, if the columns are distributed normally, then the missing entries can be found as their expectations conditioned on the known variables \cite{climate_data}. Alternatively, we can find the missing entries by choosing the linear combinations of the principal components that best fit the known matrix entries. A well-known application of matrix completion is in the field of recommender systems (also known as collaborative filtering). In this application, users rate the products they have consumed, and the task is to determine what new products they would rate highly. We obtain this problem by interpreting $A_{i, j}$ as the $j$'th user's rating of product $i$. In recommender systems, it is assumed that only a few underlying factors determine users' preferences. Hence, the data matrix $A$ should have low rank. A high profile example of recommender systems is the Netflix prize problem \cite{netflix}.

In both of these classes of problems, $\Sigma_0$ is large but should have low rank. Despite this, note that Problem \ref{main_problem} does not have a low rank assumption. Nevertheless, as our numerical results demonstrate, the spectrum of our (unregularized) covariance matrix estimator reveals low rank structure when it is present in the data. Additionally, the framework we develop in this paper naturally allows for regularization.

Having introduced Problem \ref{main_problem} and its applications, let us delve more deeply into one particular application: SPR from cryo-EM.

\subsection{Cryo-electron microscopy} \label{section:cryo_em}
Electron microscopy is an important tool for structural biologists, as it allows them to determine complex 3D macromolecular structures. A general technique in electron microscopy is called \textit{single particle reconstruction}. In the basic setup of SPR, the data collected are 2D projection images of ideally assumed identical, but randomly oriented, copies of a macromolecule. In particular, one specimen preparation technique used in SPR is called \textit{cryo-electron microscopy}, in which the sample of molecules is rapidly frozen in a thin ice layer \cite{ctf, wang06}. The electron microscope provides a top view of the molecules in the form of a large image called a micrograph. The projections of the individual particles can be picked out from the micrograph, resulting in a set of projection images. Mathematically, we can describe the imaging process as follows. Let $\scr X: \br^3 \rightarrow \br$ represent the Coulomb potential induced by the unknown molecule. We scale the problem to be dimension-free in such a way that most of the ``mass" of $\scr X$ lies within the unit ball $B \subset \br^3$ (since we later model $\scr X$ to be bandlimited, we cannot quite assume it is supported in $B$). To each copy of this molecule corresponds a rotation $\sdr R \in SO(3)$, which describes its orientation in the ice layer. The idealized forward projection operator $\scr P = \scr P(\sdr R): L^1(\br^3) \rightarrow L^1(\br^2)$ applied by the microscope is the X-ray transform
\begin{equation}
(\scr P \scr X)(x, y) = \int_{\br} \scr X(\sdr R^T r)dz,
\label{P_s}
\end{equation}
where $r = (x, y, z)^T$. Hence, $\scr P$ first rotates $\scr X$ by $\sdr R$, and then integrates along vertical lines to obtain the projection image. The microscope yields the image $\scr P \scr X$, discretized onto an $N \times N$ Cartesian grid, where each pixel is also corrupted by additive noise. Let there be $q \approx \frac{\pi}{4}N^2$ pixels contained in the inscribed disc of an $N \times N$ grid (the remaining pixels contain little or no signal because $\scr X$ is concentrated in $B$). If $\mathcal S: L^1(\br^2) \rightarrow \br^q$ is a discretization operator, then the microscope produces images $\sdr I$ given by 
\begin{equation}
\rdr I = \mathcal S \rcr P \scr X + \rdr E,
\end{equation}
with $\rdr E \sim \mathcal N(0, \sigma^2 \text{I}_q)$, where for the purposes of this paper we assume additive white Gaussian noise. The microscope has an additional blurring effect on the images, a phenomenon we will discuss shortly, but will leave out of our model. Given a set of images $\sdr I_1, \dots, \sdr I_n$, the cryo-EM problem is to estimate the orientations $\sdr R_1, \dots, \sdr R_n$ of the underlying volumes and reconstruct $\scr X$. Note that throughout this paper, we will use ``cryo-EM" and ``cryo-EM problem" as shorthand for the SPR problem from cryo-EM images; we also use ``volume" as a synonym for ``3D structure".

The cryo-EM problem is challenging for several reasons. Unlike most other imaging modalities of computerized tomography, the rotations $\sdr R_s$ are unknown, so we must estimate them before reconstructing $\scr X$. This challenge is one of the major hurdles to  reconstruction in cryo-EM. Since the images are not perfectly centered, they also contain in-plane translations, which must be estimated as well. The main challenge in rotation estimation is that the projection images are corrupted by extreme levels of noise. This problem arises because only low electron doses can scan the molecule without destroying it. To an extent, this problem is mitigated by the fact that cryo-EM datasets often have tens or even hundreds of thousands of images, which makes the reconstruction process more robust. Another issue with transmission electron microscopy in general is that technically, the detector only registers the magnitude of the electron wave exiting the specimen. Zernike realized in the 1940s that the phase information could also be recovered if the images were taken out of focus \cite{phase_contrast}. While enabling measurement of the full output of the microscope, this out-of-focus imaging technique produces images representing the convolution of the true image with a point spread function (PSF). The Fourier transform of the PSF is called the contrast transfer function (CTF). Thus the true images are multiplied by the CTF in the Fourier domain to produce the output images. Hence, the $\mathcal P_s$ operators in practice also include the blurring effect of a CTF. This results in a loss of information at the zero crossings of the (Fourier-domain) CTF and at high frequencies \cite{ctf}. In order to compensate for the former effect, images are taken with several different defocus values, whose corresponding CTFs have different zero crossings. 

The field of cryo-EM has recently seen a drastic improvement in detector technology. New direct electron detector cameras have been developed, which, according to a recent article in \textit{Science}, have ``unprecedented speed and sensitivity" \cite{resolution_revolution}. This technology has enabled SPR from cryo-EM to succeed on smaller molecules (up to size $\sim$150kDa) and achieve higher resolutions (up to 3\AA) than before. Such high resolution allows tracing of the polypetide chain and identification of residues in protein molecules \cite{cheng_resolution, scheres_resolution, henderson, groel, near_atomic}. Recently, single particle methods have provided high resolution structures of the TRPV1 ion channel \cite{cryo_em_membrane} and of the large subunit of the yeast mitochondrial ribosome \cite{mitochondrial_ribosome}. While X-ray crystallography is still the imaging method of choice for small molecules, cryo-EM now holds the promise of reconstructing larger, biomedically relevant molecules not amenable to crystallization.

The most common method for solving the basic cryo-EM problem is guessing an initial structure and then performing an iterative refinement procedure, where iterations alternate between 1) estimating the rotations of the experimental images by matching them with projections of the current 3D model; and 2) tomographic inversion producing a new 3D model based on the experimental images and their estimated rotations \cite{ctf, vanheel00, gridding}. There are no convergence guarantees for this iterative scheme, and the initial guess can incur bias in the reconstruction. An alternative is to estimate the rotations and reconstruct an accurate initial structure directly from the data. Such an ab-initio structure is a much better initialization for the iterative refinement procedure. This strategy helps avoid bias and reduce the number of refinement iterations necessary to converge \cite{zhao_singer}. In the ab-initio framework, rotations can be estimated by one of several techniques (see e.g. \cite{SS_11, lanhui_orientations} and references therein).  

\subsection{Heterogeneity problem} \label{section:het_problem}

As presented above, a key assumption in the cryo-EM problem is that the sample consists of (rotated versions of) identical molecules. However, in many datasets this assumption does not hold. Some molecules of interest exist in more than one conformational state. For example, a subunit of the molecule might be present or absent, have a few different arrangements, or be able to move continuously from one position to another. These structural variations are of great interest to biologists, as they provide insight into the functioning of the molecule. Unfortunately, standard cryo-EM methods do not account for heterogeneous samples. New techniques must be developed to map the space of molecules in the sample, rather than just reconstruct a single volume. This task is called the \textit{heterogeneity problem}. A common case of heterogeneity is when the molecule has a finite number of dominant conformational classes. In this discrete case, the goal is to provide biologists with 3D reconstructions of all these structural states. While cases of continuous heterogeneity are possible, in this paper we mainly focus on the discrete heterogeneity scenario.

\begin{figure}[h]
\begin{center}
\includegraphics[scale = 0.8]{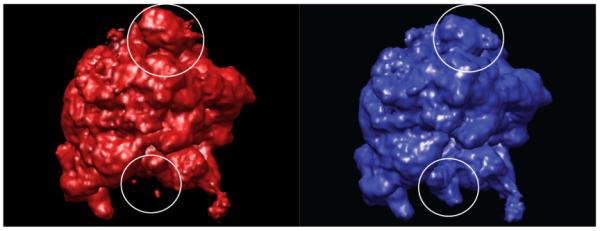}
\caption{Classical (left) and hybrid (right) states of 70S \textit{E. Coli} ribosome (image source: \cite{local})} \label{het}
\end{center}
\end{figure}

While we do not investigate the 3D rotation estimation problem in the heterogeneous case, we conjecture that this problem can be solved without developing sophisticated new tools. Consider for example the case when the heterogeneity is small; i.e., the volumes $\scr X_1,\ldots,\scr X_n$ can be rotationally aligned so they are all close to their mean (in some norm). For example, this property holds when the heterogeneity is localized (e.g., as in Figure \ref{het}). In this case, one might expect that by first assuming homogeneity, existing rotation estimation methods would yield accurate results. Even if the heterogeneity is large, an iterative scheme can be devised to alternately estimate the rotations and conformations until convergence (though this convergence is local, at best). Thus, in this publication, we assume that the 3D rotations $\sdr R_s$ (and in-plane translations) have already been estimated. 

With the discrete heterogeneity and known rotations assumptions, we can formulate the heterogeneity problem as follows. 

\begin{problem}{(Heterogeneity Problem).} \label{het_problem}
Suppose a heterogeneous molecule can take on one of $C$ different states: $\scr X^1, \dots, \scr X^C \in \mathscr B$, where $\mathscr B$ is a finite-dimensional space of bandlimited functions (see Section \ref{discrete_estimation}). Let $\Omega = \{1, 2, \dots, C\}$ be a sample space, and $p_1, \dots, p_C$ probabilities (summing to one) so that the molecule assumes state $c$ with probability $p_c$. Represent the molecule as a random field $\rcr X: \Omega \times \br^3 \rightarrow \br$, with 
\begin{equation}
\mathbb P[\rcr X= \scr X^c] = p_c, \quad c = 1, \dots, C.
\end{equation}
Let $\rdr R$ be a random rotation with some distribution over $SO(3)$, and define the corresponding random projection $\rcr P = \scr P(\rdr R)$ (see (\ref{P_s})). Finally, $\rdr E \sim \mathcal N(0, \sigma^2 \textnormal{I}_q)$. Assume that $\rcr X, \rdr R, \rdr E$ are independent. A random image of a particle is obtained via
\begin{equation}
\rdr I= \mathcal S \rcr P \rcr X + \rdr E,
\end{equation}
where $\mathcal S: L^1(\br^2) \rightarrow \br^q$ is a discretization operator. Given observations $I_1, \dots, I_n$ and $R_1, \dots, R_n$ of $\rdr I$ and $\rdr R$, respectively, estimate the number of classes $C$, the structures $\scr X^c$, and the probabilities $p_c$.
\end{problem}

Note that $\mathcal S \rcr P|_{\mathscr B}$ is a (random) linear operator between finite-dimensional spaces, and so it has a matrix version $\rdr P: \br^p \rightarrow \br^q$, where $p = \text{dim} \ \mathscr B$. If we let $\rdr X$ be the random vector on $\br^p$ obtained by expanding $\rcr X$ in the basis for $\mathscr B$, then we recover the equation $\rdr I = \rdr P \rdr X + \rdr E$ from Problem \ref{main_problem}. Thus, the main factors distinguishing Problem \ref{het_problem} from Problem \ref{main_problem} are that the former assumes a specific form for $\rdr P$ and posits a discrete distribution on $\rdr X$. As we discuss in Section \ref{finding_conformations}, Problem \ref{het_problem} can be solved by first estimating the covariance matrix as in Problem \ref{main_problem}, finding coordinates for each image with respect to the top eigenvectors of this matrix, and then applying a standard clustering procedure to these coordinates.

One of the main difficulties of the heterogeneity problem is that, compared to usual SPR, we must deal with an even lower effective signal-to-noise ratio (SNR). Indeed, the signal we seek to reconstruct is the variation of the molecules around their mean, as opposed to the mean volume itself. We propose a precise definition of SNR in the context of the heterogeneity problem in Section \ref{snr_fsc}. Another difficulty is the indirect nature of our problem. Although the heterogeneity problem is an instance of a clustering problem, it differs from usual such problems in that we do not have access to the objects we are trying to cluster -- only projections of these objects onto a lower-dimensional space are available. This makes it challenging to apply any standard clustering technique directly.

The heterogeneity problem is considered one of the most important problems in cryo-EM. In his 2013 Solvay public lecture on cryo-EM, Dr. Joachim Frank emphasized the importance of ``the ability to obtain an entire inventory of co-existing states of a macromolecule from a single sample" \cite{solvay}. Speaking of approaches to the heterogeneity problem in a review article, Frank discussed ``the potential these new technologies will have in exploring functionally relevant states of molecular machines" \cite{story_in_a_sample}. It is stressed there that much room for improvement remains; current methods cannot automatically identify the number of conformational states and have trouble distinguishing between similar conformations.

\subsection{Previous work}

Much work related to Problem \ref{main_problem} and Problem \ref{het_problem} has already been done. There is a rich statistical literature on the covariance estimation problem in the presence of missing data, a special case of Problem \ref{main_problem}. In addition, work on the \textit{low rank matrix sensing problem} (a generalization of matrix completion) is also closely related to Problem \ref{main_problem}. Regarding Problem \ref{het_problem}, several approaches to the heterogeneity problem have been proposed in the cryo-EM literature. 

\subsubsection{Work related to Problem \ref{main_problem}}

Many approaches to covariance matrix estimation from missing data have been proposed in the statistics literature \cite{little_rubin}. The simplest approach to dealing with missing data is to ignore the samples with any unobserved variables. Another simple approach is called \textit{available case analysis}, in which the statistics are constructed using all the available values. For example, the $(i, j)$ entry of the covariance matrix is constructed using all samples for which the $i$'th and $j$'th coordinates are simultaneously observed. These techniques work best under certain assumptions on the pattern of missing entries, and more sophisticated techniques are preferred \cite{little_rubin}. One of the most established such approaches is \textit{maximum likelihood estimation} (MLE). This involves positing a probability distribution on $\rdr X$ (e.g., multivariate normal) and then maximizing the likelihood of the observed partial data with respect to the parameters of the model. Such an approach to fitting models from partial observations was known as early as the 1930s, when Wilks used it for the case of a bivariate normal distribution \cite{wilks}. Wilks proposed to maximize the likelihood using a gradient-based optimization approach. In 1977, Dempster, Laird, and Rubin introduced the expectation-maximization (EM) algorithm \cite{EM} to solve maximum likelihood problems. The EM algorithm is one of the most popular methods for solving missing data problems in statistics. Also, there is a class of approaches to missing data problems called \textit{imputation}, in which the missing values are filled either by averaging the available values or through more sophisticated regression-based techniques. Finally, see \cite{high_dim_regression, lounici_2012} for other approaches to related problems.

Closely related to covariance estimation from missing data is the problem of PCA with missing data. In this problem, the task is to find the leading principal components, and not necessarily the entire covariance matrix. Not surprisingly, EM-type algorithms are popular for this problem as well. These algorithms often search directly for the low-rank factors. See \cite{pca_missing_data_2} for a survey of approaches to PCA with missing data. Closely related to PCA with missing data is the low rank matrix completion problem. Many of the statistical methods discussed above are also applicable to matrix completion. In particular, EM algorithms to solve this problem are popular, e.g., \cite{climate_data, gaussian_mc}. 

Another more general problem setup related to Problem \ref{main_problem} is the \textit{low rank matrix sensing problem}, which generalizes the low rank matrix completion problem. Let $A \in \br^{p\times n}$ be an unknown rank-$k$ matrix, and let $\mathcal M: \br^{p \times n} \rightarrow \br^d$ be a linear map, called the \textit{sensing matrix}. We would like to find $A$, but we only have access to the (possibly noisy) data $\mathcal M(A)$. Hence, the low rank matrix sensing problem can be formulated as follows \cite{lrms_alternating}:
\begin{equation}
\text{minimize} \ \norm{\mathcal M(A) - b}, \quad s.t. \quad rank(A) \leq k.
\end{equation}
Note that when $\Sigma_0$ is low rank, Problem \ref{main_problem} is a special case of the low rank matrix sensing problem. Indeed, consider putting the unknown vectors $\sdr X_1, \dots, \sdr X_n$ together as the columns of a matrix $A$. The rank of this matrix is the number of degrees of freedom in $\rdr X$ (in the cryo-EM problem, this relates to the number of heterogeneity classes of the molecule). The linear projections $\sdr P_1, \dots, \sdr P_n$ can be combined into one sensing matrix $\mathcal M$ acting on $A$. In this way, our problem falls into the realm of matrix sensing. 

One of the first algorithms for matrix sensing was inspired by the compressed sensing theory \cite{lrms_nuclear_norm}. This approach uses a matrix version of $\ell_1$ regularization called nuclear norm regularization. The nuclear norm is the sum of the singular values of a matrix, and is a convex proxy for its rank. Another approach to this problem is alternating minimization, which decomposes $A$ into a product of the form $UV^T$ and iteratively alternates between optimizing with respect to $U$ and $V$. The first proof of convergence for this approach was given in \cite{lrms_alternating}. Both the nuclear norm and alternating minimization approaches to the low rank matrix sensing problem require a \textit{restricted isometry property} on $\mathcal M$ for theoretical guarantees.

While the aforementioned algorithms are widely used, we believe they have limitations as well. EM algorithms require postulating a distribution over the data and are susceptible to getting trapped in local optima. Regarding the former point, Problem \ref{main_problem} avoids any assumptions on the distribution of $\rdr X$, so our estimator should have the same property. Matrix sensing algorithms (especially alternating minimization) often assume that the rank is known in advance. However, there is no satisfactory statistical theory for choosing the rank. By contrast, the estimator we propose for Problem \ref{main_problem} allows automatic rank estimation. 

\subsubsection{Work related to Problem \ref{het_problem}}

Several approaches to the heterogeneity problem have been proposed. Here we give a brief overview of some of these approaches.

One approach is based on the notion of common lines. By the Fourier projection slice theorem (see Theorem \ref{fourier_projection_slice}), the Fourier transforms of any two projection images of an object will coincide on a line through the origin, called a common line. The idea of Shatsky et al \cite{shat10} was to use common lines as a measure of how likely it is that two projection images correspond to the same conformational class. Specifically, given two projection images and their corresponding rotations, we can take their Fourier transforms and correlate them on their common line. From there, a weighted graph of the images is constructed, with edges weighted based on this common line measure. Then spectral clustering is applied to this graph to classify the images. An earlier common lines approach to the heterogeneity problem is described in \cite{hk_08}.

Another approach is based on MLE. It involves positing a probability distribution over the space of underlying volumes, and then maximizing the likelihood of the images with respect to the parameters of the distribution. For example, Wang et al \cite{doer13} model the heterogeneous molecules as a mixture of Gaussians and employ the EM algorithm to find the parameters. A challenge with MLE approaches is that the resulting objective functions are nonconvex and have a complicated structure. For more discussion of the theory and practice of maximum likelihood methods, see \cite{mle1} and \cite{mle2}, respectively. Also see \cite{relion} for a description of a software package which uses maximum likelihood to solve the heterogeneity problem.

A third approach to the heterogeneity problem is to use the covariance matrix of the set of original molecules. Penczek outlines a bootstrapping approach in \cite{penczek} (see also \cite{penczek2002, penczek2006, penczek2008, local}). In this approach, one repeatedly takes random subsets of the projection images and reconstructs 3D volumes from these samples. Then, one can perform principal component analysis on this set of reconstructed volumes, which yields a few dominant ``eigenvolumes''. Penczek proposes to then produce mean-subtracted images by subtracting projections of the mean volume from the images. The next step is to project each of the dominant eigenvolumes in the directions of the images, and then obtain a set of coordinates for each image based on its similarity with each of the eigenvolume projections. Finally, using these coordinates, this resampling approach proceeds by applying a standard clustering algorithm such as $K$-means to classify the images into classes.

While existing methods for the heterogeneity problem have their success stories, each suffers from its own shortcomings: the common line approach does not exploit all the available information in the images, the maximum likelihood approach requires explicit a-priori distributions and is susceptible to local optima, and the bootstrapping approach based on covariance matrix estimation is a heuristic sampling method that lacks in theoretical guarantees. 

Note that the above overview of the literature on the heterogeneity problem is not comprehensive. For example, very recently, an approach to the heterogeneity problem based on normal mode analysis was proposed \cite{hemnma}. 

\subsection{Our contribution}

In this paper, we propose and analyze a covariance matrix estimator $\Sigma_n$ to solve the general statistical problem (Problem \ref{main_problem}), and then apply this estimator to the heterogeneity problem (Problem \ref{het_problem}).

Our covariance matrix estimator has several desirable properties. First, we prove that the estimator is consistent as $n\rightarrow \infty$ for fixed $p, \ q$. Second, our estimator does not require a prior distribution on the data, unlike MLE methods. Third, when the data have low intrinsic dimension, our method does not require knowing the rank of $\Sigma_0$ in advance. The rank can be estimated from the spectrum of the estimated covariance matrix. This sets our method apart from alternating minimization algorithms that search for the low rank matrix factors themselves. Fourth, our estimator is given in closed-form and its computation requires only a single linear inversion.

To implement our covariance matrix estimator in the cryo-EM case, we must invert a high-dimensional matrix $L_n$ (see definition (\ref{L_def})). The size of this matrix is so large that typically it cannot even be stored on a computer; thus, inverting $L_n$ is the greatest practical challenge we face. We consider two possibilities of addressing this challenge. In the primary approach we consider, we replace $L_n$ by its limiting operator $L$, which does not depend on the rotations $R_s$ and is a good approximation of $L_n$ as long as these rotations are distributed uniformly enough. We then carefully construct new bases for images and volumes to make $L$ a sparse, block diagonal matrix. While $L$ has dimensions on the order of $N_{\text{res}}^6 \times N_{\text{res}}^6$, this matrix has only $O(N_{\text{res}}^9)$ total nonzero entries in the bases we construct, where $N_{\text{res}}$ is the grid size corresponding to the target resolution. These innovations lead to a practical algorithm to estimate the covariance matrix in the heterogeneity problem. The second approach we consider is an iterative inversion of $L_n$, which has a low storage requirement and avoids the requirement of uniformly distributed rotations. We compare the complexities of these two methods, and find that each has its strengths and weaknesses. 

The limiting operator $L$ is a fundamental object in tomographic problems involving variability, and we call it the \textit{projection covariance transform}. The projection covariance transform relates the covariance matrix of the imaged object to data that can be acquired from the projection images. Standard weighted back-projection tomographic reconstruction algorithms involve application of the ramp filter to the data \cite{natterer}, and we find that the inversion of $L$ entails applying a similar filter, which we call the \textit{triangular area filter}. The triangular area filter has many of the same properties as the ramp filter, but reflects the slightly more intricate geometry of the covariance estimation problem. The projection covariance transform is an interesting mathematical object in its own right, and we begin studying it in this paper. 

Finally, we numerically validate the proposed algorithm (the first algorithm discussed above). We demonstrate this method's robustness to noise on synthetic datasets by obtaining a meaningful reconstruction of the covariance matrix and molecular volumes even at low SNR levels. Excluding precomputations (which can be done once and for all), reconstructions for 10000 projection images of size $65 \times 65$ pixels takes fewer than five minutes on a standard laptop computer.

The paper is organized as follows. In Section \ref{estimator}, we construct an estimator for Problem \ref{main_problem}, state theoretical results about this estimator, and connect our problem to high-dimensional PCA. In Section \ref{covariance_cryo_em}, we specialize the covariance estimator to the heterogeneity problem and investigate its geometry. In Section \ref{finding_conformations}, we discuss how to reconstruct the conformations once we have estimated the mean and covariance matrix. In Section \ref{practical}, we discuss computational aspects of the problem and construct a basis in which $L$ is block diagonal and sparse. In Section \ref{consequences}, we explore the complexity of the proposed approach. In Section \ref{num_results}, we present numerical results for the heterogeneity problem. We conclude with a discussion of future research directions in Section \ref{future}. Appendices \ref{appendix_mat_der}, \ref{analysis_proofs}, and \ref{simplifying_integral} contain calculations and proofs. 

\section{An estimator for Problem \ref{main_problem}} \label{estimator}

\subsection{Constructing an estimator} \label{sample_cov}

We define estimators $\mu_n$ and $\Sigma_n$ through a general optimization framework based on the model (\ref{stat_model}). As a first step, let us calculate the first- and second-order statistics of $\rdr I$, conditioned on the observed matrix $\sdr P_s$ for each $s$. Using the assumptions in Problem \ref{main_problem}, we find that
\begin{equation}
\mathbb E[\rdr I| \rdr P = \sdr P_s] = \mathbb E[\rdr P \rdr X + \rdr E | \rdr P = \sdr P_s] = \mathbb E[\rdr P |\rdr P = \sdr P_s] \mathbb E[\rdr X] = P_s \mu_0.
\label{mean}
\end{equation}
and
\begin{equation}
\Var[\rdr I | \rdr P = \sdr P_s] = \Var[\rdr P \rdr X | \rdr P = P_s] + \Var[\rdr E] = P_s \Sigma_0 P_s^H + \sigma^2 \text{I}_q.
\label{cov}
\end{equation}
Note that $P_s^H$ denotes the conjugate transpose of $P_s$.

Based on (\ref{mean}) and (\ref{cov}), we devise least-squares optimization problems for $\mu_n$ and $\Sigma_n$:
\begin{equation}
\mu_n = \underset{\mu}{\operatorname{argmin}} \frac1n\sum_{s = 1}^n \norm{I_s - P_s \mu}^2;
\label{muopt}
\end{equation}
\begin{equation}
\Sigma_n= \underset{\Sigma}{\operatorname{argmin}} \frac1n\sum_{s = 1}^n \norm{(I_s - P_s \mu_n)(I_s - P_s\mu_n)^H - (P_s \Sigma P_s^H + \sigma^2 \text{I}_p)}_F^2.
\label{sigopt}
\end{equation}
Here we use the Frobenius norm, which is defined by $\norm{A}_F^2 = \sum_{i, j}|A_{ij}|^2$. 

Note that these optimization problems do not encode any prior knowledge about $\mu_0$ or $\Sigma_0$. Since $\Sigma_0$ is a covariance matrix, it must be positive semidefinite (PSD). As discussed above, in many applications $\Sigma_0$ is also low rank. The estimator $\Sigma_n$ need not satisfy either of these properties. Thus, regularization of (\ref{sigopt}) is an option worth exploring. Nevertheless, here we only consider the unregularized estimator $\Sigma_n$. Note that in most practical problems, we only are interested in the leading eigenvectors of $\Sigma_n$, and if these are estimated accurately, then it does not matter if $\Sigma_n$ is not PSD or low rank. Our numerical experiments show that in practice, the top eigenvectors of $\Sigma_n$ are indeed good estimates of the true principal components for high enough SNR.


Note that we first solve (\ref{muopt}) for $\mu_n$, and then use this result in (\ref{sigopt}). This makes these optimization problems quadratic in the elements of $\mu$ and $\Sigma$, and hence they can be solved by setting the derivatives with respect to $\mu$ and $\Sigma$ to zero. This leads to the following equations for $\mu_n$ and $\Sigma_n$ (see Appendix \ref{appendix_mat_der} for the derivative calculations):
\begin{equation}
\frac 1n\left(\sum_{s = 1}^n P_s^H P_s\right)\mu_n = \frac1n\sum_{s = 1}^n P_s^H I_s =: b_n;
\label{mueq}
\end{equation}
\begin{equation}
\frac1n\sum_{s = 1}^n P_s^H P_s \Sigma_n P_s^H P_s = \frac1n\sum_{s = 1}^n P_s^H(I_s - P_s\mu_n)(I_s - P_s\mu_n)^HP_s - \sigma^2 \frac1n\sum_{s = 1}^n P_s^H P_s =: B_n.
\label{sigeq}
\end{equation}
When $p = q$ and $\rdr P = \text{I}_p$, $\mu_n$ and $\Sigma_n$ reduce to the sample mean and sample covariance matrix. When $\rdr P$ is a coordinate-selection operator (recall the discussion following the statement of Problem \ref{main_problem}), (\ref{mueq}) estimates the mean by averaging all the available observations for each coordinate, and (\ref{sigeq}) estimates each entry of the covariance matrix by averaging over all samples for which both coordinates are observed. These are exactly the available-case estimators discussed in \cite[Section 3.4]{little_rubin}. 

Observe that (\ref{mueq}) requires inversion of the matrix
\begin{equation}
A_n = \frac1n\sum_{s=1}^n P_s^H P_s,
\label{A_def}
\end{equation}
and (\ref{sigeq}) requires inversion of the linear operator $L_n: \bc^{p \times p} \rightarrow \bc^{p \times p}$ defined by
\begin{equation}
L_n(\Sigma) = \frac1n\sum_{s = 1}^n P_s^H P_s \Sigma P_s^H P_s.
\label{L_def}
\end{equation}

Since $P_s$ are drawn independently from $\rdr P$, the law of large numbers implies that
\begin{equation}
A_n \rightarrow A \quad \text{and} \quad L_n \rightarrow L \quad \text{almost surely},
\label{as_convergence}
\end{equation}
where the convergence is in the operator norm, and
\begin{equation}
A = \mathbb E[\rdr P^H \rdr P] \quad \text{and} \quad L(\Sigma) = \mathbb E[\rdr P^H \rdr P \Sigma \rdr P^H \rdr P].
\label{as_convergence2}
\end{equation}
 The invertibilities of $A$ and $L$ depend on the distribution of $\rdr P$. Intuitively, if $\rdr P$ has a nonzero probability of ``selecting" any coordinate of its argument, then $A$ will be invertible. If $\rdr P$ has a nonzero probability of ``selecting" any pair of coordinates of its argument, then $L$ will be invertible. In this paper, we assume that $A$ and $L$ are invertible. In particular, we will find that in the cryo-EM case, $A$ and $L$ are invertible if, for example, the rotations are sampled uniformly from $SO(3)$. Under this assumption, we will prove that $A_n$ and $L_n$ are invertible with high probability for sufficiently large $n$. In the case when $A_n$ or $L_n$ are not invertible, we cannot define estimators from the above equations, so we simply set them to zero. Since the RHS quantities $b_n$ and $B_n$ are noisy, it is also not desirable to invert $A_n$ or $L_n$ when these matrices are nearly singular. Hence, we propose the following estimators:
\begin{equation}
\mu_n = 
\begin{cases}
A_n^{-1} b_n \quad &\text{if } \norm{A_n^{-1}} \leq 2\norm{A^{-1}} \\
0 \quad &\text{otherwise};
\end{cases} 
\quad
\Sigma_n = 
\begin{cases}
L_n^{-1}(B_n) \quad &\text{if } \norm{L_n^{-1}} \leq 2\norm{L^{-1}} \\
0 \quad &\text{otherwise}.
\end{cases} 
\end{equation}
The factors of 2 are somewhat arbitrary; any $\alpha > 1$ would do.

Let us make a few observations about $A_n$ and $L_n$. By inspection, $A_n$ is symmetric and positive semidefinite. We claim that $L_n$ satisfies the same properties, with respect to the Hilbert space $\bc^{p \times p}$ equipped with the inner product $\ip{A}{B} = \text{tr}(B^H A)$. Using the property $\text{tr}(AB) = \text{tr}(BA)$, we find that for any $\Sigma_1, \Sigma_2$,
\begin{equation}
\begin{split}
\ip{L_n(\Sigma_1)}{\Sigma_2} &= \text{tr}(\Sigma_2^H L_n(\Sigma_1)) = \text{tr}\left[\frac1n\sum_s\Sigma_2 ^HP_s^H P_s\Sigma_1P_s^H P_s\right] \\
&= \text{tr}\left[\frac1n\sum_sP_s^H P_s\Sigma_2^HP_s^H P_s\Sigma_1\right] = \ip{\Sigma_1}{L(\Sigma_2)}.
\end{split}
\end{equation}
Thus, $L_n$ is self-adjoint. Next, we claim that $L_n$ is positive semidefinite. Indeed,
\begin{equation}
\begin{split}
\ip{L_n(\Sigma)}{\Sigma} &= \text{tr}(\Sigma^H L_n(\Sigma)) = \text{tr}\left[\frac1n\sum_s\Sigma^H P_s^H P_s\Sigma P_s^H P_s\right] \\
&= \frac1n\sum_s \text{tr}[(P_s\Sigma P_s^H)^H(P_s\Sigma P_s^H)] = \sum_s\frac1n\norm{P_s\Sigma P_s^H}_F^2 \geq 0.
\end{split}
\end{equation}

\subsection{Consistency of $\mu_n$ and $\Sigma_n$} \label{theoretical}

In this section, we state that under mild conditions on $\rdr P, \rdr X, \rdr E$, the estimators $\mu_n$ and $\Sigma_n$ are consistent. Note that here, and throughout this paper, $\norm{\cdot}$ will denote the Euclidean norm for vectors and the operator norm for matrices. Also, define
\begin{equation}
\moment{\rdr Y}_j = \mathbb E[\norm{\rdr Y - \mathbb E[\rdr Y]}^j]^{1/j},
\end{equation}
where $\rdr Y$ is a random vector. 

\begin{proposition}\label{mu_main}
Suppose $A$ (defined in (\ref{as_convergence2})) is invertible, that $\norm{\rdr P}$ is bounded almost surely, and that $\moment{\rdr X}_2, \moment{\rdr E}_2 < \infty$. Then, for fixed $p, q$ we have 
\begin{equation}
\mathbb E\norm{\rdr \mu_n - \mu_0}= O\left(\frac 1 {\sqrt n}\right).
\end{equation}
Hence, under these assumptions, $\mu_n$ is consistent.
\end{proposition}

\begin{proposition} \label{Sigma_main}
Suppose $A$ and $L$ (defined in (\ref{as_convergence2})) are invertible, that $\norm{\rdr P}$ is bounded almost surely, and that there is a polynomial $Q$ for which
\begin{equation}
\moment{\rdr X}_j, \moment{\rdr E}_j \leq Q(j), \quad j \in \mathbb N.
\label{moment_condition}
\end{equation}
Then, for fixed $p, q$, we have
\begin{equation}
\mathbb E\norm{\rdr \Sigma_n  - \Sigma_0} = O\left(\frac{Q(\log n)}{\sqrt n}\right).
\end{equation}
Hence, under these assumptions, $\Sigma_n$ is consistent.
\end{proposition}
\begin{remark}
The moment growth condition (\ref{moment_condition}) on $\rdr X$ and $\rdr E$ is not very restrictive. For example, bounded, subgaussian, and subexponential random vectors all satisfy (\ref{moment_condition}) with $\deg Q \leq 1$ (see \cite[Sections 5.2 and 5.3]{vershynin}).
\end{remark}

See Appendix \ref{analysis_proofs} for the proofs of Propositions (\ref{mu_main}) and (\ref{Sigma_main}). We mentioned that $\mu_n$ and $\Sigma_n$ are generalizations of available-case estimators. Such estimators are known to be consistent when the data are \textit{missing completely at random} (MCAR). This means that the pattern of missingness is independent of the (observed and unobserved) data. Accordingly, in Problem \ref{main_problem}, we assume that $\rdr P$ and $\rdr X$ are independent, a generalization of the MCAR condition. The above propositions state that the consistency of $\mu_n$ and $\Sigma_n$ also generalizes to Problem \ref{main_problem}.

\subsection{Connection to high-dimensional PCA} \label{rmt}

While the previous section focused on the ``fixed $p$, large $n$" regime, in practice both $p$ and $n$ are large. Now, we consider the latter regime, which is common in modern high-dimensional statistics. In this regime, we consider the properties of the estimator $\Sigma_n$ when $\Sigma_0$ is low rank, and the task is to find its leading eigenvectors. What is the relationship between the spectra of $\Sigma_n$ and $\Sigma_0$? Can the rank of $\Sigma_0$ be deduced from that of $\Sigma_n$? To what extent do the leading eigenvectors of $\Sigma_n$ approximate those of $\Sigma_0$? In the setting of (\ref{stat_model}) when $\rdr P = \text{I}_p$, the theory of high-dimensional PCA provides insight into such properties of the sample covariance matrix (and thus of $\Sigma_n$). In particular, an existing result gives the correlation between the top eigenvectors of $\Sigma_n$ and $\Sigma_0$ for given settings of SNR and $p/n$. It follows from this result that if the SNR is sufficiently high compared to $\sqrt{p/n}$, then the top eigenvector of $\Sigma_n$ is a useful approximation of the top eigenvector of $\Sigma_0$. If generalized to the case of nontrivial $\rdr P$, this result would be a useful guide for using the estimator $\Sigma_n$ to solve practical problems, such as Problem \ref{het_problem}. In this section, we first discuss the existing high-dimensional PCA literature, and then raise some open questions about how these results generalize to the case of nontrivial $\rdr P$.

Given i.i.d. samples $I_1, \dots, I_n \in \br^p$ from a centered distribution $\rdr I$ with covariance matrix $\tilde \Sigma_0$ (called the population covariance matrix), the sample covariance matrix $\tilde \Sigma_n$ is defined by
\begin{equation}
\tilde \Sigma_n =  \frac{1}{n}\sum_{s = 1}^n I_s I_s^H.
\label{sample_covariance}
\end{equation}
We use the new tilde notation because in the context of Problem \ref{main_problem}, $\tilde \Sigma_0$ is the signal-plus-noise covariance matrix, as opposed to the covariance of the signal itself. High-dimensional PCA is the study of the spectrum of $\tilde \Sigma_n$ for various distributions of $\rdr I$ in the regime where $n, p \rightarrow \infty$ with $p/n \rightarrow \gamma$.

The first case to consider is $\rdr X = 0$, i.e., $\rdr I = \rdr E$, where $\rdr E \sim \mathcal N(0, \sigma^2 \text{I}_p)$. In a landmark paper, Mar\u{c}enko and Pastur \cite{MP} proved that the spectrum of $\tilde \Sigma_n$ converges to \textit{the Mar\u{c}enko-Pastur (MP) distribution}, which is parameterized by $\gamma$ and $\sigma^2$:
\begin{equation}
MP(x) = \frac{1}{2\pi \sigma^2}\frac{\sqrt{(\gamma_+ - x)(x - \gamma_-)}}{\gamma x} 1_{[\gamma_-, \gamma_+]}, \quad \gamma_{\pm} = \sigma^2(1 \pm \sqrt{\gamma})^2.
\label{MP_law}
\end{equation}
The above formula assumes $\gamma \leq 1$; a similar formula governs the case $\gamma > 1$. Note that there are much more general statements about classes of $\rdr I$ for which this convergence holds; see e.g., \cite{silverstein_bai}. See Figure \ref{fig:MP} for MP distributions with a few different parameter settings.

\begin{figure}
\centering
        \begin{subfigure}[b]{0.45\textwidth}
                \centering
	\includegraphics[scale = 0.1]{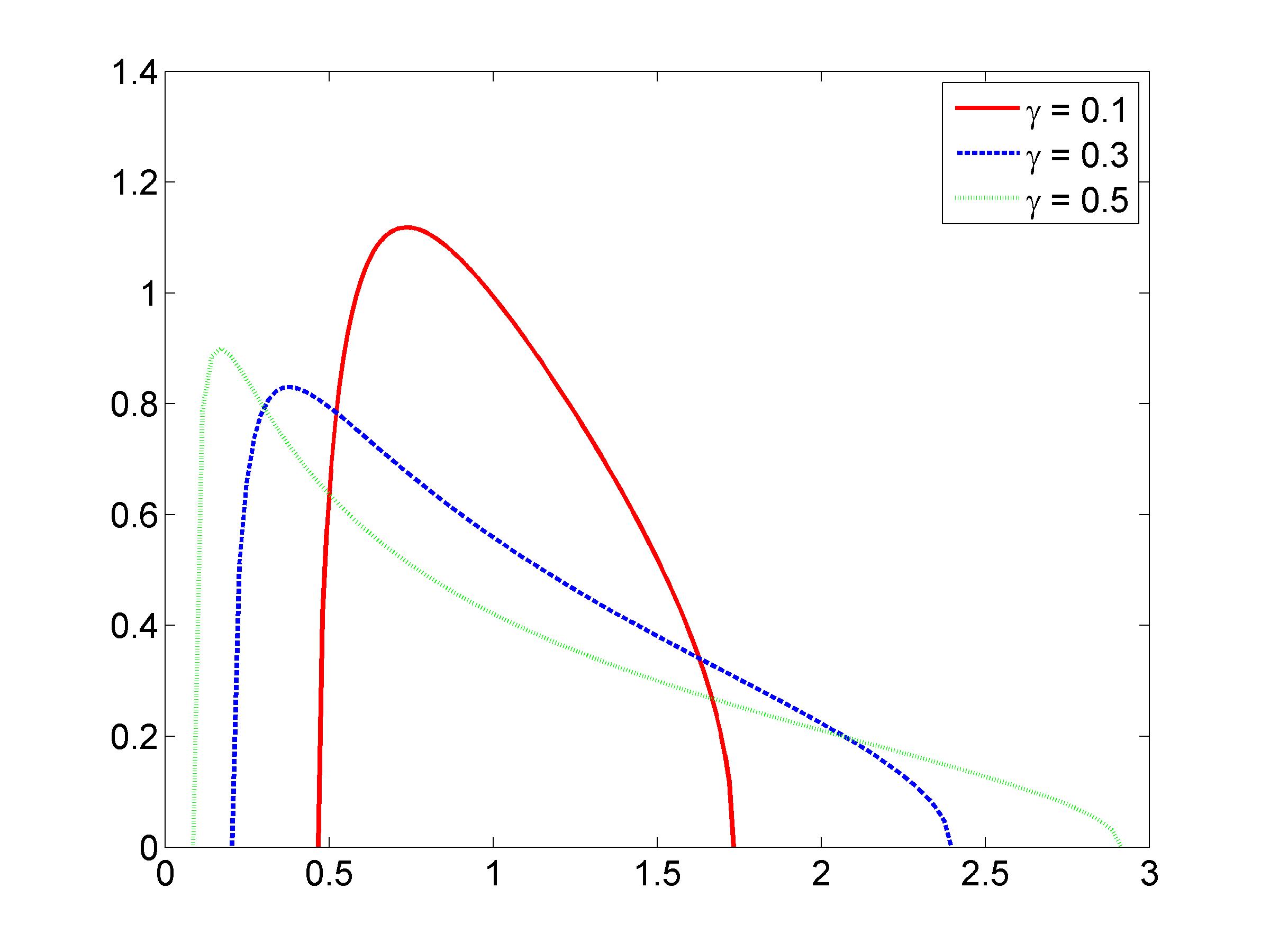}
	\caption{MP distributions (\ref{MP_law}) for $\sigma^2 = 1$}
\label{fig:MP}
        \end{subfigure}
 \begin{subfigure}[b]{0.45\textwidth}
                \centering
\includegraphics[scale = 0.1]{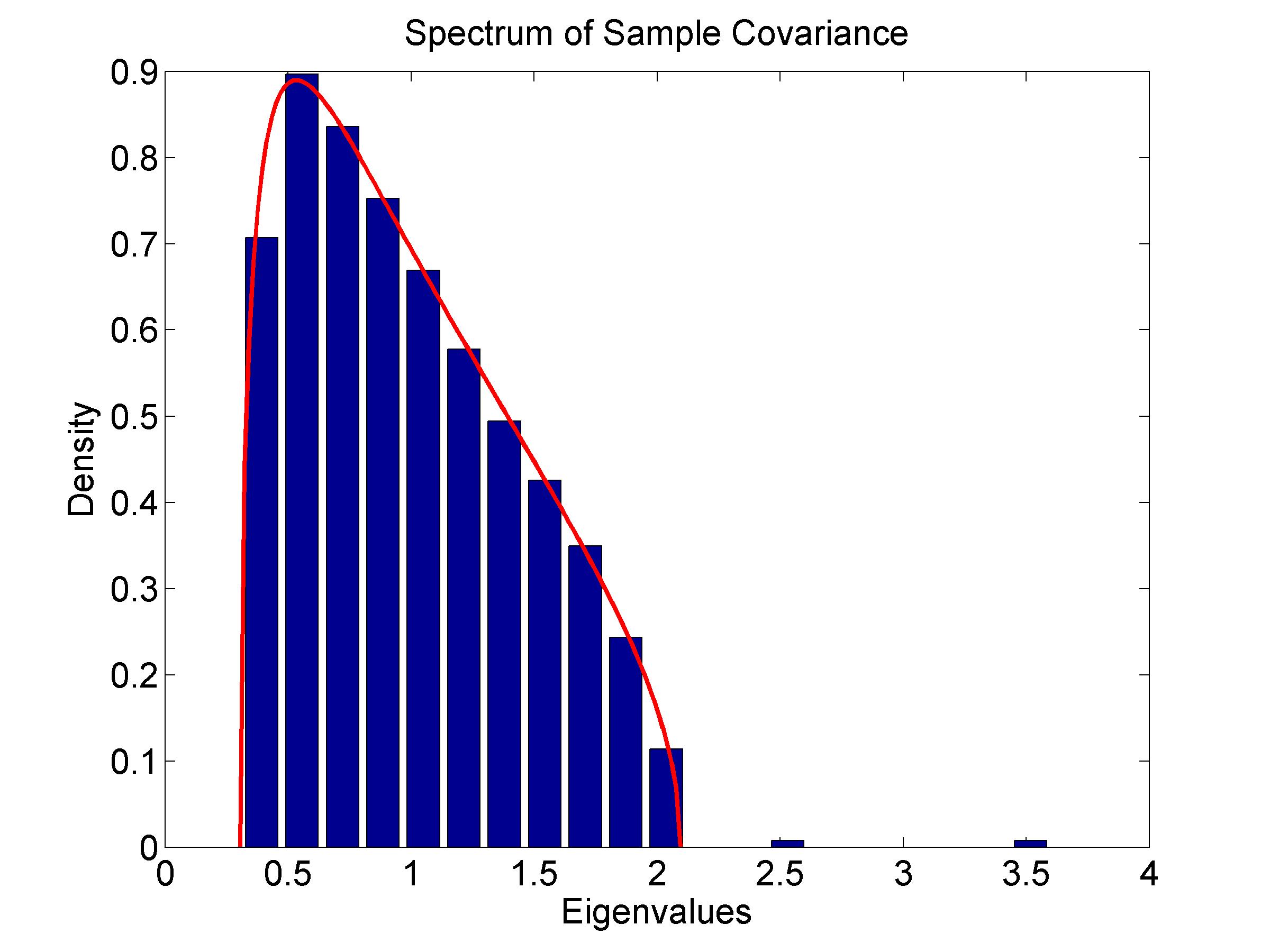}
	\caption{Empirical spectrum for spiked model}
\label{fig:spike}
        \end{subfigure}
\caption{Illustrations of high-dimensional PCA}
\end{figure}

Johnstone \cite{high_dim_pca} took this analysis a step further and considered the limiting distribution of the largest eigenvalue of $\tilde \Sigma_n$. He showed that the distribution of this eigenvalue converges to the Tracy-Widom distribution centered on the right edge of the MP spectrum. In the same paper, Johnstone considered the \textit{spiked covariance model}, in which 
\begin{equation}
\rdr I = \rdr X + \rdr E,
\label{signal_plus_noise}
\end{equation}
where $\rdr E$ is as before and $\Sigma_0 = \text{Var}[\rdr X] = \text{diag}(\tau_1^2, \dots, \tau_r^2, 0, \dots, 0)$, so that the population covariance matrix is $\tilde \Sigma_0 = \text{diag}(\tau_1^2 + \sigma^2, \dots, \tau_r^2 + \sigma^2, \sigma^2, \dots, \sigma^2)$. Here, $\rdr X$ is the signal and $\rdr E$ is the noise. In this view, the goal is to accurately recover the top $r$ eigenvectors, as these will determine the subspace on which $\rdr X$ is supported. The question then is the following: for what values of $\tau_1, \dots, \tau_r$ will the top $r$ eigenvectors of the sample covariance matrix be good approximations to the top eigenvectors of the population covariance? Since we might not know the value of $r$ a-priori, it is important to first determine for what values of $\tau_1, \dots, \tau_r$ we can detect the presence of ``spiked" population eigenvalues. In \cite{baik_spike}, the spectrum of the sample covariance matrix in the spiked model was investigated. It was found that the bulk of the distribution still obeys the MP law, whereas for each $k$ such that
\begin{equation}
\frac{\tau_k^2}{\sigma^2} \geq \sqrt{\gamma},
\label{threshold}
\end{equation}
 the sample covariance matrix will have an eigenvalue tending to $(\tau_k^2 + \sigma^2)(1 + \frac{\sigma^2}{\tau_k^2}\gamma)$. The signal eigenvalues below this threshold tend to the right edge of the noise distribution. Thus, (\ref{threshold}) defines a criterion for detection of signal. In Figure \ref{fig:spike}, we illustrate these results with a numerical example. We choose $p = 800, \ n = 4000$, and a spectrum corresponding to $r = 3$, with $\tau_1, \tau_2$ above, but $\tau_3$ below, the threshold corresponding to $\gamma = p/n = 0.2$. Figure \ref{fig:spike} is a normalized histogram of the eigenvalues of the sample covariance matrix. The predicted MP distribution for the bulk is superimposed. We see that indeed we have two eigenvalues separated from this bulk. Moreover, the eigenvalue of $\tilde \Sigma_n$ corresponding to $\tau_3$ does not pop out of the noise distribution.

It is also important to compare the top eigenvectors of the sample and population covariance matrices. Considering the simpler case of a spiked model with $r = 1$, \cite{bbp, finite_sample_pca} showed a ``phase transition" effect: as long as $\tau_1$ is above the threshold in (\ref{threshold}), the correlation of the top eigenvector ($V_{\text{PCA}})$ with the true principal component ($V$) tends to a limit between 0 and 1:
\begin{equation}
|\ip{V_{\text{PCA}}}{V}|^2 \rightarrow \frac{\frac{1}{\gamma}\frac{\tau_1^4}{\sigma^4} - 1}{\frac{1}{\gamma}\frac{\tau_1^4}{\sigma^4}  + \frac{\tau_1^2}{\sigma^2}}.
\label{correlation}
\end{equation}
Otherwise, the limiting correlation is zero. Thus, high-dimensional PCA is inconsistent. However, if $\tau_1^2/\sigma^2$ is sufficiently high compared to $\sqrt{\gamma}$, then the top eigenvector of the sample covariance matrix is still a useful approximation.

While all the statements made so far have concerned the limiting case $n, p \rightarrow \infty$, similar (but slightly more complicated) statements hold for finite $n, p$ as well (see, e.g., \cite{finite_sample_pca}). Thus, (\ref{threshold}) has a practical interpretation. Again considering the case $r = 1$, note that the quantity $\tau_1^2/\sigma^2$ is the SNR. When faced with a problem of the form (\ref{signal_plus_noise}) with a given $p$ and SNR, one can determine how many samples one needs in order to detect the signal. If $V$ represents a spatial object as in the cryo-EM case, then $p$ can reflect the resolution to which we reconstruct $V$. Hence, if we have a dataset with a certain number of images $n$ and a certain estimated SNR, then (\ref{threshold}) determines the resolution to which $V$ can be reconstructed from the data. 

This information is important to practitioners (e.g., in cryo-EM), but as of now, the above theoretical results only apply to the case when $\rdr P$ is trivial. Of course, moving to the case of more general $\rdr P$ brings additional theoretical challenges. For example, with nontrivial $\rdr P$, the empirical covariance matrix of $\rdr X$ is harder to disentangle from that of $\rdr I$, because the operator $L_n$ becomes nontrivial (see (\ref{sigeq}) and (\ref{L_def})). How can our knowledge about the spiked model be generalized to the setting of Problem \ref{main_problem}? We raise some open questions along these lines. 

\begin{enumerate}
\item In what high-dimensional parameter regimes (in terms of $n, p, q$) is there hope to detect and recover any signal from $\Sigma_n$? With the addition of the parameter $q$, the traditional regime $p \approx n$ might no longer be appropriate. For example, in the random coordinate selection case with the (extreme) parameter setting $q = 2$, it is expected that $n = p^2 \log p$ samples are needed just for $L_n$ to be invertible (by the coupon collector problem).
\item In the case when there is no signal ($\rdr X$ = 0), we have $\rdr I = \rdr E$. In this case, what is the limiting eigenvalue distribution of $\Sigma_n$ (in an appropriate parameter regime)? Is it still the MP law? How does the eigenvalue distribution depend on the distribution of $\rdr P$? This is perhaps the first step towards studying the signal-plus-noise model.
\item In the no-signal case, what is the limiting distribution of the largest eigenvalue of $\Sigma_n$? Is it still Tracy-Widom? How does this depend on $n, \ p,\ q$, and $\rdr P$? Knowing this distribution can provide p-values for signal detection, as is the case for the usual spiked model (see \cite[p. 303]{high_dim_pca}).
\item In the full model (\ref{stat_model}), if $\rdr X$ takes values in a low-dimensional subspace of $\br^p$, is the limiting eigenvalue distribution of $\Sigma_n$ a bulk distribution with a few separated eigenvalues? If so, what is the generalization of the SNR condition (\ref{threshold}) that would guarantee separation of the top eigenvalues? What would these top eigenvalues be, in terms of the population eigenvalues? Would there still be a phase-transition phenomenon in which the top eigenvectors of $\Sigma_n$ are correlated with the principal components as long as the corresponding eigenvalues are above a threshold?
\end{enumerate}

Answering these questions theoretically would require tools from random matrix theory such as the ones used by \cite{high_dim_pca, baik_spike, finite_sample_pca}. We do not attempt to address these issues in this paper, but remark that such results would be very useful theoretical guides for practical applications of our estimator $\Sigma_n$. Our numerical results show that the spectrum of the cryo-EM estimator $\Sigma_n$ has qualitative behavior similar to that of the sample covariance matrix.

At this point, we have concluded the part of our paper focused on the general properties of the estimator $\Sigma_n$. Next, we move on to the cryo-EM heterogeneity problem.

\section{Covariance estimation in cryo-EM heterogeneity problem} \label{covariance_cryo_em}

Now that we have examined the general covariance matrix estimation problem, let us specialize to the cryo-EM case. In this case, the matrices $\sdr P$ have a specific form: they are finite-dimensional versions of $\scr P$ (defined in (\ref{P_s})). We begin by describing the Fourier-domain counterpart of $\scr P$, which will be crucial in analyzing the cryo-EM covariance estimation problem. Our Fourier transform convention is
\begin{equation}
\hat f(\xi) = \int_{\br^d} f(x)e^{-ix\cdot \xi}dx; \quad f(x) =  \frac{1}{(2\pi)^d}\int_{\br^d} \hat f(\xi)e^{ix\cdot \xi}d\xi.
\label{fourier}
\end{equation}
The following classical theorem in tomography (see e.g. \cite{natterer} for a proof) shows that the operator $\scr P$ takes on a simpler form in the Fourier domain.
\begin{theorem}{(Fourier Projection Slice Theorem).} \label{fourier_projection_slice}
Suppose $\scr Y\in L^2(\br^3) \cap L^1(\br^3)$ and $\scr J: \br^2 \rightarrow \br$. Then
\begin{equation}
\scr P \scr Y = \scf J \Longleftrightarrow \scf P \scf Y = \scf J,
\label{fourier_slice}
\end{equation}
where $\hat{\mathcal P}: C(\br^3) \rightarrow C(\br^2)$ is defined by 
\begin{equation}
(\hat{\mathcal P} \scf  Y)\colvec{2}{\hat x}{\hat y} = \scf  Y\left(\sdr R^T(\hat x, \hat y, 0)^T \right) = \scf  Y\left(\hat x\sdr R^1 + \hat y\sdr R^2 \right).
\label{pstilde_1}
\end{equation}
Here, $R^i$ is the $i$'th row of $R$.
\end{theorem}

Hence, $\scf P$ rotates a function by $\sdr R$ and then restricts it to the horizontal plane $\hat z = 0$. If we let $\xi = (\hat x, \hat y, \hat z)$, then another way of viewing $\scf P$ is that it restricts a function to the plane $\xi \cdot \sdr R^3 = 0$.

\subsection{Infinite-dimensional heterogeneity problem} \label{fully_continuous}

To build intuition for the Fourier-domain geometry of the heterogeneity problem, consider the following idealized scenario, taking place in Fourier space. Suppose detector technology improves to the point that images can be measured continuously and noiselessly and that we have access to the full joint distribution of $\rdr R$  and $\rcf I$. We would like to estimate the mean $\hat m_0: \br^3 \rightarrow \bc$ and covariance function $\scf C_0: \br^3 \times \br^3 \rightarrow \bc$ of the random field $\rcf X$, defined
\begin{equation}
\hat m_0(\xi) = \mathbb E[\rcf  X(\xi)]; \quad \scf C_0(\xi_1, \xi_2) = \mathbb E[(\rcf  X(\xi_1) - \hat m_0(\xi_1))\overline{(\rcf  X(\xi_2) - \hat m_0(\xi_2))}].
\label{mean_covariance}
\end{equation}
Heuristically, we can proceed as follows. By the Fourier projection slice theorem, every image $\scf I$ provides an observation of $\scf X(\xi)$ for $\xi \in \br^3$ belonging to a central plane perpendicular to the viewing direction corresponding to $\scf I$. By abuse of notation, let $\xi \in \scf I$ if $\scf I$ carries the value of $\scf X(\xi)$, and let $\scf I(\xi)$ denote this value. Informally, we expect that we can recover $\hat m_0$ and $\scf C_0$ via
\begin{equation}
\hat m_0(\xi) = \mathbb E[\rcf I(\xi) \ | \ \xi \in \rcf I], \quad \scf C_0(\xi_1, \xi_2) = \mathbb E[(\rcf I(\xi_1) - \hat m_0(\xi_1))\overline{(\rcf I(\xi_2) - \hat m_0(\xi_2))} \ | \ \xi_1, \xi_2 \in \rcf I].
\label{heuristic}
\end{equation}
Now, let us formalize this problem setup and intuitive formulas for $\hat m_0$ and $\scf C_0$.
\begin{problem} \label{continuous_problem}
Let $\rcf  X: \Omega \times \br^3 \rightarrow \bc$ be a random field, where $(\Omega, \mathcal F, \nu)$ is a probability space. Here $\rcf  X(\omega, \cdot)$ is a Fourier volume for each $\omega \in \Omega$. Let $\rdr R: \Omega \rightarrow SO(3)$ be a random rotation, independent of $\rcf X$, having the uniform distribution over $SO(3)$. Let $\rcf P = \scf P(\rdr R)$ be the (random) projection operator associated to $\rdr R$ via (\ref{pstilde_1}). Define the random field $\rcf I: \Omega \times \br^2 \rightarrow \bc$ by 
\begin{equation}
\rcf I = \rcf P \rcf  X.
\label{fully_continuous_model}
\end{equation}
Given the joint distribution of $\rcf I$ and $\rdr R$, find the mean $\hat m_0$ and covariance function $\scf C_0$ of $\rcf X$, defined in (\ref{mean_covariance}).
Let $\rcf X$ be regular enough that 
\begin{equation}
\hat m_0 \in C_0^\infty(\br^3), \quad \scf C_0 \in C_0^\infty(\br^3 \times \br^3).
\label{continuity_condition}
\end{equation}
\end{problem}

In this problem statement, we do not assume that $\rcf X$ has a discrete distribution. The calculations that follow hold for any $\rcf X$ satisfying (\ref{continuity_condition}).

We claim that $\hat m_0$ and $\scf C_0$ can be found by solving 
\begin{equation}
\scf A (\hat m_0):= \mathbb E[\rcf P{}^* \rcf P]\hat m_0 = \mathbb E[\rcf {P}{}^* \rcf I]
\label{A_fully_continuous}
\end{equation}
and
\begin{equation}
\scf L(\scf C_0) := \mathbb E[\rcf P{}^* \rcf P\scf C_0 \rcf P{}^* \rcf P] = \mathbb E[\rcf P{}^* (\rcf I - \rcf P\hat  m_0)(\rcf I - \rcf P \hat m_0){}^* \rcf P],
\label{L_fully_continuous}
\end{equation}
equations whose interpretations we shall discuss in this section. Note that (\ref{A_fully_continuous}) and (\ref{L_fully_continuous}) can be seen as the limiting cases of (\ref{mueq}) and (\ref{sigeq}) for $\sigma^2 = 0,$ $p \rightarrow \infty$, and $n \rightarrow \infty$.

In the equations above, we define $\scf P^*: C_0^\infty(\br^2) \rightarrow C_0^\infty(\br^3)'$ by $\ip{\scf P^* \scf J}{\scf Y} := \ip{\scf J}{\scf P\scf Y}_{L^2(\br^2)}$, where $\scf J \in C_0^\infty(\br^2)$, $\scf Y \in C_0^\infty(\br^3)$ and $C_0^\infty(\br^3)'$ is the space of continuous linear functionals on $C_0^\infty(\br^3)$. Thus, both sides of (\ref{A_fully_continuous}) are elements of $C_0^\infty(\br^3)'$. To verify this equation, we apply both sides to a test function $\scf Y$:
\begin{equation}
\begin{split}
\ip{\mathbb E[\rcf {P}{}^* \rcf I]}{\scf Y} = \mathbb E\left[\ip{\rcf I}{\rcf {P}\scf Y}_{L^2(\br^2)}\right] &= \mathbb E\left[\mathbb E\left[\left.\ip{\rcf I}{\rcf {P}\scf Y}_{L^2(\br^2)} \right| \rcf P\right]\right] \\
&= \mathbb E\left[\ip{\rcf P \hat m_0}{\rcf P \scf Y}_{L^2(\br^2)}\right] = \ip{\mathbb E[\rcf P{}^* \rcf P\hat m_0]}{\scf Y}.
\end{split}
\end{equation}
Note that 
\begin{equation}
\begin{split}
\ip{\scf P^* \scf P\hat m}{\scf Y} = \ip{\scf P \hat m}{\scf P \scf Y}_{L^2(\br^2)} &= \int_{\br^2} \hat m(\hat x R^1 + \hat y R^2) \overline{\scf Y(\hat x R^1 + \hat y R^2)} d\hat x d \hat y \\
&= \int_{\br^3} \hat m(\xi)\overline{\scf Y(\xi)} \delta(\xi \cdot \sdr R^3)d\xi,
\end{split}
\label{R2R3}
\end{equation}
from which it follows that in the sense of distributions, 
\begin{equation}
(\scf P^* \scf P \hat m)(\xi) = \hat m(\xi) \delta(\xi \cdot \sdr R^3).
\label{ps_star_ps}
\end{equation}
Intuitively, this means that $\scf P^* \scf P$ inputs the volume $\hat m$ and outputs a ``truncated" volume that coincides with $\hat m$ on a plane perpendicular to the viewing angle and is zero elsewhere. This reflects the fact that the image $\scf I = \scf P \scf X$ only gives us information about $\scf X$ on a single central plane. When we aggregate this information over all possible $\sdr R$, we obtain the operator $\scf A$:
\begin{equation}
\begin{split}
\scf A \hat m (\xi) = \mathbb E[\hat m(\xi) \delta(\xi \cdot \rdr R^3)] &= \hat m(\xi) \frac{1}{4\pi} \int_{S^2} \delta(\xi \cdot \theta) d\theta \\
&= \frac{\hat m(\xi)}{|\xi|}\frac{1}{4\pi} \int_{S^2} \delta\left(\frac{\xi}{|\xi|} \cdot \theta\right) d\theta = \frac{\hat m(\xi)}{2|\xi|}.
\end{split}
\label{A_diagonal}
\end{equation}
We used the fact that $\rdr R^3$ is uniformly distributed over $S^2$ if $\rdr R$ is uniformly distributed over $SO(3)$. Here, $d\theta$ is the surface measure on $S^2$ (hence the normalization by $4\pi$). The last step holds because the integral over $S^2$ is equal to the circumference of a great circle on $S^2$, so it is $2\pi$. 

By comparing (\ref{A_fully_continuous}) and (\ref{A_def}), it is clear that $\scf A$ is the analogue of $\hat A_n$ for infinite $n$ and $p$. Also, the equation (\ref{A_fully_continuous}) echoes the heuristic formula (\ref{heuristic}). The backprojection operator $\scf P{}^*$ simply ``inserts" a 2D image into 3D space by situating it in the plane perpendicular to the viewing direction of the image, and so the RHS of (\ref{A_fully_continuous}) at a point $\xi$ is the accumulation of values $\scf I(\xi)$. Moreover, the operator $\scf A$ is diagonal, and for each $\xi$, $\scf A$  reflects the measure of the set $\xi \in \scf I$; i.e., the density of central planes passing through $\xi$ under the uniform distribution of rotations. Thus, (\ref{A_fully_continuous}) encodes the intuition from the first equation in (\ref{heuristic}). Inverting $\scf A$ involves multiplying by the radial factor $2|\xi|$. In tomography, this factor is called the ramp filter \cite{natterer}. Traditional tomographic algorithms proceed by applying the ramp filter to the projection data and then backprojecting. Note that solving $\frac{1}{2|\xi|}\hat m_0(\xi) = \mathbb E[\rcf P{}^* \rcf I]$ implies performing these operations in the reverse order; however, backprojection and application of the ramp filter commute.


Now we move on to (\ref{L_fully_continuous}). Both sides of this equation are continuous linear functionals on $C_0^\infty(\br^3) \times C_0^\infty(\br^3)$. Indeed, for $\scf Y_1, \scf Y_2 \in C_0^\infty(\br^3)$, the LHS of (\ref{L_fully_continuous}) operates on $(\scf Y_1, \scf Y_2)$ through the definition
\begin{equation}
(\scf P^* \scf P \scf C \scf P^* \scf P)(\scf Y_1, \scf Y_2) = \ip{\scf C}{(\scf P^* \scf P\scf Y_1, \scf P^* \scf P\scf Y_2)},
\end{equation}
where we view $\scf C \in C_0^\infty(\br^3 \times \br^3)$ as operating on pairs $(\eta_1, \eta_2)$ of elements in $C_0^\infty(\br^3)'$ via
\begin{equation}
\ip{\scf C}{(\eta_1, \eta_2)} := \int_{\br^3 \times \br^3}\overline{\eta_1(\xi_1)}\eta_2(\xi_2)\scf C(\xi_1, \xi_2)d\xi_1 d\xi_2.
\end{equation}
Using these definitions, we verify (\ref{L_fully_continuous}): 
\begin{equation}
\begin{split}
&\mathbb E[\rcf P{}^* (\rcf I - \rcf P\hat  m_0)(\rcf I - \rcf P \hat m_0){}^* \rcf P] (\scf Y_1, \scf Y_2) \\
&\quad := \mathbb E\left[\ip{\rcf P{}^* (\rcf I - \rcf P\hat  m_0)}{\scf Y_1}\overline{\ip{\rcf P{}^* (\rcf I - \rcf P\hat  m_0)}{\scf Y_2}}\right] \\
&\quad = \mathbb E\left[\overline{\ip{\rcf P{}^*\rcf P \scf Y_1}{\rcf X - \hat  m_0}}\ip{\rcf P{}^* \rcf P \scf Y_2}{\rcf X - \hat  m_0}\right] \\
&\quad = \mathbb E\left[\int_{\br^3 \times \br^3}\overline{\rcf P{}^*\rcf P \scf Y_1(\xi_1)} (\rcf X(\xi_1) - \hat  m_0(\xi_1))\rcf P{}^* \rcf P \scf Y_2(\xi_2)\overline{(\rcf X(\xi_2) - \hat  m_0(\xi_2))}d\xi_1 d\xi_2\right] \\
&\quad = \mathbb E\left[\int_{\br^3 \times \br^3}\overline{\rcf P{}^*\rcf P \scf Y_1(\xi_1)} \scf C_0(\xi_1, \xi_2)\rcf P{}^* \rcf P \scf Y_2(\xi_2)d\xi_1 d\xi_2\right] \\
&\quad = \mathbb E\left[\ip{\scf C_0}{(\rcf P{}^*\rcf P \scf Y_1, \rcf P{}^*\rcf P \scf Y_2)}\right] = \mathbb E\left[\rcf P{}^* \rcf P \scf C_0 \rcf P{}^* \rcf P (\scf Y_1, \scf Y_2)\right]. \\
\end{split}
\end{equation}

Substituting (\ref{ps_star_ps}) into the last two lines of the preceding calculation, we find 
\begin{equation}
(\scf P^* \scf P \scf C \scr P^* \scf P)(\xi_1, \xi_2) = \scf C(\xi_1, \xi_2)\delta(\xi_1 \cdot \sdr R^3)\delta(\xi_2 \cdot \sdr R^3).
\end{equation}
This reflects the fact that an image $\scf I$ gives us information about $\scf C_0(\xi_1, \xi_2)$ for $\xi_1, \xi_2 \in \scf I$. Taking the expectation over $\rdr R$, we find that 
\begin{equation}
\begin{split}
(\scf L \scf C)(\xi_1, \xi_2) &= \mathbb E[\scf C(\xi_1, \xi_2) \delta(\xi_1 \cdot \rdr R^3)\delta(\xi_2 \cdot \rdr R^3)] \\
&= \scf C(\xi_1, \xi_2)\frac{1}{4\pi}\int_{S^2}\delta(\xi_1 \cdot \theta)\delta(\xi_2\cdot \theta)d\theta =: \scf C(\xi_1, \xi_2)\mathcal K(\xi_1, \xi_2).
\end{split}
\label{L_diagonal}
\end{equation}
Like $\scf A$, the operator $\scf L$ is diagonal. $\scf L$ is a fundamental operator in tomographic inverse problems involving variability; we term it the \textit{projection covariance transform}. In the same way that (\ref{A_fully_continuous}) reflected the first equation of (\ref{heuristic}), we see that \ref{L_fully_continuous}) resembles the second equation of (\ref{heuristic}). In particular, the kernel value $\mathcal K(\xi_1, \xi_2)$ reflects the density of central planes passing through $\xi_1, \xi_2$.

To understand this kernel, let us compute it explicitly. We have
\begin{equation}
\mathcal K(\xi_1, \xi_2) = \frac{1}{4\pi}\int_{S^2} \delta(\xi_1\cdot \theta) \delta(\xi_2 \cdot \theta) d\theta.
\label{K_def}
\end{equation}
For fixed $\xi_1$, note that $\delta(\xi_1 \cdot \theta)$ is supported on the great circle of $S^2$ perpendicular to $\xi_1$. Similarly, $\delta(\xi_2 \cdot \theta)$ corresponds to a great circle perpendicular to $\xi_2$. Choose $\xi_1, \xi_2 \in \br^3$ so that $|\xi_1 \times \xi_2| \neq 0$. Then, note that these two great circles intersect in two antipodal points $\theta = \pm (\xi_1 \times \xi_2)/|\xi_1 \times \xi_2|$, and the RHS of (\ref{K_def}) corresponds to the total measure of $\delta(\xi_1 \cdot \theta)\delta(\xi_2 \cdot \theta)$ at those two points. 

To calculate this measure explicitly, let us define the approximation to the identity $\delta_\eps(t) = \frac{1}{2\eps}\chi_{[-\eps, \eps]}(t)$. Fix $\eps_1, \eps_2 > 0$. Note that $\delta_{\eps_1}(\xi_1 \cdot \theta)$ is supported on a strip of width $2\eps_1/|\xi_1|$ centered at the great circle perpendicular to $\xi_1$. $\delta_{\eps_2}(\xi_2 \cdot \theta)$ is supported on a strip of width $2\eps_2/|\xi_2|$ intersecting the first strip transversely. For small $\eps_1, \eps_2$, the intersection of the two strips consists of two approximately parallelogram-shaped regions, $S_1$ and $S_2$ (see Figure \ref{fig:triangular_area_filter}). The sine of the angle between the diagonals of each of these regions is $|\xi_1 \times \xi_2|/|\xi_1||\xi_2|$, and a simple calculation shows that the area of one of these regions is $2\eps_1 2\eps_2/|\xi_1\times \xi_2|$. It follows that
\begin{equation}
\begin{split}
\mathcal K(\xi_1, \xi_2) &= \frac{1}{4\pi}\int_{S^2} \delta(\xi_1\cdot \theta) \delta(\xi_2\cdot \theta) d\theta = \lim_{\eps_1, \eps_2 \rightarrow 0} \frac{1}{4\pi}\int_{S^2} \delta_{\eps_1}(\xi_1 \cdot \theta) \delta_{\eps_2}(\xi_2\cdot \theta) d\theta \\
&= \lim_{\eps_1, \eps_2 \rightarrow 0} \frac{1}{4\pi}\int_{S_1 \cup S_2} \frac{1}{2\eps_1}\frac{1}{2\eps_2} d\theta = \lim_{\eps_1, \eps_2 \rightarrow 0} \frac{1}{4\pi}2\frac{4\eps_1 \eps_2}{|\xi_1 \times \xi_2|}\frac{1}{2\eps_1}\frac{1}{2\eps_2}\\
&= \frac{1}{4\pi}\frac{2}{|\xi_1 \times \xi_2|}.
\end{split}
\end{equation}

This analytic form of $\mathcal K$ sheds light on the geometry of $\scf L$. Recall that $\mathcal K(\xi_1, \xi_2)$ is a measure of the density of central planes passing through $\xi_1$ and $\xi_2$. Note that this density is nonzero everywhere, which reflects the fact that there is a central plane passing through each pair of points in $\br^3$. The denominator in $\mathcal K$ is proportional to the magnitudes $|\xi_1|$ and $|\xi_2|$, which indicates that there is a greater density of planes passing through pairs of points nearer the origin. Finally, note that $\mathcal K$ varies inversely with the sine of the angle between $\xi_1$ and $\xi_2$; indeed, a greater density of central planes pass through a pair of points nearly collinear with the origin. In fact, there is a singularity in $\mathcal K$ when $\xi_1, \xi_2$ are linearly dependent, reflecting the fact that infinitely many central planes pass through collinear points. As a way to sum up the geometry encoded in $\mathcal K$, note that except for the factor of $1/4\pi$, $1/\mathcal K$ is the area of the triangle spanned by the vectors $\xi_1$ and $\xi_2$. For this reason, we call $1/\mathcal K$ the \textit{triangular area filter}.

\begin{figure}[h]
\centering
\includegraphics[scale = 0.4]{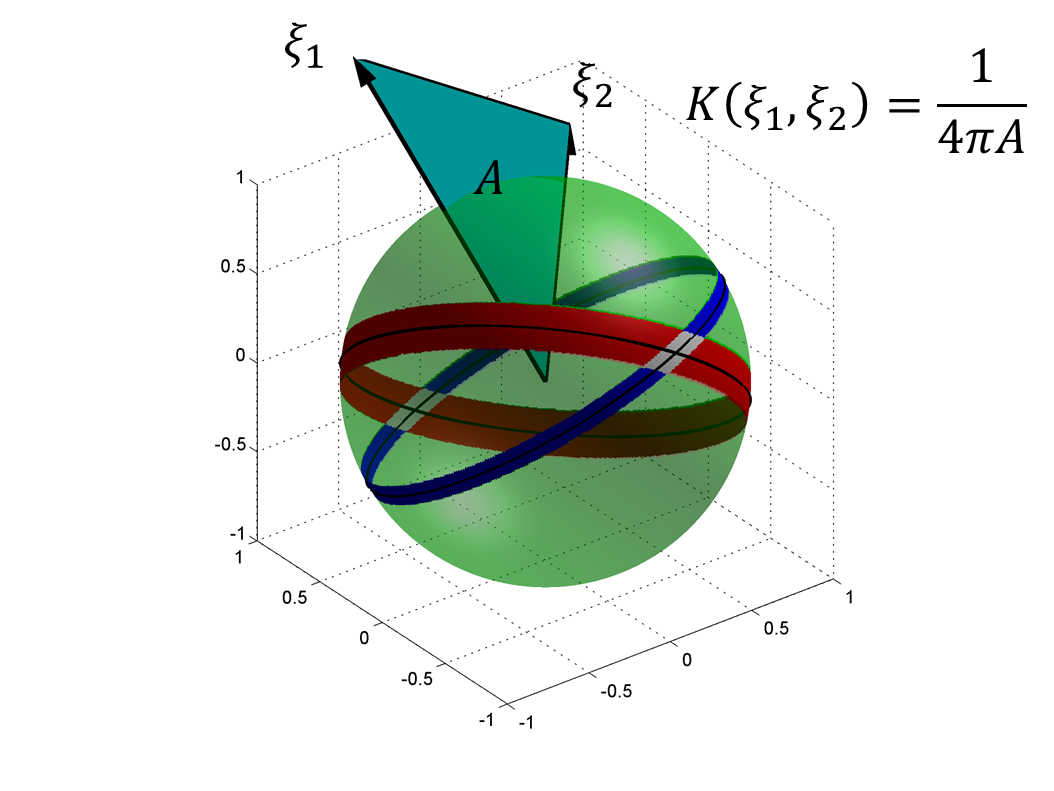}
\caption{The triangular area filter. $\xi_1$ induces a strip on $S^2$ of width proportional to $1/|\xi_1|$ (blue); $\xi_2$ induces a strip of width proportional to $1/|\xi_2|$ (red). The strips intersect in two parallelogram-shaped regions (white), each with area proportional to $1/|\xi_1 \times \xi_2|$. Hence, $\mathcal K(\xi_1, \xi_2)$ is inversely proportional to the area of the triangle spanned by $\xi_1, \xi_2$ (cyan).}
\label{fig:triangular_area_filter}
\end{figure}

Note that the triangular area filter is analogous to the ramp filter: it grows linearly with the frequencies $|\xi_1|$ and $|\xi_2|$ to compensate for the loss of high frequency information incurred by the geometry of the problem. So, this filter is a generalization of the ramp filter appearing in the estimation of the mean to the covariance estimation problem. The latter has a somewhat more intricate geometry, which is reflected in $\mathcal K$.

The properties of $\mathcal K$ translate into the robustness of inverting $\scf L$ (supposing we added noise to our model). In particular, the robustness of recovering $\scf C_0(\xi_1, \xi_2)$ grows with $\mathcal K(\xi_1, \xi_2)$. For example, recovering higher frequencies in $\scf C_0$ is more difficult. However, the fact that $\mathcal K$ is everywhere positive means that $\scf L$ is at least invertible. This statement is important in proving theoretical results about our estimators, as we saw in Section \ref{theoretical}. Note that an analogous problem of estimating the covariance matrix of 2D objects from their 1D line projections would not satisfy this condition, because for most pairs of points in $\mathbb{R}^2$, there is not a line passing through both points as well as the origin. 

\subsection{The discrete covariance estimation problem} \label{discrete_estimation}

The calculation in the preceding section shows that if we could sample images continuously and if we had access to projection images from all viewing angles, then $\scf L$ would become a diagonal operator. In this section, we explore the modifications necessary for the realistic case where we must work with finite-dimensional representations of volumes and images.

Our idea is to follow what we did in the fully continuous case treated above and estimate the covariance matrix in the Fourier domain. One possibility is to choose a Cartesian basis in the Fourier domain. With this basis, a tempting way to define $\hat P_s$ would be to restrict the Fourier 3D grid to the pixels of a 2D central slice by nearest-neighbor interpolation. This would make $\hat P_s$ a coordinate-selection operator, making $\hat L_n$ diagonal. However, this computational simplicity comes at a great cost in accuracy; numerical experiments show that the errors induced by such a coarse interpolation scheme are unacceptably large. Such an interpolation error should not come as a surprise, considering similar interpolation errors in computerized tomography \cite{natterer}. Hence, we must choose other bases for the Fourier volumes and images.

The finite sampling rate of the images limits the 3D frequencies we can hope to reconstruct. Indeed, since the images are sampled on an $N \times N$ grid confining a disc of radius 1, the corresponding Nyquist bandlimit is $\omega_{\text{Nyq}} = N\pi/2$. Hence, the images carry no information past this 2D bandlimit. By the Fourier slice theorem, this means that we also have no information about $\scr X$ past the 3D bandlimit $\omega_{\text{Nyq}}$. In practice, the exponentially decaying envelope of the CTF function renders even fewer frequencies possible to reconstruct. Moreover, we saw in Section \ref{fully_continuous} and will see in Section \ref{condition_number} that reconstruction of $\Sigma_0$ becomes more ill-conditioned as the frequency increases. Hence, it often makes sense to take a cutoff $\omega_{\max} < \omega_{\text{Nyq}}$. We can choose $\omega_{\max}$ to correspond to an effective grid size of $N_{\text{res}}$ pixels, where $N_{\text{res}} \leq N$. In this case, we would choose $\omega_{\max} = N_{\text{res}}\pi/2$. Thus, it is natural to search for $\scr X$ in a space of functions bandlimited in $B_{\omega_{\max}}$ (the ball of radius $\omega_{\max}$) and with most of their energy contained in the unit ball. The optimal space $\mathscr B$ with respect to these constraints is spanned by a finite set of 3D Slepian functions \cite{prolate4}. For a given bandlimit $\omega_{\max}$, we have
\begin{equation}
p = \text{dim}(\mathscr B) = \frac{2}{9\pi}\omega_{\max}^3.
\label{shannon_3}
\end{equation}
This dimension is called the \textit{Shannon number}, and is the trace of the kernel in \cite[eq. 6]{prolate4}.

For the purposes of this section, let us work abstractly with the finite-dimensional spaces $\ssf V \subset C_0(B_{\omega_{\max}})$ and $\ssf{I} \subset C_0(D_{\omega_{\max}})$, which represent Fourier volumes and Fourier images, respectively ($D_{\omega_{\max}} \subset \br^2$ is the disc of radius $\omega_{\max}$). For example, $\ssf V$ could be spanned by the Fourier transforms of the 3D Slepian functions. Let 
\begin{equation}
 \ssf V = \text{span}\{\hat h_j\}, \quad \ssf{I} = \text{span}\{\hat g_i\},
\end{equation}
with dim$(\ssf V) = \hat p$  and dim$(\ssf I) = \hat q$.  Assume that for all $\sdr R$, $\hat{\mathcal P}(\ssf V) \subset \ssf{I}$ (i.e., we do not need to worry about interpolation). Denote by $\sdf P$ the matrix expression of $\scf P|_{\ssf V}$. Thus, $\sdf P \in \bc^{\hat q \times \hat p}$. Let $\sdf X_1, \dots, \sdf X_n$ be the representations of $\scf X_1, \dots, \scf X_n$ in the basis for $\ssf V$. 

Since we are given the images $\sdr I_s$ in the pixel basis $\br^q$, let us consider how to map these images into $\ssf I$. Let $Q_1: \br^q \rightarrow \ssf{I}$ be the mapping which fits (in the least-squares sense) an element of $\ssf I$ to the pixel values defined by a vector in $\br^q$. It is easiest to express $Q_1$ in terms of the reverse mapping $Q_2: \ssf{I} \rightarrow \br^q$. The $i$'th column of $Q_2$ consists of the evaluations of $g_i$ at the real-domain gridpoints inside the unit disc. It is easy to see that the least-squares method of defining $Q_1$ is $Q_1 = Q_2^+ = (Q_2^H Q_2)^{-1}Q_2^H$. 

Now, note that
\begin{equation}
\rdr I = \scr S\rcr P \rcr X + \rdr E \Rightarrow Q_1\rdr I = Q_1 \scr S\rcr P \rcr X + Q_1 \rdr E \approx \rdf P \rdf X + Q_1 \rdr E.
\label{approx_1}
\end{equation}
The last approximate equality is due to the Fourier slice theorem. The inaccuracy comes from the discretization operator $\scr S$. Note that $\Var[Q_1 \rdr E] = \sigma^2 Q_1 Q_1^H = \sigma^2 (Q_2^H Q_2)^{-1}$. We would like the latter matrix to be a multiple of the identity matrix so that the noise in the images remains white. Let us calculate the entries of $Q_2^H Q_2$ in terms of the basis functions $g_i$. Given the fact that we are working with volumes $h_i$ which have most of their energy concentrated in the unit ball, it follows that $g_i$ have most of their energy concentrated in the unit disc. If $x_1, \dots, x_q$ are the real-domain image gridpoints, it follows that
\begin{equation}
\begin{split}
(Q_2^H Q_2)_{ij} = \sum_{r =1}^q \overline{ g_i(x_r)} g_j(x_r) &\approx \frac{q}{\pi}\int_{|x| \leq 1} \overline{ g_i(x)} g_j(x)dx \\
&\approx \frac{q}{\pi}\overline{\ip{ g_i}{ g_j}_{L^2(\br^2)}} = \frac{q}{\pi}\frac{1}{(2\pi)^2}\overline{\ip{\hat g_i}{\hat g_j}_{L^2(\br^2)}}.
\label{identity_calculation}
\end{split}
\end{equation}
It follows that in order for $Q_2^H Q_2$ to be (approximately) a multiple of the identity matrix, we should require $\{\hat g_i\}$ to be an orthonormal set in $L^2(\br^2)$. If we let $c_q = 4\pi^3/q$, then we find that
\begin{equation}
Q_1 Q_1^H \approx c_q \text{I}_{\hat q}.
\label{approx_2}
\end{equation}
It follows that, if we make the approximations in (\ref{approx_1}) and (\ref{approx_2}), we can formulate the heterogeneity problem entirely in the Fourier domain as follows:
\begin{equation}
\rdf I = \rdf P \rdf X + \rdf E,
\end{equation}
where $\Var[\rdf E] = \sigma^2{c_q}\text{I}_{\hat q}$. Thus, we have an instance of Problem (\ref{main_problem}), with $\sigma^2$ replaced by $\sigma^2 c_q$, $q$ replaced by $\hat q$, and $p$ replaced by $\hat p$. We seek $\hat \mu_0 = \mathbb E[\rdf X]$ and $\hat \Sigma_0 = \Var[\rdf X]$. Equations (\ref{mueq}) and (\ref{sigeq}) become
\begin{equation}
\sdf A_n \hat \mu_n := \left(\frac1n\sum_{s = 1}^n \sdf P_s^H \sdf P_s\right)\hat \mu_n = \frac1n\sum_{s = 1}^n \sdf P_s^H \sdf I_s.
\label{mueq_new}
\end{equation}
and
\begin{equation}
\begin{split}
\sdf L_n \hat \Sigma_n :&= \frac1n\sum_{s = 1}^n \sdf P_s^H \sdf P_s \hat \Sigma_n \sdf P_s^H \sdf P_s \\
&= \frac1n\sum_{s = 1}^n \sdf P_s^H (\sdf I_s -  \sdf P_s \hat \mu_n)(\hat I_s -  \sdf P_s \hat \mu_n)^H \sdf P_s - \sigma^2c_q \sdf A_n =: \hat B_n.
\label{sigeq_new}
\end{split}
\end{equation}


\subsection{Exploring $\hat A$ and $\hat L$} \label{exploring}

In this section, we seek to find expressions for $\hat A$ and $\hat L$ like those in (\ref{A_diagonal}) and (\ref{L_diagonal}). The reason for finding these limiting operators is two-fold. First of all, recall that the theoretical results in Section \ref{theoretical} depend on the invertibility of these limiting operators. Hence, knowing $\hat A$ and $\hat L$ in the cryo-EM case will allow us to verify the assumptions of Propositions \ref{mu_main} and \ref{Sigma_main}. Secondly, the law of large numbers guarantees that for large $n$, we have $\sdf A_n \approx \sdf A$ and $\sdf L_n \approx \sdf L$. We shall see in Section \ref{practical} that approximating $\sdf A_n$ and $\sdf L_n$ by their limiting counterparts makes possible the tractable implementation of our algorithm.

In Section \ref{fully_continuous}, we worked with functions $\hat m: \br^3 \rightarrow \bc$ and $\scf C: \br^3 \times \br^3 \rightarrow \bc$. Now, we are in a finite-dimensional setup, and we have formulated (\ref{mueq_new}) and (\ref{sigeq_new}) in terms of vectors and matrices. Nevertheless, in the finite-dimensional case we can still work with functions as we did in Section \ref{fully_continuous} via the identifications
\begin{equation}
\hat \mu \in \bc^{\hat p} \leftrightarrow \hat m = \sum_{i = 1}^{\hat p} \hat \mu_i \hat h_i \in \ssf V, \quad \hat \Sigma \in \bc^{\hat p \times \hat p} \leftrightarrow \scf C = \sum_{i, j = 1}^{\hat p} \hat \Sigma_{i, j} \hat h_i \otimes \hat h_j \in \ssf V \otimes \ssf V,
\label{identifications}
\end{equation}
where we define
\begin{equation}
(\hat h_i \otimes \hat h_j)(\xi_1, \xi_2) = \hat h_i (\xi_1) \overline{\hat h_j(\xi_2)},
\end{equation}
and $\ssf V \otimes \ssf V = \text{span}\{\hat h_i \otimes \hat h_j\}$.
Thus, we identify $\bc^{\hat p}$ and $\bc^{\hat p \times \hat p}$ with spaces of bandlimited functions. For these identifications to be isometries, we must endow $\ssf V$ with an inner product for which $\hat h_i$ are orthonormal. We consider a family of inner products, weighted by radial functions $w(|\xi|)$:
\begin{equation}
\ip{\hat h_i}{\hat h_j}_{L_w^2(\br^3)} = \int_{\br^3} \hat h_i(\xi) 
\overline{\hat h_j(\xi)}w(|\xi|)d\xi = \delta_{ij}.
\label{assumption}
\end{equation}
The inner product on $\ssf V \otimes \ssf V$ is inherited from that of $\ssf V$.


Note that $\hat A_n$ and $\hat L_n$ both involve the projection-backprojection operator $\hat P_s^H \hat P_s$. Let us see how to express $\hat P_s^H \hat P_s$ as an operator on $\ssf V$. The $i$'th column of $\hat P_s$ is the representation of $\mathcal{\hat P}_s \hat h_i$ in the orthonormal basis for $\ssf{I}$. Hence, using the isomorphism $\hat \bc^{\hat q} \leftrightarrow \ssf{I}$ and reasoning along the lines of (\ref{R2R3}), we find that
\begin{equation}
(\hat P_s^H \hat P_s)_{i, j} = \overline{\ip{\mathcal{\hat P}_s \hat h_i}{\mathcal{\hat P}_s \hat h_j}_{L^2(\br^2)}} = \int_{\br^3} \overline{\hat h_i(\xi)} \hat h_j(\xi) \delta(\xi \cdot \sdr R_s^3)d\xi.
\label{pshps}
\end{equation}
Note that here and throughout this section, we perform manipulations (like those in Section \ref{fully_continuous}) that involve treating elements of $\ssf V$ as test functions for distributions. We will ultimately construct $\ssf V$ so that its elements are continuous, but not in $C_0^\infty(\br^3)$, as assumed in Section \ref{fully_continuous}. Nevertheless, since we are only dealing with distributions of order zero, continuity of the elements of $\ssf V$ is sufficient.

From (\ref{pshps}), it follows that if $\hat \mu \in \bc^{\hat p} \leftrightarrow \hat m \in \ssf V$, then
\begin{equation}
\begin{split}
(\hat P_s^H \hat P_s)\hat \mu \leftrightarrow \sum_{i = 1}^{\hat p}\hat h_i \sum_{j = 1}^{\hat p}(\hat P_s^H \hat P_s)_{ij}\hat \mu_j  &= \sum_{i = 1}^{\hat p}\hat h_i \sum_{j = 1}^{\hat p}\int_{\br^3} \overline{\hat h_i(\xi)} \hat \mu_{j}\hat h_j(\xi) \delta(\xi \cdot \sdr R_s^3)d\xi \\
&= \sum_{i = 1}^{\hat p}\hat h_i \int_{\br^3} \left(\hat m(\xi) \delta(\xi \cdot \sdr R_s^3)\right)\overline{\hat h_i(\xi)} d\xi \\
&=: \pi_{\ssf V}\left(\hat m(\xi)\delta (\xi \cdot \sdr R_s^3)\right),
\label{proj_backproj_discrete_precursor}
\end{split}
\end{equation}
where $\pi_{\ssf V}: C_0^\infty(\br^3)' \rightarrow \ssf V$ is defined via
\begin{equation}
\pi_{\ssf V}(\eta) = \sum_{i} \hat h_i \ip{\eta}{\hat h_i}, \quad \eta \in C_0^\infty(\br^3)'
\label{pi_V_def}
\end{equation}
is a projection onto the finite-dimensional subspace $\ssf V$.
%

In analogy with (\ref{A_fully_continuous}), we have
\begin{equation}
\begin{split}
\hat A \hat \mu \leftrightarrow \mathbb E\left[\pi_{\ssf V}\left(\hat m(\xi)\delta (\xi \cdot \rdr R^3)\right)\right] =  \pi_{\ssf V}\left(\frac{\hat m(\xi)}{2|\xi|} \right)
\end{split}
\label{A_coords}
\end{equation}
Note $\hat A$ resembles the operator $\scf A$ obtained in (\ref{A_fully_continuous}), with the addition of the ``low-pass filter" $\pi_{\ssf V}$.  As a particular choice of weight, one might consider $w(|\xi|) = 1/|\xi|$ in order to cancel the ramp filter. For this weight, note that
\begin{equation}
\begin{split}
\hat A \hat \mu \leftrightarrow \pi_{\ssf V}\left(\frac{\hat m(\xi)}{2|\xi|}\right) = \pi^w_{\ssf V} \left(\frac12\hat m(\xi)\right) = \frac12 \hat m(\xi), \quad w(|\xi|) = 1/|\xi|,
\end{split}
\label{ramp_cancellation}
\end{equation}
where $\pi^w_{\ssf V}$ is the \textit{orthogonal} projection onto $\ssf V$ with respect to the weight $w$.
Thus, for this weight we find that $\hat A = \frac12 \text{I}_{\hat p}$.

A calculation analagous to (\ref{proj_backproj_discrete_precursor}) shows that for $\hat \Sigma \in \bc^{\hat p \times \hat p}\leftrightarrow \scf C \in \ssf V \otimes \ssf V$, 
\begin{equation}
\begin{split}
\hat P_s^H \hat P_s \hat \Sigma \hat P_s^H \hat P_s \leftrightarrow \pi_{\ssf V \otimes \ssf V}\left(\scf C(\xi_1, \xi_2)\delta (\xi_1 \cdot \sdr R_s^3)\delta (\xi_2 \cdot \sdr R_s^3)\right).
\end{split}
\end{equation}
Then, taking the expectation over $\rdr R^3$, we find that  
\begin{equation}
\begin{split}
\hat L \hat \Sigma \leftrightarrow \pi_{\ssf V \otimes \ssf V}\left(\scf C(\xi_1, \xi_2)\mathcal K(\xi_1, \xi_2)\right).
\label{L_coords}
\end{split}
\end{equation}
This shows that between $\hat L$ is linked to $\scf L$ via the low-pass-filter $\pi_{\ssf V \otimes \ssf V}$, which is defined analogously to (\ref{pi_V_def}).
\subsection{Properties of $\hat A$ and $\hat L$} \label{cryo_em_theoretical}

In this section, we will prove several results about $\hat A$ and $\hat L$, defined in (\ref{A_coords}) and (\ref{L_coords}). We start by proving a useful lemma:
\begin{lemma} \label{useful_lemma}
For $\eta \in C_0^\infty(\br^3)'$ and $\scf Y \in \ssf V$, we have
\begin{equation}
\ip{\pi_{\ssf V}\eta}{\scf Y}_{L^2_w(\br^3)} = \ip{\eta}{\scf Y}.
\end{equation}
Likewise, if $\eta \in C_0^\infty(\br^3 \times \br^3)'$ and $\scf C \in \ssf V \otimes \ssf V$, we have
\begin{equation}
\ip{\pi_{\ssf V \otimes \ssf V}\eta}{\scf C}_{L^2_w(\br^3 \times \br^3)} = \ip{\eta}{\scf C}.
\end{equation}

\end{lemma}
\begin{proof}
Indeed, we have
\begin{equation}
\begin{split}
\ip{\pi_{\ssf V}\eta}{\scf Y}_{L^2_w(\br^3)} &= \sum_{i = 1}^{\hat p} \ip{\eta}{\hat h_i} \ip{\hat h_i}{\scf Y}_{L^2_w(\br^3)} \\
&= \ip{\eta}{\sum_{i = 1}^{\hat p}\ip{\scf Y}{\hat h_i}_{L^2_w(\br^3)}\hat h_i} = \ip{\eta}{\scf Y}.
\end{split}
\end{equation}
The proof of the second claim is similar.
\end{proof}

Note that $\hat A$ and $\hat L$ are self-adjoint and positive semidefinite because each $\hat A_n$ and $\hat L_n$ satisfies this property. In the next proposition, we bound the minimum eigenvalues of these two operators from below.

\begin{proposition} \label{min_eig}
Let $M_{w}(\omega_{\max}) = \max_{|\xi| \leq \omega_{\max}}|\xi|w(|\xi|)$. Then,
\begin{equation}
\lambda_{\min}(\hat A) \geq \frac{1}{2M_{w}(\omega_{\max})}; \quad \lambda_{\min}(\hat L) \geq \frac{1}{2\pi M^2_{w}(\omega_{\max})}.
\end{equation}
\end{proposition}
\begin{proof}
Let $\hat \mu \in \bc^{\hat p}\leftrightarrow \hat m\in \ssf V$. Using the isometry $\bc^{\hat p} \leftrightarrow \ssf V$, Lemma (\ref{useful_lemma}), and (\ref{A_coords}), we find
\begin{equation}
\begin{split}
\ip{\hat A\hat \mu}{\hat \mu}_{\bc^{\hat p}} &= \ip{\pi_{\ssf V}\left(\hat m \frac{1}{2|\xi|}\right)}{\hat m}_{L^2_w(\br^3)} 
= \ip{\hat m\frac{1}{2|\xi|}}{\hat m} \\
&= \int_{B_{\omega_{\max}}}|\hat m(\xi)|^2 \frac{1}{2|\xi|w(|\xi|)}w(|\xi|)d\xi \geq \frac{1}{2M_{w}(\omega_{\max})}\norm{\hat m}_{L^2_w(\br^3)} ^2 = \frac{1}{2M_{w}(\omega_{\max})}\norm{\hat \mu}^2.
\end{split}
\end{equation}

The bound on the minimum eigenvalue of $\hat L$ follows from a similar argument, using (\ref{L_coords}) and the following bound:
\begin{equation}
\min_{\xi_1, \xi_2 \in B_{\omega_{\max}}} \frac{K(\xi_1, \xi_2)}{w(|\xi_1|)w(|\xi_2|)} = \min_{\xi_1, \xi_2 \in B_{\omega_{\max}}} \frac{1}{2\pi |\xi_1 \times \xi_2|w(|\xi_1|)w(|\xi_2|)} \geq \frac{1}{2\pi M^2_{w}(\omega_{\max})}.
\label{still_singular}
\end{equation}
\end{proof}

By inspecting $M_w(\omega_{\max})$, we see that choosing $w = 1/|\xi|$ leads to better conditioning of both $\hat A$ and $\hat L$, as compared to $w = 1$. This is because the former weight compensates for the loss of information at higher frequencies. We see from (\ref{ramp_cancellation}) that for $w = 1/|\xi|$, $\hat A$ is perfectly conditioned. This weight also cancels the linear growth of the triangular area filter with radial frequency. However, it does not cancel $\mathcal K$ altogether, since the dependency on $\sin \gamma$ in the denominators in (\ref{still_singular}) remains, where $\gamma$ is the angle between $\xi_1$ and $\xi_2$.

The maximum eigenvalue of $\hat L$ cannot be bounded as easily, since the quotient in (\ref{still_singular}) is not bounded from above. A bound on $\lambda_{\max}(\hat L)$ might be obtained by using the fact that a bandlimited $\scf C$ can only be concentrated to a limited extent around the singular set $\{\xi_1, \xi_2 : |\xi_1 \times \xi_2| = 0\}$.

Finally, we prove another property of $\hat A$ and $\hat L$: they commute with rotations. Let us define the group action of $SO(3)$ on functions $\br^3 \rightarrow \bc$ as follows: for $\sdr R \in SO(3)$ and $\scf Y: \br^3 \rightarrow \bc$, let $\sdr R . \scf Y (\xi) = \scf Y(\sdr R^T \xi)$. Likewise, define the group action of $SO(3)$ on functions $\scf C: \br^3\times \br^3 \rightarrow \bc$ via $\sdr R . \scf C (\xi_1, \xi_2) = \scf C(\sdr R^T \xi_1, \sdr R^T \xi_2)$. 
\begin{proposition} \label{rot_invariance}
Suppose that the subspace $\ssf V$ is closed under rotations. Then, for any $\scf Y \in \ssf V$, $\scf C \in  \ssf V \otimes \ssf V$ and $\sdr R \in SO(3)$, we have
\begin{equation}
R.(\hat A \scf Y) = \hat A (R.\scf Y), \quad R.(\hat L \scf C) = \hat L (R.\scf C),
\label{commuting_rotation}
\end{equation}
where $\hat A \scf Y$ and $\hat L \scf C$ are understood via the identifications (\ref{identifications}).
\end{proposition}
\begin{proof}
We begin by proving the first half of (\ref{commuting_rotation}). First of all, extend the group action of SO(3) to the space $C_0^\infty(\br^3)'$ via 
\begin{equation}
\ip{R.\eta}{\scf Y} := \ip{\eta}{R^{-1}.\scf Y}, \quad \scf Y \in C_0^\infty(\br^3).
\end{equation}
We claim that for any $\eta \in C_0^\infty(\br^3)'$, we have $R.(\pi_{\ssf V}\eta) = \pi_{\ssf V}(R.\eta)$. Since $\ssf V$ is closed under rotations, both sides of this equation are elements of $\ssf V$. We can verify their equality by taking an inner product with an arbitrary element $\scf Y \in \ssf V$. Using Lemma \ref{useful_lemma} and the fact that $\ssf V$ is closed under rotations, we obtain
\begin{equation}
\begin{split}
\ip{R.(\pi_{\ssf V}\eta)}{\scf Y}_{L^2_w(\br^3)} = \ip{\pi_{\ssf V}\eta}{R^{-1}.\scf Y}_{L^2_w(\br^3)} &= \ip{\eta}{R^{-1}.\scf Y} \\
&= \ip{R.\eta}{\scf Y} = \ip{\pi_{\ssf V}(R.\eta)}{\scf Y}_{L^2_w(\br^3)}.
\end{split}
\end{equation}
Next, we claim that for any $\scf Y \in \ssf V$, we have $R.(\scf A \scf Y) = \scf A (R.\scf Y)$. To check whether these two elements of $C_0^\infty(\br^3)'$ are the same, we apply them to a test function $\scf Z \in C_0^\infty(\br^3)$:
\begin{equation}
\begin{split}
\ip{R.(\scf A \scf Y)}{\scf Z} = \ip{\scf A \scf Y}{R^{-1}.\scf Z} &= \int_{\br^3}\frac{\scf Y(\xi)}{2|\xi|}\scf Z(R\xi)d\xi \\
&= \int_{\br^3}\frac{\scf Y(R^H\xi)}{2|\xi|}\scf Z(\xi)d\xi = \ip{\scf A (R.\scf Y)}{\scf Z}.
\end{split}
\end{equation}
Putting together what we have, we find that
\begin{equation}
R.(\hat A \scf Y) = R.(\pi_{\ssf V}(\scf A \scf Y)) = \pi_{\ssf V}(R.(\scf A \scf Y)) = \pi_{\ssf V}(\scf A(R. \scf Y)) = \hat A(R. \scf Y),
\end{equation}
which proves the first half of (\ref{commuting_rotation}). The second half is proved analogously.
\end{proof}

This property of $\hat A$ and $\hat L$ is to be expected, given the rotationally symmetric nature of these operators. This suggests that $\hat L$ can be studied further using the representation theory of $SO(3)$.

Finally, let us check that the assumptions of Propositions \ref{mu_main} and \ref{Sigma_main} hold in the cryo-EM case. It follows from Proposition \ref{min_eig} that as long as $M_{w}(\omega_{\max}) < \infty$, the limiting operators $\hat A$ and $\hat L$ are invertible. Of course, it is always possible to choose such a weight $w$. In particular the weights already considered, $w = 1, 1/|\xi|$ satisfy this property. Moreover, by rotational symmetry, $\norm{\hat P(\sdr R)}$ is independent of $\sdr R$, and so of course this quantity is uniformly bounded. Thus, we have checked all the necessary assumptions to arrive at the following conclusion: 
\begin{proposition} \label{cryo_em_consistency}
If we neglect the errors incurred in moving to the Fourier domain and assume that the rotations are drawn uniformly from $SO(3)$, then the estimators $\hat \mu_n$ and $\hat \Sigma_n$ obtained from (\ref{mueq_new}) and (\ref{sigeq_new}) are consistent. 
\end{proposition}

\section{Using $\hat \mu_n, \hat \Sigma_n$ to determine the conformations} \label{finding_conformations}

To solve Problem \ref{het_problem}, we must do more than just estimate $\hat \mu_0$ and $\hat \Sigma_0$. We must also estimate $C$, $\hat X^c$, and $p_c$, where $\sdf X^c$ is the coefficient vector of $\scf X^c$ in the basis for $\ssf V$. Once we solve (\ref{mueq_new}) and (\ref{sigeq_new}) for $\hat \mu_n$ and $\hat \Sigma_n$, we perform the following steps.

From the discussion on high-dimensional PCA in Section \ref{rmt}, we expect to determine the number of structural states by inspecting the spectrum of $\hat \Sigma_n$. We expect the spectrum of $\hat \Sigma_n$ to consist of a bulk distribution along with $C - 1$ separate eigenvalues (assuming the SNR is sufficiently high), a fact confirmed by our numerical results. Hence, given $\hat \Sigma_n$, we can estimate $C$.

Next, we discuss how to reconstruct $\hat X^1, \dots, \hat X^C$ and $p_1, \dots, p_C$. Our approach is similar to Penczek's \cite{penczek}. By the principle of PCA, the leading eigenvectors of $\hat \Sigma_0$ span the space of mean subtracted volumes $\hat X^1 - \hat \mu_0, \dots, \hat X^C - \hat \mu_0$. If $\hat V_n^1, \dots, \hat V_n^{C-1}$ are the leading eigenvectors of $\hat \Sigma_n$, we can write
\begin{equation}
\hat X_s \approx \hat \mu_n + \sum_{c' = 1}^{C-1} \alpha_{s, c'} \hat V_n^{c'}.
\label{factorial}
\end{equation}
Note that there is only approximate equality because we have replaced the mean $\hat \mu_0$ by the estimated mean $\hat \mu_n$, and the eigenvectors of $\hat \Sigma_0$ by those of $\hat \Sigma_n$.
We would like to recover the coefficients $\alpha_{s} = (\alpha_{s, 1}, \dots, \alpha_{s, C-1})$, but the $\hat X_s$ are unknown. Nevertheless, if we project the above equation by $\hat P_s$, then we get
\begin{equation}
\sum_{c' = 1}^{C-1}\alpha_{s, c'} \hat P_s \hat V_n^{c'} \approx \hat P_s  \hat X_s -  \hat P_s \hat \mu_n = (\hat I_s - \hat P_s \hat \mu_n) - \hat \eps_s.
\label{least_squares}
\end{equation}
For each $s$, we can now solve this equation for the coefficient vector $\alpha_s$ in the least-squares sense. This gives us $n$ vectors in $\bc^{C-1}$. These should be clustered around $C$ points $\alpha^c = (\alpha^c_1, \dots, \alpha^c_{C-1})$ for $c = 1, \dots, C$, corresponding to the $C$ underlying volumes. At this point, Penczek proposes to perform $K$-means clustering on $\alpha_s$ in order to deduce which image corresponds to which class. However, if the images are too noisy, then it would be impossible to separate the classes via clustering. Note that in order to reconstruct the original volumes, all we need are the means of the $C$ clusters of coordinates. If the mean volume and top eigenvectors are approximately correct, then the main source of noise in the coordinates is the Gaussian noise in the images. It follows that the distribution of the coordinates in $\bc^{C-1}$ is a mixture of Gaussians. Hence, we can find the means $\alpha^c$ of each cluster using either an EM algorithm (of which the $K$-means algorithm used by Penczek is a limiting case \cite{bishop_book}) or the method of moments, e.g. \cite{disentangling}. In the current implementation, we use an EM algorithm. Once we have the $C$ mean vectors, we can reconstruct the original volumes using (\ref{factorial}). Putting these steps together, we arrive at a high-level algorithm to solve the heterogeneity problem (see Algorithm \ref{high_level_algorithm}).

\begin{algorithm}[h]
\caption{High-level algorithm for heterogeneity problem (Problem \ref{het_problem}).}
\label{high_level_algorithm}
\begin{algorithmic}[1]
\STATE \textbf{Input:} $n$ images $I_s$ and the corresponding rotations $\sdr R_s$
\STATE Estimate the noise level $\sigma^2$ from the corner regions of the images.
\STATE Choose bases for $\ssf I$ and $\ssf V$.
\STATE Map the images $I_s$ into $\sdf I_s \in \bc^{\hat q}$.
\STATE Estimate $\hat \mu_n, \hat \Sigma_n$ by solving (\ref{mueq_new}) and (\ref{sigeq_new}).
\STATE Compute the eigendecomposition of $\hat \Sigma_n$ and estimate its rank $r$. Set $C = r + 1$.
\STATE Estimate each $\alpha_{s} \in \bc^{C-1}$ by solving (\ref{least_squares}) using least squares.
\STATE Find $\alpha^c$ and $p_c$ by applying either EM or a method of moments algorithm to $\alpha_s$.
\STATE Using $\alpha^c$, find $\scf X^1, \dots, \scf X^C$ from (\ref{factorial}). Map volumes back to real domain for visualization.
\STATE \textbf{Output:} $C$, $\scr X^1, \dots, \scr X^C$, $p_1, \dots, p_C$.
\end{algorithmic}
\end{algorithm}

\section{Implementing Algorithm \ref{high_level_algorithm}} \label{practical}

In this section, we confront the practical challenges of implementing Algorithm \ref{high_level_algorithm}. We consider different approaches to addressing these challenges and choose one approach to explore further.

\subsection{Computational challenges and approaches} \label{comp_challenges_approaches}

The main computational challenge in Algorithm \ref{high_level_algorithm} is solving for $\hat \Sigma_n$ in
\begin{equation}
\hat L_n(\hat \Sigma_n) = \hat B_n,
\label{big_system}
\end{equation}
given the immense size of this problem. Two possibilities for inverting $\hat L_n$ immediately come to mind. The first is to treat (\ref{big_system}) as a large system of linear equations, viewing $\hat \Sigma_n$ as a vector in $\bc^{\hat p^2}$ and $\hat L_n$ as a matrix in $\bc^{\hat p^2 \times \hat p^2}$. In this scheme, the matrix $\hat L_n$ could be computed once and stored. However, this approach has an unreasonably large storage requirement. Since $\hat p = O(N_{\text{res}}^3)$, it follows that $\hat L_n$ has size $N_{\text{res}}^6 \times N_{\text{res}}^6$. Even for a small $N_{\text{res}}$ value such as 17, each dimension of $\hat L_n$ is $1.8\times 10^6$. Storing such a large $\hat L_n$ requires over 23 terabytes. Moreover, inverting this matrix na\"{i}vely is completely intractable. 

The second possibility is to abandon the idea of forming $\hat L_n$ as a matrix, and instead to use an iterative algorithm, such as the conjugate gradient (CG) algorithm, based on repeatedly applying $\hat L_n$ to an input matrix. From (\ref{sigeq_new}), we see that applying $\hat L_n$ to a matrix is dominated by $n$ multiplications of a $\hat q \times \hat p$ matrix by a $\hat p \times \hat p$ matrix, which costs $n \hat q \hat p^2 = O(n N_{\text{res}}^8)$. If $\kappa_n$ is the condition number of $\hat L_n$, then CG will converge in $O(\sqrt{\kappa_n})$ iterations (see, e.g., \cite{trefethen_numerical}). Hence, while the storage requirement of this alternative algorithm is only $O(\hat p^2) = O(N_{\text{res}}^6)$, the computational complexity is $O(nN_{\text{res}}^8\sqrt{\kappa_n})$. Thus, the price to pay for reducing the storage requirement is that $n$ matrix multiplications must be performed at each iteration. While this computational complexity might render the algorithm impractical for a regular computer, one can take advantage of the fact that the $n$ matrix multiplications can be performed in parallel.

We propose a third numerical scheme, one which requires substantially less storage than the first scheme above and does not require $O(n)$ operations at each iteration. We assume that $R_s$ are drawn from the uniform distribution over $SO(3)$, and so for large $n$, the operator $\hat L_n$ does not differ much from its limiting counterpart, $\hat L$ (defined in (\ref{L_coords})). Hence, if we replace $\hat L_n$ by $\hat L$ in (\ref{big_system}), we would not be making too large an error. Of course, $\hat L$ is a matrix of the same size as $\hat L_n$, so it is also impossible to store on a computer. However, we leverage the analytic form of $\hat L$ in order to invert it more efficiently. At this point, we have not yet chosen the spaces $\ssf V$ and $\ssf I$, and by constructing these carefully we give $\hat L$ a special structure. This approach also entails a tradeoff: in practice the approximation $\hat L_n \approx \hat L$ is accurate to the extent that $\sdr R_1^3, \dots, \sdr R_n^3$ are uniformly distributed on $S^2$. Hence, we must extract a subset of the given rotations whose viewing angles are approximately uniformly distributed on the sphere. Thus, the sacrifice we make in this approach is a reduction in the sample size. Moreover, since the subselected viewing directions are no longer statistically independent, the theoretical consistency result stated in Proposition \ref{cryo_em_consistency} does not necessarily extend to this numerical scheme.

Nevertheless, the latter approach is promising because the complexity of inverting $\hat L$ is independent of the number of images, and this computation might be tractable for reasonable values of $N_{\text{res}}$ if $\hat L$ has enough structure. It remains to construct $\ssf V$ and $\ssf I$ to induce a special structure in $\hat L$, which we turn to next.

\subsection{Choosing $\ssf V$ to make $\hat L$ sparse and block diagonal}

In this section, we write down an expression for an individual element of $\hat L$, and discover that for judiciously chosen basis functions $\hat h_i$, the matrix $\hat L$ becomes sparse and block diagonal. 

First, let us fix a functional form for the basis elements $\hat h_i$: let 
\begin{equation}
\hat h_i(r, \alpha) = f_i(r)a_i(\alpha), \quad r \in \br^+, \ \alpha \in S^2,
\label{radial_angular}
\end{equation}
where $f_i: \br^+ \rightarrow \br$ are radial functions and $a_i: S^2 \rightarrow \bc$ are spherical harmonics. Note, for example, that the 3D Slepian functions have this form \cite[eq. 110]{prolate4}. If $\hat h_i$ are orthogonal with respect to the weight $w$, then 
\begin{equation}
\ip{f_i}{f_j}_{L^2_{r^2 w(r)}}\ip{a_i}{a_j}_{L^2(S^2)} = \delta_{ij},
\label{orthogonality}
\end{equation}
where we use $L^2_w$ as a shorthand for $L^2_w(\br^+)$. The 3D Slepian functions satisfy the above condition with $w = 1$, because they are orthogonal in $L^2(\br^3)$.

Next, we write down the formula for an element $\hat L_{i_1, i_2, j_1, j_2}$ (here, $j_1, j_2$ are the indices of the input matrix, and $i_1, i_2$ are the indices of the output matrix). From (\ref{L_coords}) and Lemma \ref{useful_lemma}, we find
\begin{equation}
\begin{split}
\hat L_{i_1, i_2, j_1, j_2} &= \ip{\pi_{\ssf V \otimes \ssf V}\left((\hat h_{j_1} \otimes \hat h_{j_2})\mathcal K\right)}{\hat h_{i_1} \otimes \hat h_{i_2}}_{L^2_w(\br^3 \times \br^3)}\\
&= \int_{\br^3 \times \br^3}(\hat h_{j_1} \otimes \hat h_{j_2})(\xi_1, \xi_2)\mathcal K(\xi_1, \xi_2)\overline{(\hat h_{i_1} \otimes \hat h_{i_2})(\xi_1, \xi_2)}d\xi_1 d\xi_2\\
&= \int_{S^2\times S^2}\int_{\br^+ \times \br^+} (\hat h_{j_1} \otimes \hat h_{j_2})(\xi_1, \xi_2)\overline{(\hat h_{i_1} \otimes \hat h_{i_2})(\xi_1, \xi_2)}\frac{1}{2\pi r_1 r_2 |\alpha \times \beta|}r_1^2 r_2^2 dr_1 dr_2 d\alpha d\beta\\
&= \ip{f_{j_1}}{f_{i_1}}_{L^2_r}\ip{f_{j_2}}{f_{i_2}}_{L^2_r}\int_{S^2 \times S^2}(a_{j_1} \otimes a_{j_2})(\alpha, \beta)\overline{(a_{i_1} \otimes a_{i_2})(\alpha, \beta)}\frac{1}{2\pi|\alpha \times \beta|}d\alpha d\beta.
\label{L_indices}
\end{split}
\end{equation}
Thus, to make many of the radial inner products in $\hat L$ vanish, we see from (\ref{orthogonality}) that the correct weight is
\begin{equation}
w(r) = \frac{1}{r}.
\label{weight}
\end{equation}
Recall that this is the weight needed to cancel the ramp filter in $\hat A$ (see (\ref{ramp_cancellation})). We obtain a cancellation in $\hat L$ as well because the kernel of this operator also grows linearly with radial frequency. From this point on, $w$ will represent the weight above, and we will work in the corresponding weighted $L^2$ space.

What are sets of functions of the form (\ref{radial_angular}) that are orthonormal in $L^2_w(\br^3)$? If we chose 3D Slepian functions, we would get the functional form
\begin{equation}
\hat h_{k, \ell, m}(r, \alpha) = f_{k, \ell}(r)Y_\ell^m(\alpha).
\label{slepian_basis}
\end{equation}
However, these functions are orthonormal with weight $w = 1$ instead of $w = 1/r$. Consider modifying this construction by replacing the $f_{k, \ell}(r)$ by the radial functions arising in the 2D Slepian functions. These satisfy the property
\begin{equation}
\ip{f_{k_1, \ell_1}}{f_{k_2, \ell_2}}_{L^2_r} = 0 \quad \text{if} \quad \ell_1 = \ell_2, k_1 \neq k_2.
\end{equation}
With this property (\ref{slepian_basis}) becomes orthonormal in $L^2_w(\br^3)$. This gives $\hat L$ a certain degree of sparsity. However, note that the construction (\ref{slepian_basis}) has different families of $L^2_r$-orthogonal radial functions corresponding to each angular function. Thus, we only have orthogonality of the radial functions $f_{k_1, \ell_1}$ and $f_{k_2, \ell_2}$ when $\ell_1 = \ell_2$. Thus, many of the terms $\ip{f_{j}}{f_{i}}_{L^2_r}$ in (\ref{L_indices}) are still nonzero. 

A drastic improvement on (\ref{slepian_basis}) would be to devise an orthogonal basis in $L^2_w$ that used one set of $r$-weighted orthogonal functions $f_k$ for all the angular functions, rather than a separate set for each angular function. Namely, suppose we chose
\begin{equation}
\hat h_{k, \ell, m}(r, \alpha) = f_{k}(r)Y_\ell^m(\alpha), \quad (k, \ell, m) \in J,
\label{our_basis}
\end{equation}
where $J$ is some indexing set. Note that $f_k$ and $J$ need to be carefully constructed so that $\text{span}\{h_{k, \ell, m}\} \approx \mathscr B$ (see Section \ref{properties} for this construction). 
We have
\begin{equation}
f_{k}(r)Y_{\ell, m}(\alpha) = \hat h_{k, \ell, m}(r, \alpha) =\hat h_{k, \ell, m}(-r, -\alpha) = f_{k}(-r)Y_{\ell, m}(-\alpha) = (-1)^\ell f_k(-r)Y_{\ell, m}(\alpha).
\end{equation} 
Here, we assume that each $f_k$ is either even or odd at the origin, and we extend $f_k(r)$ to $r \in \br$ according to this parity. The above calculation implies that $f_k$ should have the same parity as $\ell$. Let us suppose that $f_k$ has the same parity as $k$. Then, it follows that $(k, \ell, m) \in J$ only if $k = \ell$ mod 2. Thus, $h_{k, \ell, m}$ will be orthonormal in $L^2_w$ if 
\begin{equation}
\{f_k : k = 0 \ \text{mod } 2\} \quad \text{and} \quad \{f_k : k = 1 \ \text{mod } 2\} \quad \text{are orthonormal in } L^2_r. 
\label{even_odd}
\end{equation}
If we let $k_i$ be the radial index corresponding to $i$, then we claim that the above construction implies
\begin{equation}
\begin{split}
\hat L_{i_1, i_2, j_1, j_2} &= \delta_{k_{i_1}k_{j_1}}\delta_{k_{i_2}k_{j_2}}\int_{S^2 \times S^2}(a_{j_1} \otimes a_{j_2})(\alpha, \beta)\overline{(a_{i_1} \otimes a_{i_2})(\alpha, \beta)}\frac{1}{2\pi|\alpha \times \beta|}d\alpha d\beta.
\label{block_diagonal}
\end{split}
\end{equation}
This statement does not follow immediately from (\ref{even_odd}), because we still need to check the case when $k_{i_1} \neq k_{j_1}$ mod 2. Note that in this case, the dependence on $\alpha$ in the integral over $S^2 \times S^2$ is odd, and so indeed $\hat L_{i_1, i_2, j_1, j_2} = 0$ in that case as well. If $\ssf V_k$ is the space spanned by $f_k(r)Y_{\ell}^m(\alpha)$ for all $\ell, m$, then the above implies that $\hat L$ operates separately on each $\ssf V_{k_1} \otimes \ssf V_{k_2}$. In the language of matrices, this means that if we divide $\hat \Sigma_n$ into blocks $\hat \Sigma_n^{k_1, k_2}$ based on radial indices, $\hat L$ operates on these blocks separately. We denote each of the corresponding ``blocks" of $\hat L$ by $\hat L^{k_1, k_2}$. Let us re-index the angular functions so that $a_i^k$ denotes the $i$'th angular basis function paired with $f_k$. From (\ref{block_diagonal}), we have
\begin{equation}
\begin{split}
\hat L^{k_1, k_2}_{i_1, i_2, j_1, j_2} &= \int_{S^2 \times S^2}(a_{j_1}^{k_1} \otimes a_{j_2}^{k_2})(\alpha, \beta)\overline{(a_{i_1}^{k_1} \otimes a_{i_2}^{k_2})(\alpha, \beta)}\frac{1}{2\pi|\alpha \times \beta|}d\alpha d\beta.\\
\end{split}
\label{L_block}
\end{equation}
This block diagonal structure of $\hat L$ makes it much easier to invert. Nevertheless, each block $\hat L^{k_1, k_2}$ is a square matrix with dimension $O(k_1^2 k_2^2)$. Hence, inverting the larger blocks of $\hat L$ can be difficult. Remarkably, it turns out that each block of $\hat L$ is sparse. In Appendix \ref{simplifying_integral}, we simplify the above integral over $S^2 \times S^2$. Then, (\ref{L_block}) becomes
\begin{equation}
\begin{split}
\hat L^{k_1, k_2}_{i_1, i_2, j_1, j_2} = \sum_{\ell, m} c(\ell)C_{\ell, m}(\overline{a_{i_1}^{k_1}}a_{j_1}^{k_1}) \overline{C_{\ell, m}(\overline{a_{i_2}^{k_2}}a_{j_2}^{k_2})},
\label{L_hat_indices}
\end{split}
\end{equation}
where the constants $c(\ell)$ are defined in (\ref{c_ell}) and $C_{\ell, m}(\hat \psi)$ is the $\ell, m$ coefficient in the spherical harmonic expansion of $\hat \psi: S^2 \rightarrow \bc$. It turns out that the above expression is zero for most sets of indices. To see why, recall that the functions $a_i^k$ are spherical harmonics. It is known that the product $Y_{\ell}^m Y_{\ell'}^{m'}$ can be expressed as a linear combination of harmonics $Y_L^M$, where $M = m + m'$ and $|\ell - \ell'| \leq L \leq \ell + \ell'$. Thus, $C_\ell^m\left(\overline{a_{i}}a_{j}\right)$ are sparse vectors, which shows that each block $\hat L^{k_1, k_2}$ is sparse. For example, $\hat L^{15, 15}$ has each dimension approximately $2\times 10^4$. However, only about $10^7$ elements of this block are nonzero, which is only about 3\% of its the total number of entries. This is about the same number of elements as a $3000\times 3000$ full matrix.

Thus, we have found a way to tractably solve the covariance matrix estimation problem: reconstruct $\hat \Sigma_n$ (approximately) by solving the sparse linear systems
\begin{equation}
\hat L^{k_1, k_2} \hat \Sigma_n^{k_1, k_2} = \hat B_n^{k_1, k_2},
\label{linear_system}
\end{equation}
where we recall that $\hat B_n$ is the RHS of (\ref{sigeq_new}). Also, using the fact that $\hat A_n \approx \hat A = \frac12 \text{I}_{\hat q}$, we can estimate $\hat \mu_n$ from 
\begin{equation}
\hat \mu_n = \frac{2}{n}\sum_{s = 1}^n \hat P_s^H \hat I_s.
\label{mu_system}
\end{equation}

In the next two sections, we discuss how to choose the radial components $f_k(r)$ and define $\ssf{I}$ and $\ssf V$ more precisely.

\subsection{Constructing $f_k(r)$ and the space $\ssf V$} \label{properties}

We have discussed so far that 
\begin{equation}
\ssf V = \text{span}(\{f_k(r)Y_\ell^m(\theta, \vp): (k, \ell, m) \in J\}),
\end{equation}
with $(k, \ell, m) \in J$ only if $k = \ell$ mod 2. Moreover, we have required the orthonormality condition (\ref{even_odd}). However, recall that we initially assumed that the real-domain functions $\scr X_s$ belonged to the space of 3D Slepian functions $\mathscr B$. Thus, we must choose $\ssf V$ to approximate the image of $\mathscr B$ under the Fourier transform. Hence, the basis functions $f_k(r)Y_\ell^m(\theta, \vp)$ should be supported in the ball of radius $\omega_{\max}$ and have their inverse Fourier transforms concentrated in the unit ball. Moreover, we must have $\text{dim}(\ssf V) \approx \text{dim}(\mathscr B)$. Finally, the basis functions $\hat h_i$ should be analytic at the origin (they are the truncated Fourier transforms of compactly supported molecules). We begin by examining this condition.

Expanding $\hat h_i$ in a Taylor series near the origin up to a certain degree, we can approximate it locally as a finite sum of homogeneous polynomials. By \cite[Theorem 2.1]{stein_weiss}, a homogeneous polynomial of degree $d$ can be expressed as 
\begin{equation}
H_d(\xi) = r^d(c_d Y_d(\alpha) + c_{d - 2}Y_{d-2}(\alpha) + \cdots),
\label{homogeneous}
\end{equation}
where each $Y_\ell$ represents a linear combination of spherical harmonics of degree $\ell$.
Hence, if $(k, \ell, m) \in J$, then we require that $f_k(r) = \alpha_\ell r^\ell + \alpha_{\ell +2}r^{\ell +2} + \cdots$, where some coefficients can be zero. We satisfy this requirement by constructing $f_0, f_1, \dots$ so that 
\begin{equation}
f_k(r) = \alpha_{k, k} r^k + \alpha_{k, k+2}r^{k+2} + \cdots
\end{equation}
for small $r$ with $\alpha_{k, k} \neq 0$, and combine $f_k$ with $Y_\ell^m$ if $k = \ell$ mod 2 and $\ell \leq k$. This leads to the following set of 3D basis functions:
\begin{equation}
\{\hat h_i\} = \{f_0Y_0^0, f_1 Y_1^{-1}, f_1 Y_1^0, f_1 Y_1^1, f_2 Y_0^0, f_2 Y_2^{-2}, \dots, f_2 Y_2^2, \dots \}.
\label{h_def}
\end{equation}
Written another way, we define
\begin{equation}
\ssf V = \text{span}\left(\left\{f_k(r)Y_\ell^m(\theta, \vp): 0 \leq k \leq K, \ \ell = k \text{ (mod 2)}, \ 0 \leq \ell \leq k, \ |m| \leq \ell\right\}\right).
\label{V_def}
\end{equation}
Following the reasoning preceding (\ref{homogeneous}), it can be seen that near the origin, this basis spans the set of polynomial functions up to degree $K$.

Now, consider the real- and Fourier-domain content of $\hat h_i$. The bandlimitedness requirement on $\scr X_s$ is satisfied if and only if the functions $f_k$ are supported in the interval $[0, \omega_{\max}]$. To deal with the real domain requirement, we need the inverse Fourier transform of $f_k(r)Y_\ell^m(\theta, \vp)$. With the Fourier convention (\ref{fourier}), it follows from \cite{hankel} that
\begin{equation}
\begin{split}
\mathcal F^{-1}\left(f_k(r)Y_\ell^m(\theta, \vp)\right)(r_x, \theta_x, \vp_x) &= \frac{1}{2\pi^2}i^\ell \left(\int_{0}^\infty f_k(r )j_\ell(r r_x)r^2 dr \right)Y_\ell^m(\theta_x, \vp_x) \\
&= \frac{1}{2\pi^2}i^\ell (S_\ell f_k)(r_x)Y_\ell^m(\theta_x, \vp_x).
\end{split}
\end{equation}
Here, $j_\ell$ is the spherical Bessel function of order $\ell$, and $S_\ell$ is the spherical Hankel transform. Also note that $(r, \theta, \vp)$ are Fourier-domain spherical coordinates, while $(r_x, \theta_x, \vp_x)$ are their real-domain counterparts. Thus, satisfying the real-domain concentration requirement amounts to maximizing the percentage of the energy of $S_\ell f_k$ that is contained in $[0, 1]$ for $0 \leq k \leq K, \ 0 \leq \ell \leq k, \ \ell = k$ mod 2.

Finally, we have arrived at the criteria we would like $f_k(r)$ to satisfy:
\begin{enumerate}
\item $\supp f_k \subset [0, \omega_{\max}]$;
\item $\{f_k : k \text{ even}\}$ and $\{f_k : k \text{ odd}\}$ orthonormal in $L^2(\br^+, r)$;
\item $f_k(r) = \alpha_{k,k} r^k + \alpha_{k, k+2}r^{k+2} + \cdots$ near $r = 0$;
\item Under the above conditions, maximize the percentage of the energy of $S_\ell f_k$ in $[0, 1]$, for $0 \leq k \leq K,\  0 \leq \ell \leq k, \ \ell = k$ mod 2.
\end{enumerate}
While it might be possible to find an optimal set of such functions $\{f_k\}$ by solving an optimization problem, we can directly construct a set of functions that satisfactorily satisfies the above criteria.

Note that since $\ell$ ranges in $[0, k]$, it follows that for larger $k$, we need to have higher-order spherical Hankel transforms $S_\ell f_k$ remain concentrated in $[0, 1]$. Since higher-order spherical Hankel transforms tend to be less concentrated for oscillatory functions, it makes sense to choose $f_k$ to be less and less oscillatory as $k$ increases. Note that the functions $f_k$ cannot all have only few oscillations because the even and odd functions must form orthonormal sets. Using this intuition, we construct $f_k$ as follows. Since the even and odd $f_k$ can be constructed independently, we will illustrate the idea by constructing the even $f_k$. For simplicity, let us assume that $K$ is odd, with $K = 2K_0 + 1$. Define the cutoff $\chi = \chi([0, \omega_{\max}])$. First, consider the sequence
\begin{equation}
J_{0}(z_{0, K_0 + 1}r/\omega_{\max})\chi, J_{2}(z_{2, K_0}r/\omega_{\max})\chi, \dots,
J_{2K_0}(z_{2K_0,1} r/\omega_{\max})\chi,
\end{equation}
where $z_{k, m}$ is the $m$th positive zero of $J_k$ (the $k$th order Bessel function). Note that the functions in this list satisfy criteria 1 (by construction) and 3 (due to the asymptotics of the Bessel function at the origin). Also note that we have chosen the scaling of the arguments of the Bessel functions so that the number of zero crossings decreases as the list goes on. Thus, the functions become less and less oscillatory, which is the pattern that might lead to satisfying criterion 4. However, since these functions might not be orthogonal with respect to the weight $r$, we need to orthonormalize them with respect to this weight (via Gram-Schmidt). We need to be careful to orthonormalize them in such a way as to preserve the properties that they already satisfy. This can be achieved by running the ($r$-weighted) Gram-Schmidt algorithm from higher $k$ towards lower $k$. This preserves the supports of the functions, their asymptotics at the origin, and the oscillation pattern. Moreover, the orthogonality property now holds as well. See Figure \ref{radial_basis} for the first several even radial basis functions. Constructing the odd radial functions requires following an analogous procedure. Also, changing the parity of $K$ requires the obvious modifications.

\begin{figure}
        \centering
        \begin{subfigure}[b]{0.24\textwidth}
                \centering
                \includegraphics[scale = 0.04]{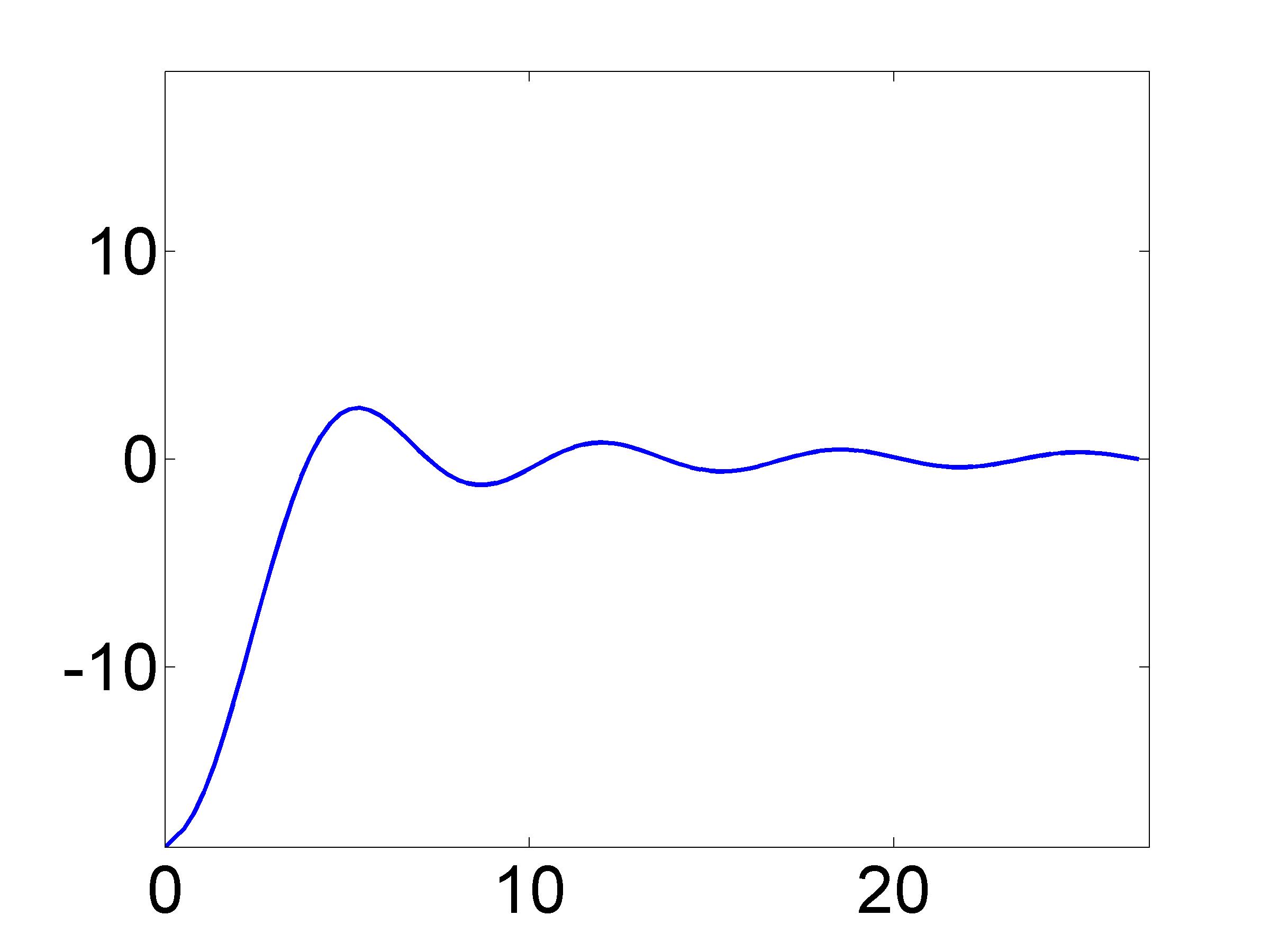}
	\caption{$f_0(r)$}
        \end{subfigure}
	\begin{subfigure}[b]{0.24\textwidth}
                \centering
                \includegraphics[scale = 0.04]{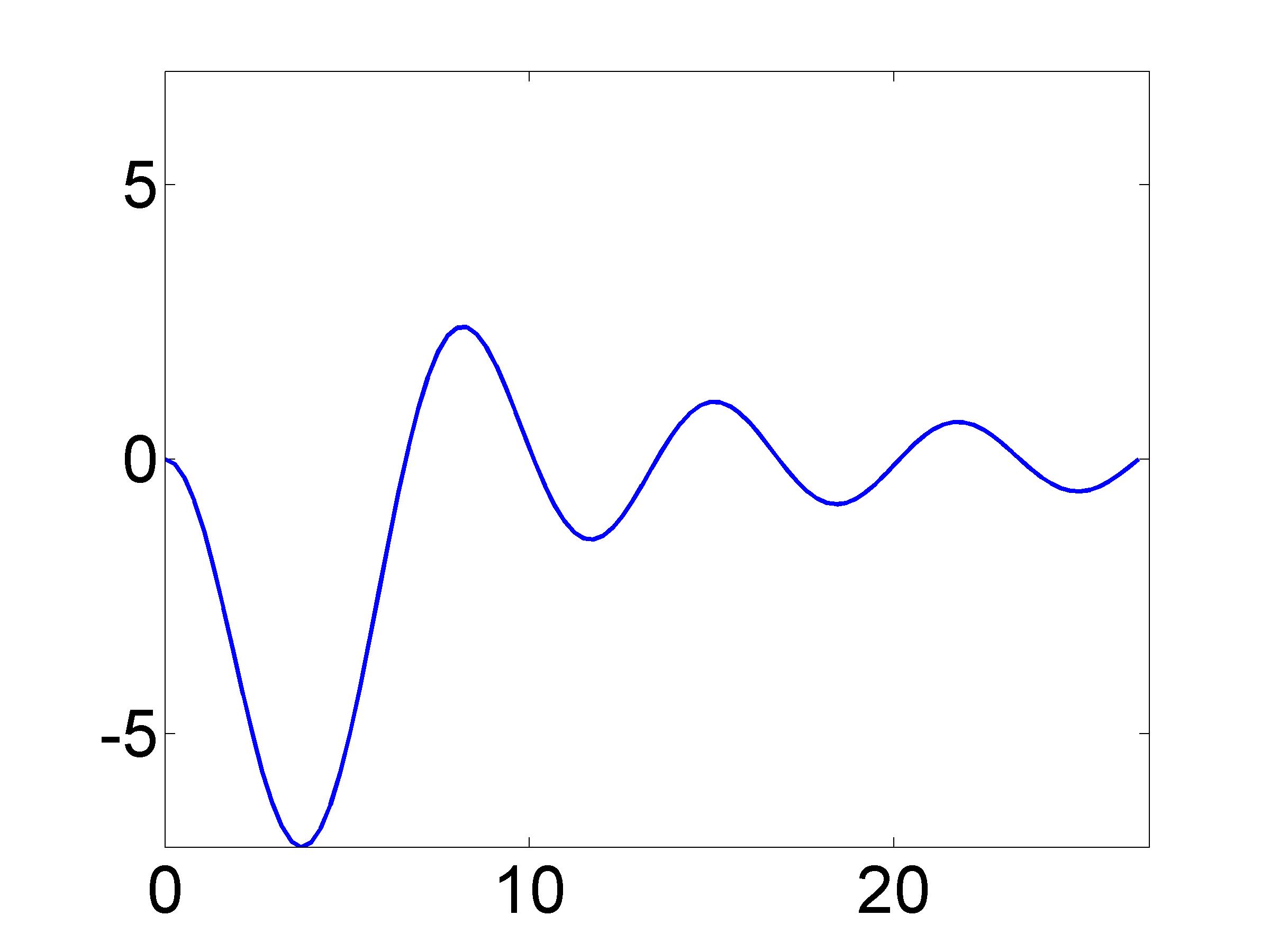}
	\caption{$f_2(r)$}
        \end{subfigure}
	\begin{subfigure}[b]{0.24\textwidth}
                \centering
                \includegraphics[scale = 0.04]{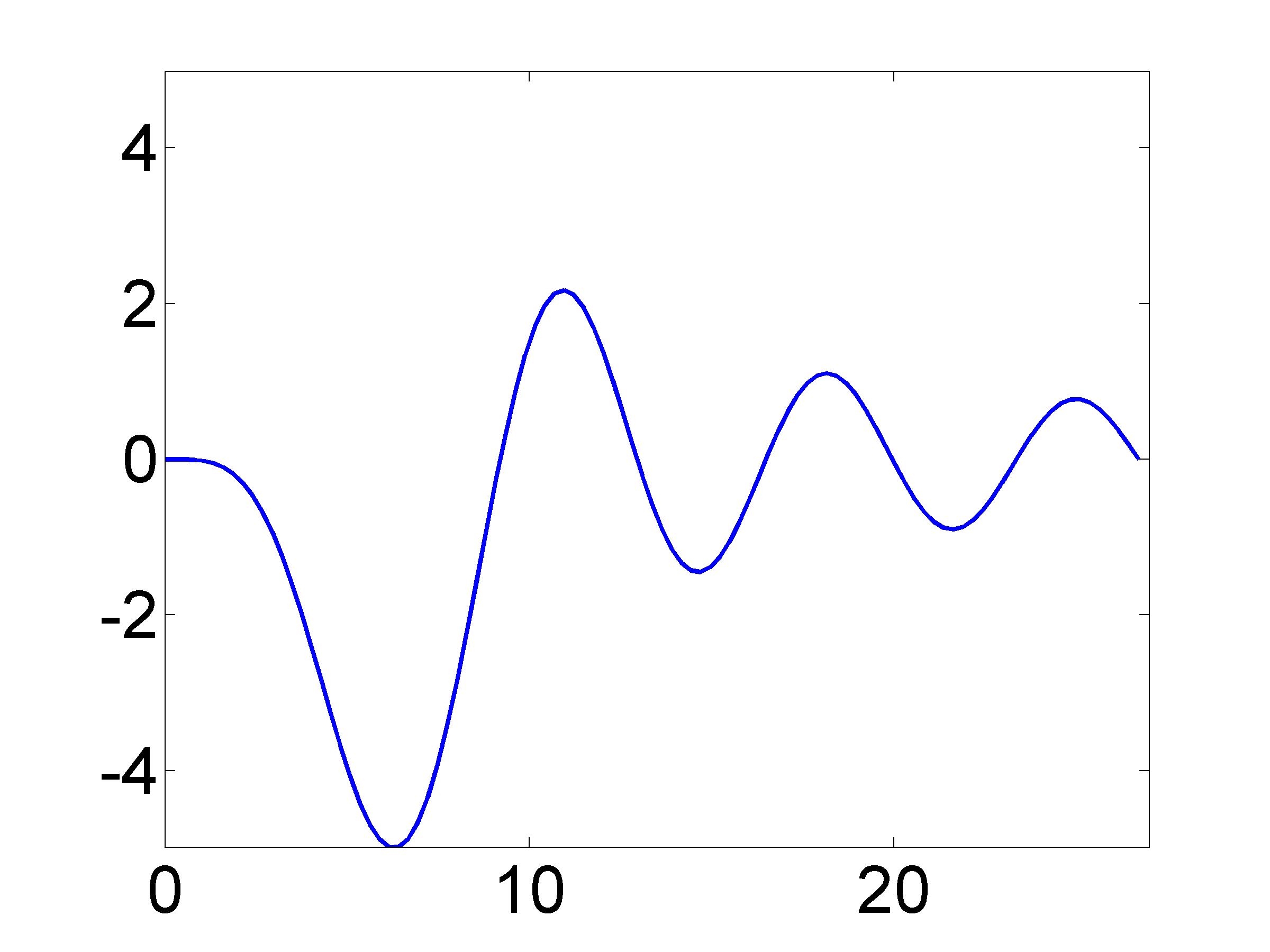}
	\caption{$f_4(r)$}
        \end{subfigure}
	\begin{subfigure}[b]{0.24\textwidth}
                \centering
                \includegraphics[scale = 0.04]{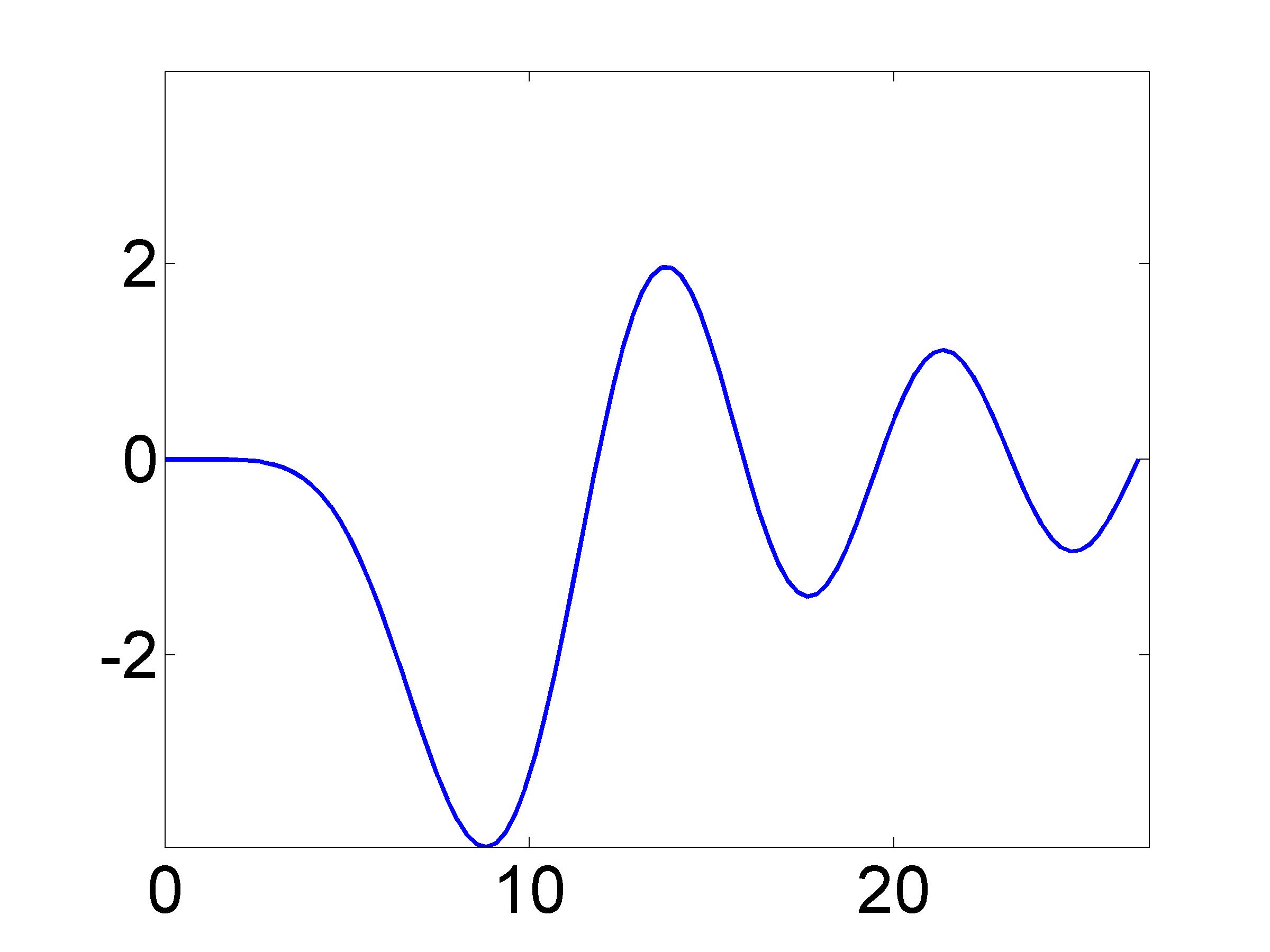}
	\caption{$f_6(r)$}
        \end{subfigure}
	\\
	\begin{subfigure}[b]{0.24\textwidth}
                \centering
                \includegraphics[scale = 0.04]{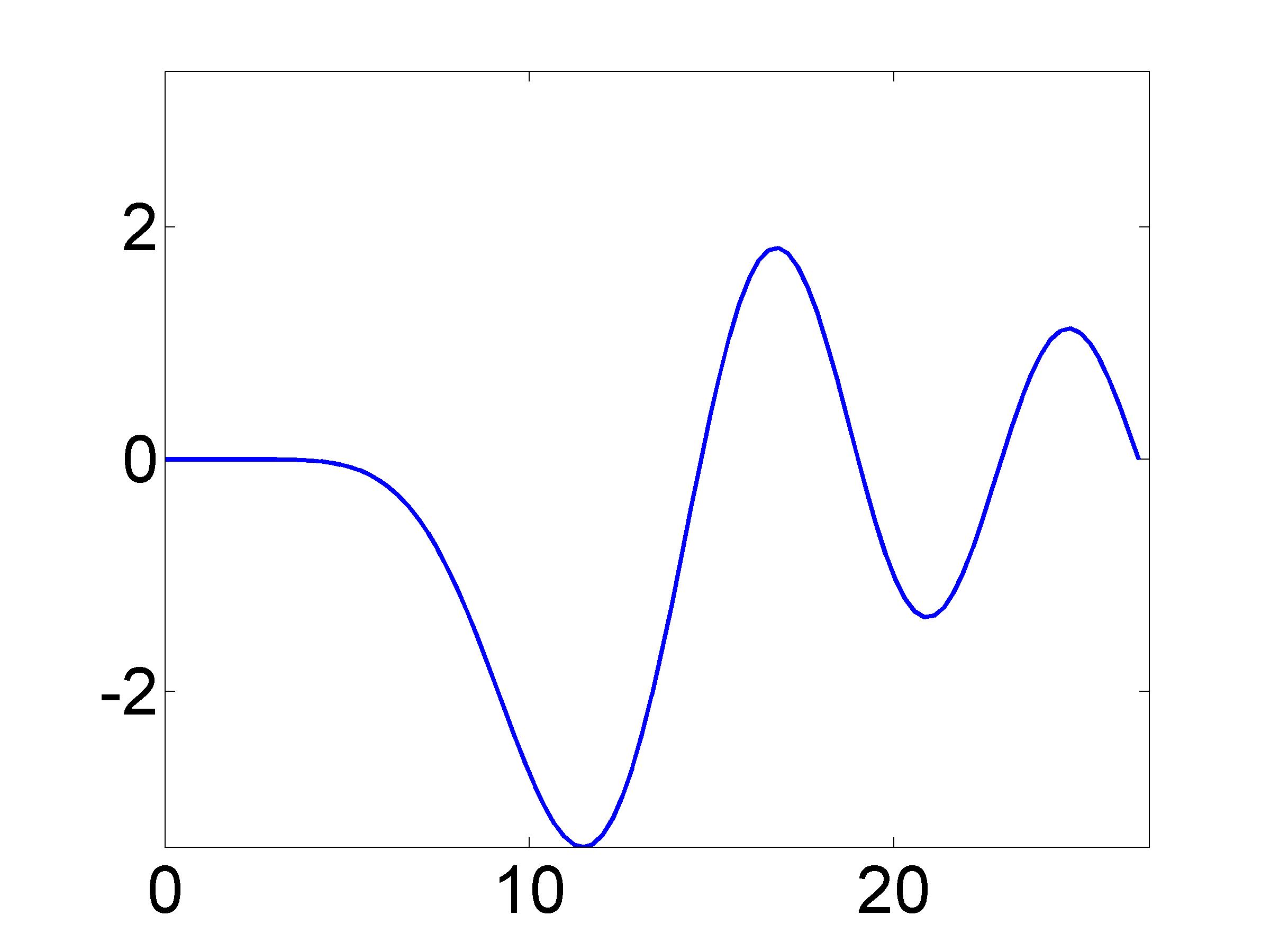}
	\caption{$f_8(r)$}
        \end{subfigure}
	\begin{subfigure}[b]{0.24\textwidth}
                \centering
                \includegraphics[scale = 0.04]{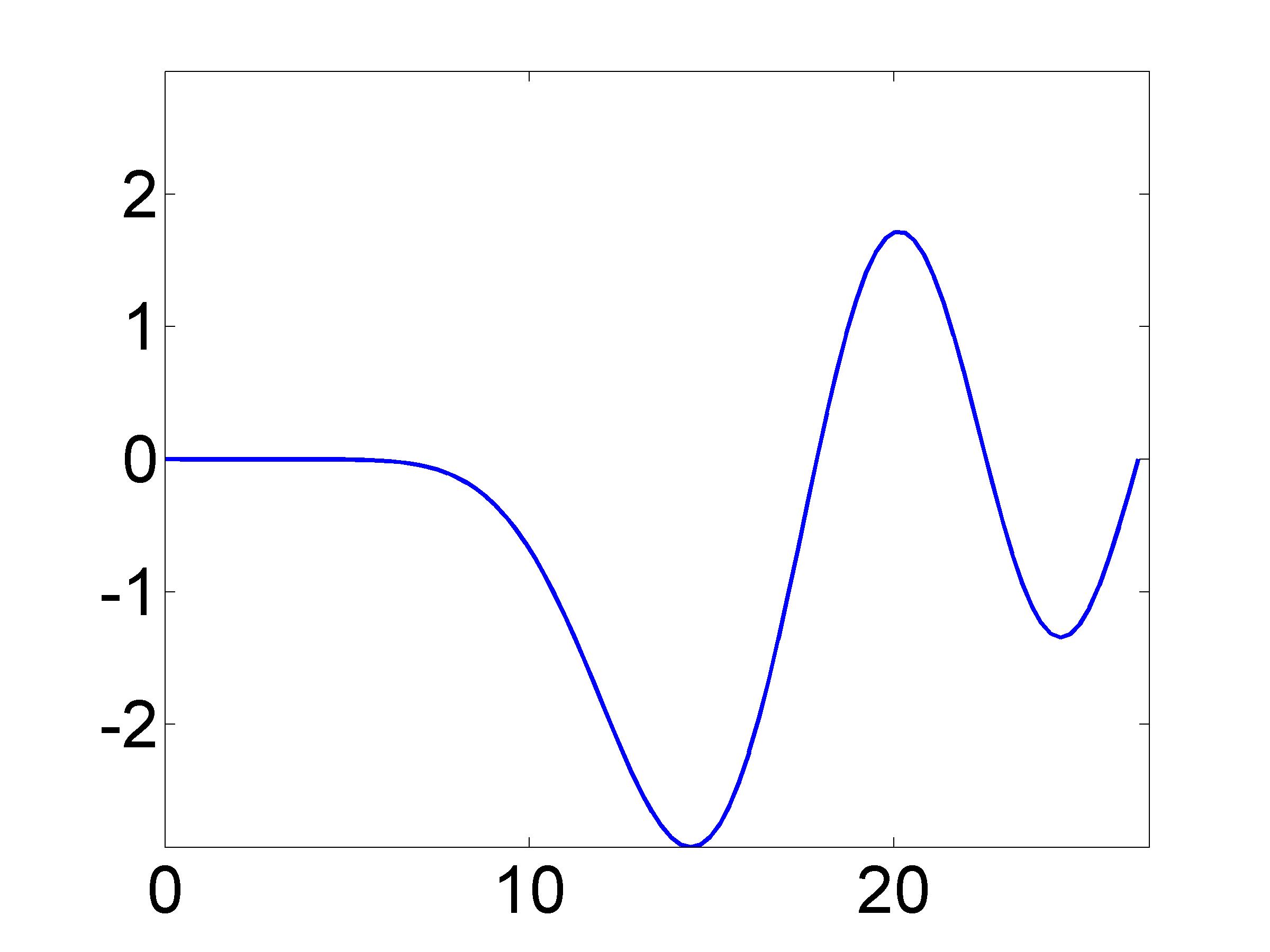}
	\caption{$f_{10}(r)$}
        \end{subfigure}
	\begin{subfigure}[b]{0.24\textwidth}
                \centering
                \includegraphics[scale = 0.04]{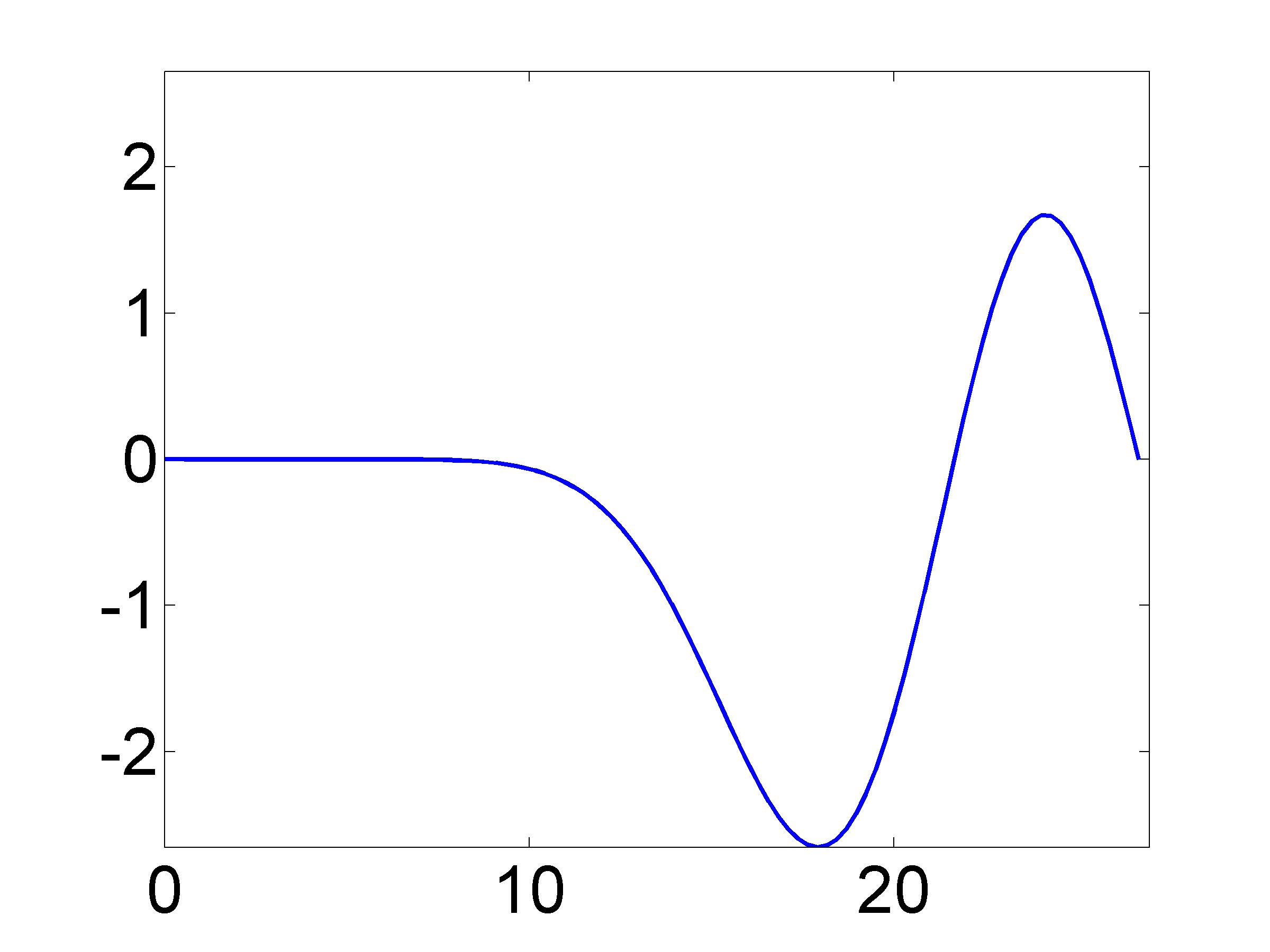}
	\caption{$f_{12}(r)$}
        \end{subfigure}
	\begin{subfigure}[b]{0.24\textwidth}
                \centering
                \includegraphics[scale = 0.04]{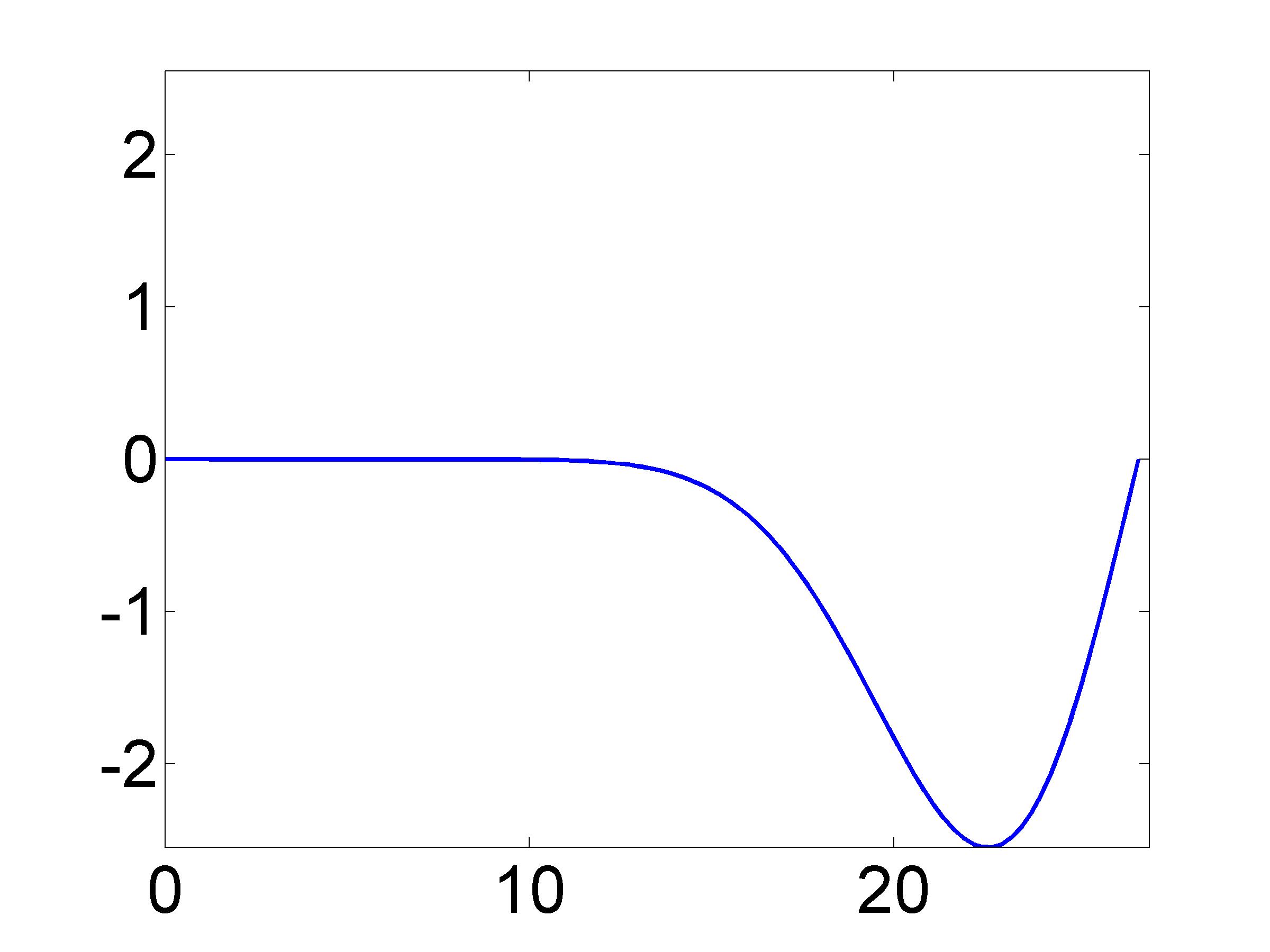}
	\caption{$f_{14}(r)$}
	\end{subfigure}
\caption{The even basis functions up to $f_{14}(r)$. Note that they become less oscillatory as $k$ increases, and that $f_k(r) \sim r^k$ at the origin. The odd basis functions have a similar structure and so are not pictured.}
\label{radial_basis}
\end{figure}

It remains to choose $K$. We do this based on how well criterion 4 is satisfied. For example, we can calculate how much energy of $S_\ell f_k$ is contained in the unit interval for all $0 \leq k \leq K, \ 0 \leq \ell \leq k, \ \ell = k$ mod 2. Numerical experiments show that $K = N_{\text{res}} - 2$ is a reasonable value. For each value of $N_{\text{res}}$ that we tested, this choice led to $S_\ell f_k$ having at least 80\% of its energy concentrated in the unit interval for each relevant $(\ell, k)$, and at least 95\% on average over all such pairs $(\ell, k)$. Thus our experiments show that for our choice of $f_k$, choosing roughly $K \approx N_{\text{res}}$ leads to acceptable satisfaction of criterion 4. A short calculation yields
\begin{equation}
\hat p = \text{dim}(\ssf V) = \sum_{k = 0}^K \frac{(k+1)(k+2)}{2} = \frac{(K+1)(K+2)(K+3)}{6} \approx \frac{N_{\text{res}}^3}{6} = \frac{4\omega_{\max}^3}{3\pi^3}.
\label{p_dim}
\end{equation}
Recall from (\ref{shannon_3}) that $p = \text{dim}(\mathscr B) = \frac{2}{9\pi}\omega_{\max}^3$. Hence, we have $\hat p/p = 6/\pi^2 \approx 0.6$. Hence, the dimension of the space $\ssf V$ we have constructed is within a constant factor of the dimension of $\mathscr B$. This factor is the price we pay for the computational simplicity $\ssf V$ provides. 

Note that a different construction of $f_k$ might have even better results. Choosing better radial functions can be the topic of further research. In any case, the specific choice of $f_k$ does not affect the structure of our algorithm at all because $\hat L$ is independent of these functions, as can be seen from (\ref{L_block}). Thus, the selection of the radial basis functions can be viewed as an independent module in our algorithm. The radial functions we choose here work well in numerical experiments; see Section \ref{num_results}.

\subsection{Constructing $\ssf{I}$} \label{properties_I}

Finally, the remaining piece in our construction is the finite-dimensional space of Fourier images, $\ssf I$. To motivate our construction, consider applying $\scf P_s$ to a basis element of $\ssf V$. The first observation to make is that the radial components $f_k(r)$ factor through $\hat{\mathcal P}_s$ completely:
\begin{equation}
\hat{\mathcal P}_s(f_k(r) Y_\ell^m(\theta, \vp)) = f_k(r)\hat{\mathcal P}_s(Y_\ell^m(\theta, \vp)).
\label{factor_through}
\end{equation}
Note that the $\hat{\mathcal P}_s$ on the LHS should be intepreted as $C(\br^3)\rightarrow C(\br^2)$, whereas the one on the RHS is the restricted map $C(S^2) \rightarrow C(S^1)$, which we also call $\hat{\mathcal P}_s$. The correct interpretation should be clear in each case. Viewed in this new way, $\hat{\mathcal P}_s: C(S^2) \rightarrow C(S^1)$ rotates a function on the sphere by $\sdr R_s \in SO(3)$, and then restricts the result to the equator. 
\begin{figure}
\centering
\epsfig{file = 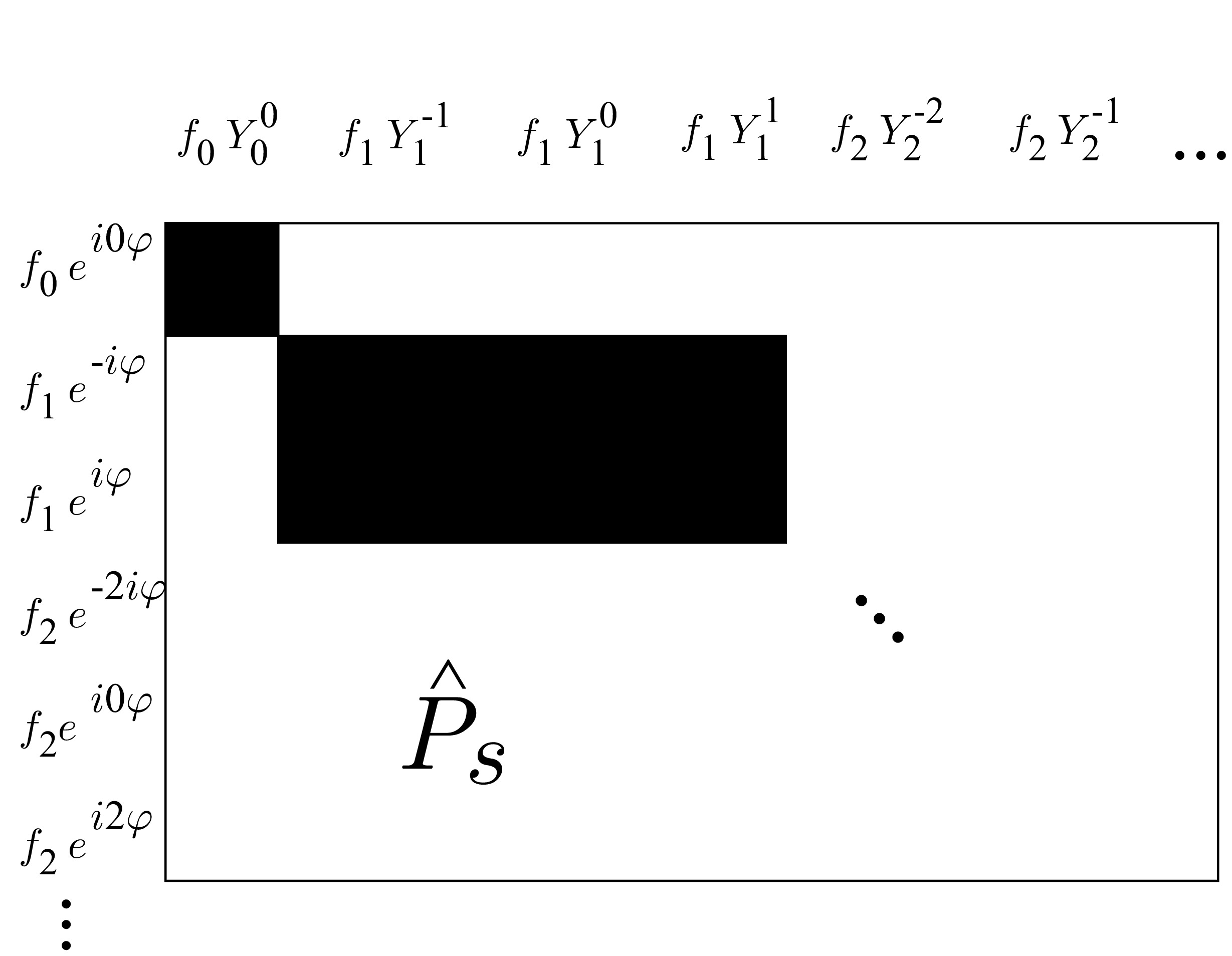, height = 5cm}
\caption{Block diagonal structure of $\hat P_s$. The shaded rectangles represent the nonzero entries. For an explanation of the specific pairing of angular and radial functions, see (\ref{g_def}) and (\ref{h_def}) and the preceding discussion. A short calculation shows that the $k$th block of $\hat P_s$ has size $(k+1) \times \frac{(k+1)(k+2)}{2}$.}
\label{tilde_P_s}
\end{figure}

By the rotational properties of spherical harmonics, a short calculation shows that
\begin{equation}
\begin{split}
\hat{\mathcal P}_s(Y_\ell^m(\theta, \vp)) = \sum_{\substack{|m'| \leq \ell\\ m' = \ell \text{ mod }2}} c_{\ell, m, m'}(R_s)\frac{1}{\sqrt{2\pi}}e^{im'\vp},
\end{split}
\end{equation}
where the constants $c_{\ell, m, m'}$ depend on the Wigner D matrices $D^\ell$ \cite{wigner}. Hence, $\scf P_s(\ssf V) \subset \ssf I$ if 
\begin{equation}
f_k(r) Y_\ell^m(\theta, \vp) \in \ssf V \Rightarrow \frac{1}{\sqrt{2\pi}}f_k(r) e^{im\vp} \in \ssf{I}, \quad m = -\ell, -\ell +2, \dots, \ell - 2, \ell. 
\label{restriction}
\end{equation}
Thus, we construct $\ssf I$ by pairing $f_k$ with $\frac{1}{\sqrt{2\pi}}e^{im\vp}$ if $k = m$ mod 2 and $m \leq k$. This leads to the 2D basis functions
\begin{equation}
\begin{split}
\{\hat g_i\} &= \left\{\frac{1}{\sqrt{2\pi}}f_0(r), \frac{1}{\sqrt{2\pi}}f_1(r)e^{-i\vp}, \frac{1}{\sqrt{2\pi}}f_1(r)e^{i\vp}\right., \\
&\quad\quad \left.\frac{1}{\sqrt{2\pi}}f_2(r)e^{-2i\vp}, \frac{1}{\sqrt{2\pi}}f_2(r),\frac{1}{\sqrt{2\pi}}f_2(r)e^{2i\vp}, \dots \right\}.
\end{split}
\label{g_def}
\end{equation}
Written another way, we construct
\begin{equation}
\ssf{I} = \text{span}\left(\left\{\frac{1}{\sqrt{2\pi}}f_k(r)e^{im\vp}: 0 \leq k \leq K,\  m = k \text{ (mod 2)}, \ |m| \leq k\right\}\right).
\label{I_def}
\end{equation}
If $\ssf I_k$ is the subspace of $\ssf I$ spanned by the basis functions with radial component $f_k$, (\ref{factor_through}) shows that $\scf P_s(\ssf V_k) \subset \ssf I_k$ for each $k$. Thus, $\sdf P_s$ has a block-diagonal structure, as depicted in Figure \ref{tilde_P_s}.

Let us now compare the dimension of $\ssf I$ to that of the corresponding space of 2D Slepian functions, as we did the previous section. We have
\begin{equation}
\hat q = \text{dim}(\ssf{I}) = \sum_{k = 0}^K (k+1) = \frac{(K+1)(K+2)}{2} \approx \frac{N_{\text{res}}^2}{2} = \frac{2\omega_{\max}^2}{\pi^2}.
\label{q_dim}
\end{equation}
The Shannon number in 2D corresponding to the bandlimit $\omega_{\max}$ is $\omega_{\max}^2/4$. Thus, we are short of this dimension by a constant factor of $8/\pi^2 \approx 0.8$. Another comparison to make is that the number grid points in the disc inscribed in the $N_{\text{res}} \times N_{\text{res}}$ grid is $\frac \pi 4 N_{\text{res}}^2 = \omega_{\max}^2/\pi$. Thus, $\text{dim}(\ssf{I})$ is short of this number by a factor of $\frac 2 \pi$. Note that this is the same factor that was obtained in a similar situation in \cite{fb}, so $\ssf I$ is comparable in terms of approximation to the Fourier-Bessel space constructed there.

Thus, by this point we have fully specified our algorithm for the heterogeneity problem. After finding $\hat \Sigma_n$ numerically via (\ref{linear_system}), we can proceed as in steps 6-9 of Algorithm \ref{high_level_algorithm} to solve Problem \ref{het_problem}. 

\section{Algorithm complexity} \label{consequences}

In this section, we explore the consequences of the constructions of $\ssf V$ and $\ssf I$ for the complexity of the proposed algorithm. We also compare this complexity with that of the straightforward CG approach discussed in Section (\ref{comp_challenges_approaches}).

To calculate the computational complexity of inverting the sparse matrix $\hat L^{k_1, k_2}$ via the CG algorithm, we must bound the number of nonzero elements of this matrix and its condition number.

\subsection{Sparsity of $\hat L$ and storage complexity} \label{sparsity_of_L}
Preliminary numerical experiments confirm the following conjecture:
\begin{conjecture} \label{L_sparsity_conjecture}
\begin{equation}
\text{\textnormal{nnz}}(\hat L^{k_1, k_2}) \leq \frac{1}{k_1 + k_2 + 1} \left(\frac{(k_1 +1)(k_1 + 2)(k_2 + 1)(k_2 + 2)}{4}\right)^2,
\label{L_sparsity}
\end{equation}
where $\text{\textnormal{nnz}}(A)$ is the number of nonzero elements in a matrix $A$, and the term involving the square is the total number of elements in $\hat L^{k_1, k_2}$.
\end{conjecture}

Hence, the percentage of nonzero elements in each block of $\hat L$ decays linearly with the frequencies associated with that block. This conjecture remains to be verified theoretically.

We pause here to note the storage complexity of the proposed algorithm, which is dominated by the cost of storing $\hat L$. In fact, since we process all the blocks separately, only storing one $\hat L^{k_1, k_2}$ at a time will suffice. Hence, the storage complexity is the memory required to store the largest block of $\hat L$, which is $\text{nnz}(\hat L^{K, K}) = O(K^7) = O(N_{\text{res}}^7)$.
%
Compare this to the required storage for a full matrix of the size of $\hat L$, which is $O(N_{\text{res}}^{12})$.

\subsection{Condition number of $\hat L$} \label{condition_number}

Here we find the condition number of each $\hat L^{k_1, k_2}$. We already proved in Proposition \ref{min_eig} that $\lambda_{\min}(\hat L) \geq 1/2\pi$. For any $k_1, k_2$, this implies that $\lambda_{\min}(\hat L^{k_1, k_2}) \geq 1/2\pi$. This is confirmed by a numerical experiment: in Figure \ref{fig:sm_eigs} are plotted the minimum eigenvalues of $\hat L^{k, k}$ for $0 \leq k \leq 15$. Note that the eigenvalues actually approach the value $1/2\pi$ (marked with a horizontal line) as $k$ increases.
\begin{figure}
\centering
        \begin{subfigure}[b]{0.45\textwidth}
                \centering
	\includegraphics[scale = 0.3]{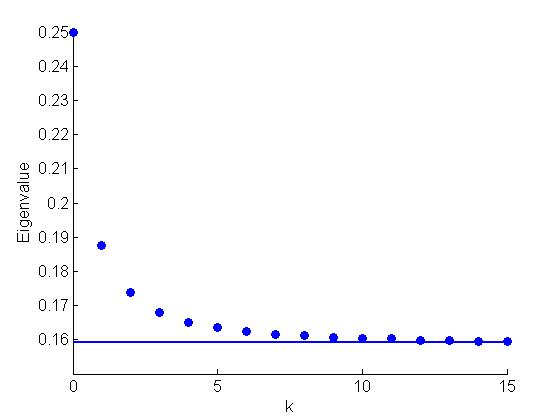}
	\caption{Smallest eigenvalues}
\label{fig:sm_eigs}
        \end{subfigure}
 \begin{subfigure}[b]{0.45\textwidth}
                \centering
	\includegraphics[scale = 0.3]{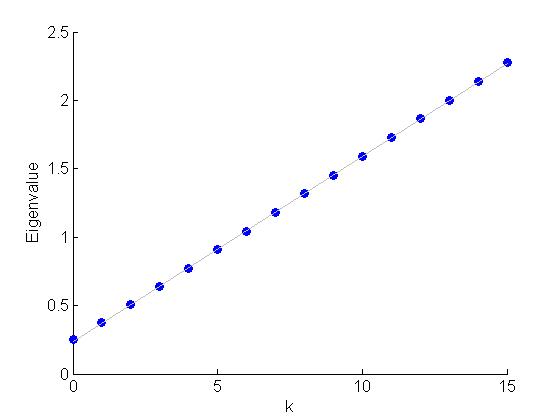}
	\caption{Largest eigenvalues}
\label{fig:lg_eigs}
        \end{subfigure}
\caption{The smallest and largest eigenvalues of (the continuous version of) $\hat L^{k, k}$, for $0 \leq k \leq 15$. The smallest eigenvalues approach their theoretical lower bound of $1/2\pi$ as $k$ increases. The largest eigenvalues show a clear linear dependence on $k$.}
\label{fig:Lcond}
\end{figure}
We remarked in Section \ref{cryo_em_theoretical} that an upper bound on the maximum eigenvalue is harder to find. Nevertheless, numerical experiments have led us to the following conjecture:
\begin{conjecture}\label{conjecture}
The maximal eigenvalue of $\hat L^{k_1, k_2}$ grows linearly with $\min(k_1, k_2)$.
\end{conjecture}

Moreover, a plot of the maximal eigenvalue of $\hat L^{k, k}$ shows a clear linear dependence on $k$. See Figure \ref{fig:lg_eigs}. The line of best fit is approximately
\begin{equation}
\text{maximum eigenvalue of }  \hat L^{k, k} =  0.2358 +  0.1357 k.
\end{equation}

Taken together, Proposition \ref{min_eig} and Conjecture \ref{conjecture} imply the following conjecture about the condition number of $\hat L^{k_1, k_2}$, which we denote by $\kappa(\hat L^{k_1, k_2})$:
\begin{conjecture}\label{cond_number}
\begin{equation}
\kappa(\hat L^{k_1, k_2}) \leq 1.4818 + 0.8524 \min(k_1, k_2).
\end{equation}
\end{conjecture}

In particular, this implies that 
\begin{equation}
\kappa(\hat L) \leq 1.4818 + 0.8524K.
\label{cond_L}
\end{equation}

\subsection{Algorithm complexity}

Using the above results, we estimate the computational complexity of Algorithm \ref{high_level_algorithm}. We proceed step by step through the algorithm and estimate the complexity at each stage. Before we do so, note that due to the block-diagonal structure of $\hat P_s$ (depicted in Figure \ref{tilde_P_s}), it can be easily shown that an application of $\hat P_s$ or $\hat P_s^H$ costs $O(K^4)$. 

Sending the images from the pixel domain into $\ssf I$ requires $n$ applications of the matrix $Q_1 \in \bc^{\hat q \times q}$, which costs $O(nq\hat q) = O(nN^2N_{\text{res}}^2)$. Note that this complexity can be improved using an algorithm of the type \cite{special_function_transforms}, but in this paper we do not delve into the details of this alternative.

Finding $\hat \mu_n$ from (\ref{mu_system}) requires $n$ applications of the matrix $\hat P_s^H$, and so has complexity $O(nK^4) = O(n N_{\text{res}}^4)$. 

Next, we must compute the matrix $\hat B_n$. Note that the second term in $\hat B_n$ can be replaced by a multiple of the identity matrix by (\ref{ramp_cancellation}), so only the first term of $\hat B_n$ must be computed. Note that $\hat B_n$ is a sum of $n$ matrices, and each matrix can be found as the outer product of $\hat P_s^H(\hat I_s - \hat P_s \hat \mu_n) \in \bc^{\hat p}$ with itself. Calculating this vector has complexity $O(K^4)$, from which it follows that calculating $\hat B_n$ costs $O(nK^4) = O(nN_{\text{res}}^4)$. 

Next, we must invert $\hat L$. As mentioned in Section \ref{comp_challenges_approaches}, the inversion of a matrix $A$ via CG takes $\sqrt{\kappa(A)}$ iterations. If $A$ is sparse, than applying it to a vector has complexity $\text{nnz}(A)$. Hence, the total complexity for inverting a sparse matrix is  $\sqrt{\kappa(A)}\text{nnz}(A)$. Conjectures \ref{L_sparsity_conjecture} and \ref{cond_number} imply that 
\begin{equation}
\begin{split}
\text{complexity of inverting $\hat L$} &\lesssim \sum_{k_1, k_2 = 0}^K \sqrt{\kappa(\hat L^{k_1, k_2})}\text{\textnormal{nnz}}(\hat L^{k_1, k_2}) \\
&\lesssim \sum_{k_1, k_2 = 0}^K\sqrt{\min(k_1, k_2)}\frac{1}{k_1 + k_2 + 1} \left(\frac{(k_1 +1)(k_1 + 2)(k_2 + 1)(k_2 + 2)}{4}\right)^2 \\
&\lesssim \sum_{k_1, k_2 = 0}^K(k_1 k_2)^{1/4}\frac{1}{\sqrt{k_1 k_2}} k_1^4 k_2^4\\
&\lesssim \sum_{k_1 = 0}^K k_1^{3.75}\sum_{k_2 = 0}^K k_2^{3.75} \lesssim K^{4.75}K^{4.75} = K^{9.5}.
\end{split}
\end{equation}
Since $\hat L$ has size of the order $K^6 \times K^6$, note that the complexity of inverting a full matrix of this size would be $K^{18}$. Thus, our efforts to make $\hat L$ sparse have saved us a $K^{8.5}$ complexity factor.  Moreover, the fact that $\hat L$ is block diagonal makes its inversion parallelizable.

Assuming that $C = O(1)$, solving each of the $n$ least-squares problems (\ref{least_squares}) is dominated by a constant number of applications of $\hat P_s$ to a vector. Thus, finding $\alpha_s$ for $s = 1, \dots, n$ costs $O(nN_{\text{res}}^4)$. 

Next, we must fit a mixture of Gaussians to $\alpha_s$ to find $\alpha^c$. An EM approach to this problem to this problem requires $O(n)$ operations per iteration. Assuming that the number of iterations is constant, finding $\alpha^c$ has complexity $O(n)$. 

Finally, reconstructing $\hat X^c$ via (\ref{factorial}) has complexity $O(N_\text{res}^3)$.

Hence, neglecting lower-order terms, we find that the total complexity of our algorithm is 
\begin{equation}
O(nN^2N_{\text{res}}^2 + N_{\text{res}}^{9.5}).
\label{complexity}
\end{equation}

\subsection{Comparison to straightforward CG approach}

We mentioned in Section \ref{comp_challenges_approaches} that a CG approach is possible in which at each iteration, we apply $\hat L_n$ to $\hat \Sigma$ using the definition (\ref{sigeq_new}). This approach has the advantage of not requiring uniformly spaced viewing directions. While the condition number of $\hat L_n$ depends on the rotations $R_1, \dots, R_n$, let us assume here that $\kappa(\hat L_n) \approx \kappa(\hat L)$. We estimated the computational complexity of this approach in Section \ref{comp_challenges_approaches}, but at that point we assumed that each $\hat P_s$ was a full matrix. If we use the bases $\ssf V$ and $\ssf I$, we reap the benefit of the block-diagonal structure of $\hat P_s$. Hence, for each $s$, evaluating $\hat P_s^H \hat P_s \hat \Sigma \hat P_s^H \hat P_s$ is dominated by the multiplication $\hat P_s \hat \Sigma$, which has complexity $N_{\text{res}}^7$. Hence, applying $\hat L_n$ to $\hat \Sigma$ has complexity $n N_{\text{res}}^7$. By (\ref{cond_L}), we assume that $\kappa(\hat L_n) = O(N_{\text{res}})$. Hence, the full complexity of inverting $\hat L$ using the conjugate gradient approach is 
\begin{equation}
O(n N_{\text{res}}^{7.5}). 
\end{equation}
Compare this to a complexity of $O(N_{\text{res}}^{9.5})$ for inverting $\hat L$. Given that $n$ is usually on the order of $10^5$ or $10^6$, for moderate values of $N_{\text{res}}$ we have $N_{\text{res}}^{9.5} \ll n N_{\text{res}}^{7.5}$. Nevertheless, both algorithms have possibilities for parallelization, which might change their relative complexities. As for memory requirements, note that the straightforward CG algorithm only requires $O(N_{\text{res}}^6)$ storage, whereas we saw in Section \ref{sparsity_of_L} that the proposed algorithm requires $O(N_{\text{res}}^7)$ storage.

In summary, these two algorithms each have their strengths and weaknesses, and it would be interesting to write parallel implementations for both and compare their performances. In the present paper, we have implemented and tested only the algorithm based on inverting $\hat L$. 

\section{Numerical results}\label{num_results}

Here, we provide numerical results illustrating Algorithm \ref{high_level_algorithm}, with the bases $\ssf I$ and $\ssf V$ chosen so as to make $\hat L$ sparse, as discussed in Section \ref{practical}. The results presented below are intended for proof-of-concept purposes, and they demonstrate the qualitative behavior of the algorithm. They are not, however, biologically significant results. We have considered an idealized setup in which there is no CTF effect, and have assumed that the rotations $R_s$ (and translations) have been estimated perfectly. In this way, we do not perform a ``full-cycle" experiment, starting from only the noisy images. Therefore, we cannot gauge the overall effect of noise on our algorithm because we do not account for its contribution to the misspecification of rotations; we investigate the effect of noise on the algorithm only after the rotation estimation step. Moreover, we use simulated data instead of experimental data. The application of our algorithm to experimental datasets is left for a separate publication.

\subsection{An appropriate definition of SNR} \label{snr_fsc}

Generally, the definition of SNR is
\begin{equation}
\text{SNR} = \frac{P(\text{signal})}{P(\text{noise})},
\end{equation}
where $P$ denotes power. In our setup, we will find appropriate definitions for both $P(\text{signal})$ and $P(\text{noise})$. Let us consider first the noise power. The standard definition is $P(\text{noise}) = \sigma^2$. However, note that in our case, the noise has a power of $\sigma^2$ in each pixel of an $N \times N$ grid, but we reconstruct the volumes to a bandlimit $\omega_{\max}$, corresponding to $N_{\text{res}}$. Hence, if we downsampled the $N \times N$ images to size $N_{\text{res}} \times N_{\text{res}}$, then we would still obey the Nyquist criterion (assuming the volumes actually are bandlimited by $\omega_{\max}$). This would have the effect of reducing the noise power by a factor of $N_{\text{res}}^2/N^2$. Hence, in the context of our problem, we define
\begin{equation}
P(\text{noise}) = \frac{N_{\text{res}}^2}{N^2}\sigma^2.
\end{equation}

Now, consider $P(\text{signal})$. In standard SPR, a working definition of signal power is 
\begin{equation}
P(\text{signal}) = \frac{1}{n}\sum_{s = 1}^n \frac{1}{q}\norm{P_s X_s}^2, 
\label{signal_power}
\end{equation}
However, in the case of the heterogeneity problem, the object we are trying to reconstruct is not the volume itself, but rather the deviation from the average volume, due to heterogeneity. Thus, the relevant signal to us is not the images themselves, but the parts of the images that correspond to projections of the deviations of $X_s$ from $\mu_0$. Hence, a natural definition of signal power in our case is
\begin{equation}
P(\text{signal}_{\text{het}}) = \frac{1}{n}\sum_{s=1}^n\frac1q \norm{P_s(X_s - \mu_0)}^2.
\label{signal_power_het}
\end{equation}
Using the above definitions, let us define SNR$_{\text{het}}$ in our problem by
\begin{equation}
\text{SNR}_\text{het} = \frac{P(\text{signal}_{\text{het}})}{P(\text{noise})} = \frac{\frac{1}{qn}\sum_{s=1}^n \norm{P_s(X_s - \mu_0)}^2}{\sigma^2 N_{\text{res}}^2/ N^2}.
\label{SNR_het_def}
\end{equation}
Even with the correction factor $N_{\text{res}}^2/ N^2$, SNR$_{\text{het}}$ values are lower than the SNR values usually encountered in structural biology. Hence, we also define
\begin{equation}
\text{SNR} = \frac{P(\text{signal})}{P(\text{noise})} = \frac{\frac{1}{n}\sum_{s = 1}^n \frac{1}{q}\norm{P_s X_s}^2}{\sigma^2 N_{\text{res}}^2/ N^2}.
\label{SNR_def}
\end{equation}
We will present our numerical results primarily using SNR$_{\text{het}}$, but we will also provide the corresponding SNR values in parentheses.

\begin{figure}[h]
        \centering
        \begin{subfigure}[b]{0.24\textwidth}
                \centering
                \includegraphics[scale = 0.04]{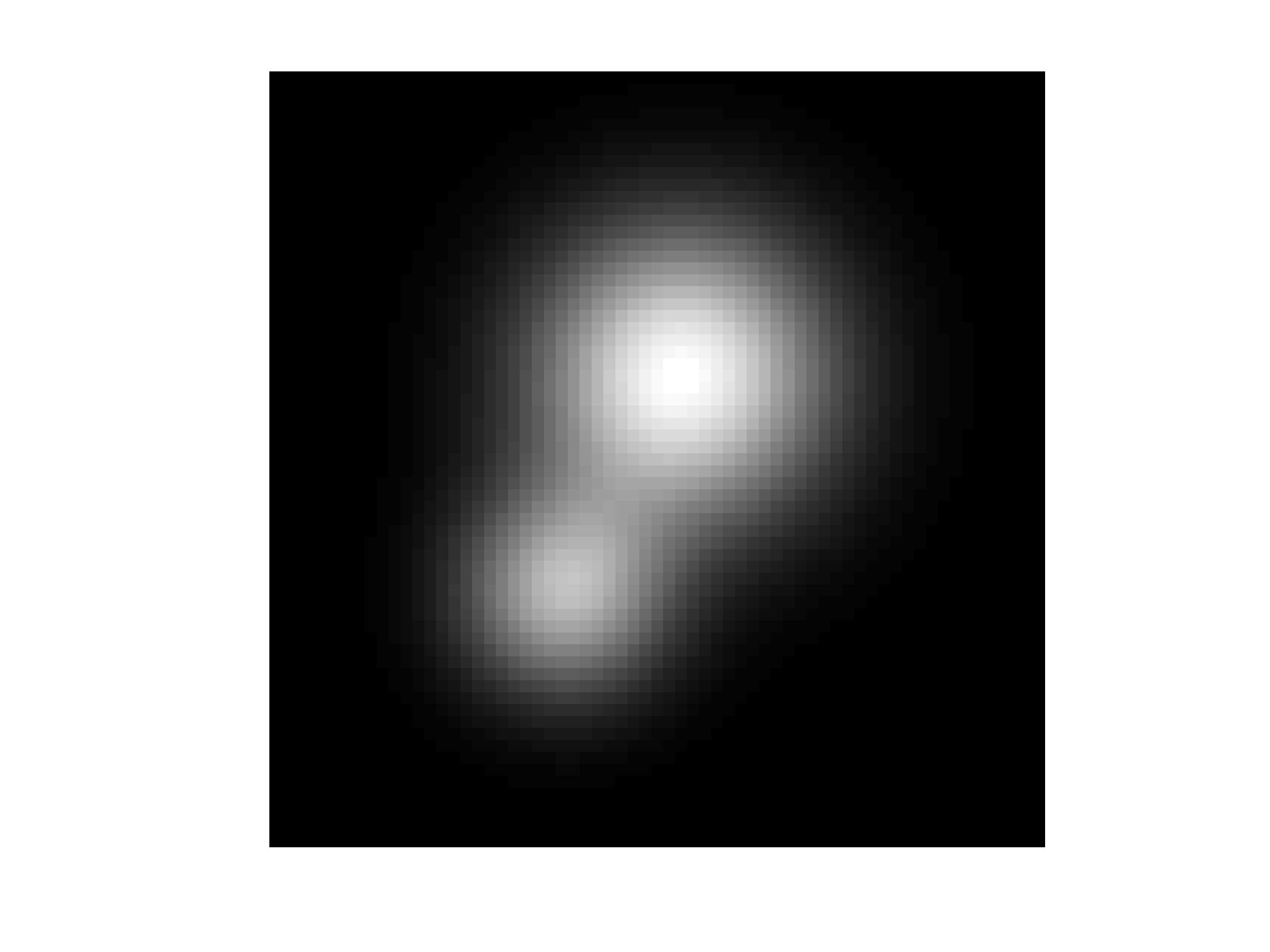}
        \end{subfigure}
	\begin{subfigure}[b]{0.24\textwidth}
                \centering
                \includegraphics[scale = 0.04]{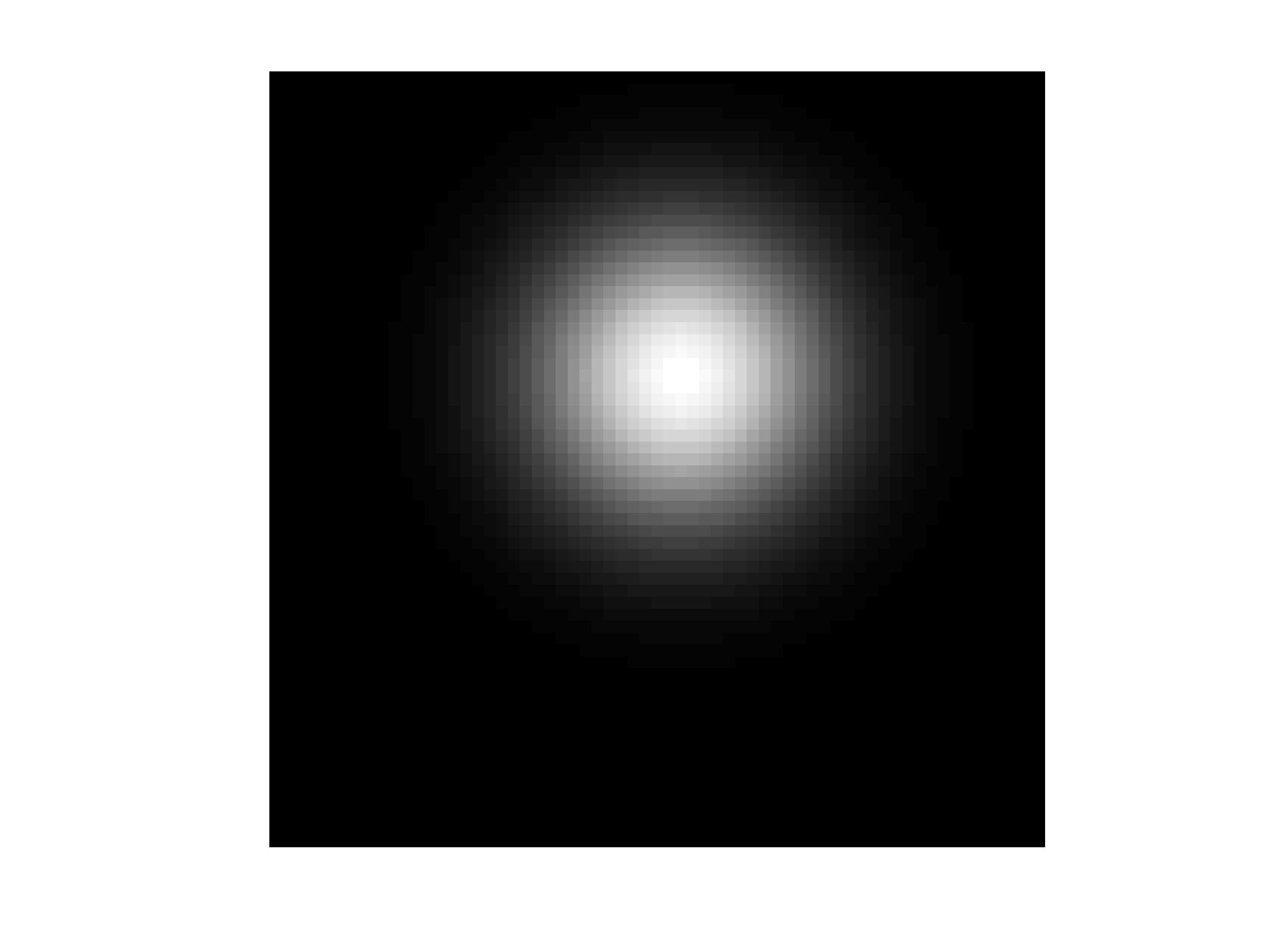}
        \end{subfigure}
	\begin{subfigure}[b]{0.24\textwidth}
                \centering
                \includegraphics[scale = 0.04]{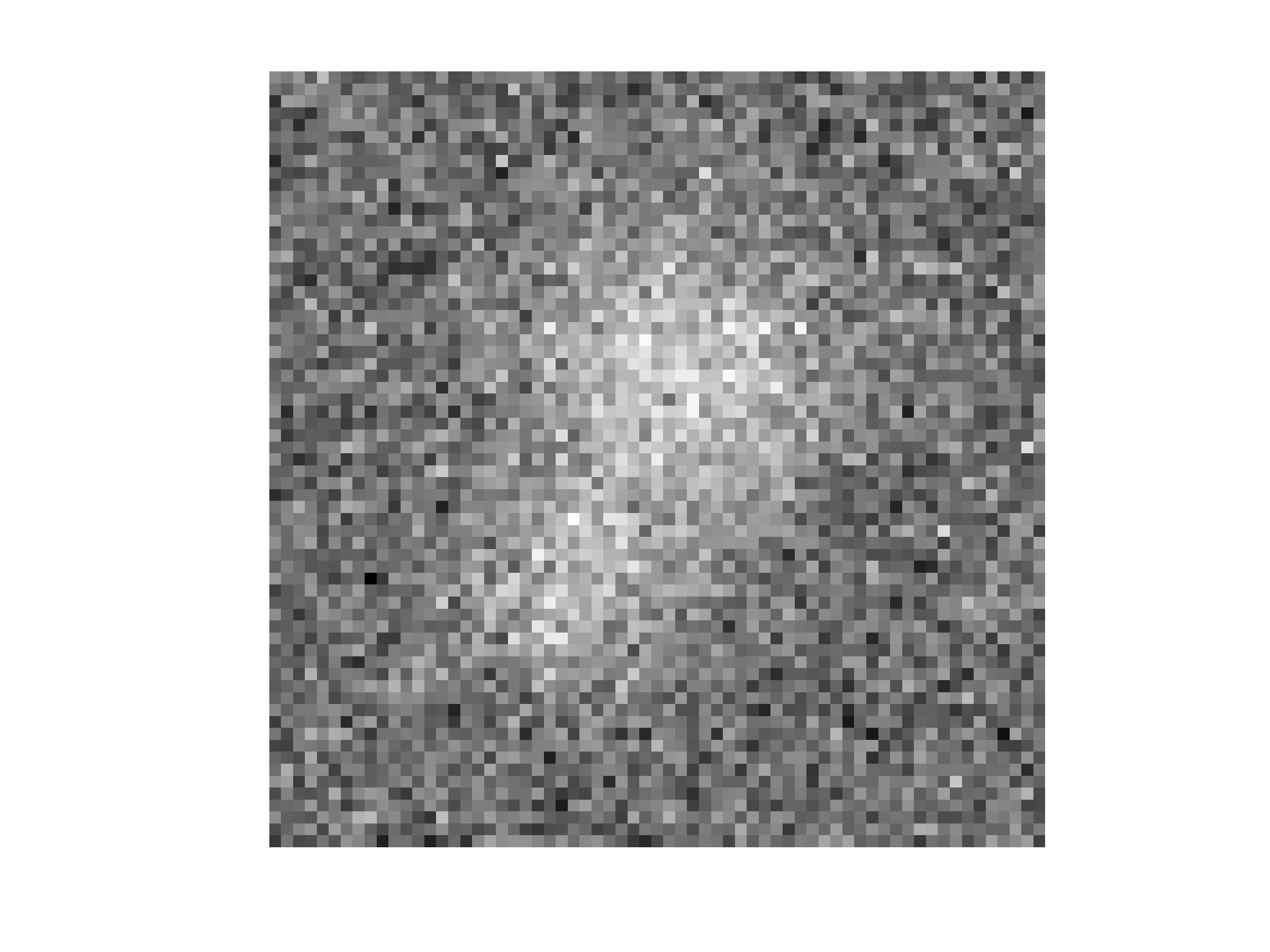}
        \end{subfigure}
	\begin{subfigure}[b]{0.24\textwidth}
                \centering
                \includegraphics[scale = 0.04]{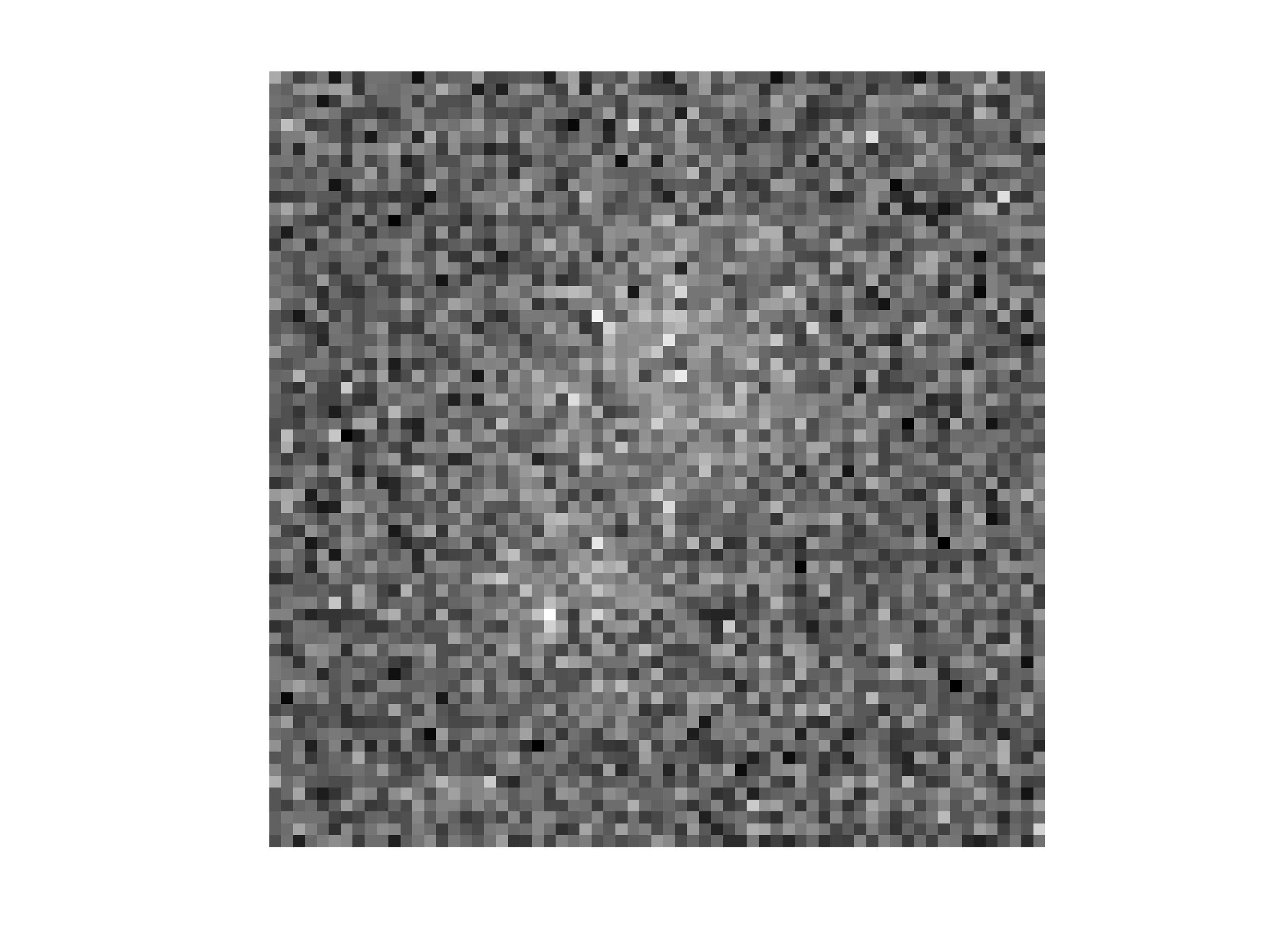}
       \end{subfigure} \\

	\begin{subfigure}[b]{0.24\textwidth}
		\centering
		\includegraphics[scale = 0.04]{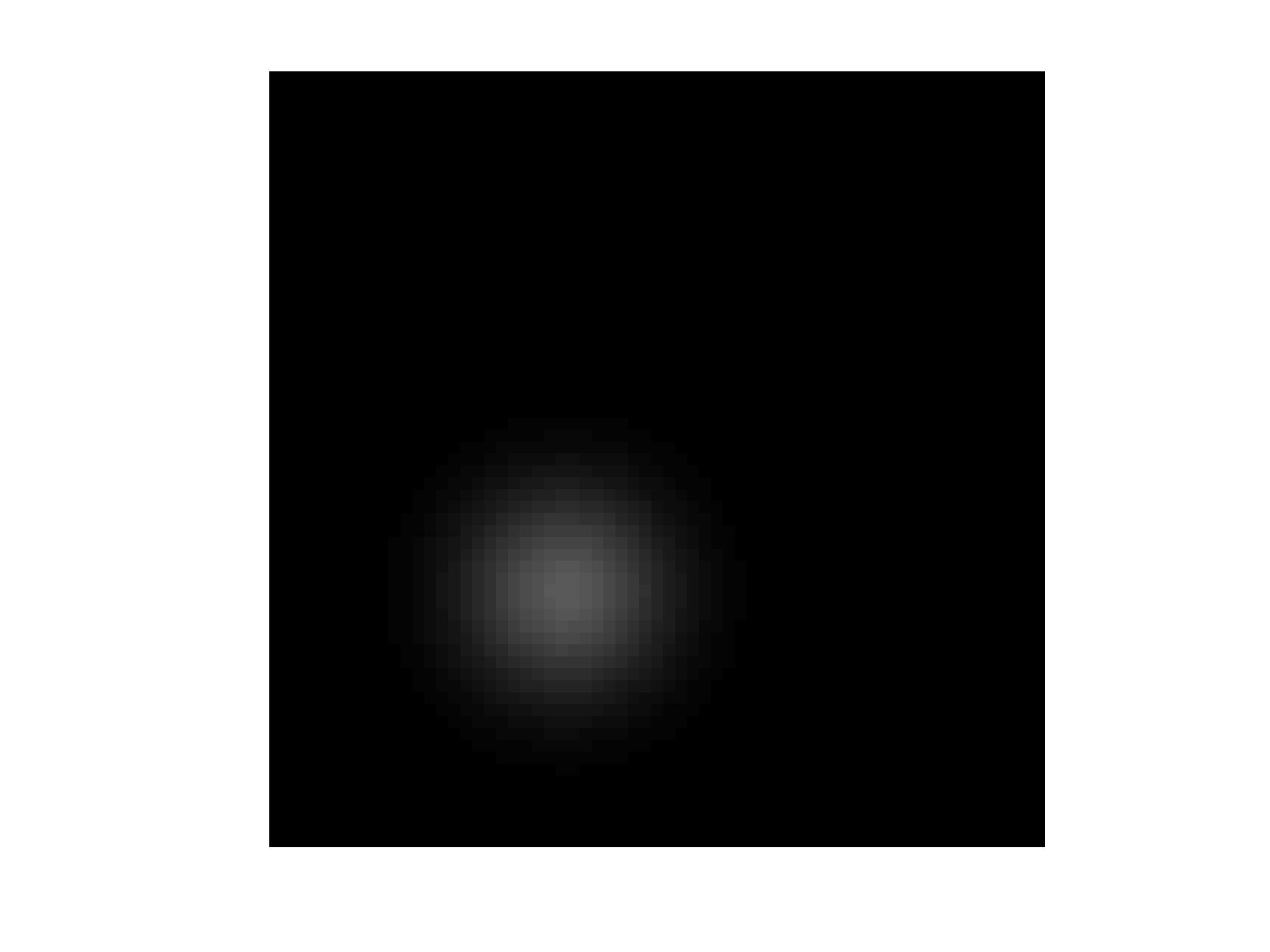}
		\caption{Class 1 (Clean)}
        \end{subfigure}
\begin{subfigure}[b]{0.24\textwidth}
                \centering
                \includegraphics[scale = 0.17]{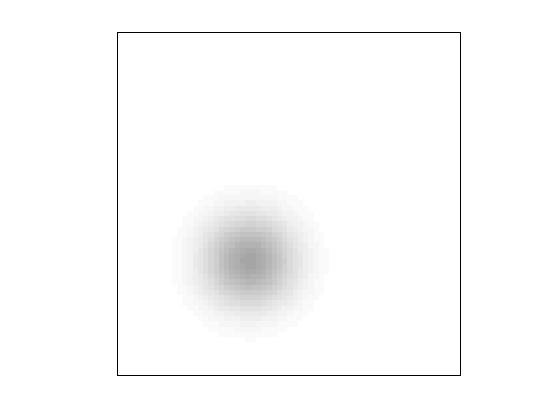}
		\caption{Class 2 (Clean)}
        \end{subfigure}
	\begin{subfigure}[b]{0.24\textwidth}
		\centering
		\includegraphics[scale = 0.04]{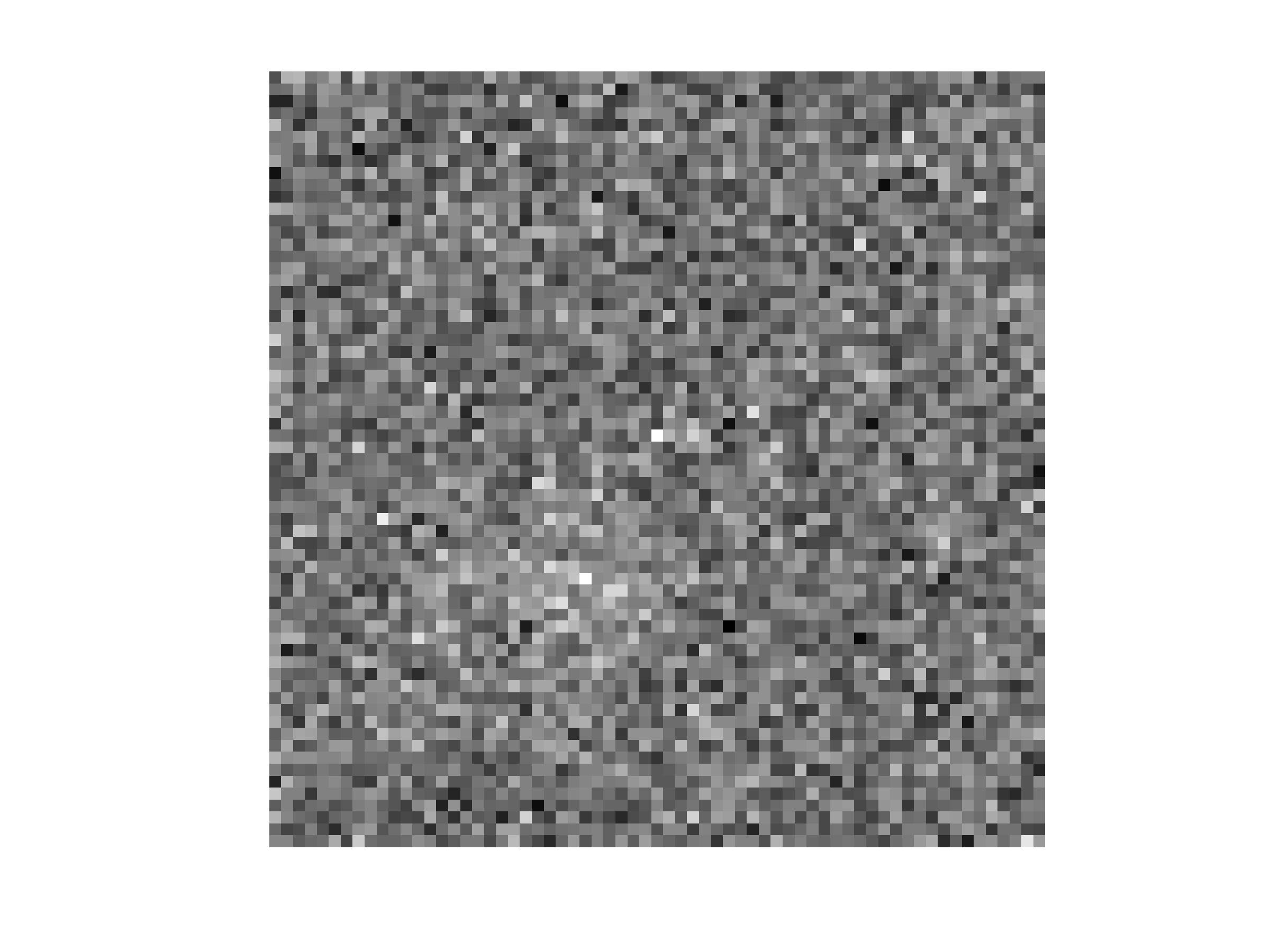}
	\caption{SNR = 0.96}
        \end{subfigure}
	\begin{subfigure}[b]{0.24\textwidth}
                \centering
                \includegraphics[scale = 0.04]{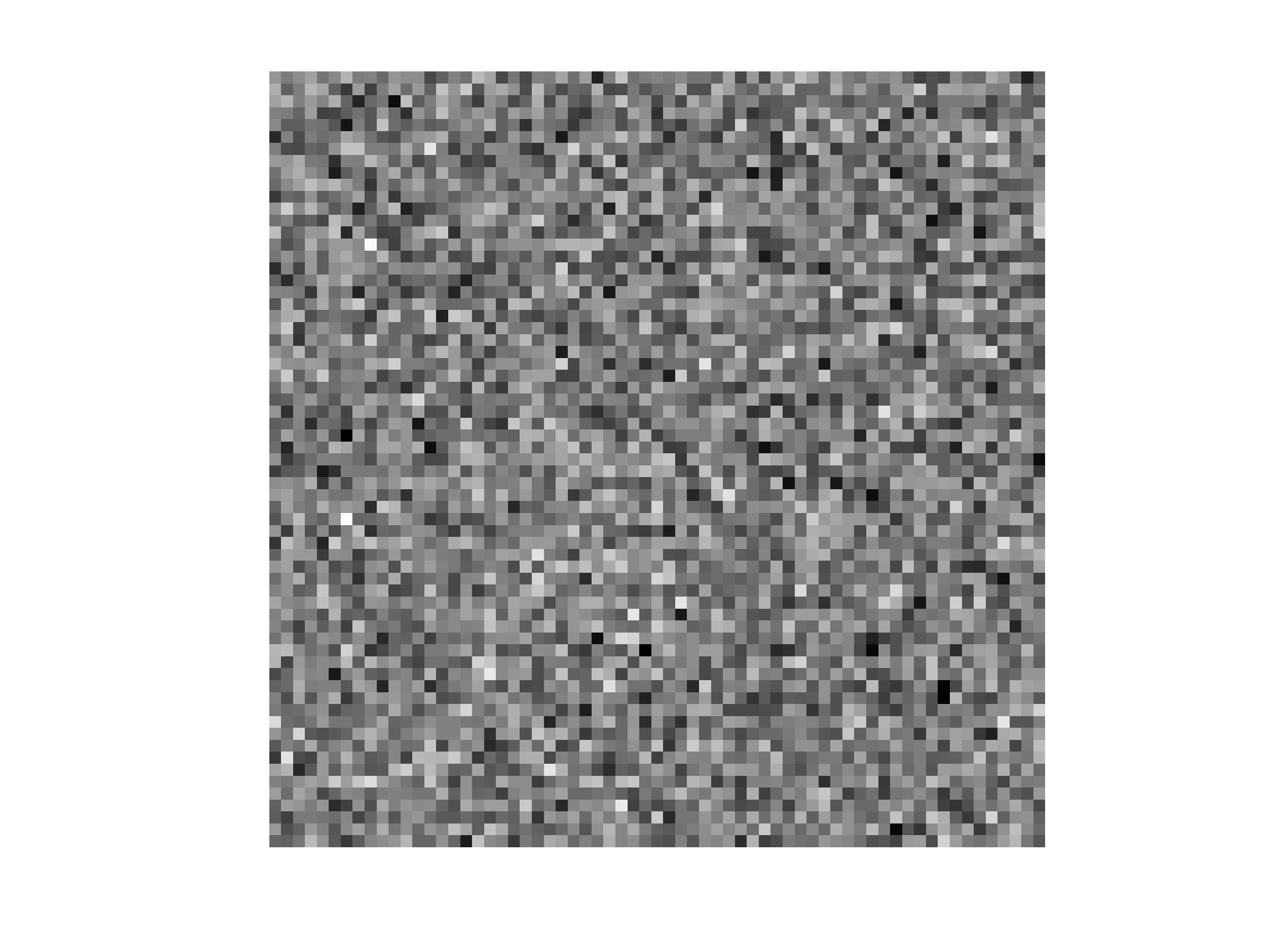}
	\caption{SNR = 0.19}
        \end{subfigure}
\caption{This figure depicts the effect of mean-subtraction on projection images in the context of a two-class heterogeneity. The bottom row projections obtained from the top row by mean-subtraction. Columns (a) and (b) are clean projection images of the two classes from a fixed viewing angle. Columns (c) and (d) are both noisy versions of column (a). The image in the top row of column (c) has an SNR of 0.96, but the SNR of the corresponding mean-subtracted image is only 0.05. In column (d), the top image has an SNR of 0.19, but the mean-subtracted image has SNR 0.01. Note: the SNR values here are not normalized by $N_{\text{res}}^2/N^2$ in order to illustrate the signal present in a projection image. }
\label{noisy_projs}
\end{figure}

To get a sense of the difference between this definition of SNR and the conventional one, compare the signal strength in a projection image to that in a mean-subtracted projection image in Figure \ref{noisy_projs}. 

\subsection{Experimental procedure}

We performed three numerical experiments: one with two heterogeneity classes, one with three heterogeneity classes, and one with continuous variation along the perimeter of a triangle defined by three volumes. The first two demonstrate our algorithm in the setup of Problem \ref{het_problem}, and the third shows that we can estimate the covariance matrix and discover a low-dimensional structure in more general setups than the discrete heterogeneity case. 

As a first step in each of the experiments, we created a number of phantoms analytically. We chose the phantoms to be linear combinations of Gaussian densities:
\begin{equation}
\scr X^c(r) = \sum_{i = 1}^{M_c} a_{i, c}\exp\left(-\frac{\norm{r - r_{i, c}}^2}{2\sigma_{i, c}^2}\right), \quad r_{i, c} \in \br^3, a_{i, c}, \sigma_{i, c}, \in \br_+, \quad c = 1, \dots, C.
\label{phantom_eq}
\end{equation}

For the discrete heterogeneity cases, we chose probabilities $p_1, \dots, p_C$ and generated $\scr X_1, \dots, \scr X_n$ by sampling from $\scr X^1, \dots, \scr X^C$ accordingly. For the continuous heterogeneity case, we generated each $\scr X_s$ by choosing a point uniformly at random from the perimeter of the triangle defined by $\scr X^1, \scr X^2, \scr X^3$. 

For all of our experiments, we chose $n = 10000, \ N = 65, \ N_{\text{res}} = 17, \ K = 15$, and selected the set of rotations $\sdr R_s$ to be approximately uniformly distributed on $SO(3)$. For each $\sdr R_s$, we calculated the clean continuous projection image $\mathcal P_s \scr X_s$ analytically, and then sampled the result on an $N\times N$ grid. Then, for each SNR level, we used (\ref{SNR_het_def}) to find the noise power $\sigma^2$ to add to the images. 

After simulating the data, we ran Algorithm \ref{high_level_algorithm} on the images $I_s$ and rotations $\sdr R_s$ on an Intel i7-3615QM CPU with 8 cores, and 8 GB of RAM. The runtime for the entire algorithm with the above parameter values (excluding precomputations) is 257 seconds. For the continuous heterogeneity case, we stopped the algorithm after computing the coordinates $\alpha_{s}$ (we did not attempt to reconstruct individual volumes in this case). To quantify the resolution of our reconstructions, we use the Fourier Shell Correlation (FSC), defined as the correlation of the reconstruction with the ground truth on each spherical shell in Fourier space \cite{fsc}. For the discrete cases, we calculated FSC curves for the mean, the top eigenvectors, and the mean-subtracted reconstructed volumes. We also plotted the correlations of the mean, top eigenvectors, and mean-subtracted volumes with the corresponding ground truths for a range of SNR values. Finally, we plotted the coordinates $\alpha_{s}$. For the continuous heterogeneity case, we tested the algorithm on only a few different SNR values. By plotting $\alpha_{s}$ in this case, we recover the triangle used in constructing $\scr X_s$.

\subsection{Experiment: two classes} \label{two_classes}

In this experiment, we constructed two phantoms $\scr X^1$ and $\scr X^2$ of the form (\ref{phantom_eq}), with $M_1 = 1, M_2 = 2$.
Cross-sections of $\scr X^1$ and $\scr X^2$ are depicted in the top row panels (c) and (d) in Figure \ref{reconstructions}. We chose the two heterogeneity classes to be equiprobable: $p_1 = p_2 = 1/2$. Note that the theoretical covariance matrix in the two-class heterogeneity problem has rank 1, with dominant eigenvector proportional to the difference between the two volumes.
%

Figure \ref{reconstructions} shows the reconstructions of the mean, top eigenvector, and two volumes for SNR$_{\text{het}}$ = 0.013, 0.003, 0.0013 (0.25, 0.056, 0.025). In Figure \ref{eig_hists}, we display eigenvalue histograms of the reconstructed covariance matrix for the above SNR values. Figure \ref{fig:fsc} shows the FSC curves for these reconstructions. Figure \ref{correlations} shows the correlations of the computed means, top eigenvectors, and (mean-subtracted) volumes with their true values for a broader range of SNR values. In Figure \ref{coords_2}, we plot a histogram of the coordinates $\alpha_{s}$ from step 7 of Algorithm \ref{high_level_algorithm}.

\begin{figure}[H]
        \centering
        \begin{subfigure}[b]{0.24\textwidth}
                \centering
                \includegraphics[scale = 0.04]{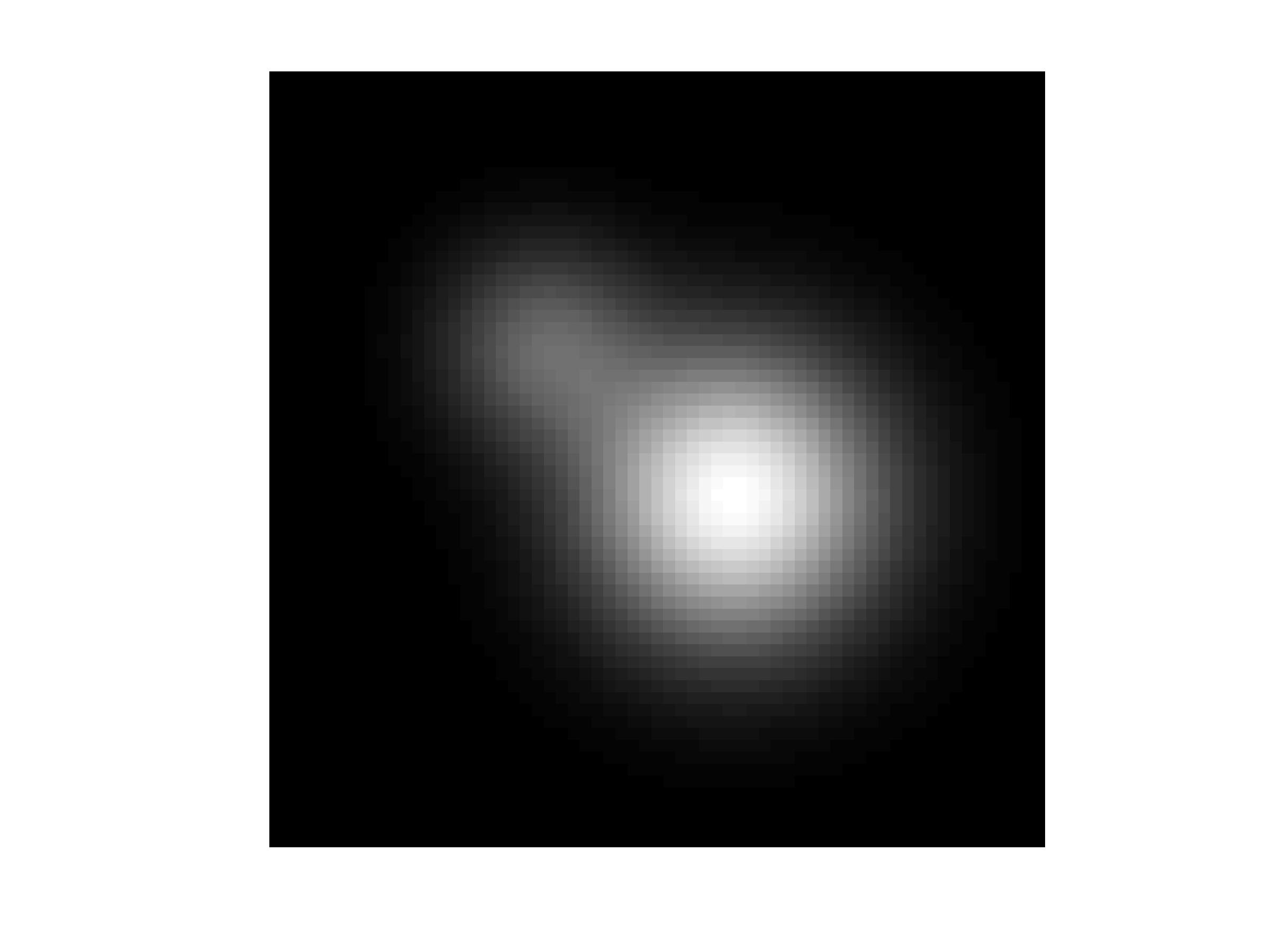}
        \end{subfigure}
\begin{subfigure}[b]{0.24\textwidth}
                \centering
                \includegraphics[scale = 0.04]{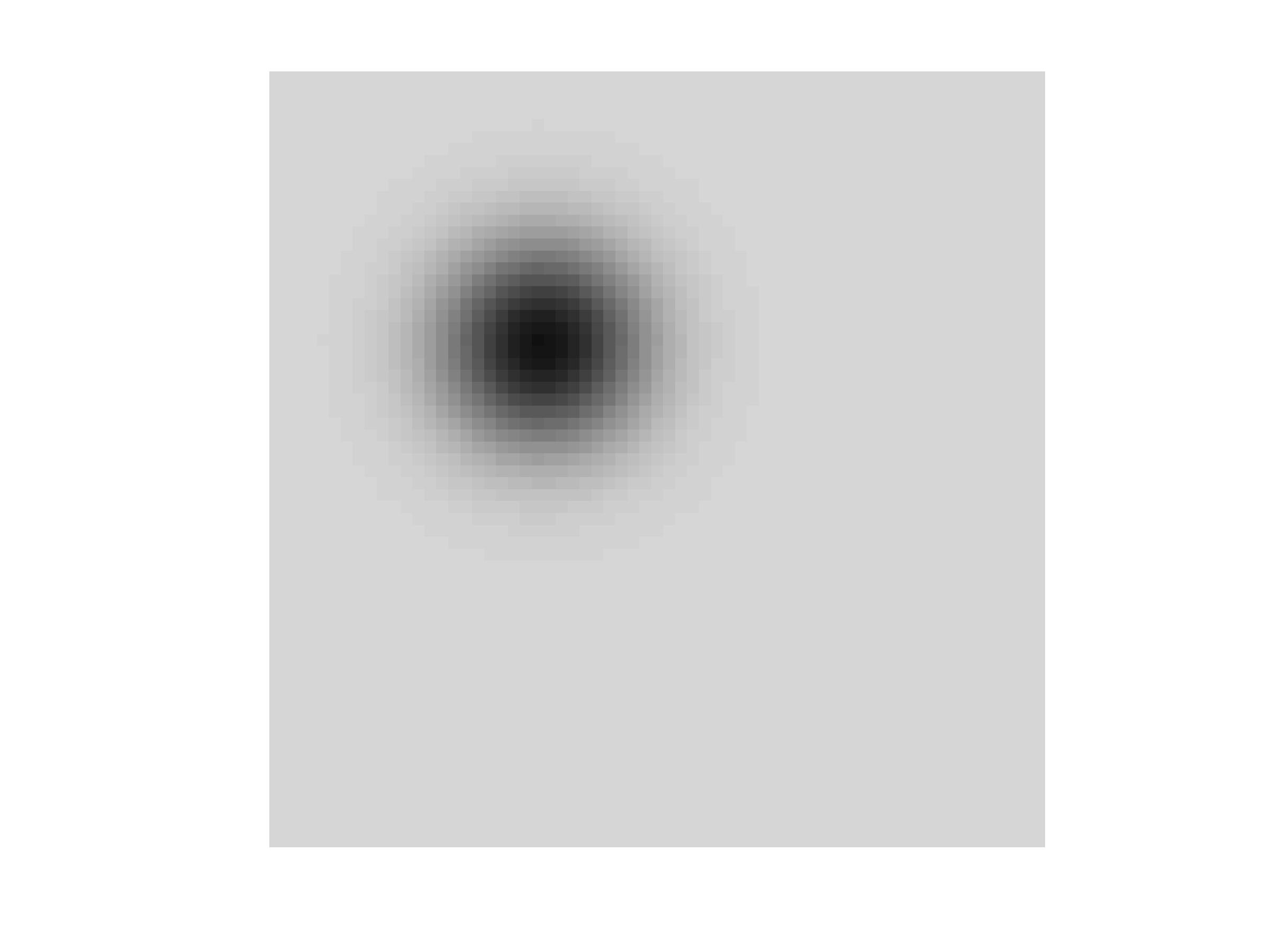}
        \end{subfigure}
\begin{subfigure}[b]{0.24\textwidth}
                \centering
                \includegraphics[scale = 0.04]{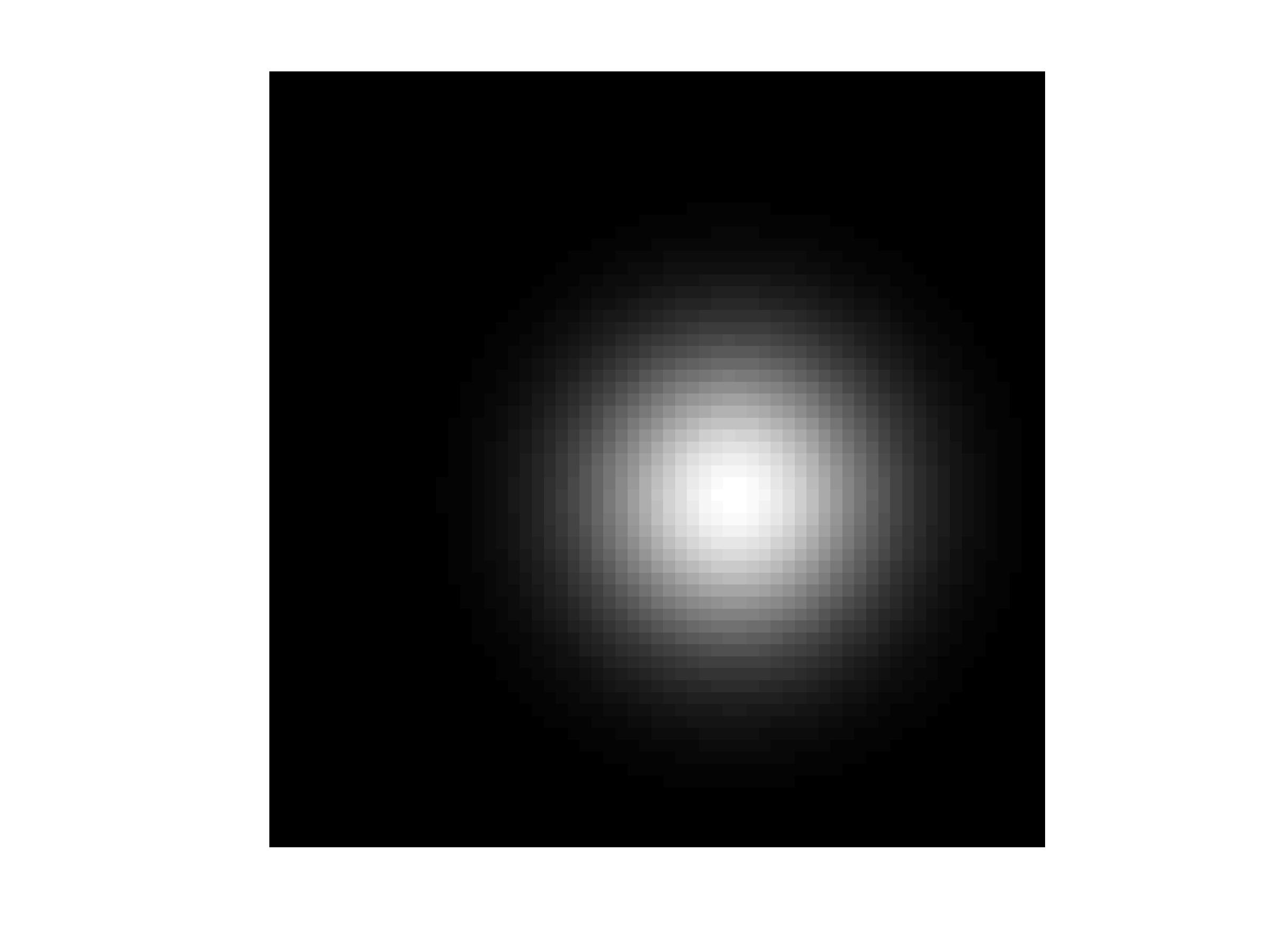}
        \end{subfigure}
        \begin{subfigure}[b]{0.24\textwidth}
                \centering
                \includegraphics[scale = 0.04]{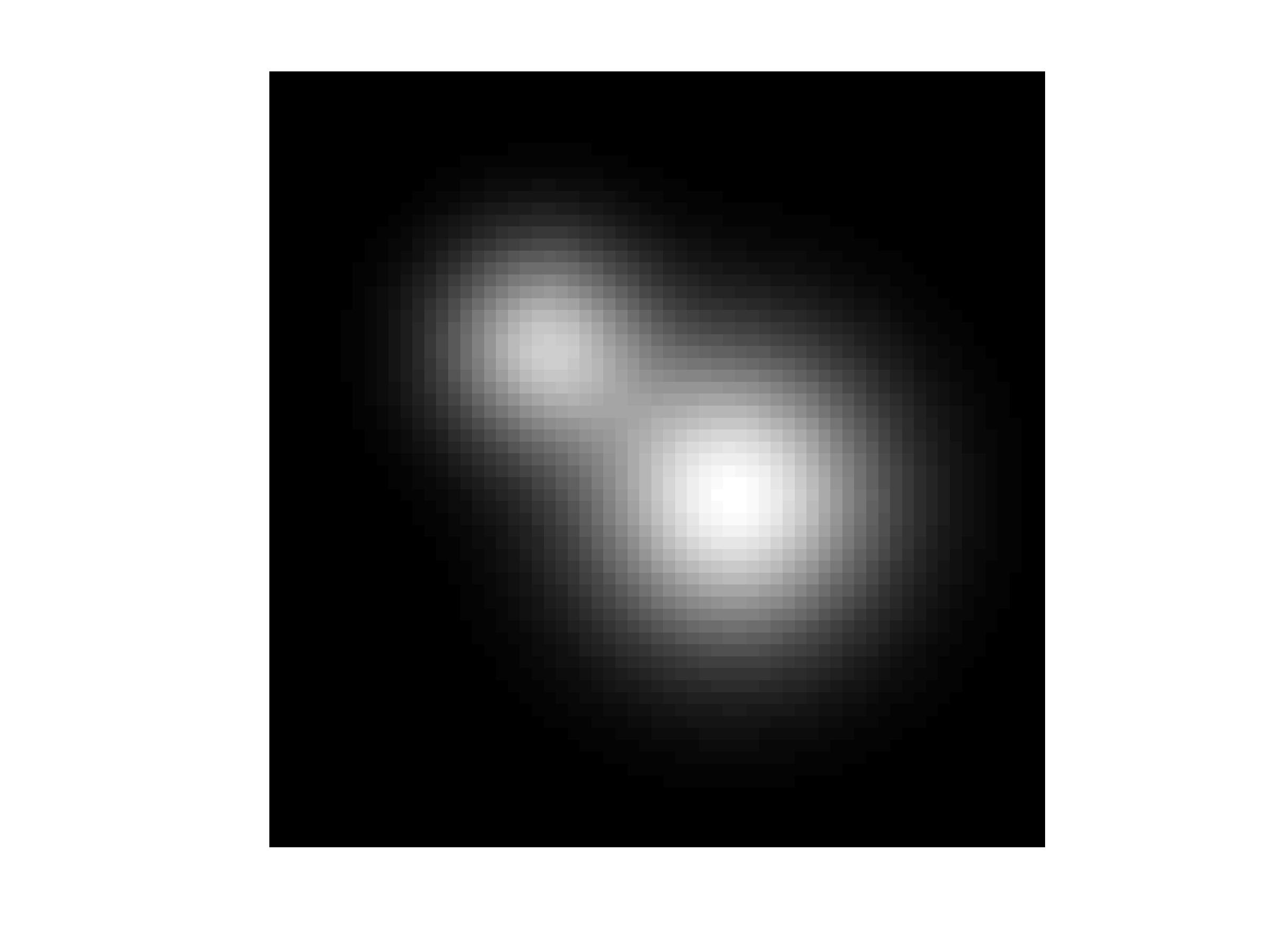}
        \end{subfigure} \\

	\begin{subfigure}[b]{0.24\textwidth}
                \centering
                \includegraphics[scale = 0.04]{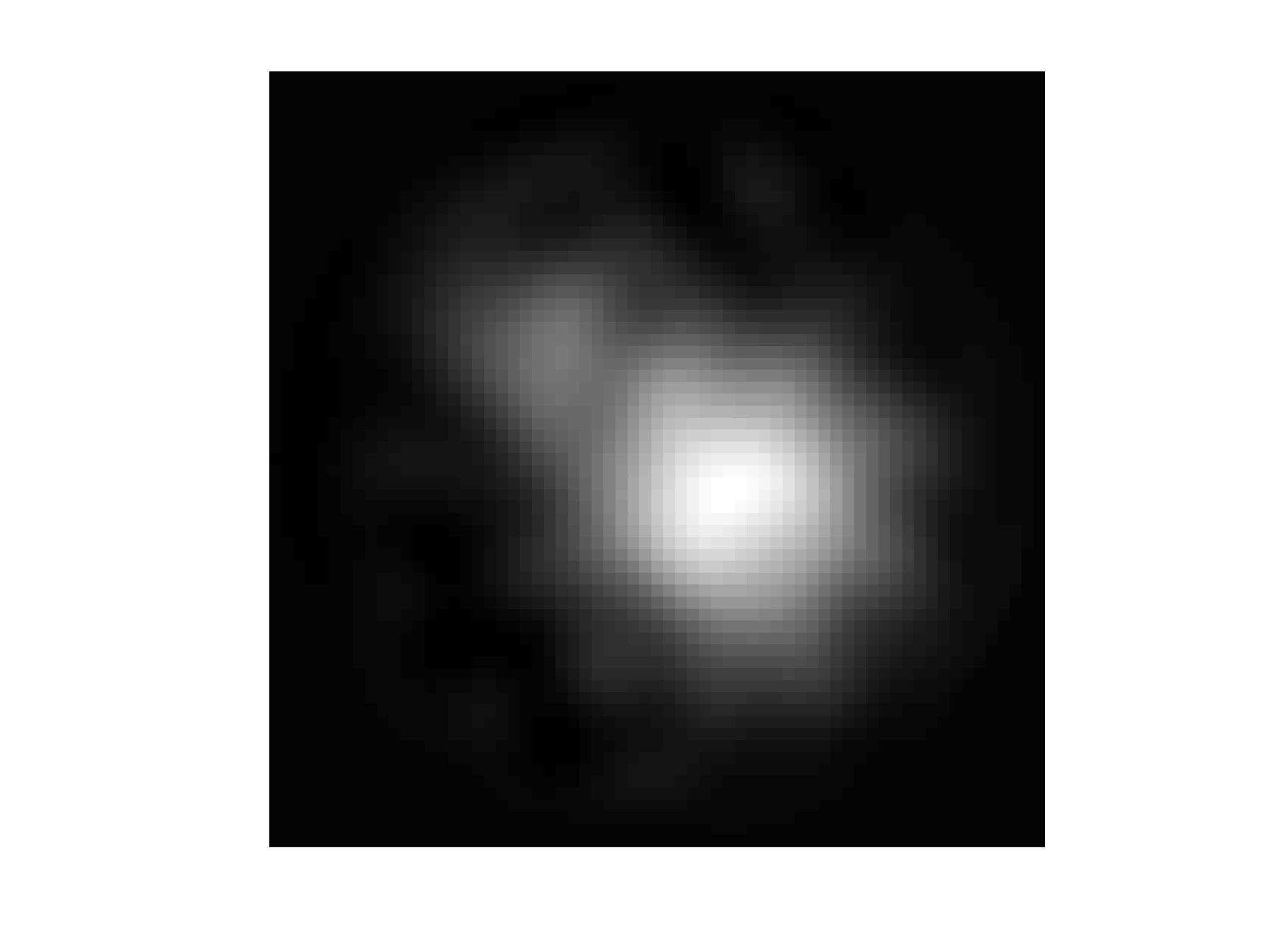}
        \end{subfigure}
	\begin{subfigure}[b]{0.24\textwidth}
                \centering
                \includegraphics[scale = 0.04]{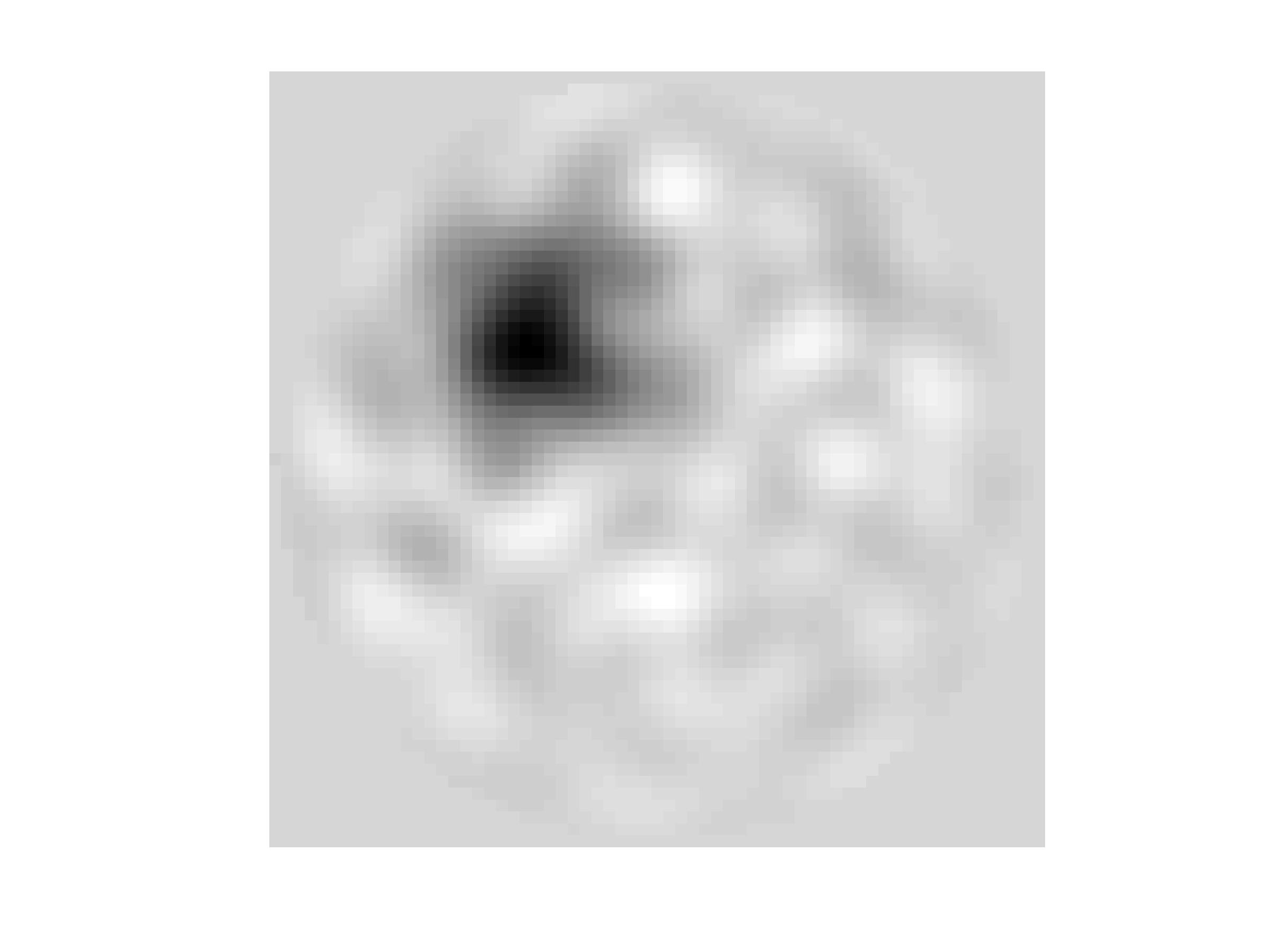}
        \end{subfigure}
	\begin{subfigure}[b]{0.24\textwidth}
                \centering
                \includegraphics[scale = 0.04]{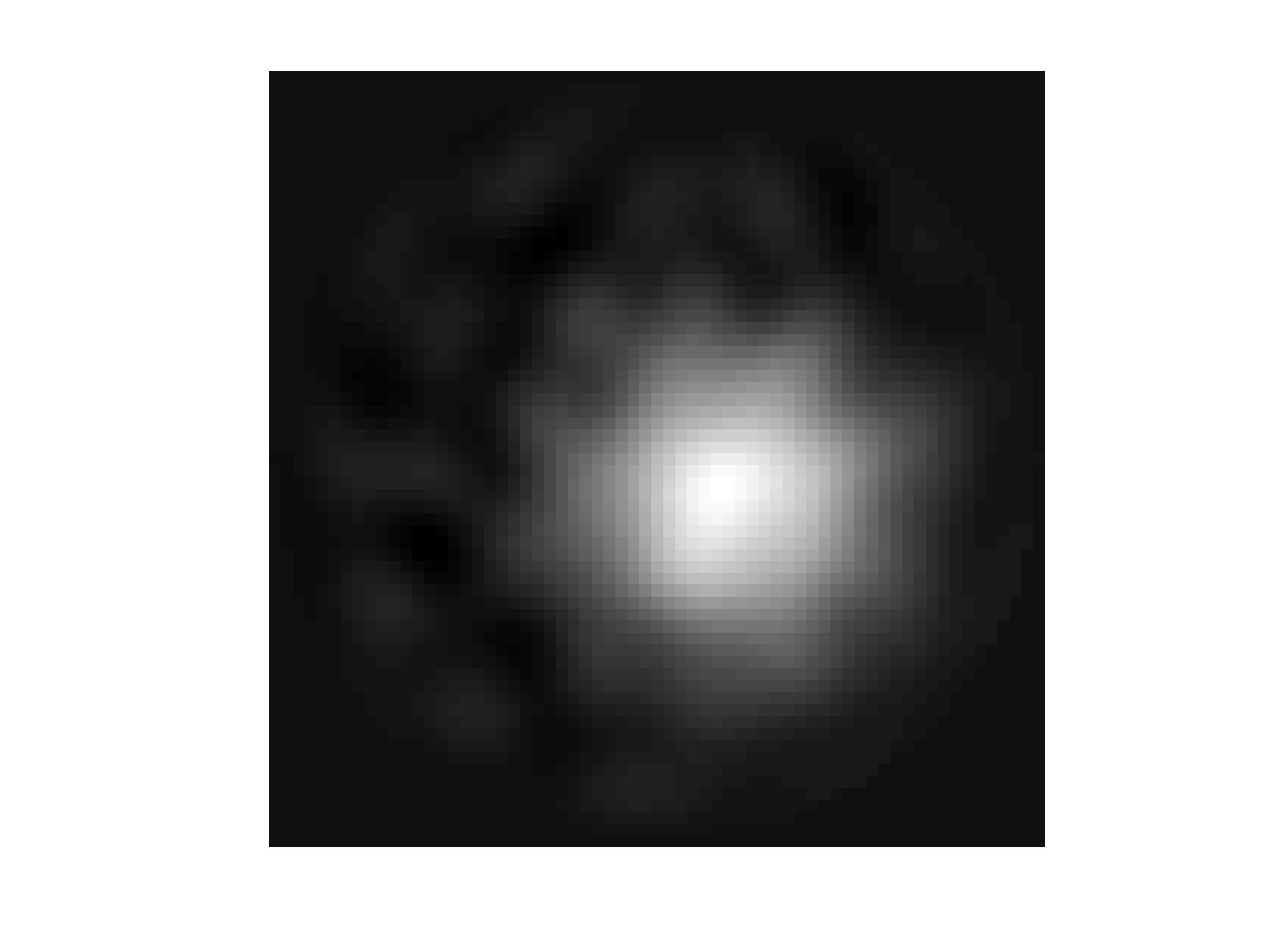}
        \end{subfigure}
	\begin{subfigure}[b]{0.24\textwidth}
                \centering
                \includegraphics[scale = 0.04]{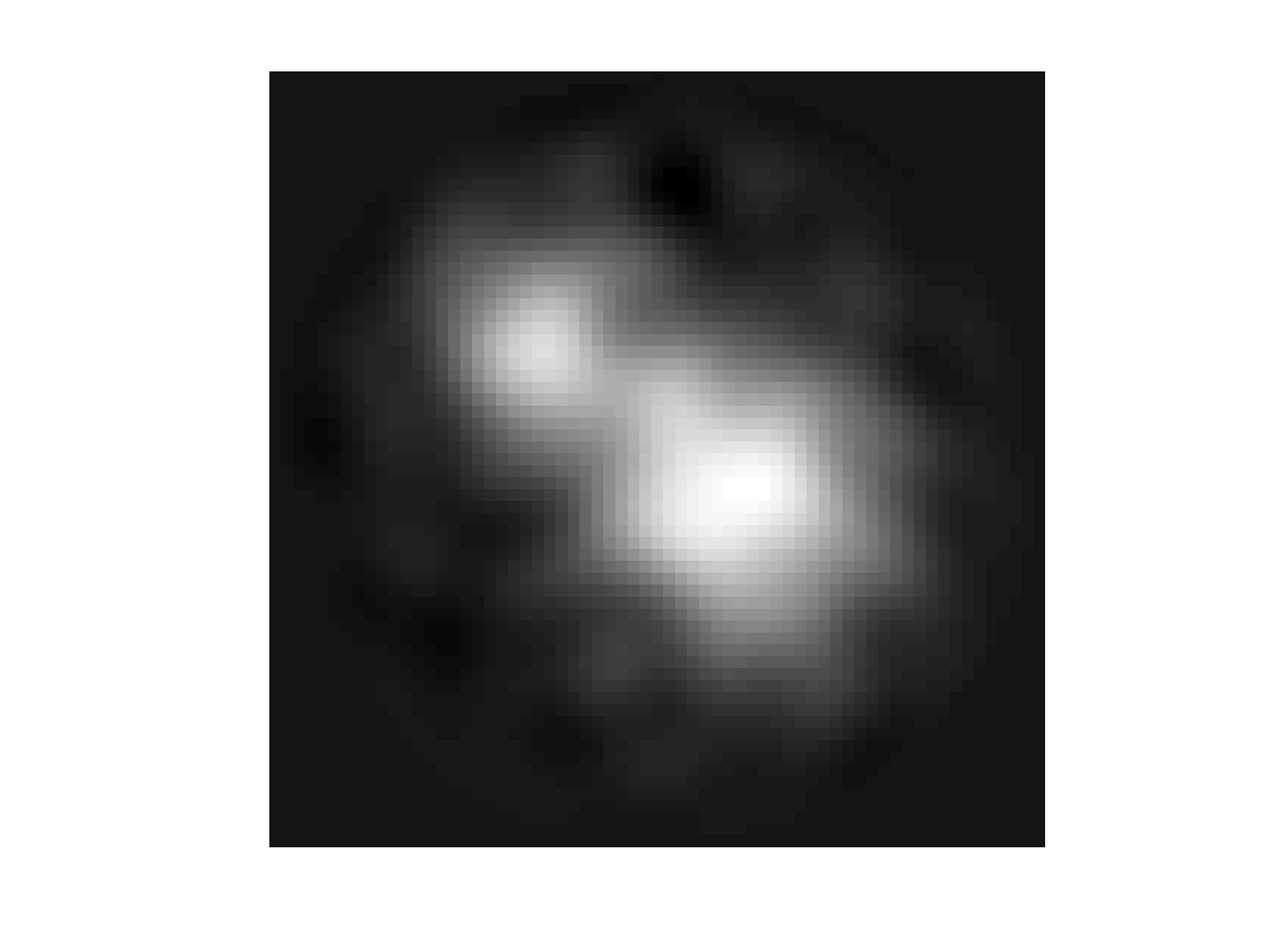}
        \end{subfigure}

\begin{subfigure}[b]{0.24\textwidth}
                \centering
                \includegraphics[scale = 0.04]{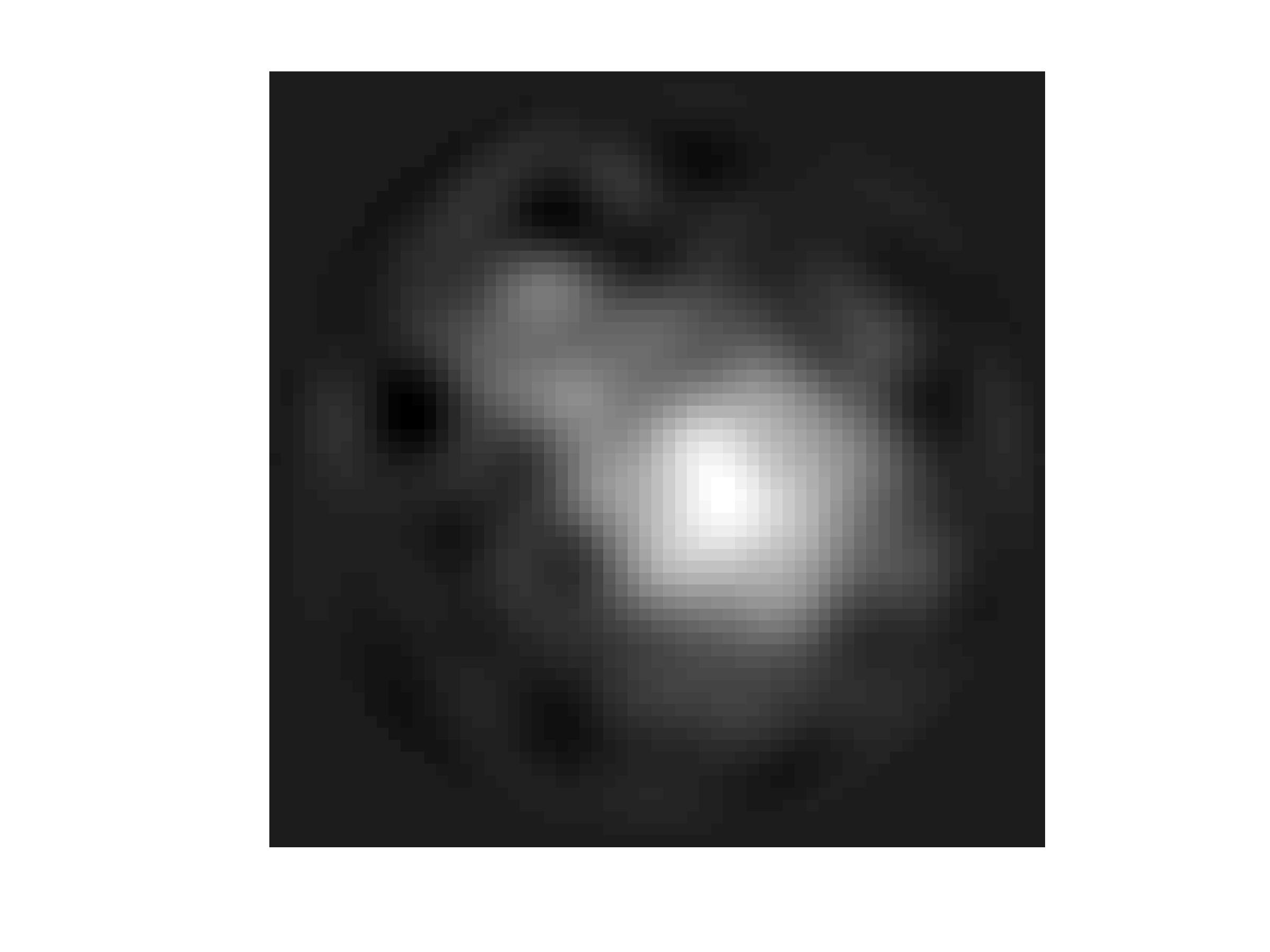}
        \end{subfigure}
\begin{subfigure}[b]{0.24\textwidth}
                \centering
                \includegraphics[scale = 0.04]{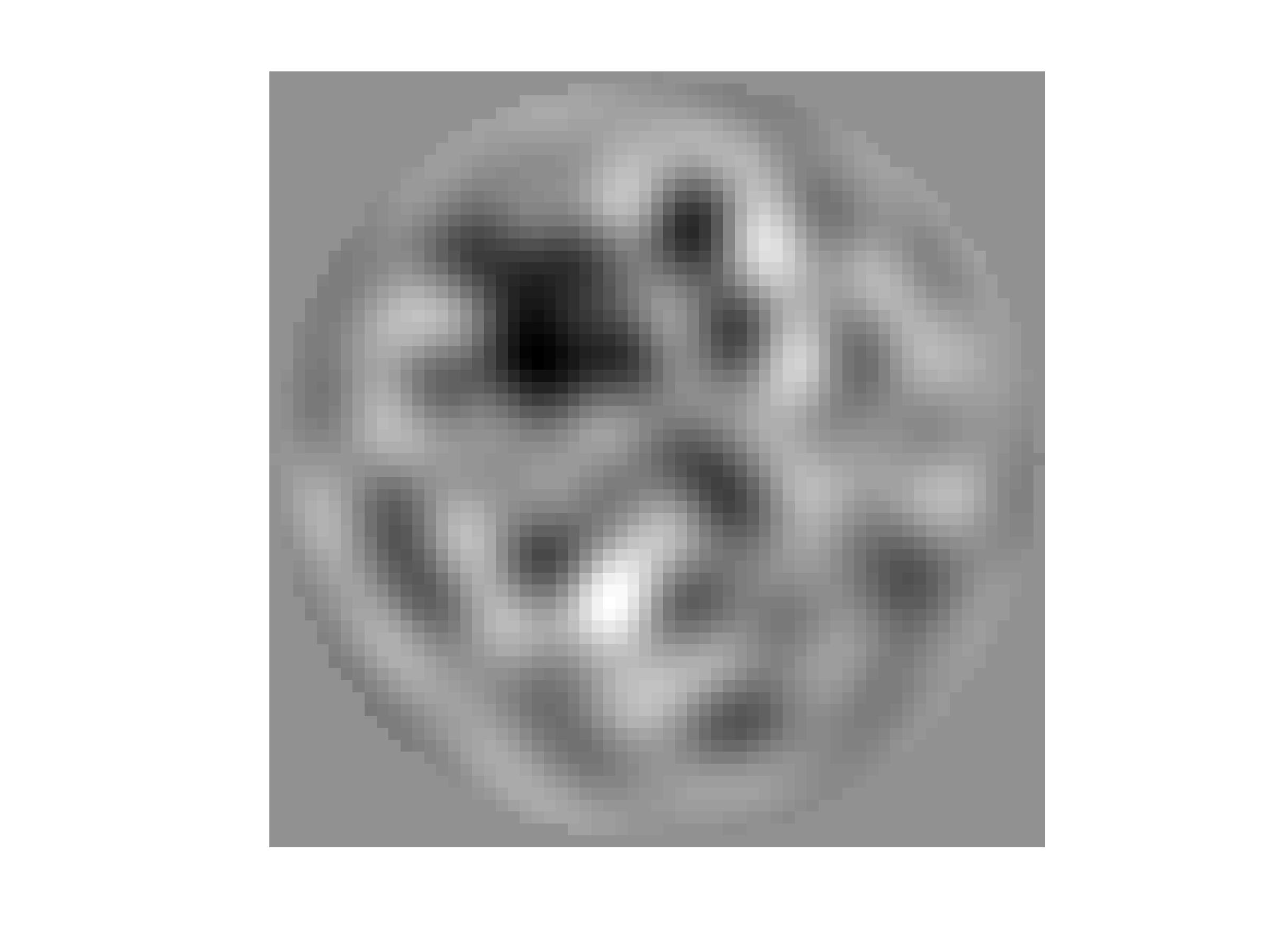}
        \end{subfigure}
\begin{subfigure}[b]{0.24\textwidth}
                \centering
                \includegraphics[scale = 0.04]{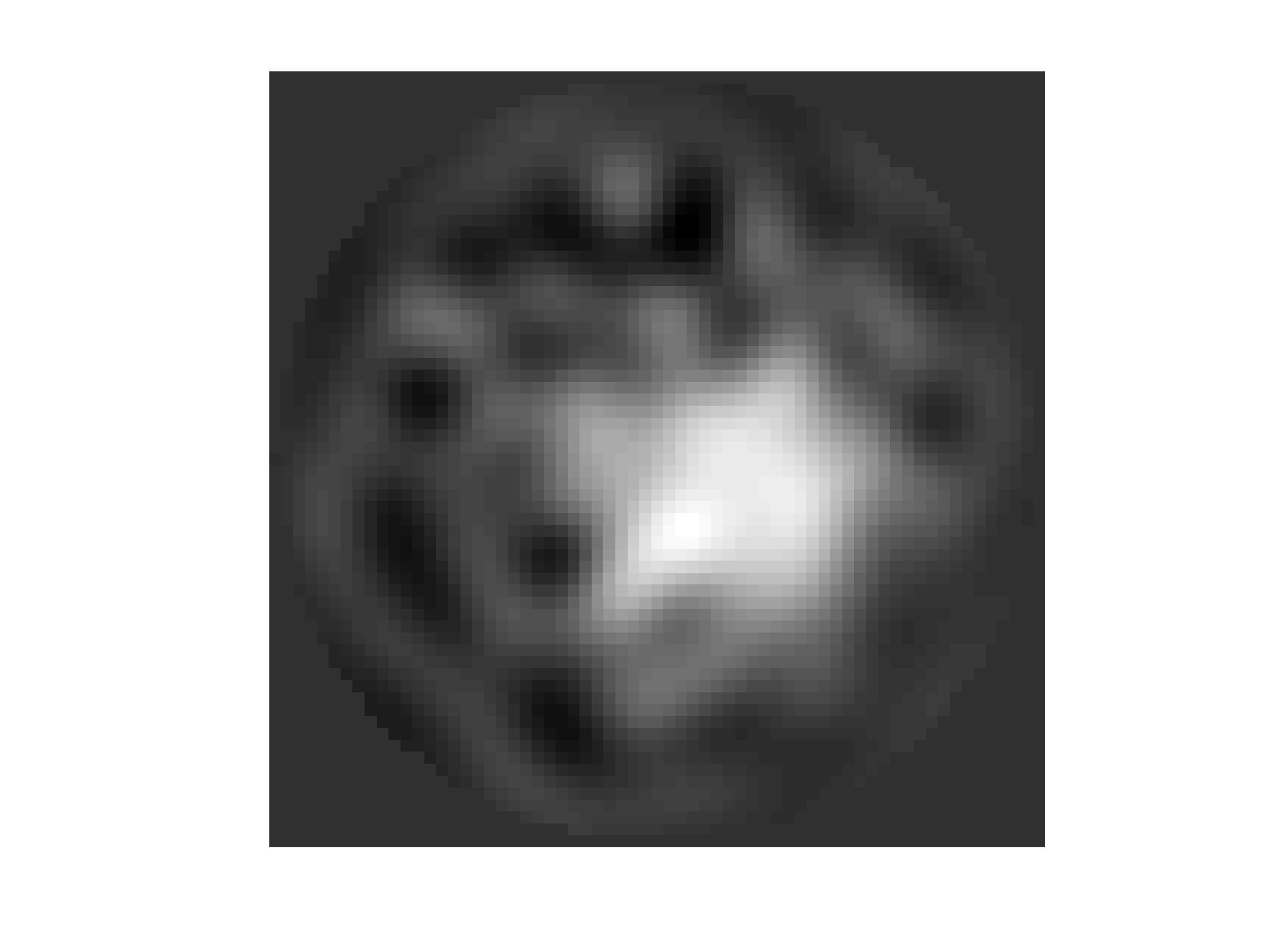}
        \end{subfigure}
\begin{subfigure}[b]{0.24\textwidth}
                \centering
                \includegraphics[scale = 0.04]{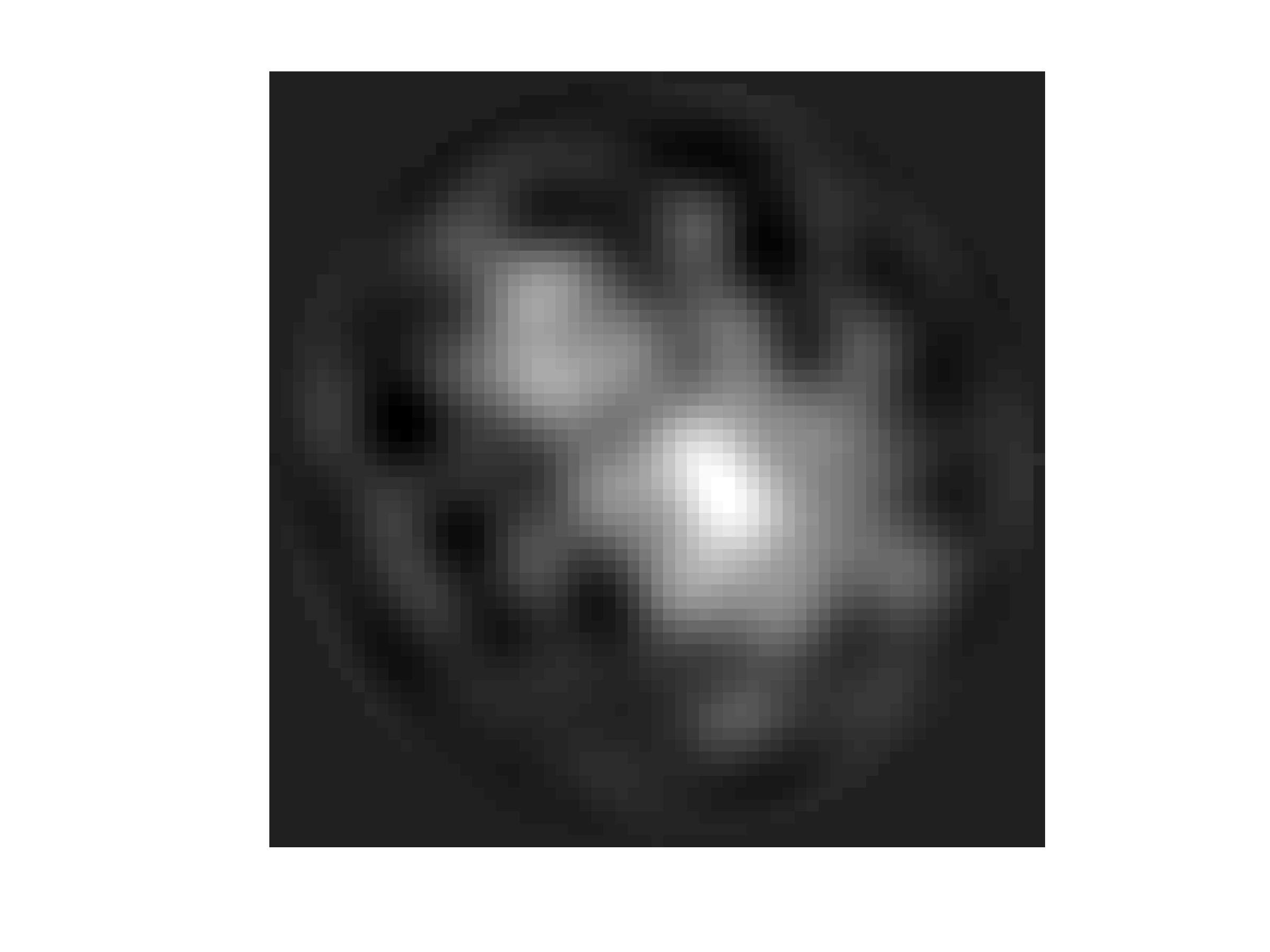}
        \end{subfigure} \\

\begin{subfigure}[b]{0.24\textwidth}
                \centering
                \includegraphics[scale = 0.04]{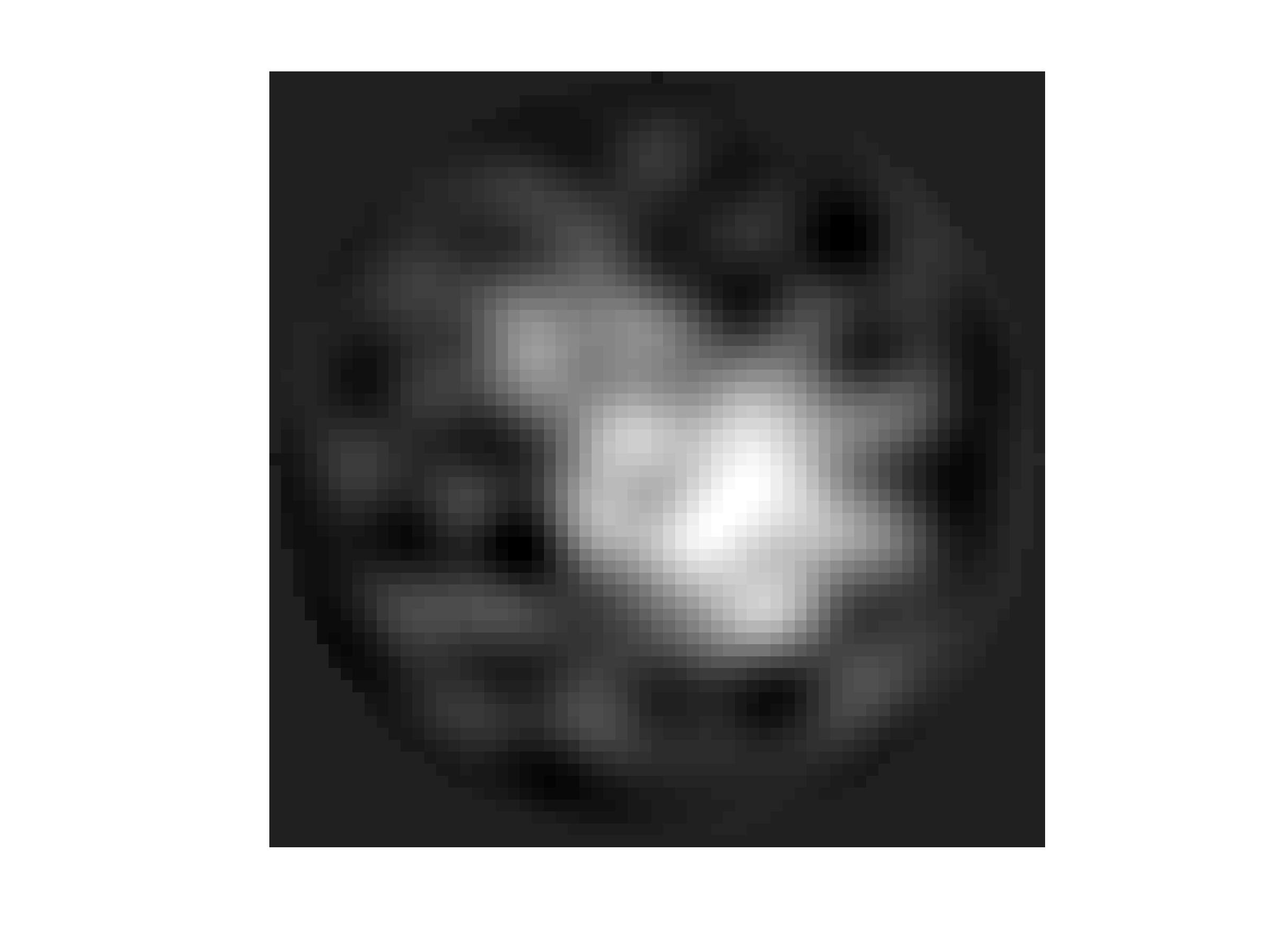}
	\caption{Mean}
        \end{subfigure} 
\begin{subfigure}[b]{0.24\textwidth}
                \centering
                \includegraphics[scale = 0.04]{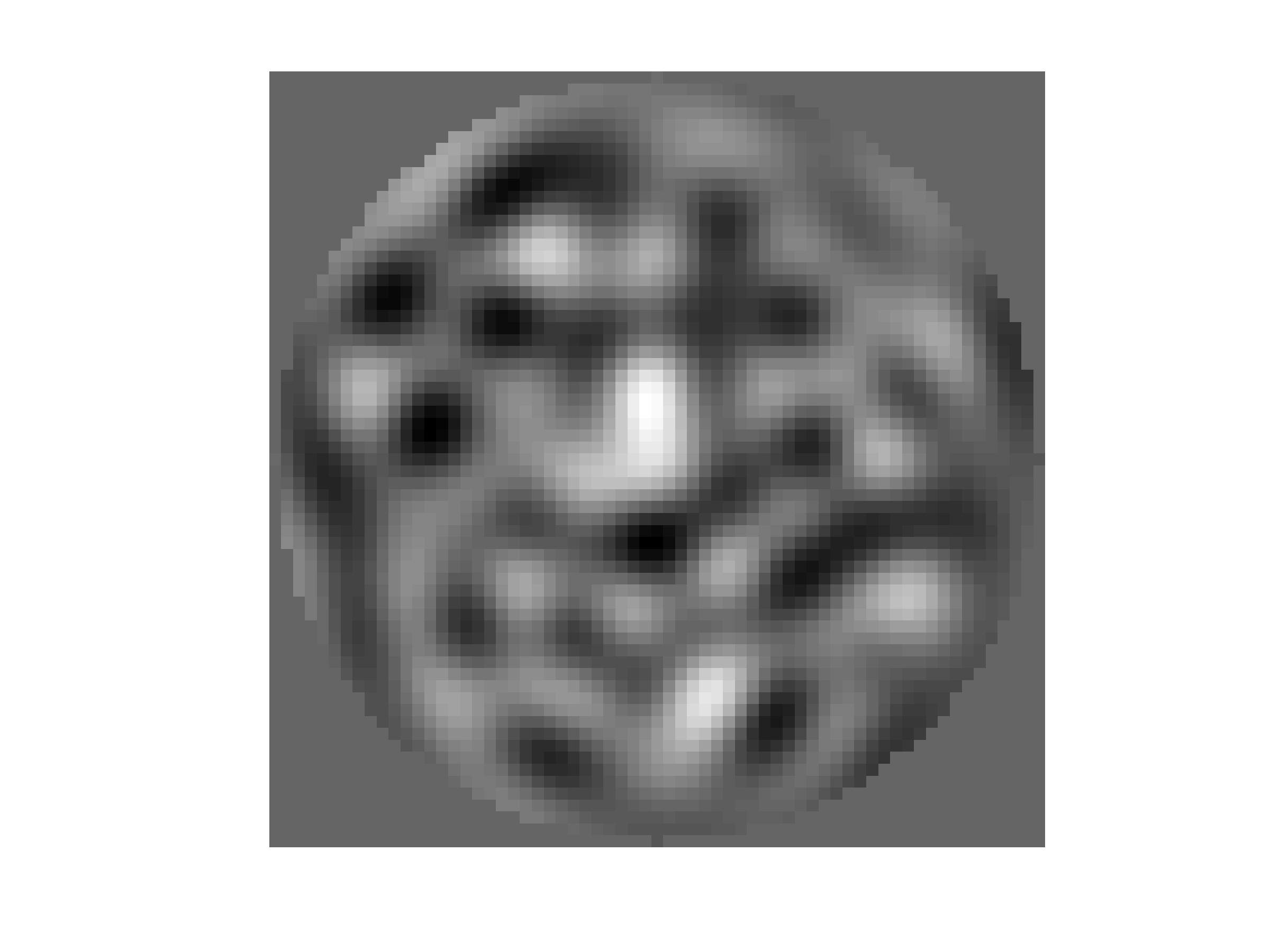}
	\caption{Eigenvector}
        \end{subfigure} 
\begin{subfigure}[b]{0.24\textwidth}
                \centering
                \includegraphics[scale = 0.04]{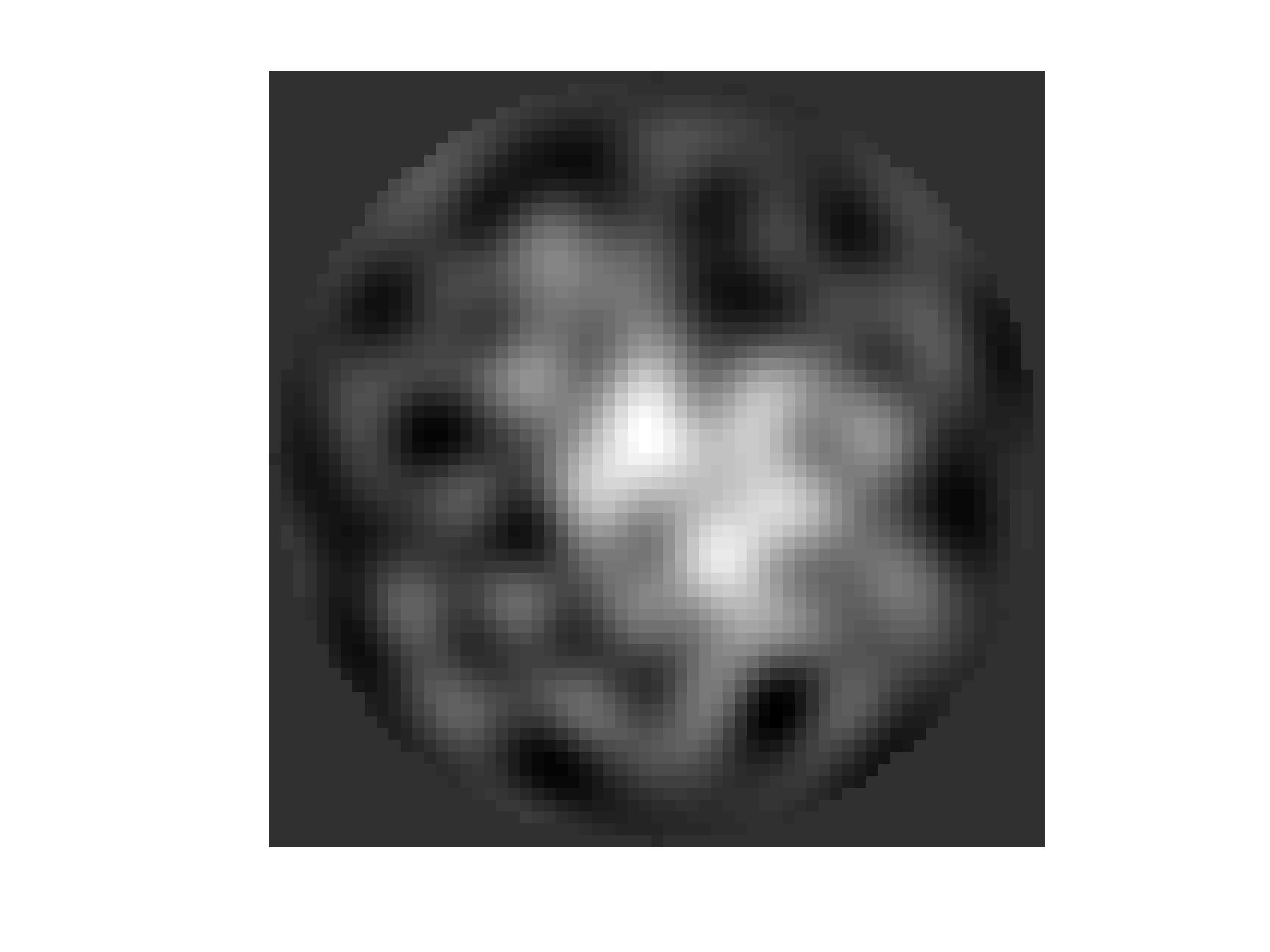}
	\caption{Volume 1}
        \end{subfigure} 
\begin{subfigure}[b]{0.24\textwidth}
                \centering
                \includegraphics[scale = 0.04]{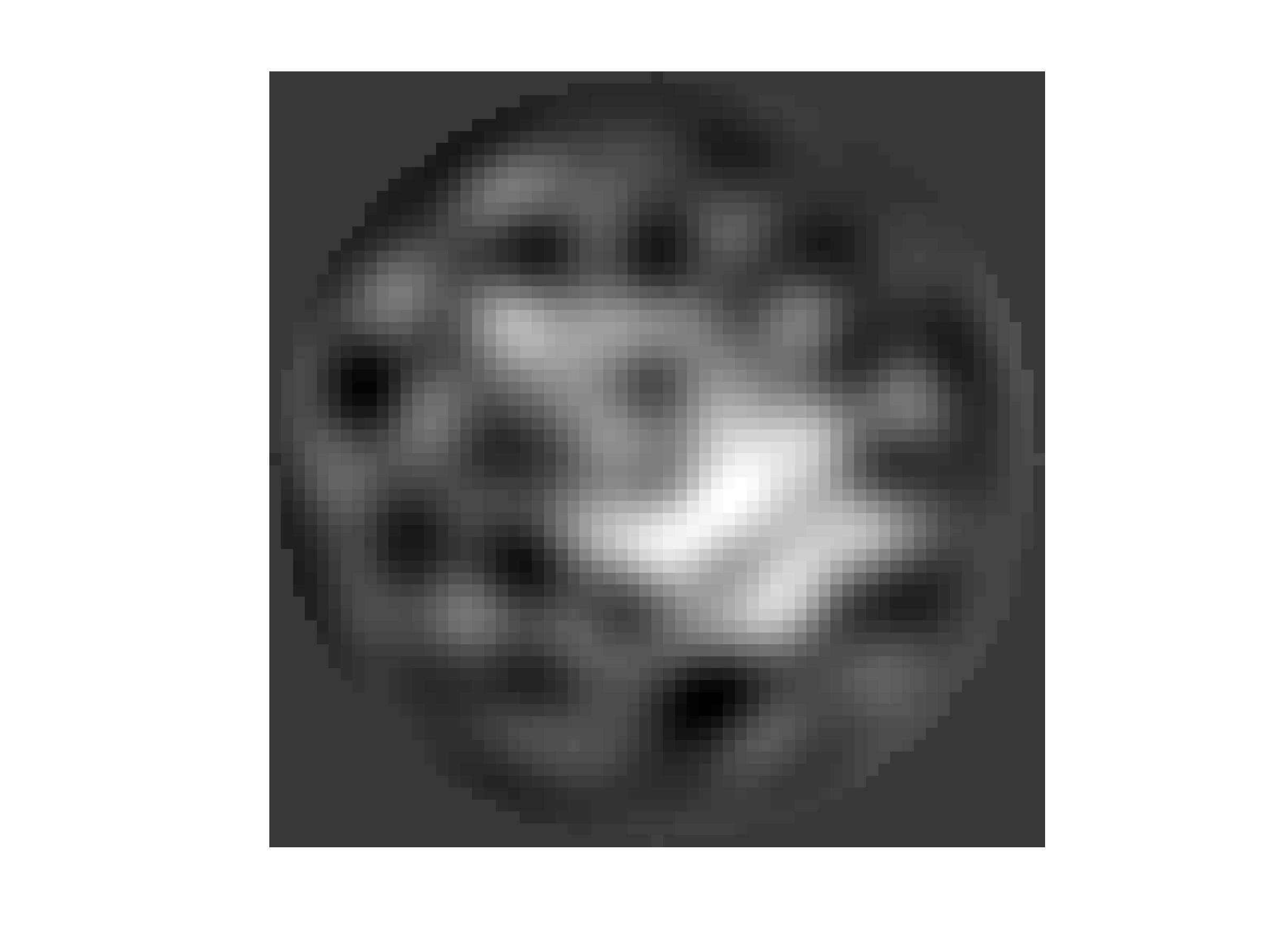}
	\caption{Volume 2}
        \end{subfigure} 
\caption{Cross-sections of reconstructions of the mean, top eigenvector, and two volumes for three different SNR values. The top row is clean, the second row corresponds to SNR$_\text{het}$ = 0.013 (0.25), the third row to SNR$_\text{het}$ = 0.003 (0.056), and the last row to SNR$_\text{het}$ = 0.0013 (0.025).}
\label{reconstructions}
\end{figure}

Our algorithm was able to meaningfully reconstruct the two volumes for SNR$_{\text{het}}$ as low as about 0.003 (0.06). Note that the means were always reconstructed with at least a 94\% correlation to their true values. On the other hand, the eigenvector reconstruction shows a phase-transition behavior, with the transition occurring between SNR$_{\text{het}}$ values of 0.001 (0.002) and 0.003 (0.006). Note that this behavior is tied to the spectral gap (separation of top eigenvalues from the bulk) of $\hat \Sigma_n$. Indeed, the disappearance of the spectral gap going from panel (b) to panel (c) of Figure \ref{eig_hists} coincides with the estimated top eigenvector becoming uncorrelated with the truth, as reflected in Figures \ref{fig:fsc}(b) and \ref{correlations}(a). This phase transition behavior is very similar to that observed in the usual high-dimensional PCA setup, described in Section \ref{rmt}.

\begin{figure}[H]
	\begin{subfigure}[b]{0.32\textwidth}
                \centering
                \includegraphics[scale = 0.132]{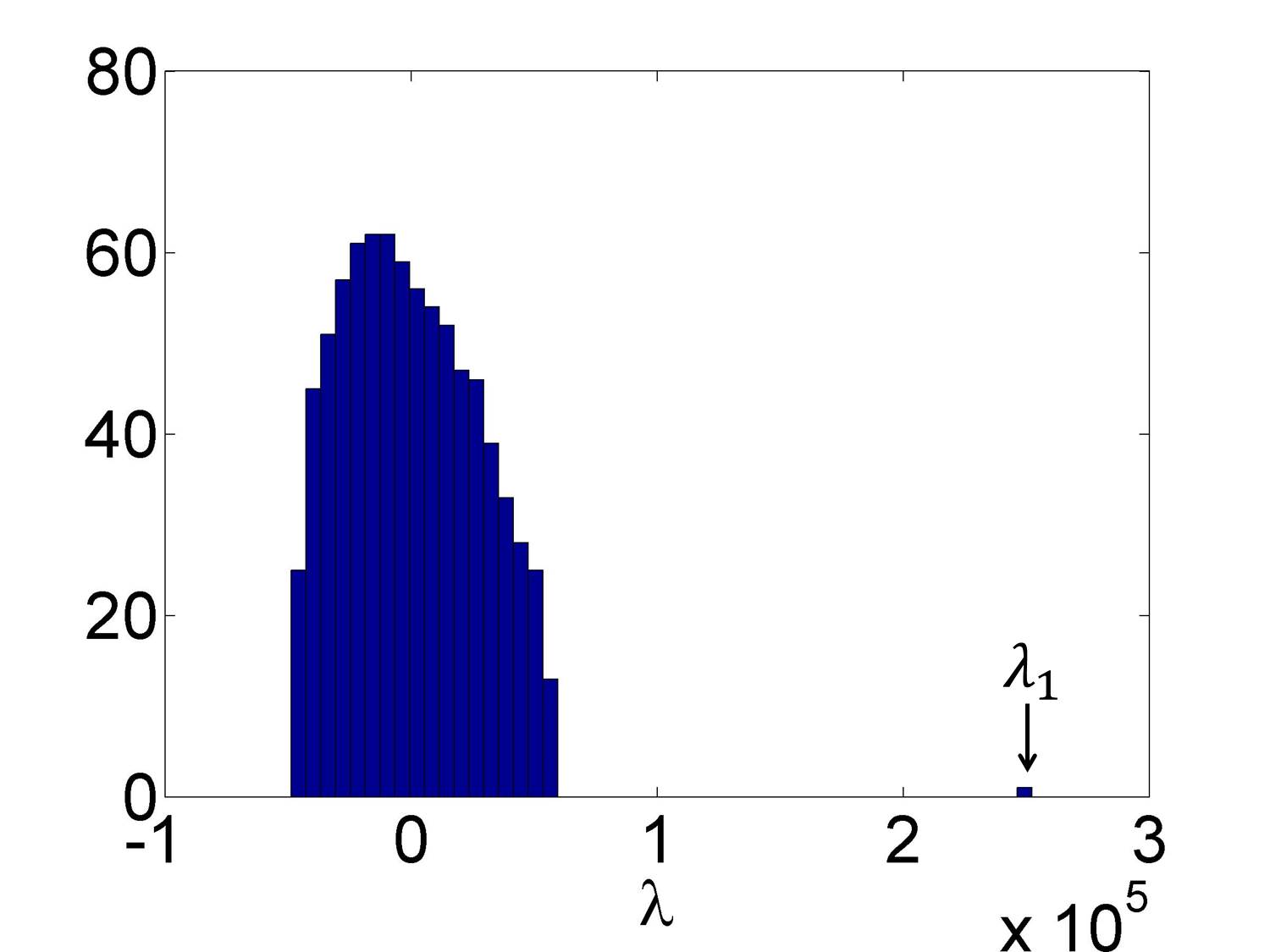}
	\caption{SNR$_{\text{het}}$ = 0.013 (0.25)}
        \end{subfigure}
\begin{subfigure}[b]{0.32\textwidth}
                \centering
                \includegraphics[scale = 0.132]{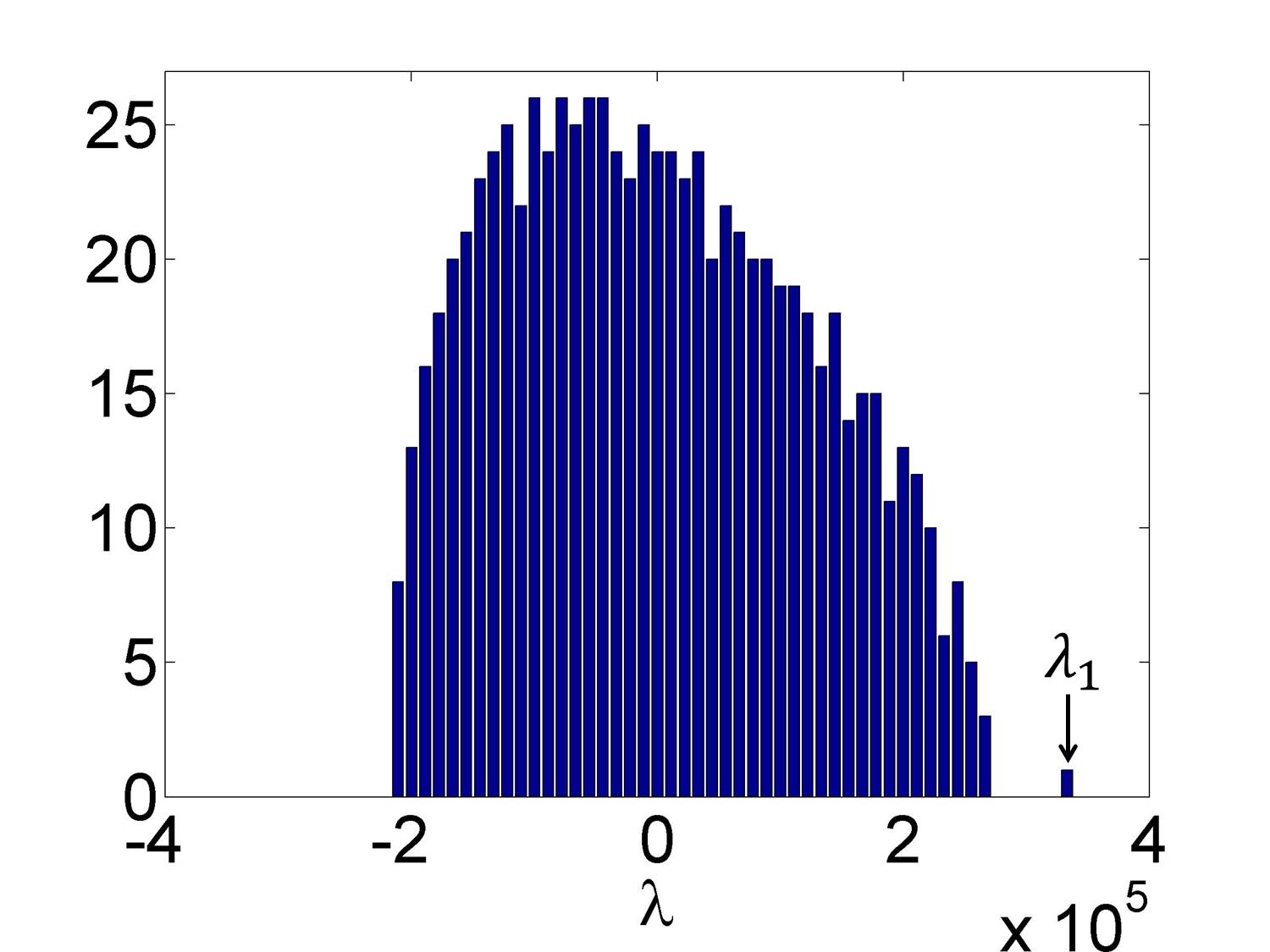}
	\caption{SNR$_{\text{het}}$ = 0.003 (0.056)}
        \end{subfigure}
\begin{subfigure}[b]{0.32\textwidth}
                \centering
                \includegraphics[scale = 0.0525]{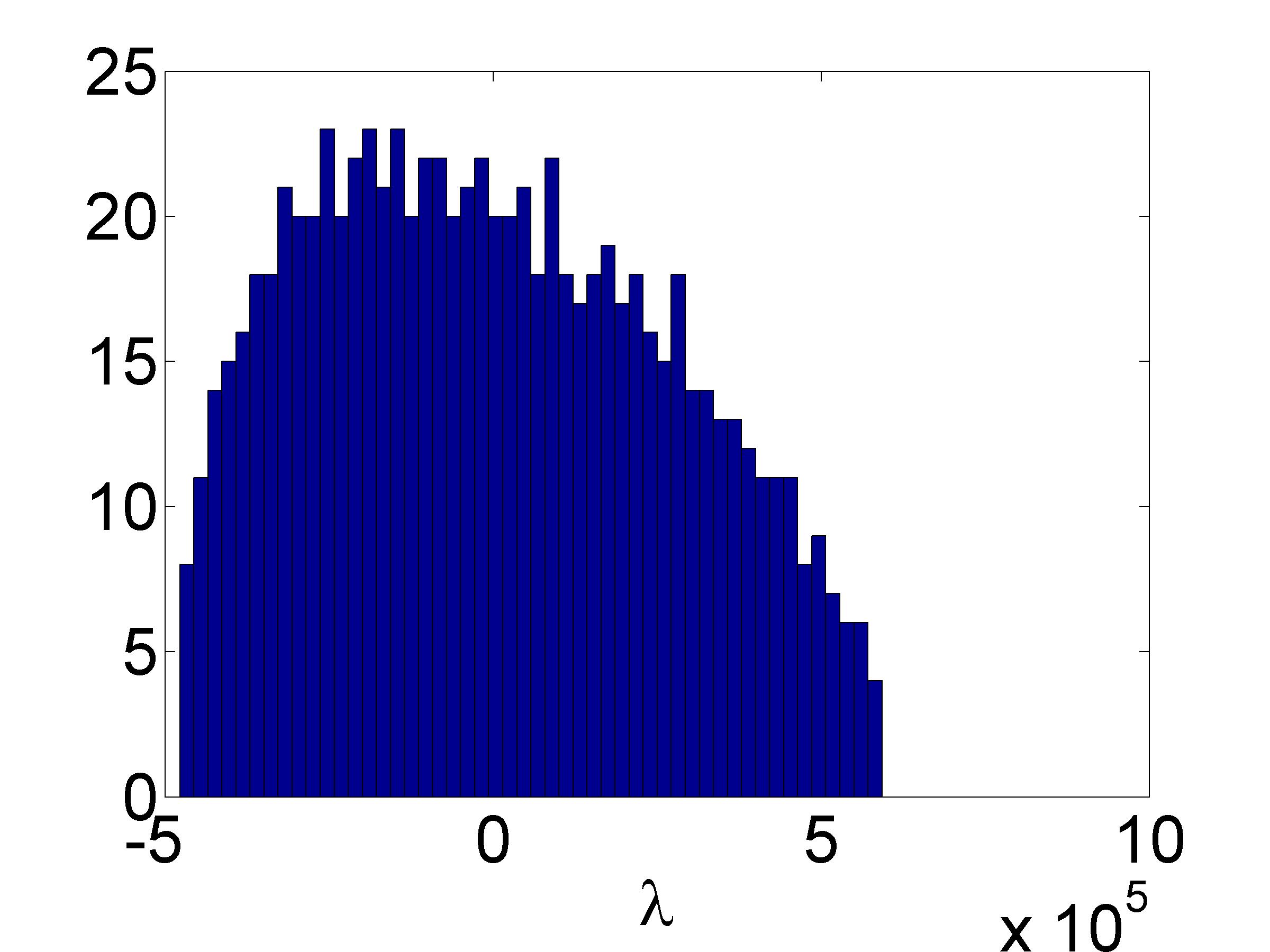}
	\caption{SNR$_{\text{het}}$ = 0.0013 (0.025)}
        \end{subfigure}
\caption{Eigenvalue histograms of $\hat \Sigma_n$ in two volume case for three SNR values. Note that as the SNR decreases, the distribution of eigenvalues associated with noise comes increasingly closer to the top eigenvalue that corresponds to the structural variability, and eventually the latter is no longer distinguishable.}
\label{eig_hists}
\end{figure}

\begin{figure}[H]
 	\centering
\begin{subfigure}[b]{0.3\textwidth}
                \centering
	\includegraphics[scale = 0.05]{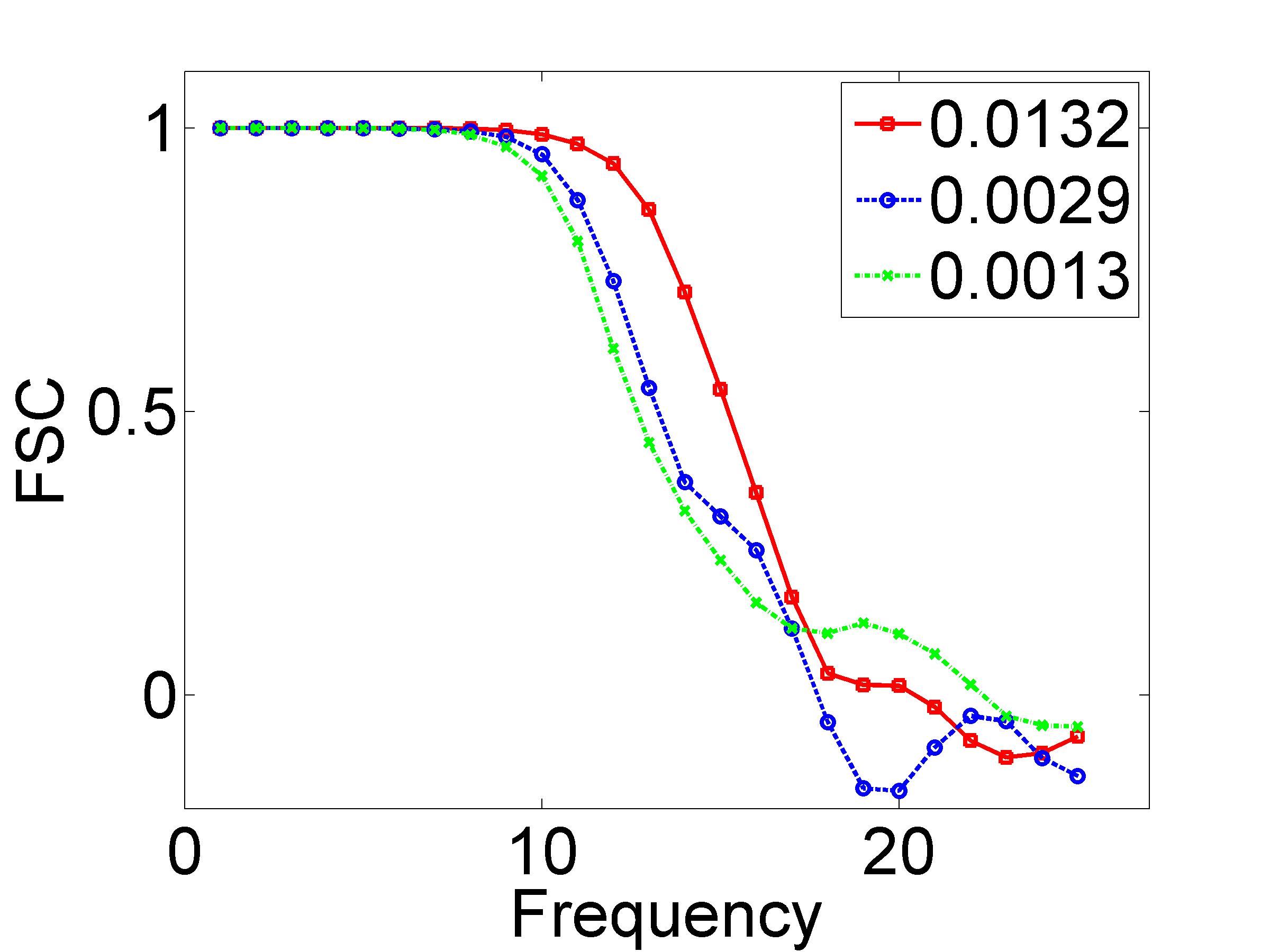}
\caption{Mean}
        \end{subfigure}	
\begin{subfigure}[b]{0.3\textwidth}
	\centering
	\includegraphics[scale = 0.05]{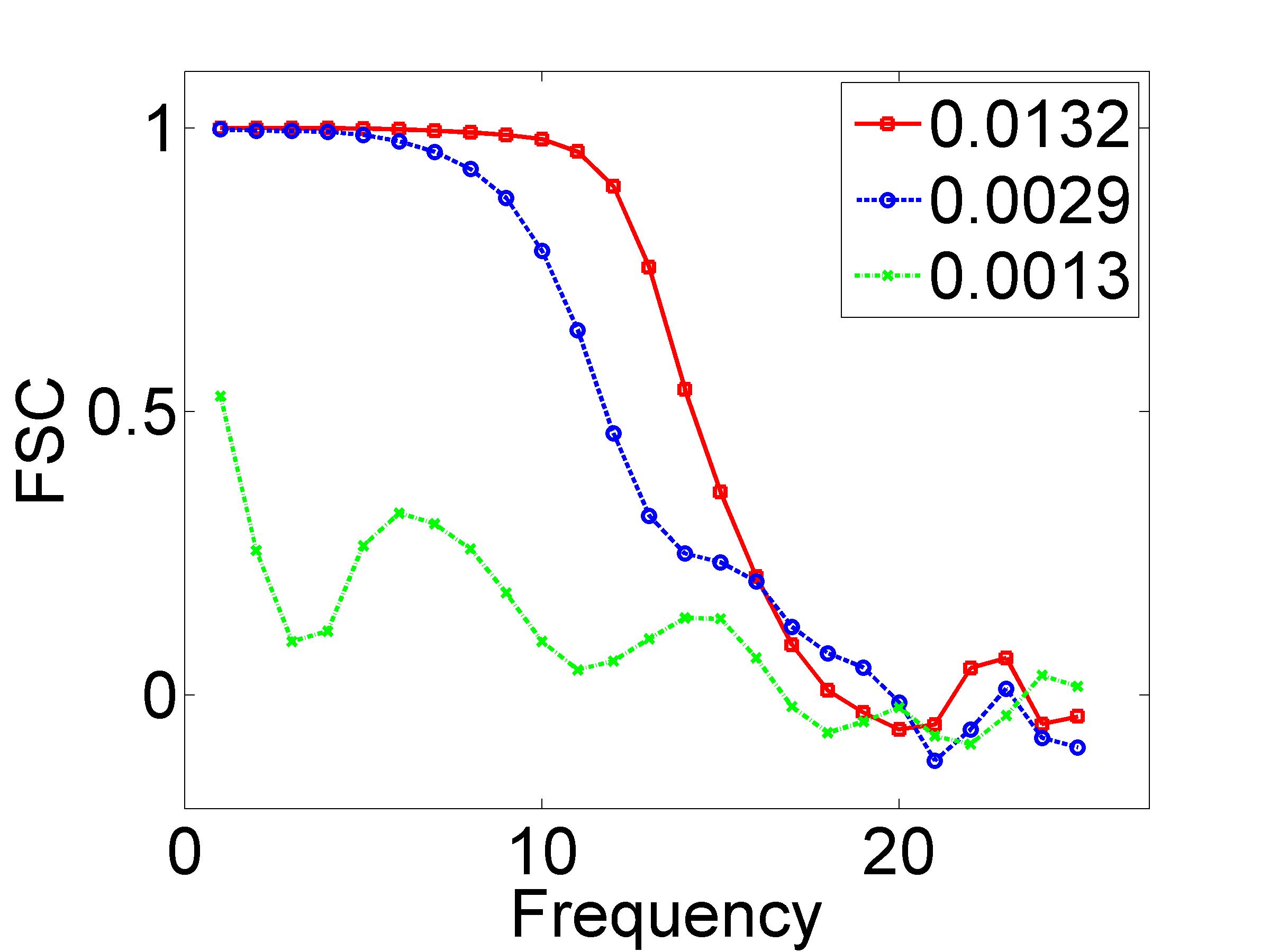}
\caption{Top eigenvector}
\end{subfigure}
\begin{subfigure}[b]{0.3\textwidth}
	\centering
	\includegraphics[scale = 0.05]{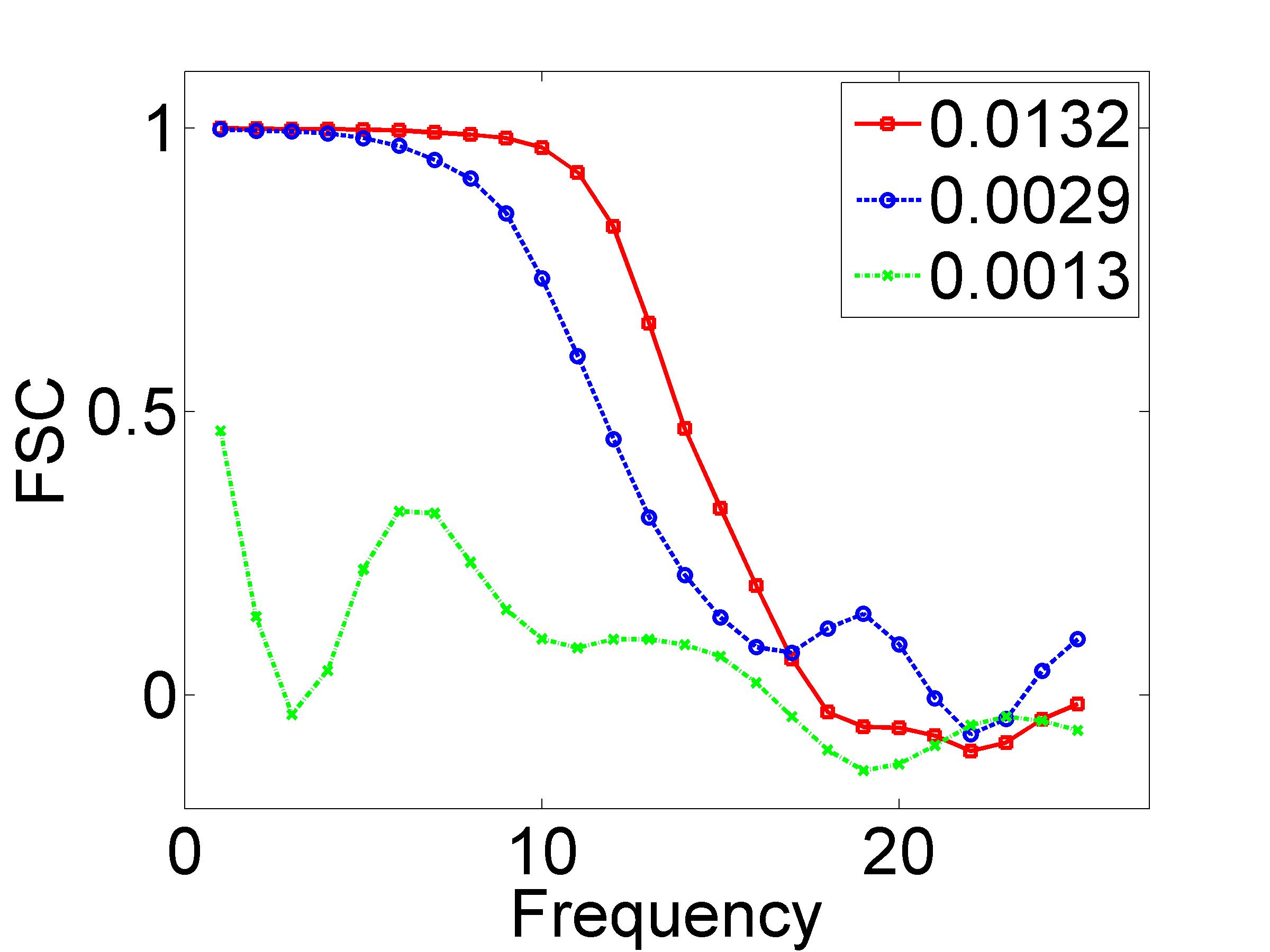}
\caption{Volume 1}
\end{subfigure}
\caption{FSC curves for the mean volume, top eigenvector, and one mean-subtracted volume at the same three SNRs as in Figure \ref{reconstructions}. Note that the mean volume is reconstructed successfully for all three SNR levels. On the other hand, the top eigenvector and volume are recovered at the highest two SNR levels but not at the lowest SNR.}
\label{fig:fsc}
\end{figure}

\begin{figure}[H]
 \begin{subfigure}[b]{0.49\textwidth}
                \centering
                \includegraphics[scale = 0.05]{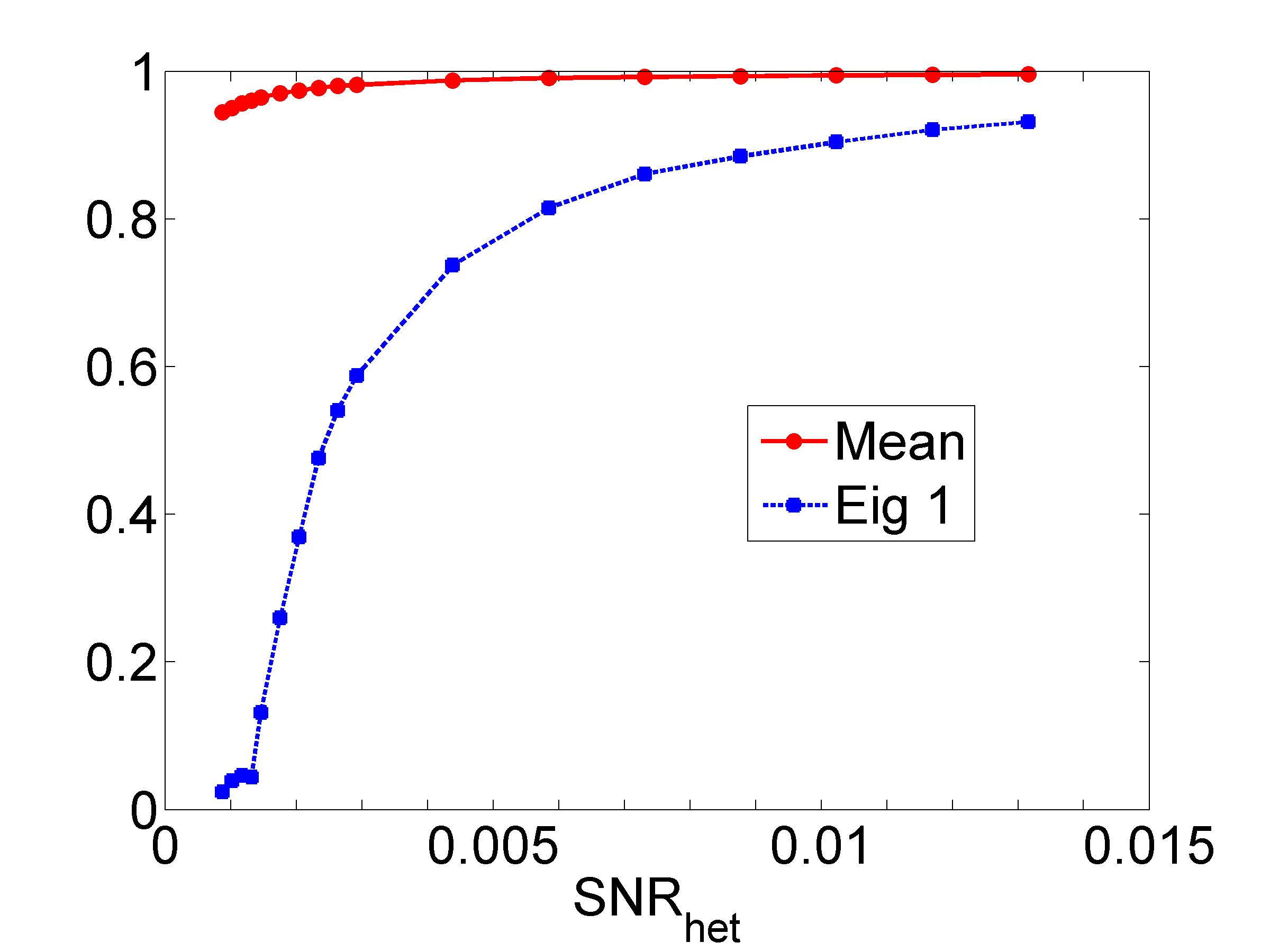}
	\caption{Mean and eigenvector correlations}
        \end{subfigure}
\begin{subfigure}[b]{0.49\textwidth}
	\centering
	\includegraphics[scale = 0.05]{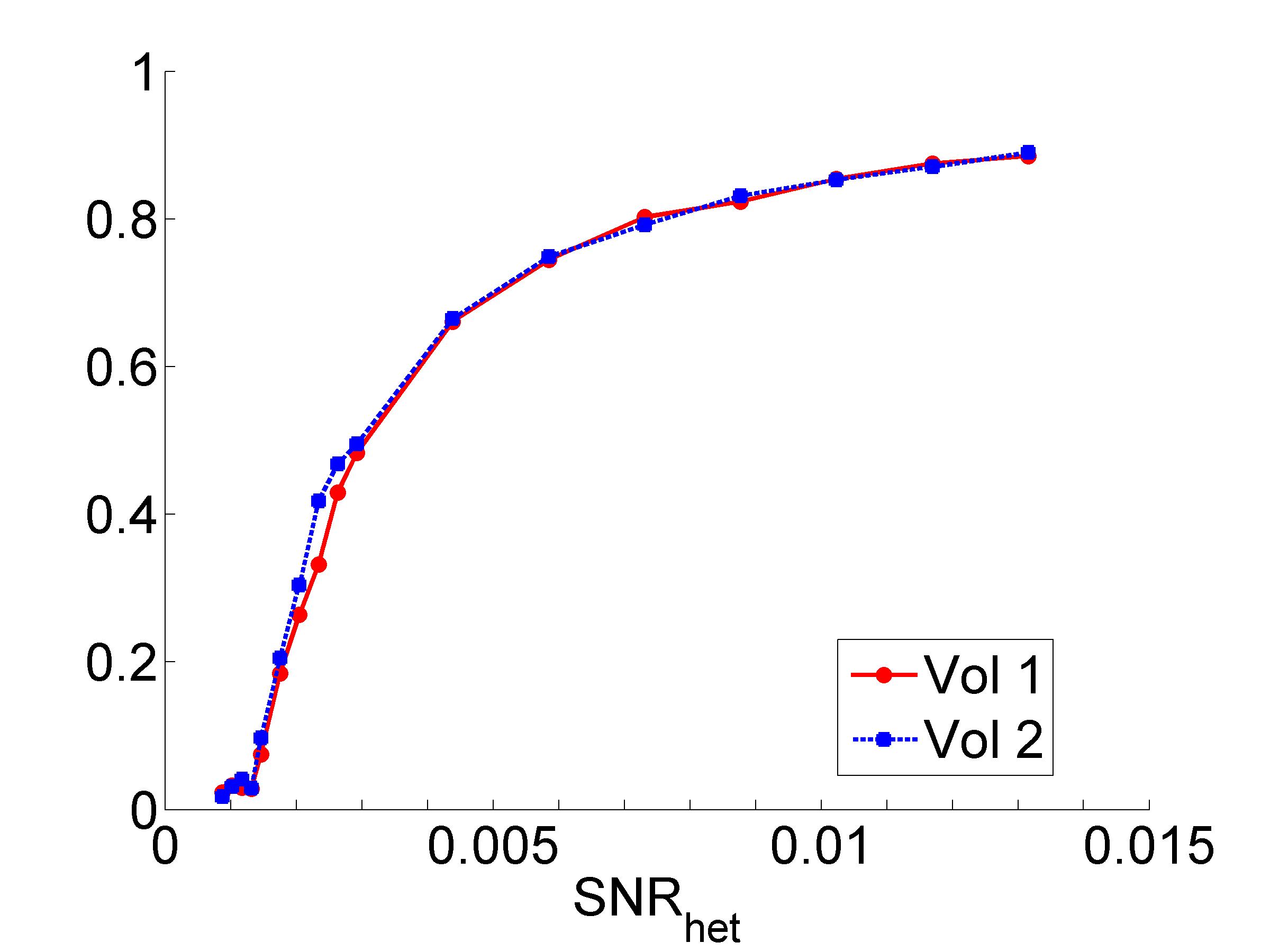}
	\caption{Volume correlations} \label{correlations_c}
\end{subfigure}
\caption{Correlations of computed quantities with their true values for different SNRs (averaged over 10 experiments) for the two volume case. Note that in the two volumes case, the mean-subtracted volume correlations are essentially the same as the eigenvector correlation (the only small discrepancy is that we subtract the true mean rather than the computed mean to obtain the former).}
\label{correlations}
\end{figure}


Regarding the coefficients $\alpha_{s}$ depicted in Figure \ref{coords_2}, note that in the noiseless case, there should be a distribution composed of two spikes. By adding noise to the images, the two spikes start blurring together. For SNR values up to a certain point, the distribution is still visibly bimodal. However, after a threshold the two spikes coalesce into one. The proportions $p_c$ are reliably estimated until this threshold.

\begin{figure}[H]
        \centering
        \begin{subfigure}[b]{0.3\textwidth}
                \centering
                \includegraphics[scale = 0.05]{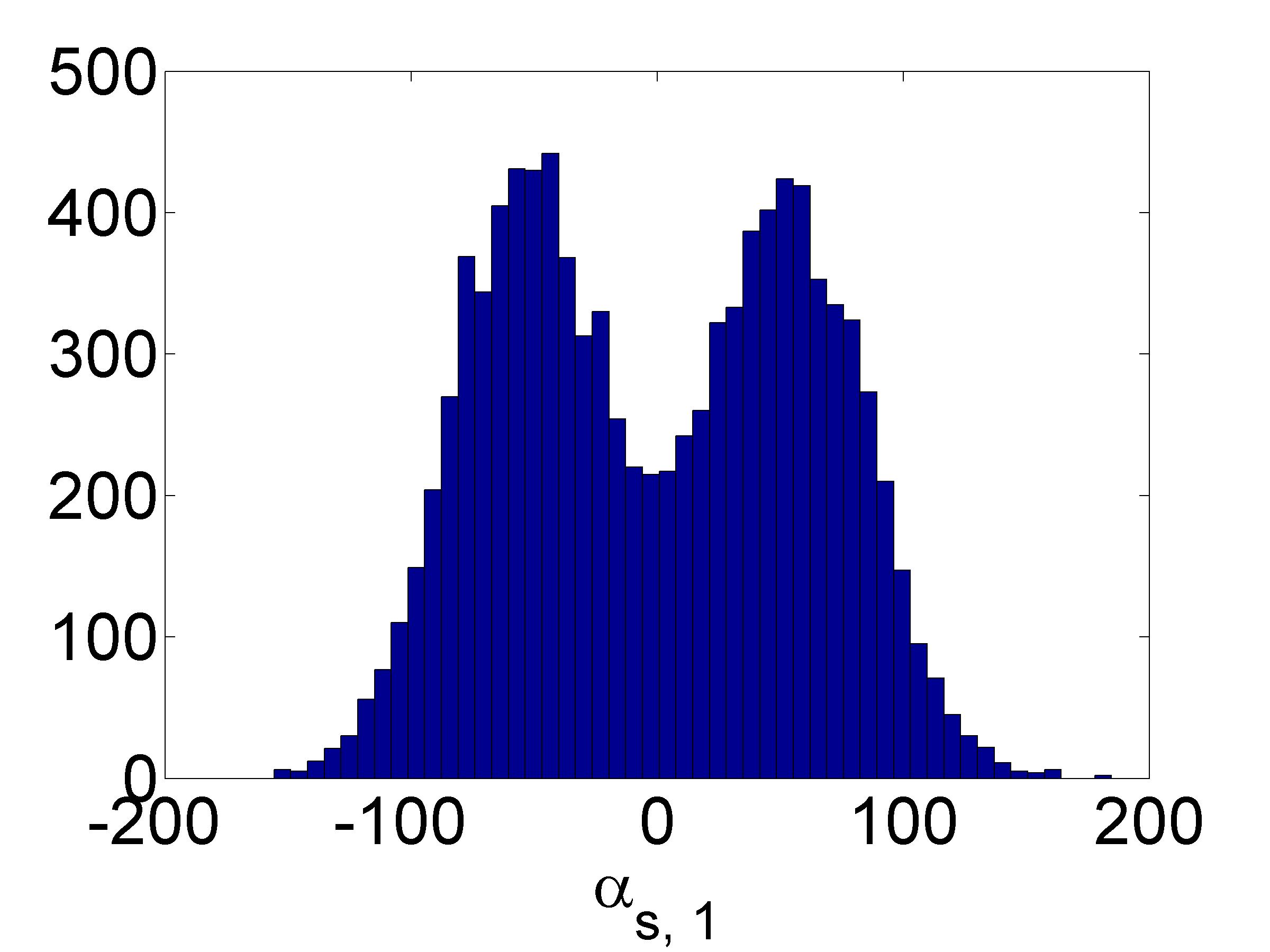}
	\caption{SNR$_{\text{het}}$ = 0.013 (0.25)}
        \end{subfigure}
	\begin{subfigure}[b]{0.3\textwidth}
                \centering
                \includegraphics[scale = 0.05]{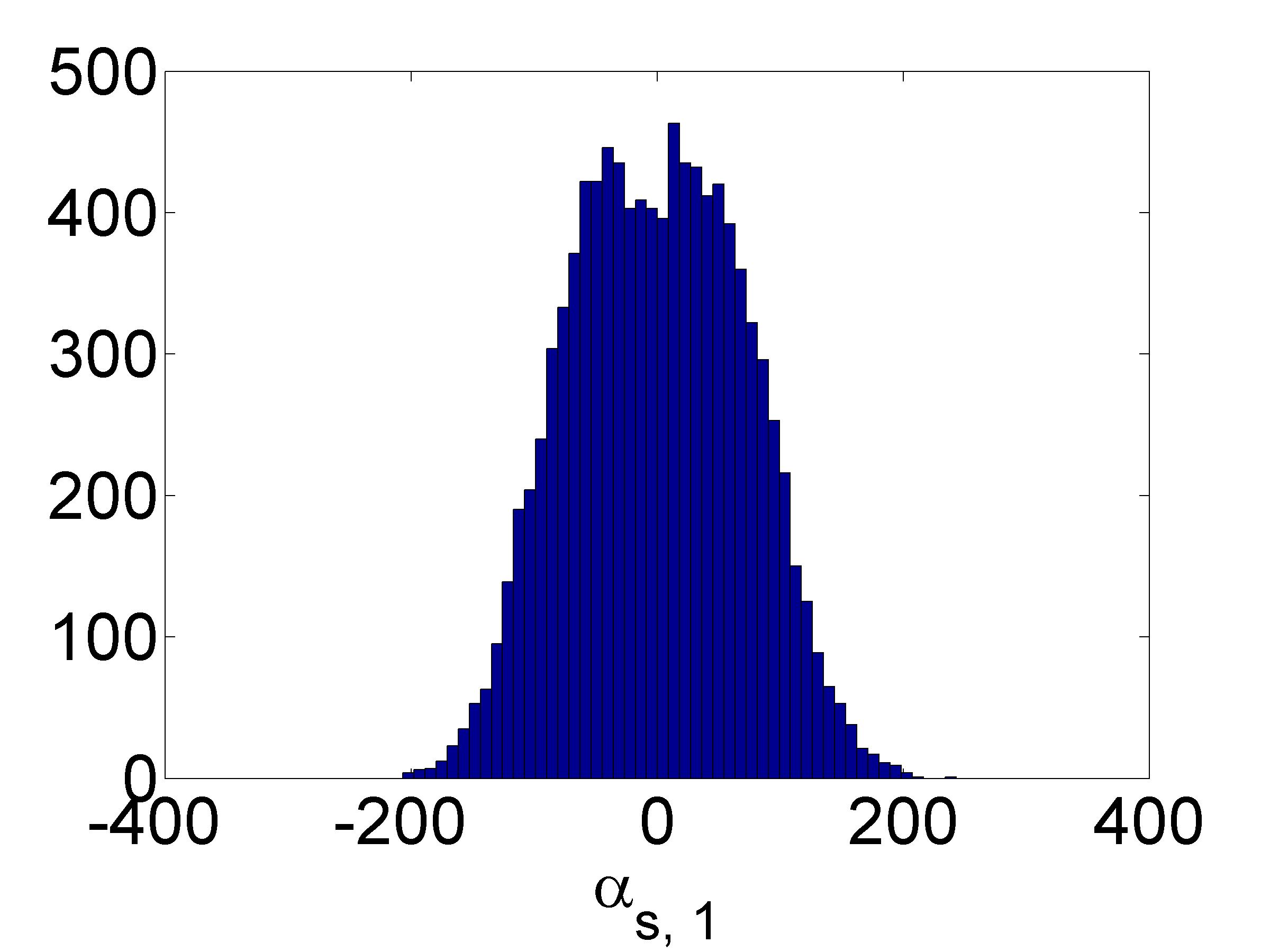}
	\caption{SNR$_{\text{het}}$ = 0.0058 (0.11)}
        \end{subfigure}
	\begin{subfigure}[b]{0.3\textwidth}
                \centering
                \includegraphics[scale = 0.05]{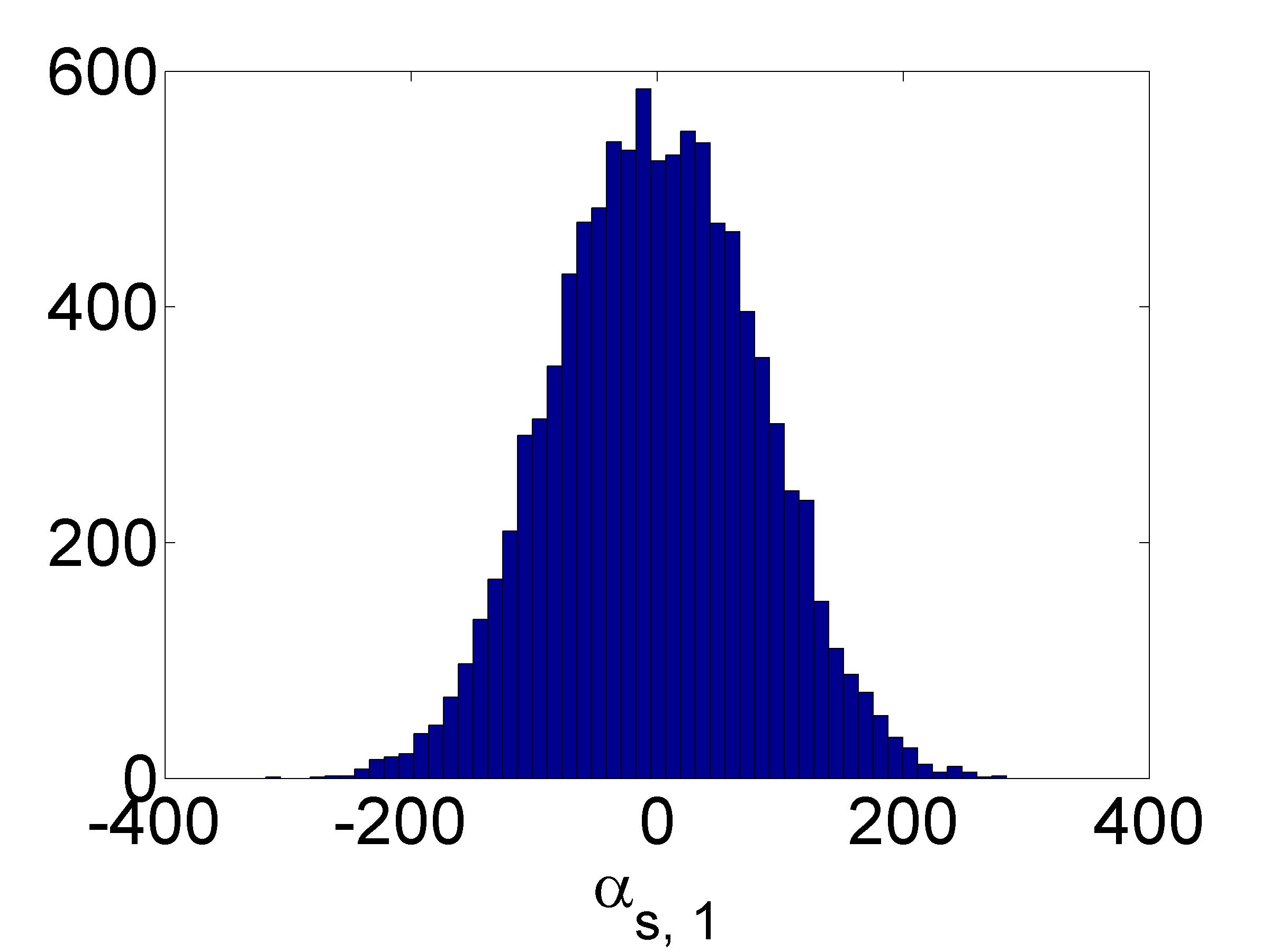}
	\caption{SNR$_{\text{het}}$ = 0.003 (0.056)}
        \end{subfigure}
\caption{Histograms of $\alpha_{s}$ for two class case. Note that (a) has a bimodal distribution corresponding to two heterogeneity classes, but these two distributions merge as SNR decreases.}
\label{coords_2}
\end{figure}


\subsection{Experiment: three classes}

In this experiment, we constructed three phantoms $\scr X^1, \scr X^2, \scr X^3$ of the form (\ref{phantom_eq}), with $M_1 = 2, M_2 = 2, M_3 = 1$. 
The cross-sections of $\scr X^1, \scr X^2, \scr X^3$ are depicted in Figure \ref{reconstructions_3} (top row, panels (d)-(f)). We chose the three classes to be equiprobable: $p_1 = p_2 = p_3 = 1/3$. Note that the theoretical covariance matrix in the three-class heterogeneity problem has rank 2.



Figures \ref{reconstructions_3},\ref{eig_hists_3}, \ref{fig:fsc_3}, \ref{correlations_3},\ref{coords_3}are the three-class analogues of Figures \ref{reconstructions}, \ref{eig_hists}, \ref{fig:fsc}, \ref{correlations}, \ref{coords_2} in the two-class case.

Qualitatively, we observe behavior similar to that in the two class case. The mean is reconstructed with at least 90\% accuracy for all SNR values considered, while both top eigenvectors experience a phase-transition phenomenon (Figure \ref{correlations_3}(a)). As with the two class case, we see that the disappearance of the eigengap coincides with the phase-transition behavior in the reconstruction of the top eigenvectors. However, in the three class case we have two eigenvectors, and we see that the accuracy of the second eigenvector decays more quickly than that of the first eigenvector. This reflects the fact that the top eigenvalue of the true covariance $\hat \Sigma_0$ is $2.1\times 10^5$, while the second eigenvalue is $1.5 \times 10^5$. These two eigenvalues differ because $\scr X^1 - \scr X^3$ has greater norm than $\scr X^2 - \scr X^3$, which means that the two directions of variation have different associated variances. Hence, recovering the second eigenvector is less robust to noise. In particular, there are SNR values for which the top eigenvector can be recovered, but the second eigenvector cannot. SNR$_{\text{het}}$ = 0.0044 (0.03) is such an example. We see in Figure \ref{eig_hists_3} that for this SNR value, only the top eigenvector pops out of the bulk distribution. In this case, we would incorrectly estimate the rank of the true covariance as 1, and conclude that $C = 2$.

The coefficients $\alpha_{s}$ follow a similar trend to those in the two class case. For high SNRs, there is a clearly defined clustering of the coordinates around three points, as in Figure \ref{coords_3}(a). As the noise is increased, the three clusters become increasingly less defined. In Figure \ref{coords_3}(b), we see that in this threshold case, the three clusters begin merging into one. As in the two class case, this is the same threshold up to which $p_c$ are accurately estimated. By the time SNR = 0.0044 (0.03), there is no visible cluster separation, just as we observed in the two class case. Although the SNR threshold for finding $p_c$ from the $\alpha_{s}$ coefficients comes earlier than the one for the eigengap, the quality of volume reconstruction roughly tracks the quality of the eigenvector reconstruction. This suggests that the estimation of cluster means is more robust than that of the probabilities $p_c$. 


\begin{figure}[H]
        \centering
        \begin{subfigure}[b]{0.15\textwidth}
                \centering
                \includegraphics[scale = 0.04]{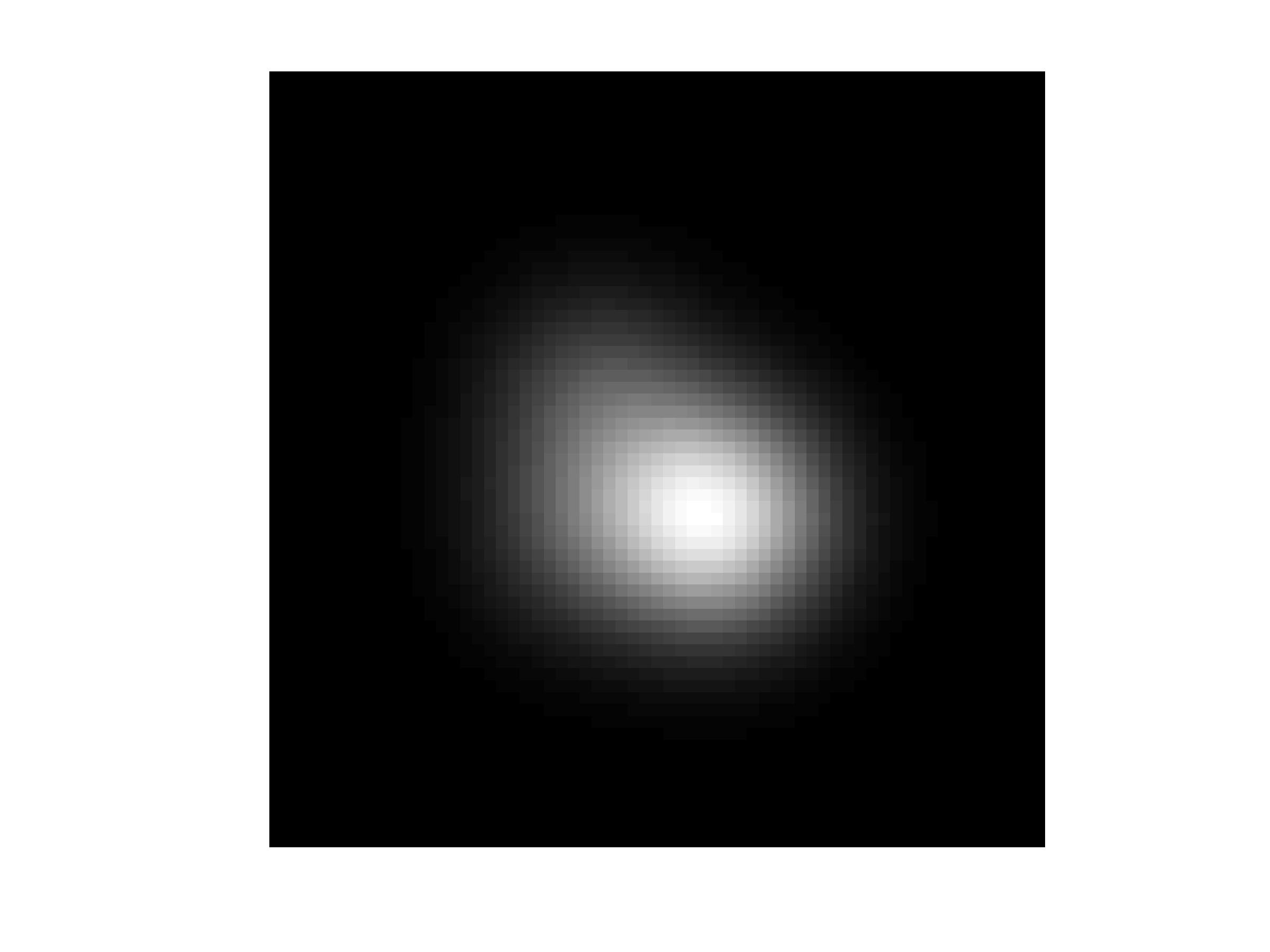}
        \end{subfigure}
\begin{subfigure}[b]{0.15\textwidth}
                \centering
                \includegraphics[scale = 0.04]{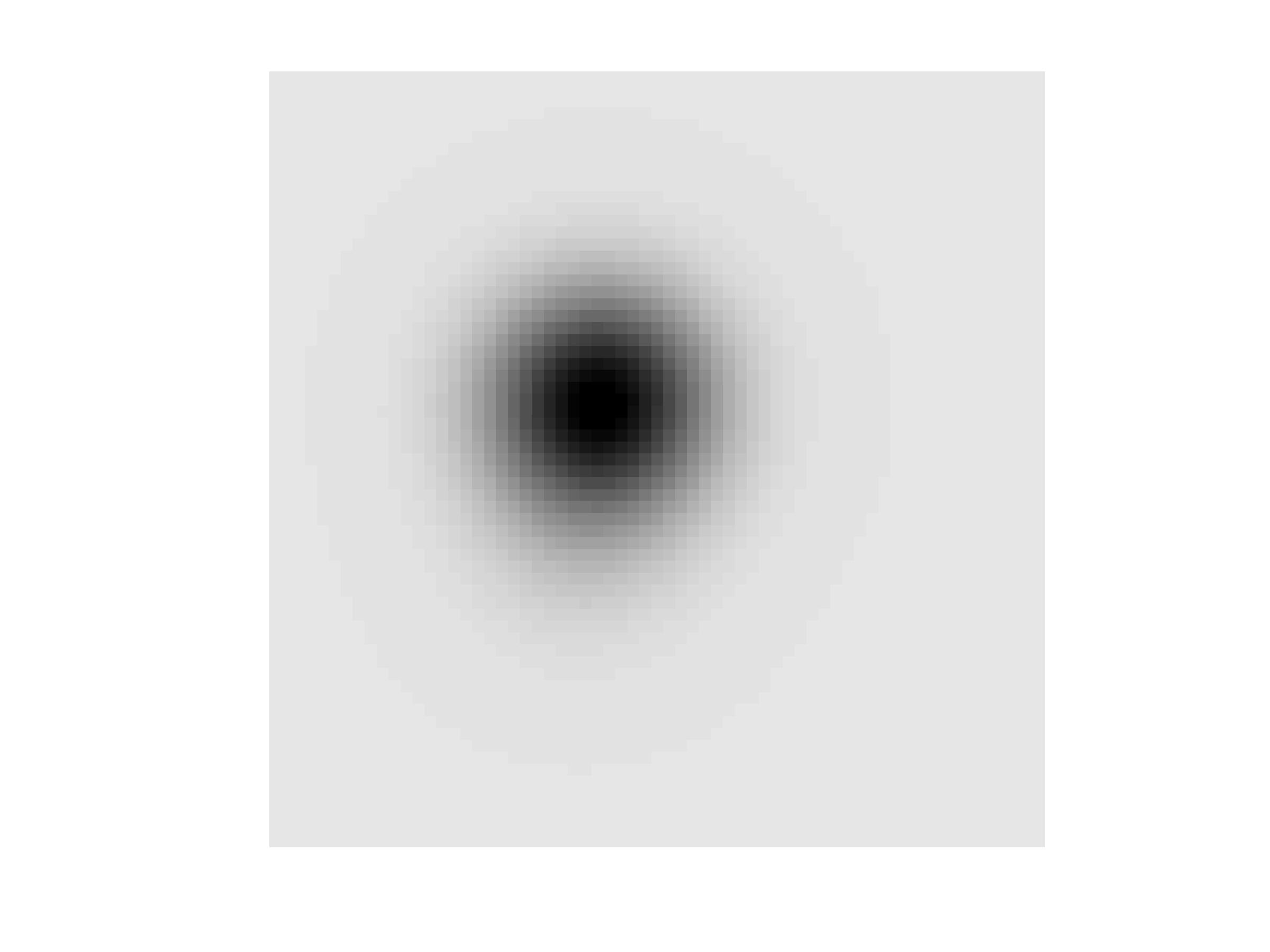}
        \end{subfigure}
\begin{subfigure}[b]{0.15\textwidth}
                \centering
                \includegraphics[scale = 0.04]{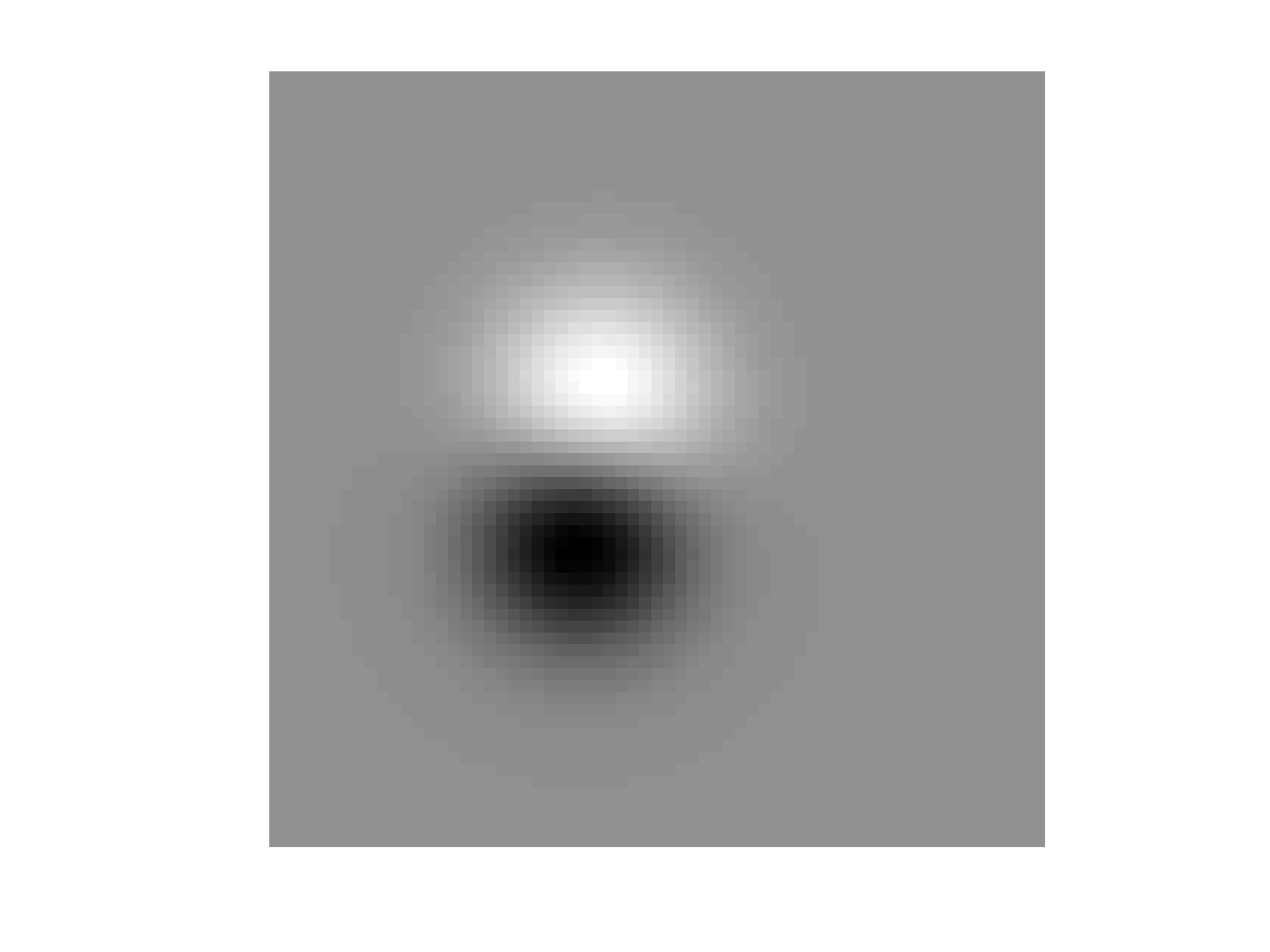}
        \end{subfigure}
\begin{subfigure}[b]{0.15\textwidth}
                \centering
                \includegraphics[scale = 0.04]{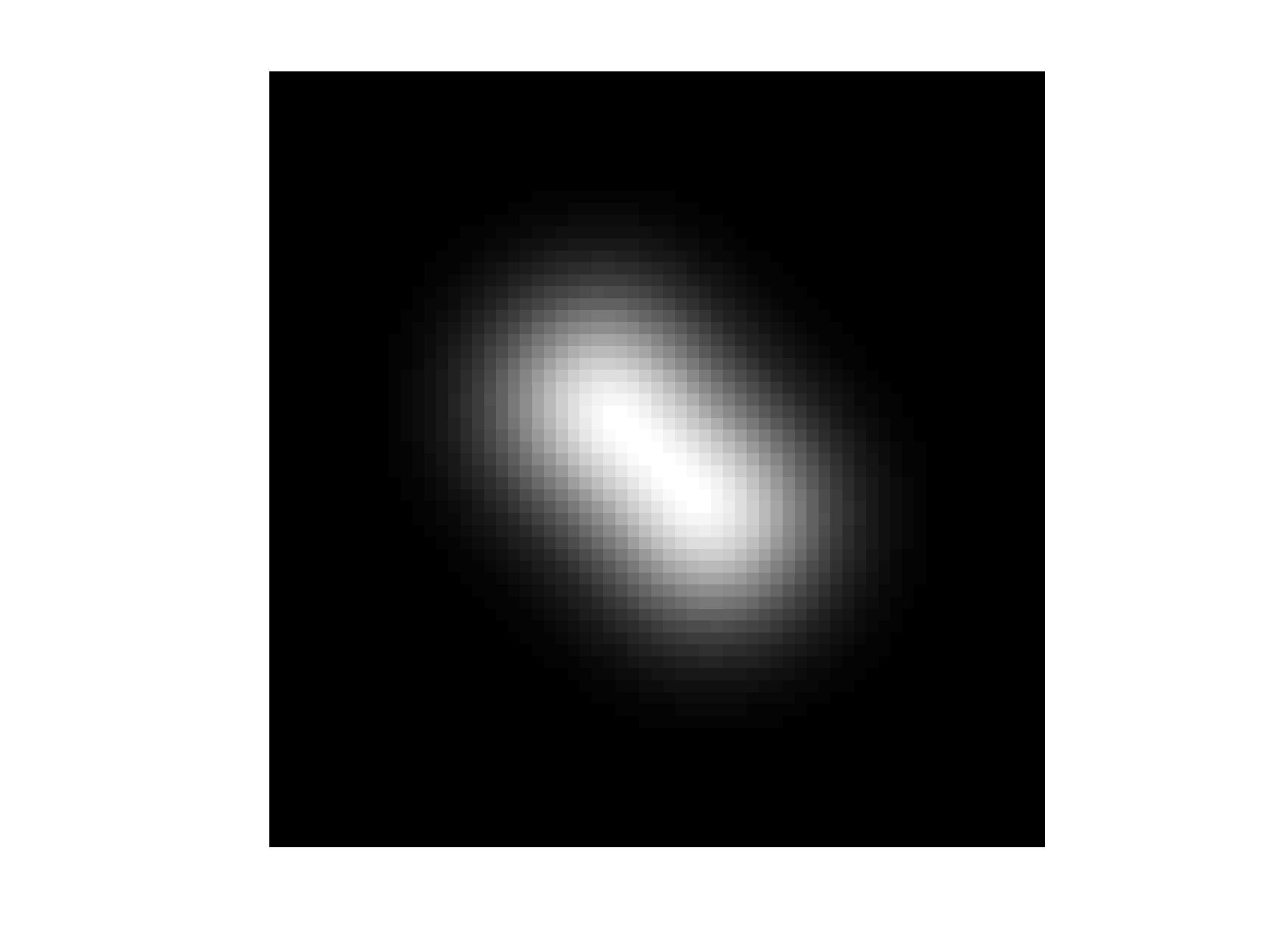}
        \end{subfigure}
      \begin{subfigure}[b]{0.15\textwidth}
                \centering
                \includegraphics[scale = 0.04]{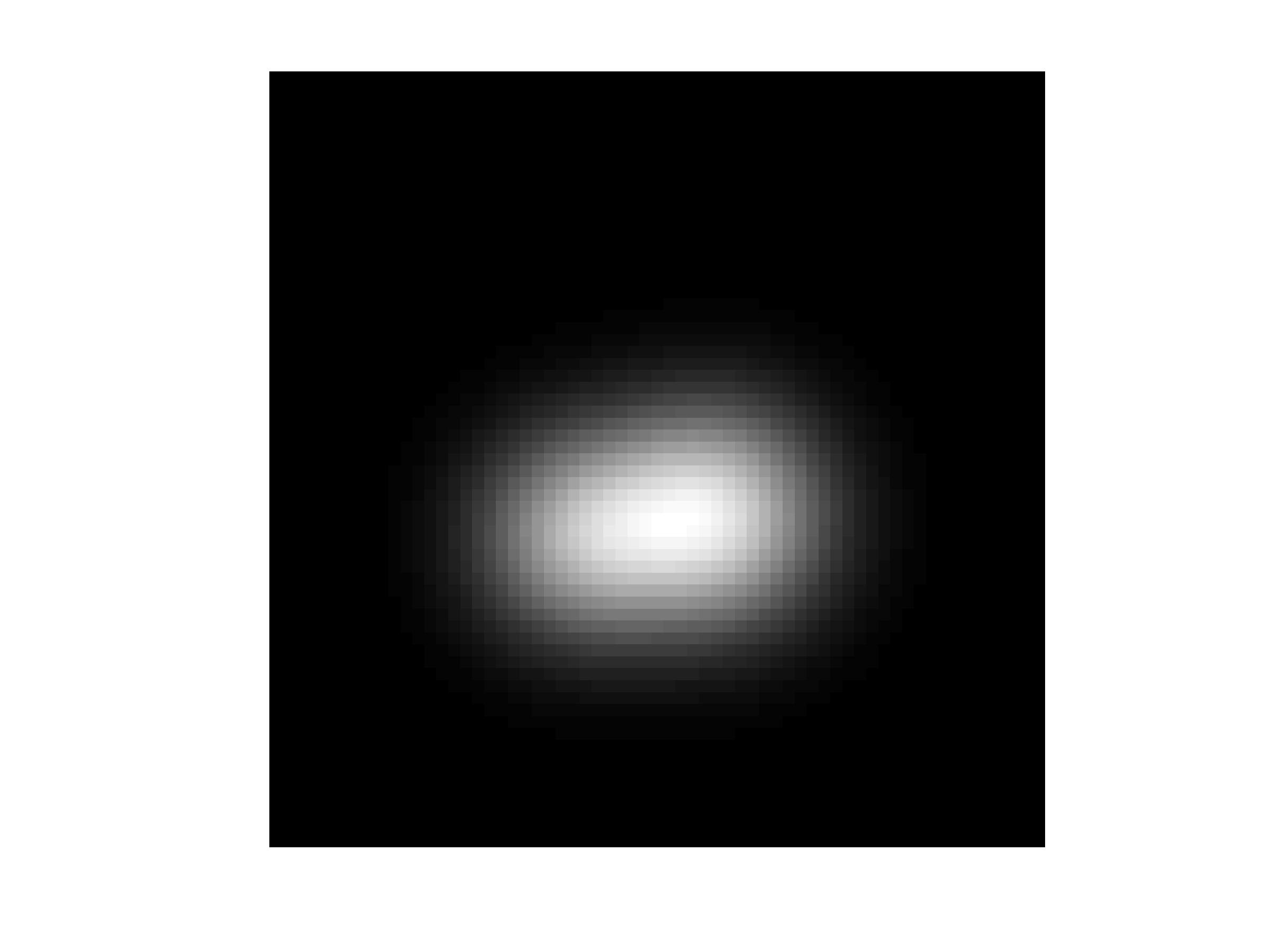}
        \end{subfigure}
        \begin{subfigure}[b]{0.15\textwidth}
                \centering
                \includegraphics[scale = 0.04]{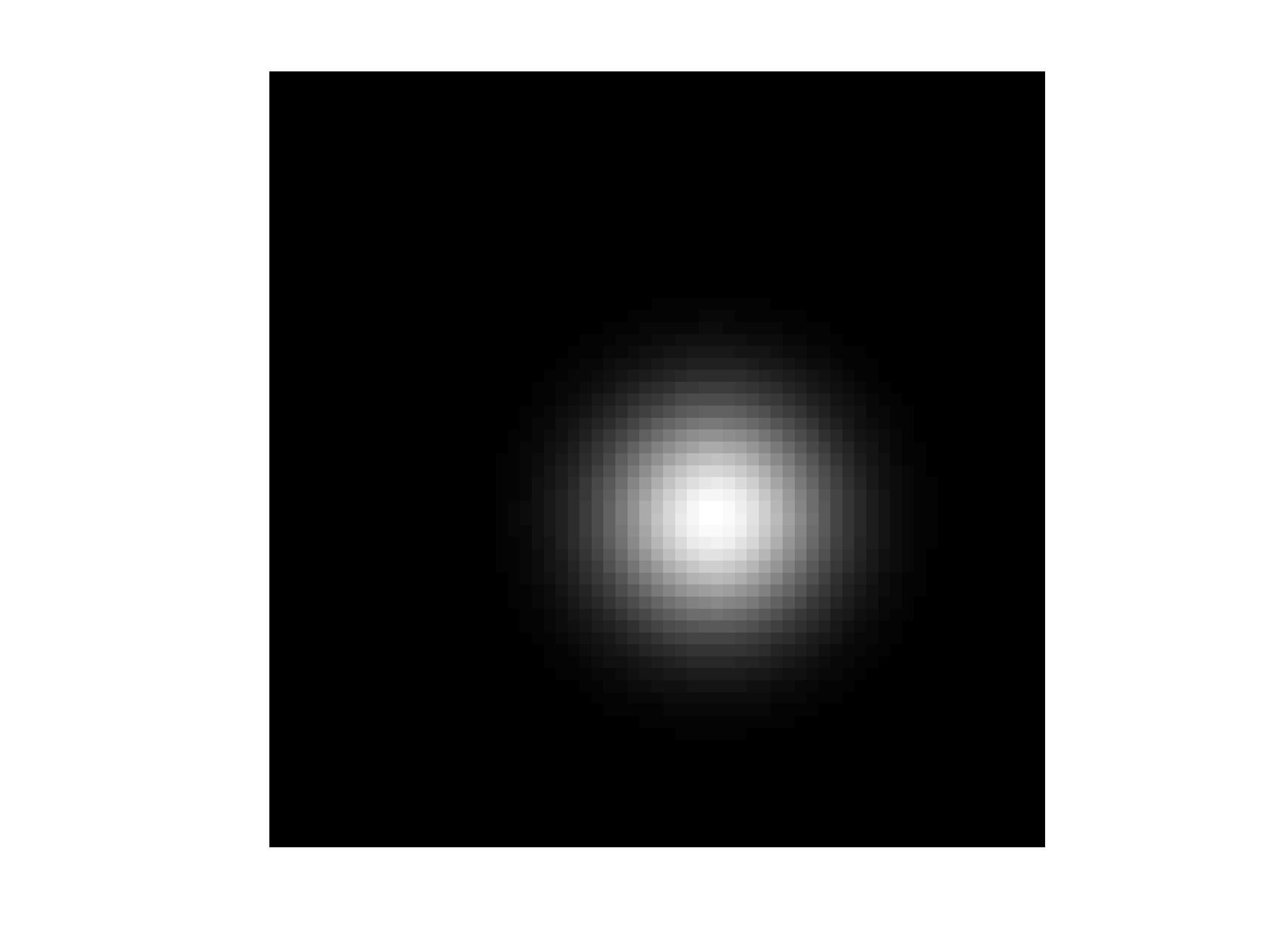}
        \end{subfigure} \\

	\begin{subfigure}[b]{0.15\textwidth}
                \centering
                \includegraphics[scale = 0.04]{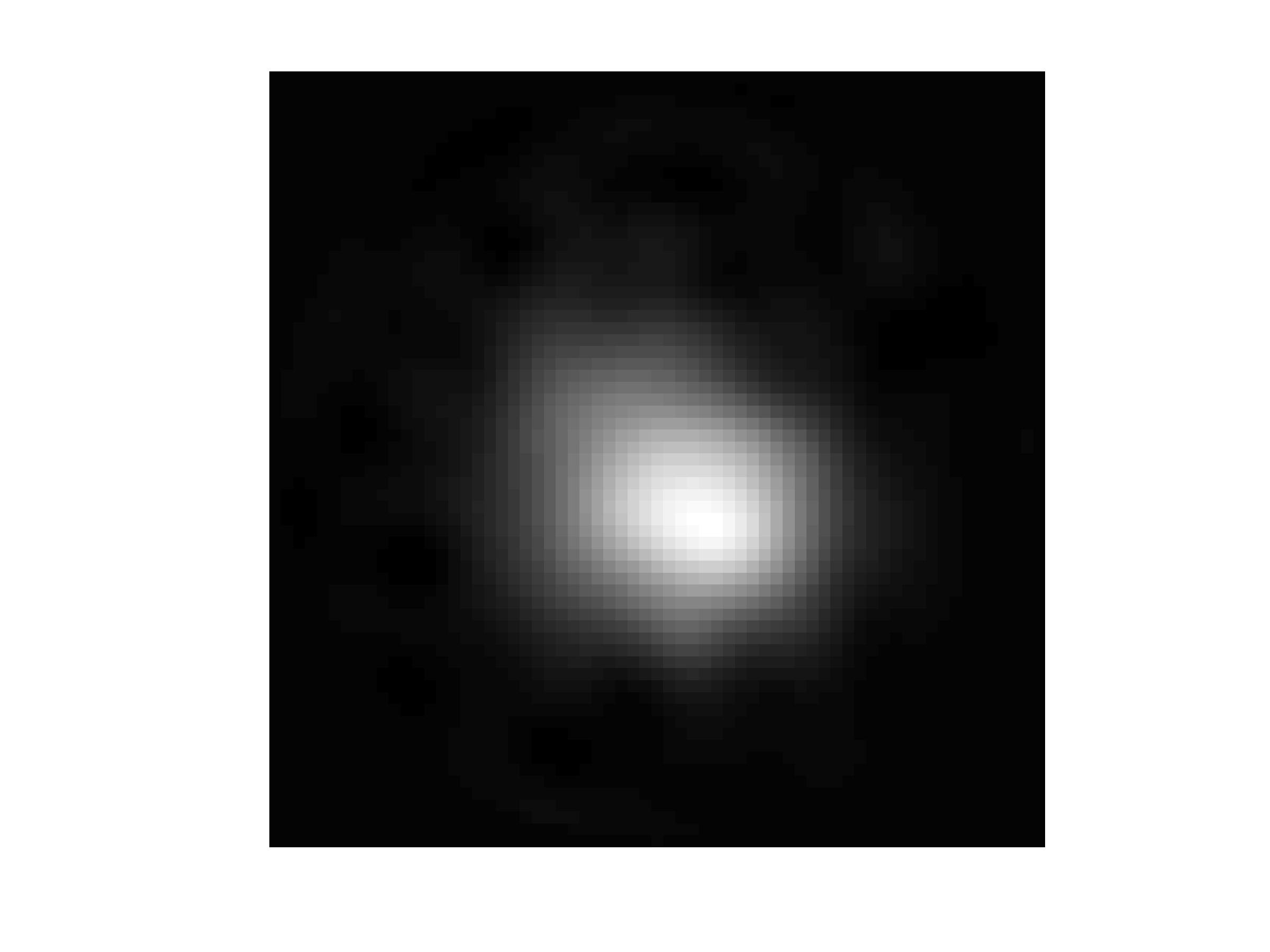}
        \end{subfigure}
	\begin{subfigure}[b]{0.15\textwidth}
                \centering
                \includegraphics[scale = 0.04]{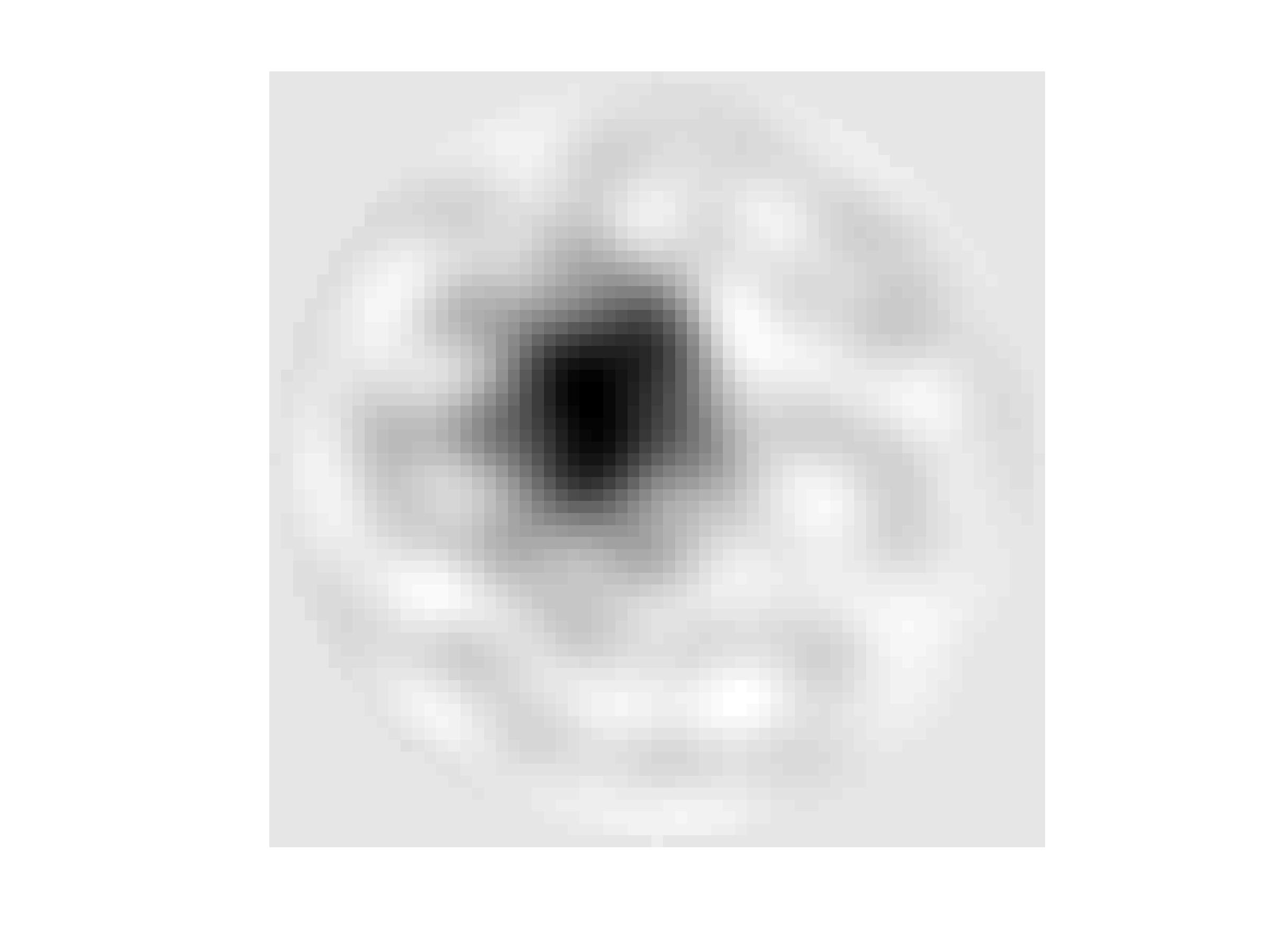}
        \end{subfigure}
	\begin{subfigure}[b]{0.15\textwidth}
                \centering
                \includegraphics[scale = 0.04]{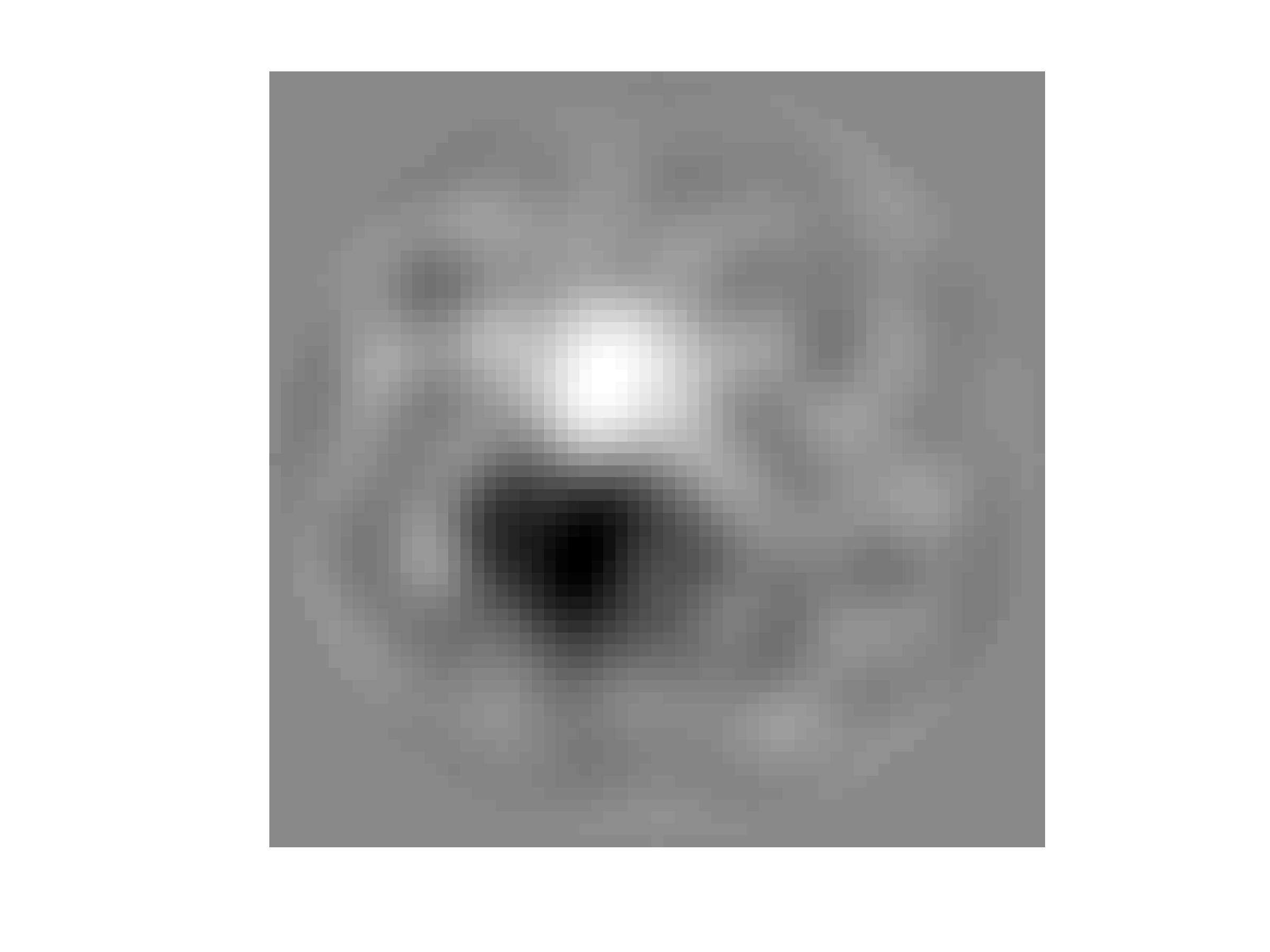}
        \end{subfigure}
	\begin{subfigure}[b]{0.15\textwidth}
                \centering
                \includegraphics[scale = 0.04]{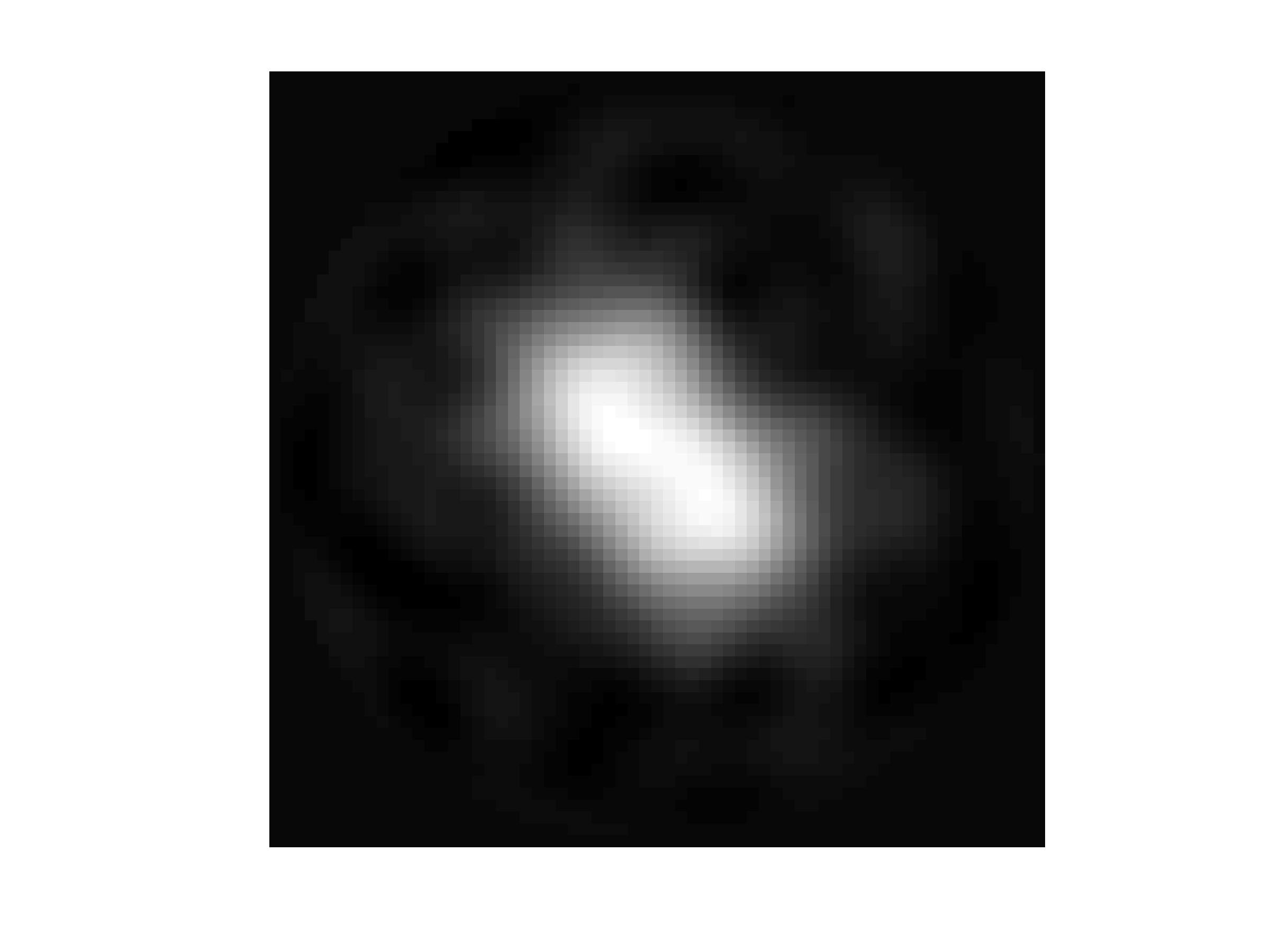}
        \end{subfigure}
	\begin{subfigure}[b]{0.15\textwidth}
                \centering
                \includegraphics[scale = 0.04]{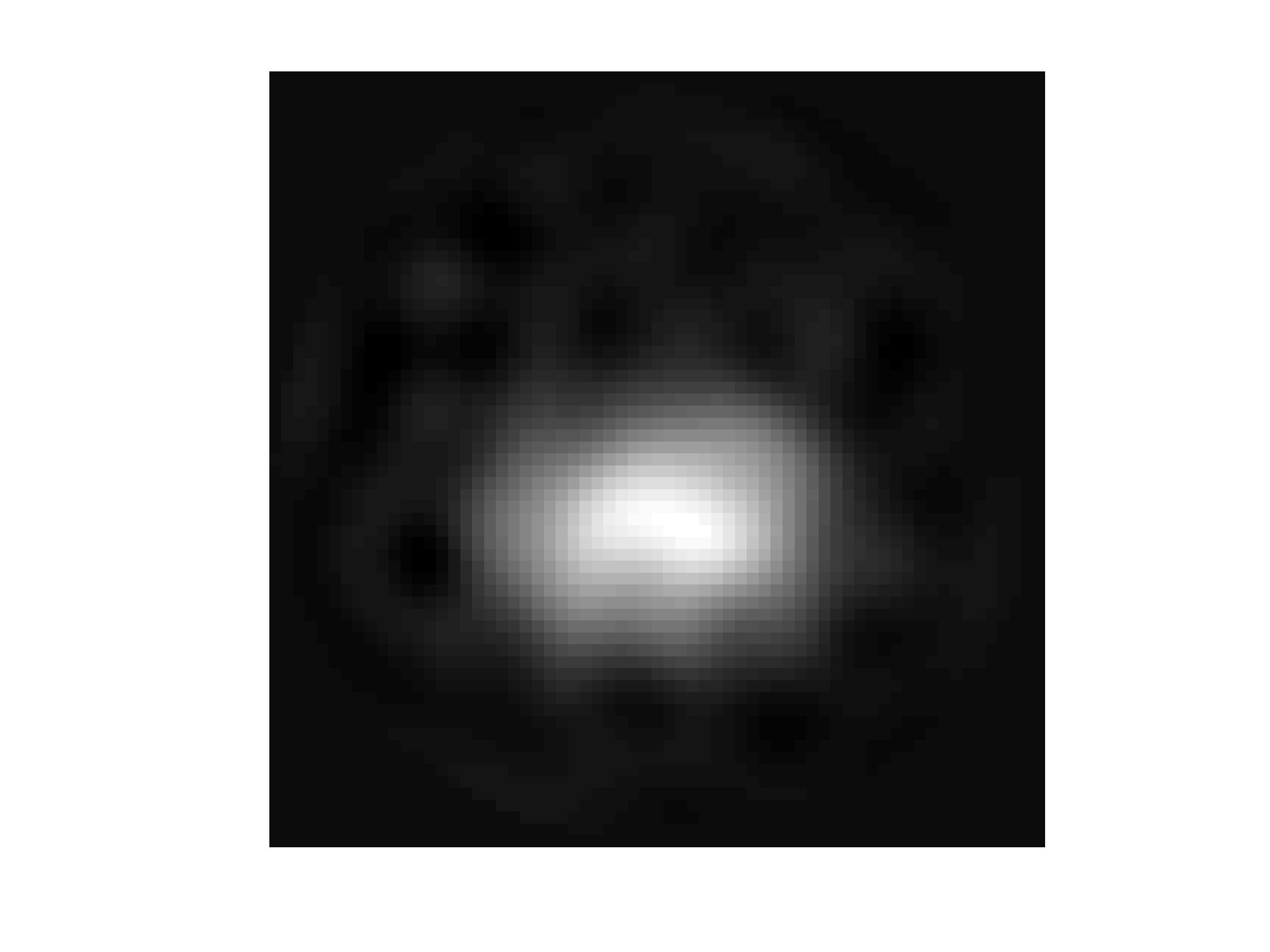}
        \end{subfigure}
	\begin{subfigure}[b]{0.15\textwidth}
                \centering
                \includegraphics[scale = 0.04]{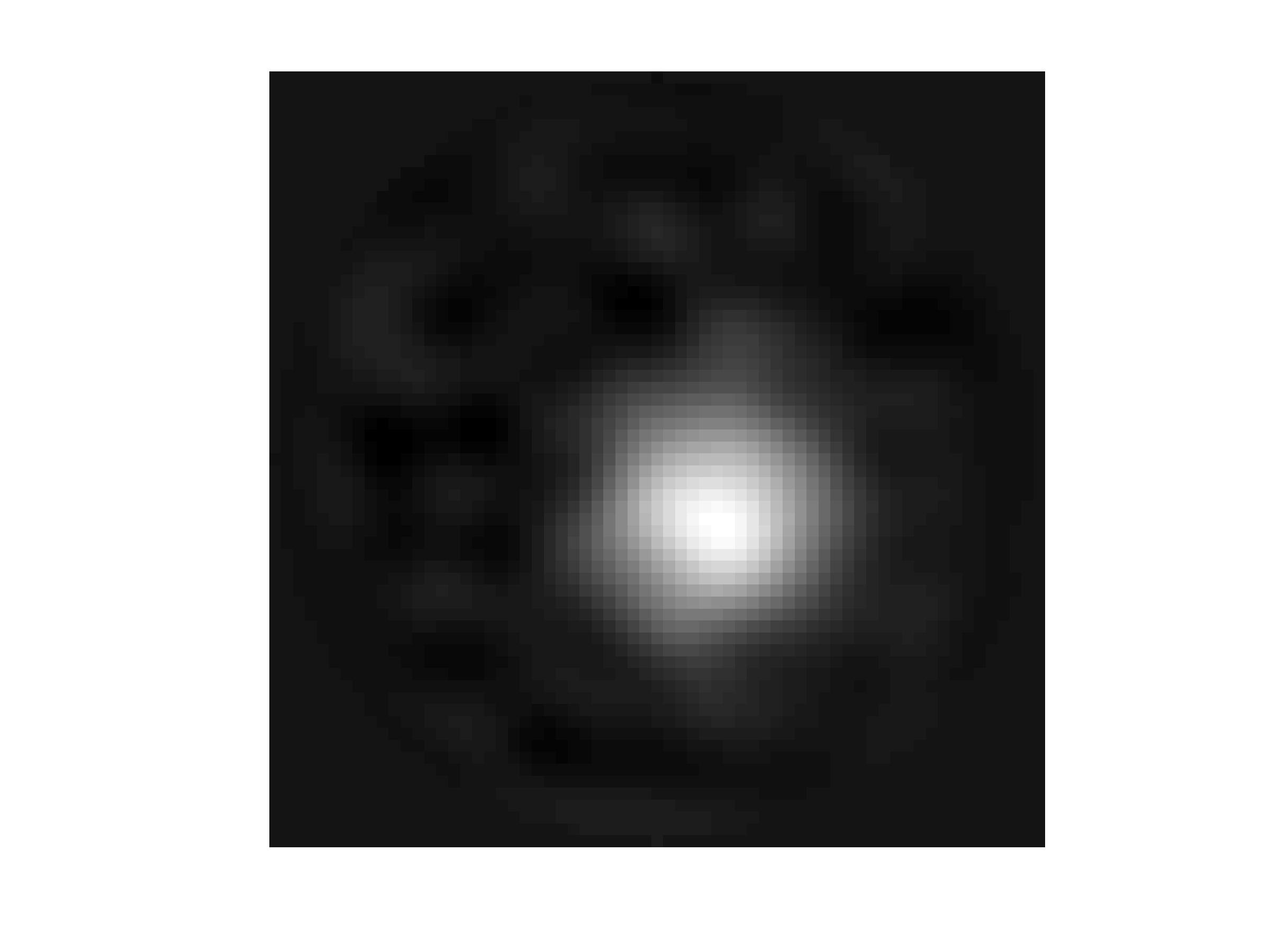}
        \end{subfigure} \\

\begin{subfigure}[b]{0.15\textwidth}
                \centering
                \includegraphics[scale = 0.04]{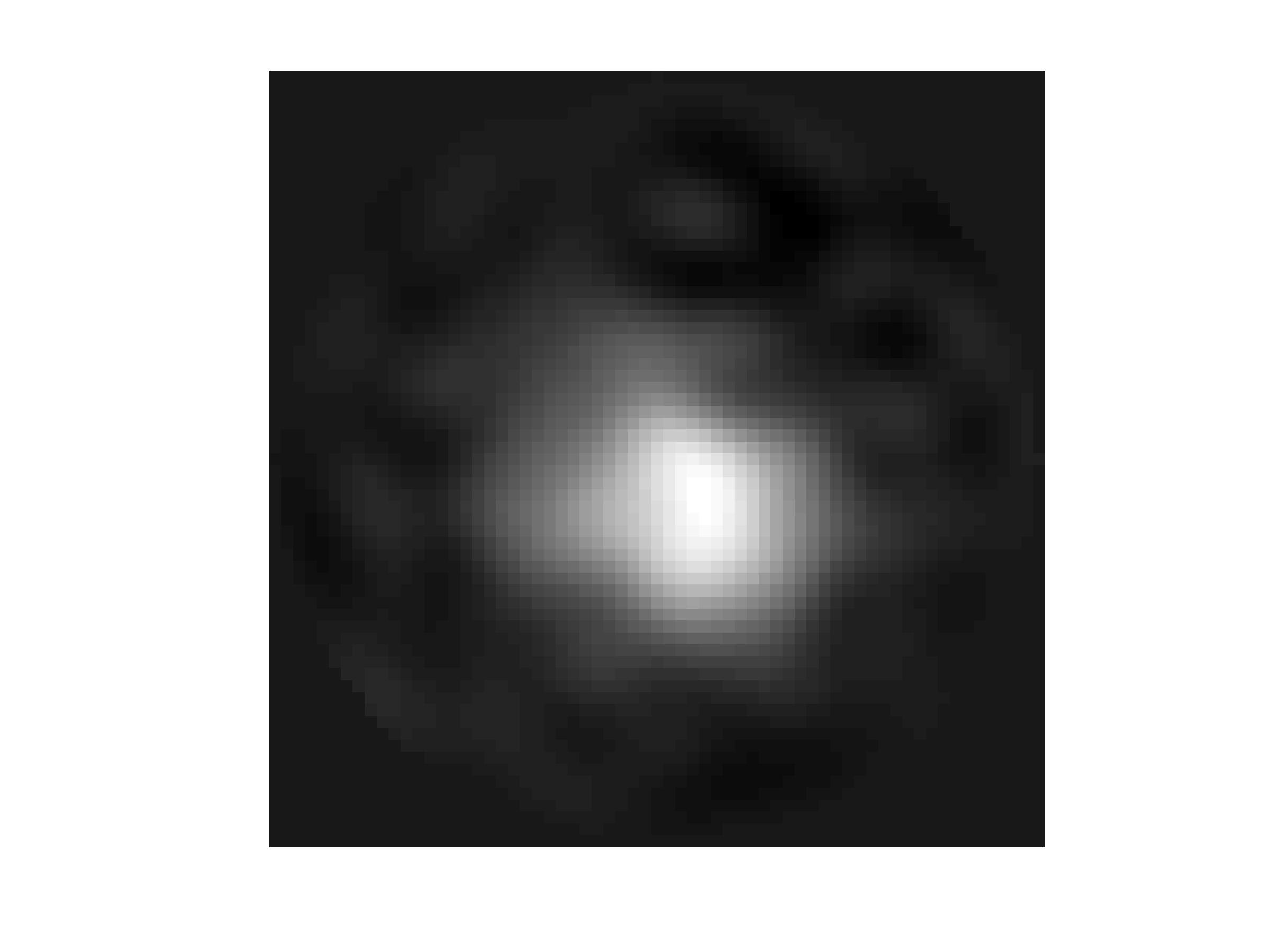}
        \end{subfigure}
\begin{subfigure}[b]{0.15\textwidth}
                \centering
                \includegraphics[scale = 0.04]{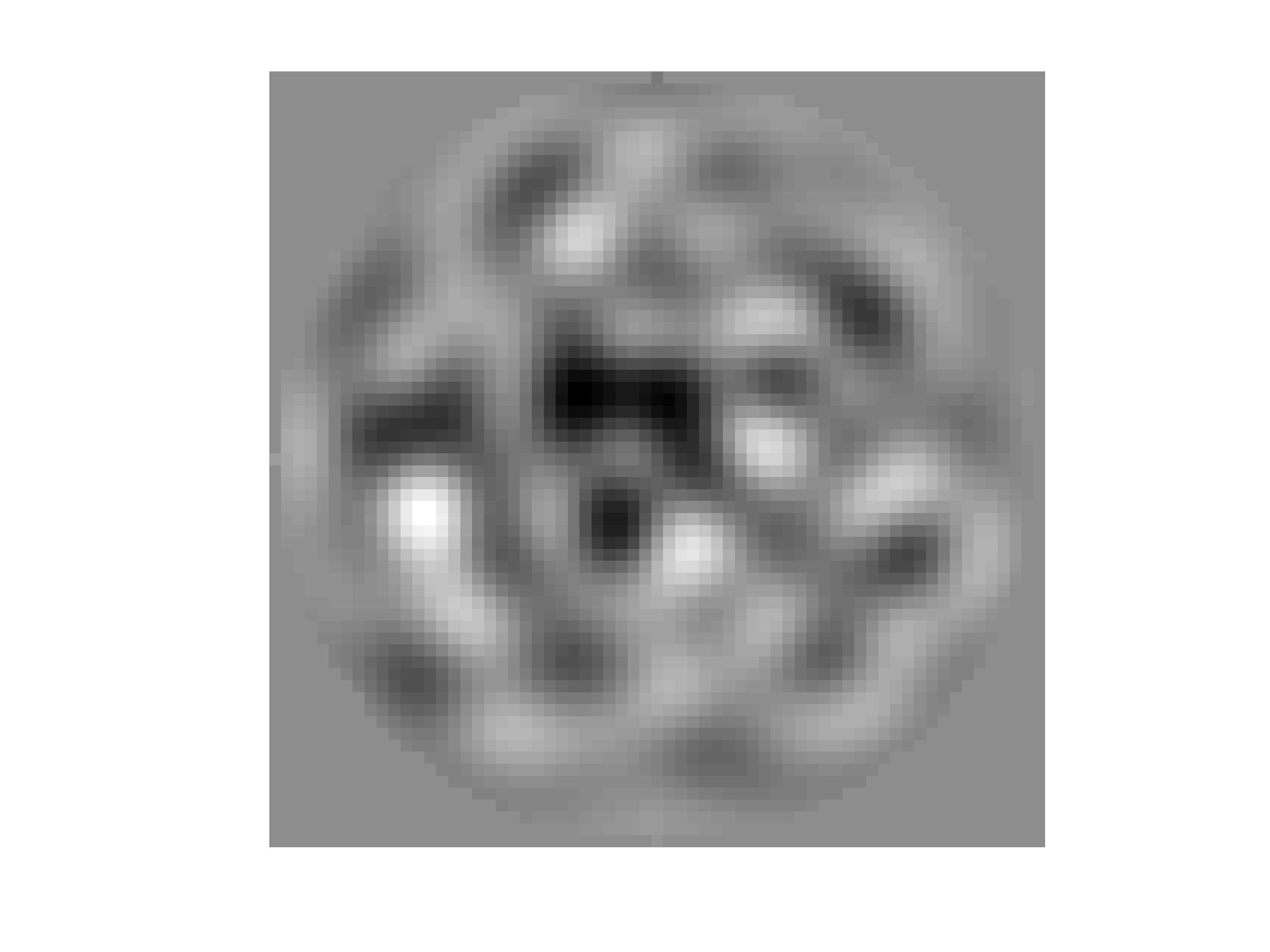}
        \end{subfigure}
\begin{subfigure}[b]{0.15\textwidth}
                \centering
                \includegraphics[scale = 0.04]{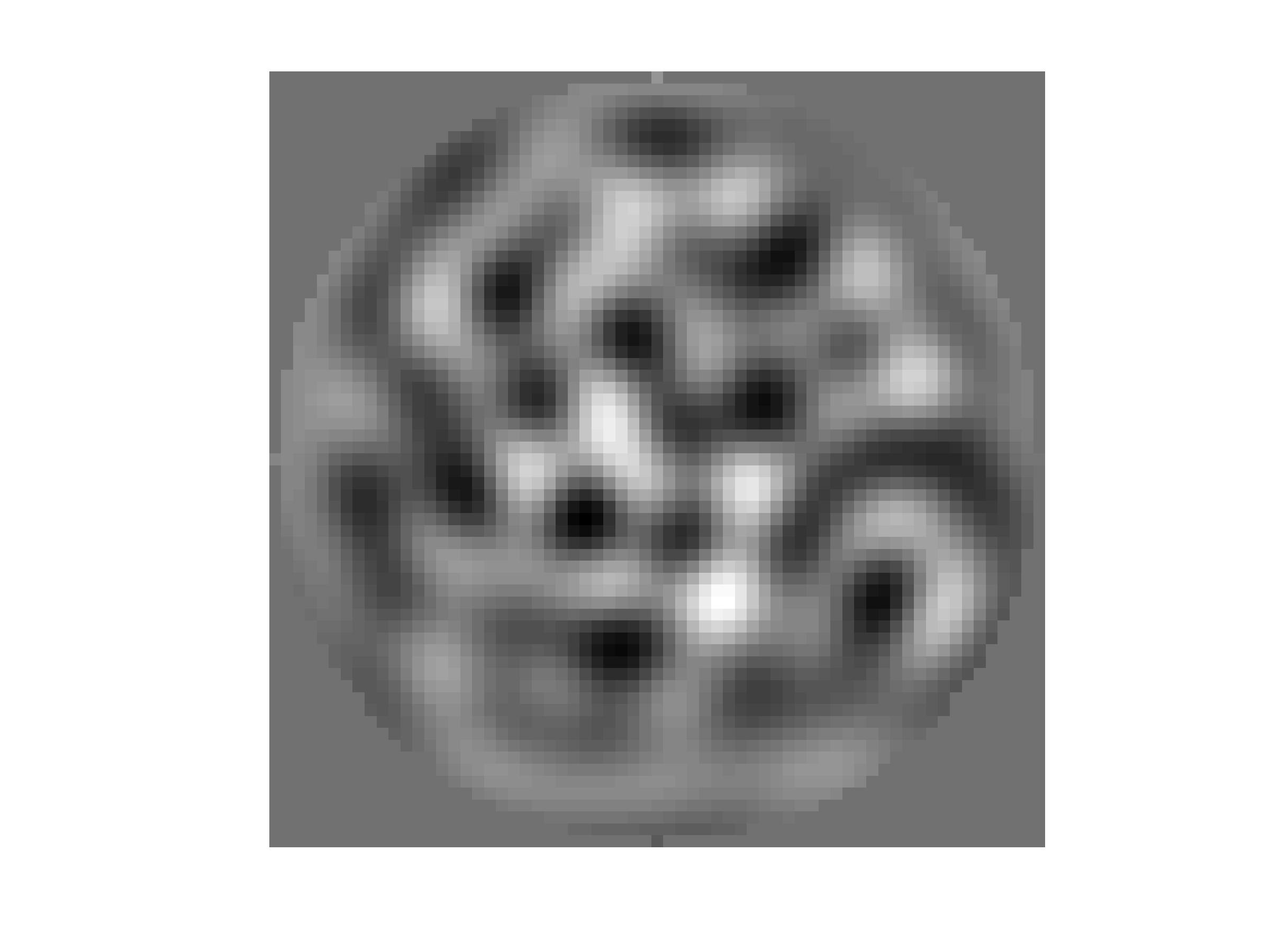}
        \end{subfigure}
\begin{subfigure}[b]{0.15\textwidth}
                \centering
                \includegraphics[scale = 0.04]{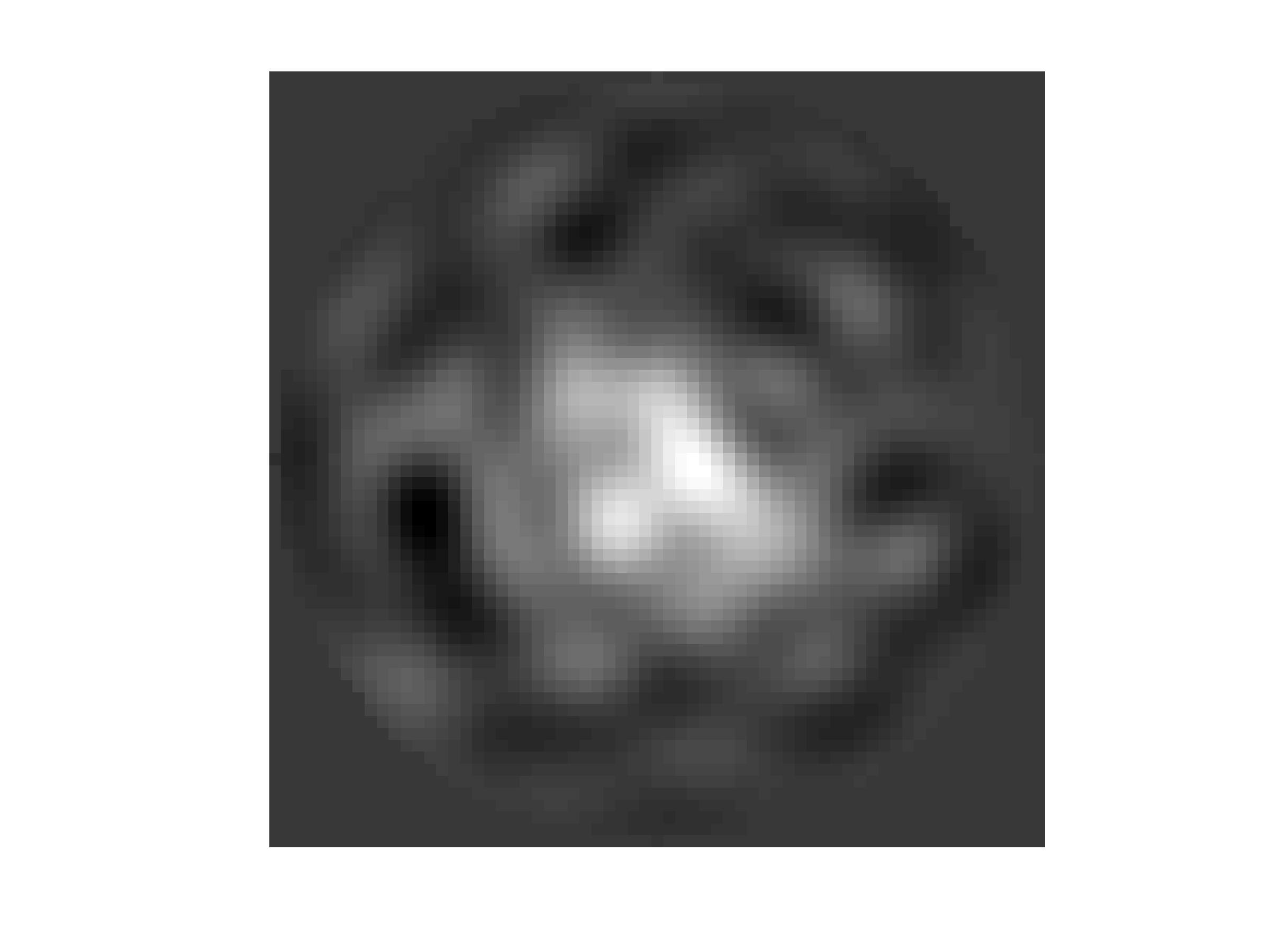}
        \end{subfigure}
\begin{subfigure}[b]{0.15\textwidth}
                \centering
                \includegraphics[scale = 0.04]{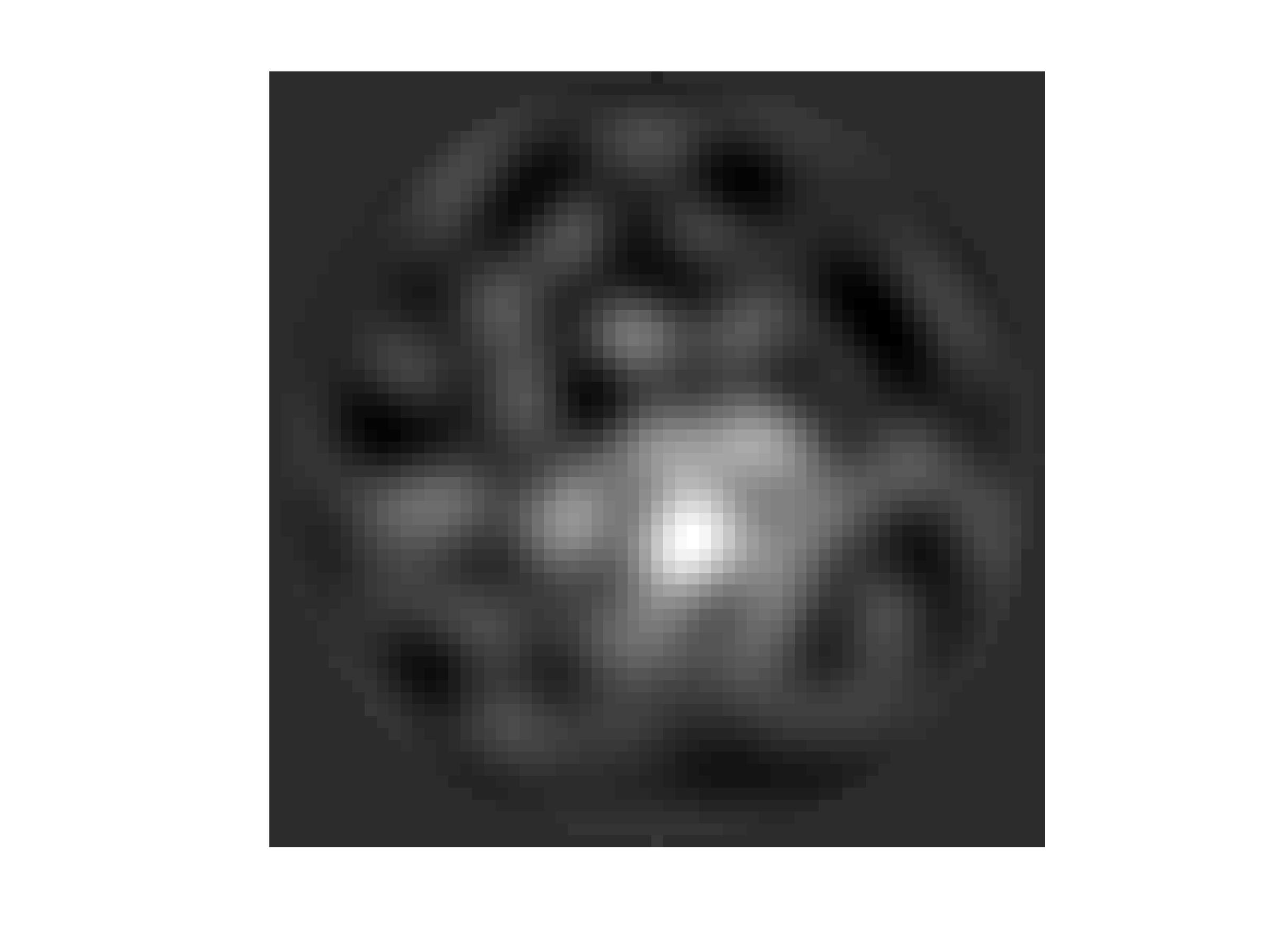}
        \end{subfigure}
\begin{subfigure}[b]{0.15\textwidth}
                \centering
                \includegraphics[scale = 0.04]{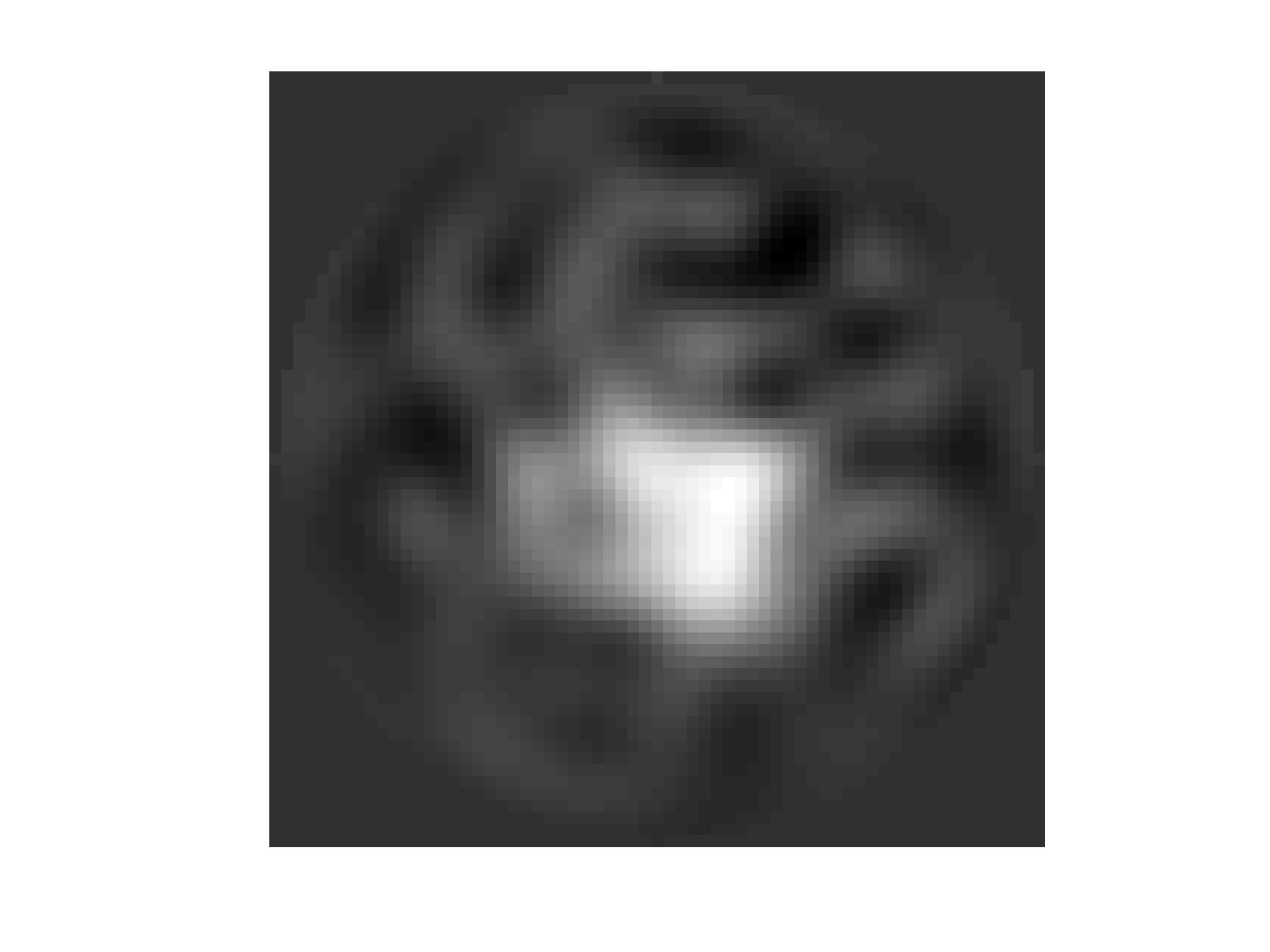}
        \end{subfigure} \\

\begin{subfigure}[b]{0.15\textwidth}
                \centering
                \includegraphics[scale = 0.04]{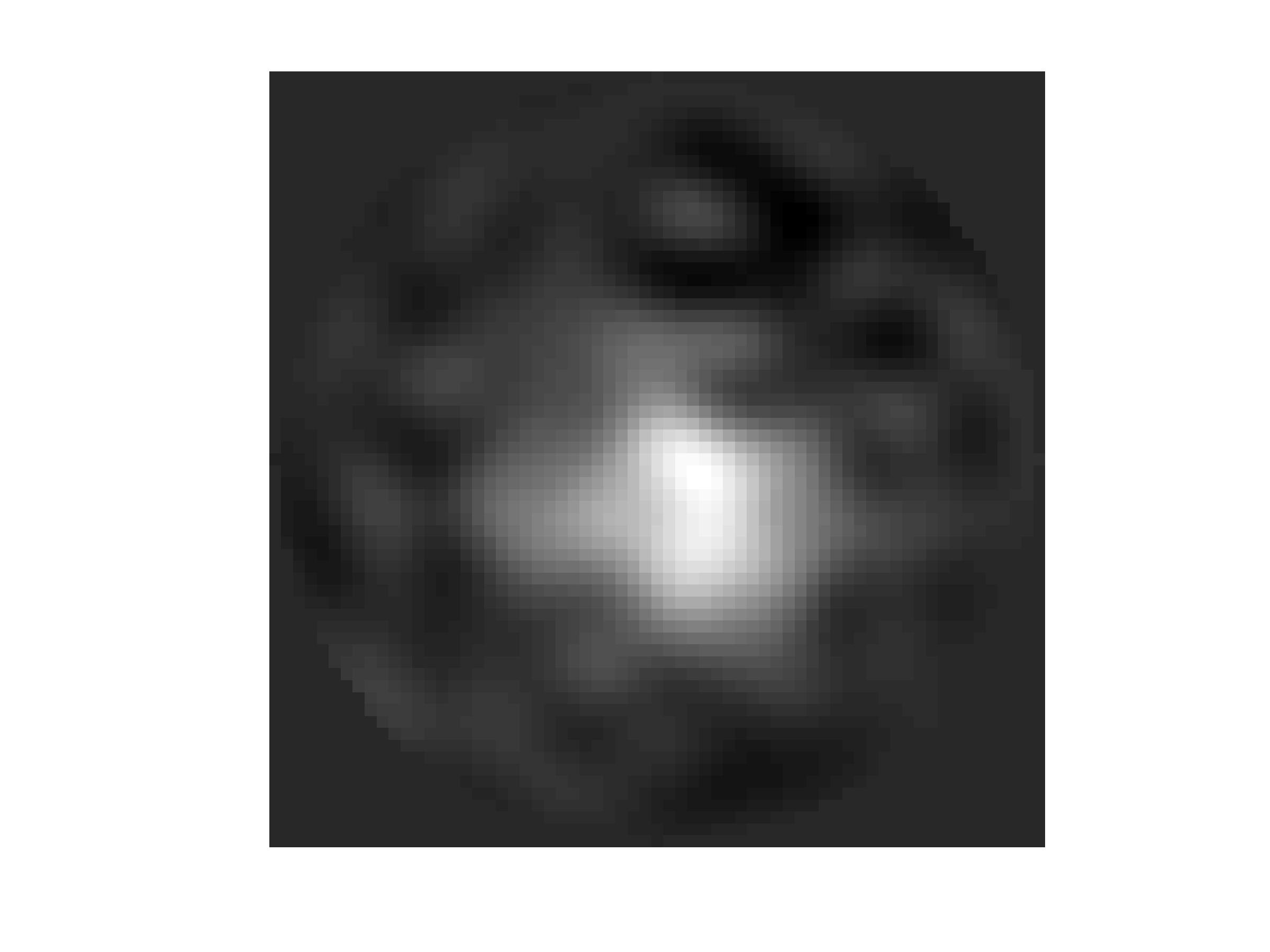}
\caption{Mean}
        \end{subfigure} 
\begin{subfigure}[b]{0.15\textwidth}
                \centering
                \includegraphics[scale = 0.04]{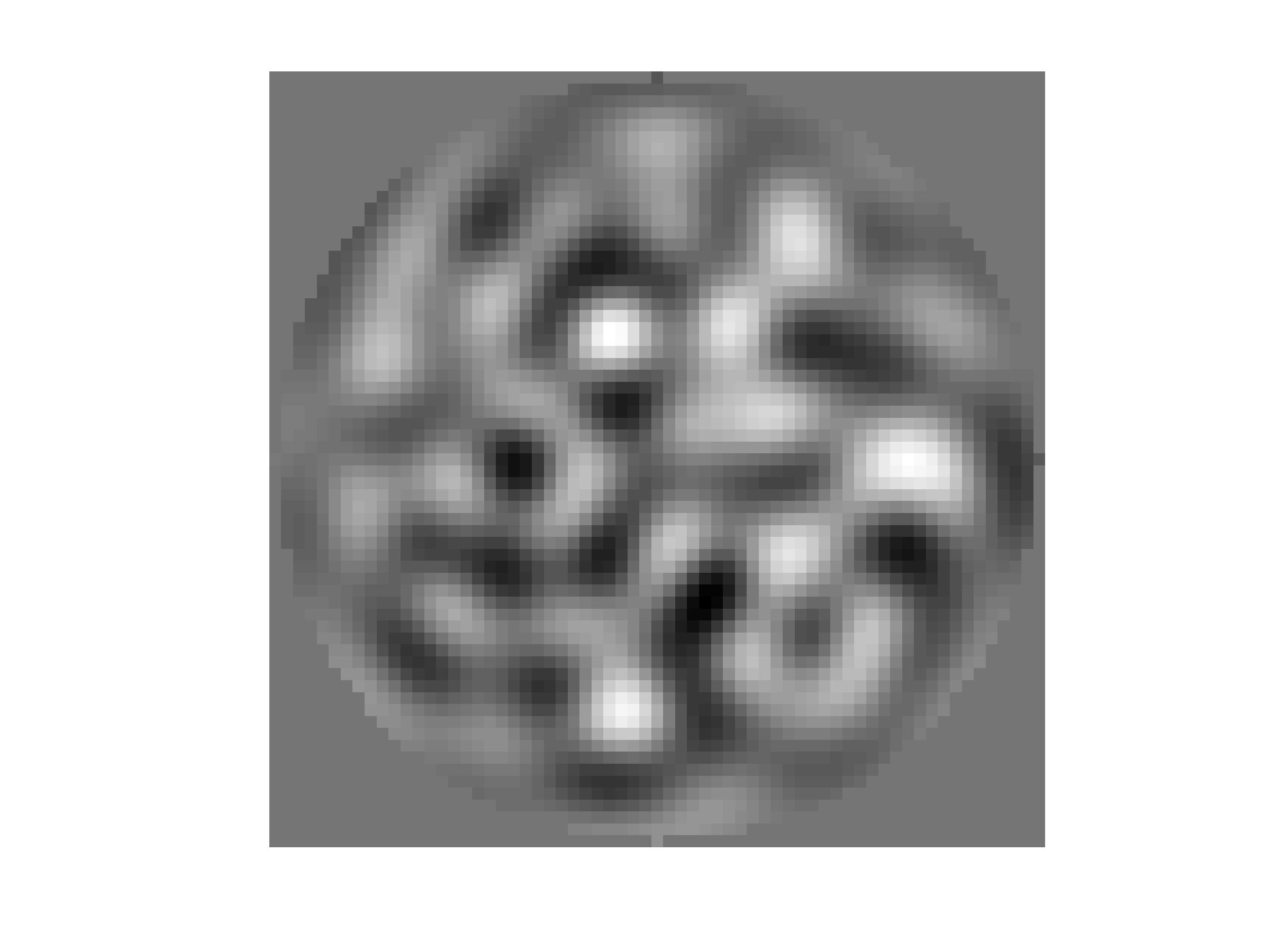}
\caption{Eigenv. 1}
        \end{subfigure}
\begin{subfigure}[b]{0.15\textwidth}
                \centering
                \includegraphics[scale = 0.04]{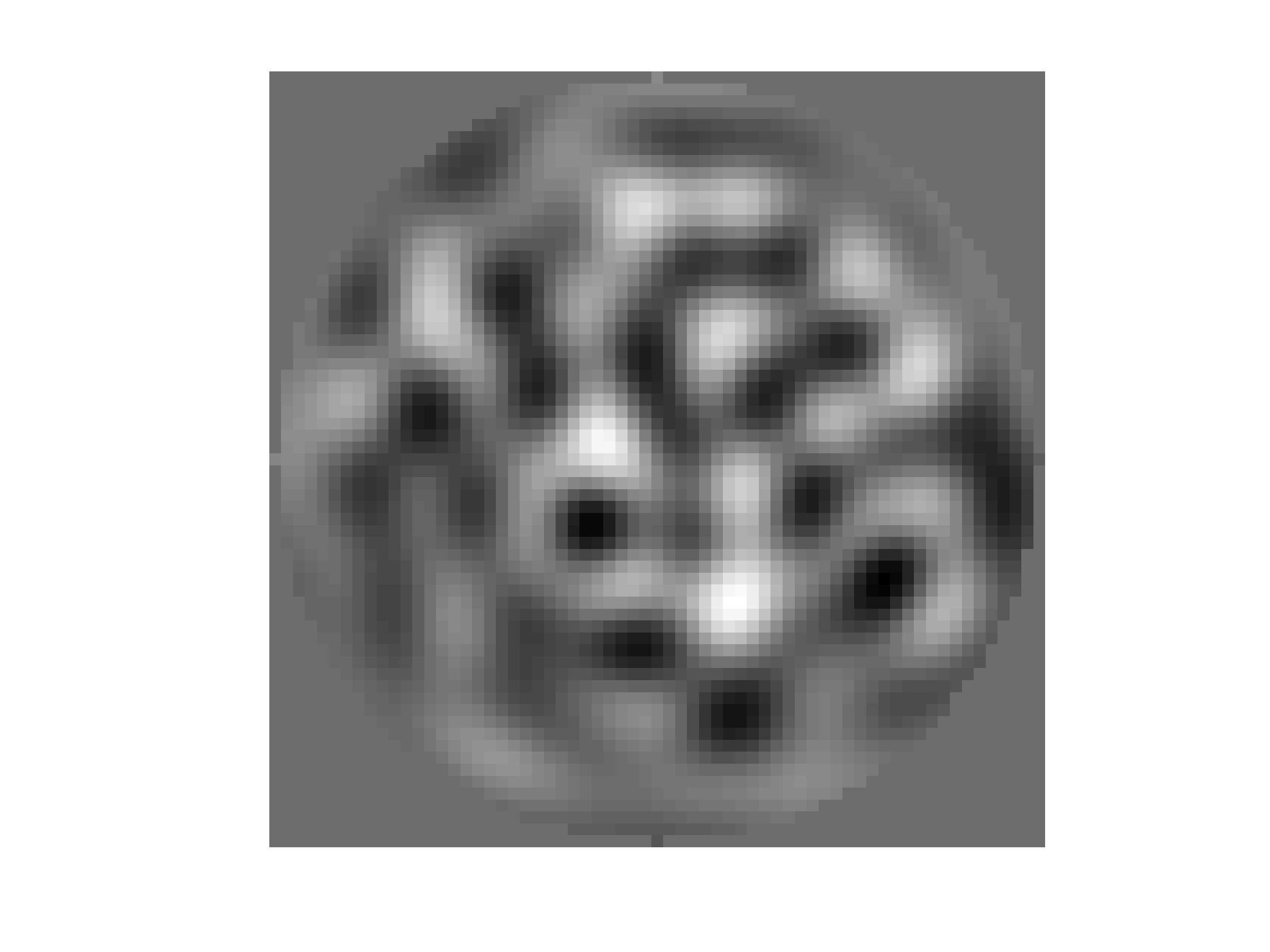}
\caption{Eigenv. 2}
        \end{subfigure}
\begin{subfigure}[b]{0.15\textwidth}
                \centering
                \includegraphics[scale = 0.04]{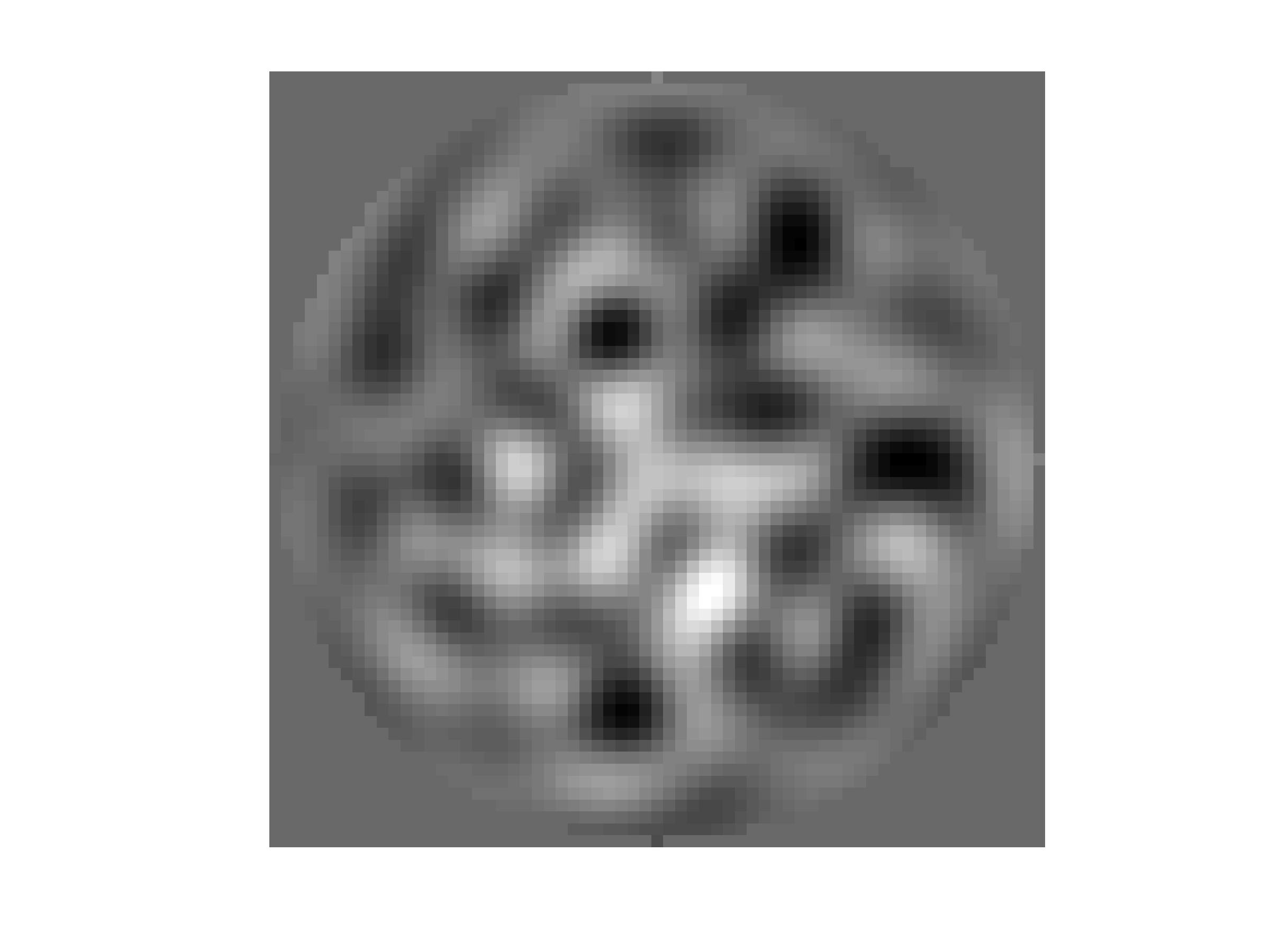}
\caption{Volume 1}
        \end{subfigure}
\begin{subfigure}[b]{0.15\textwidth}
                \centering
                \includegraphics[scale = 0.04]{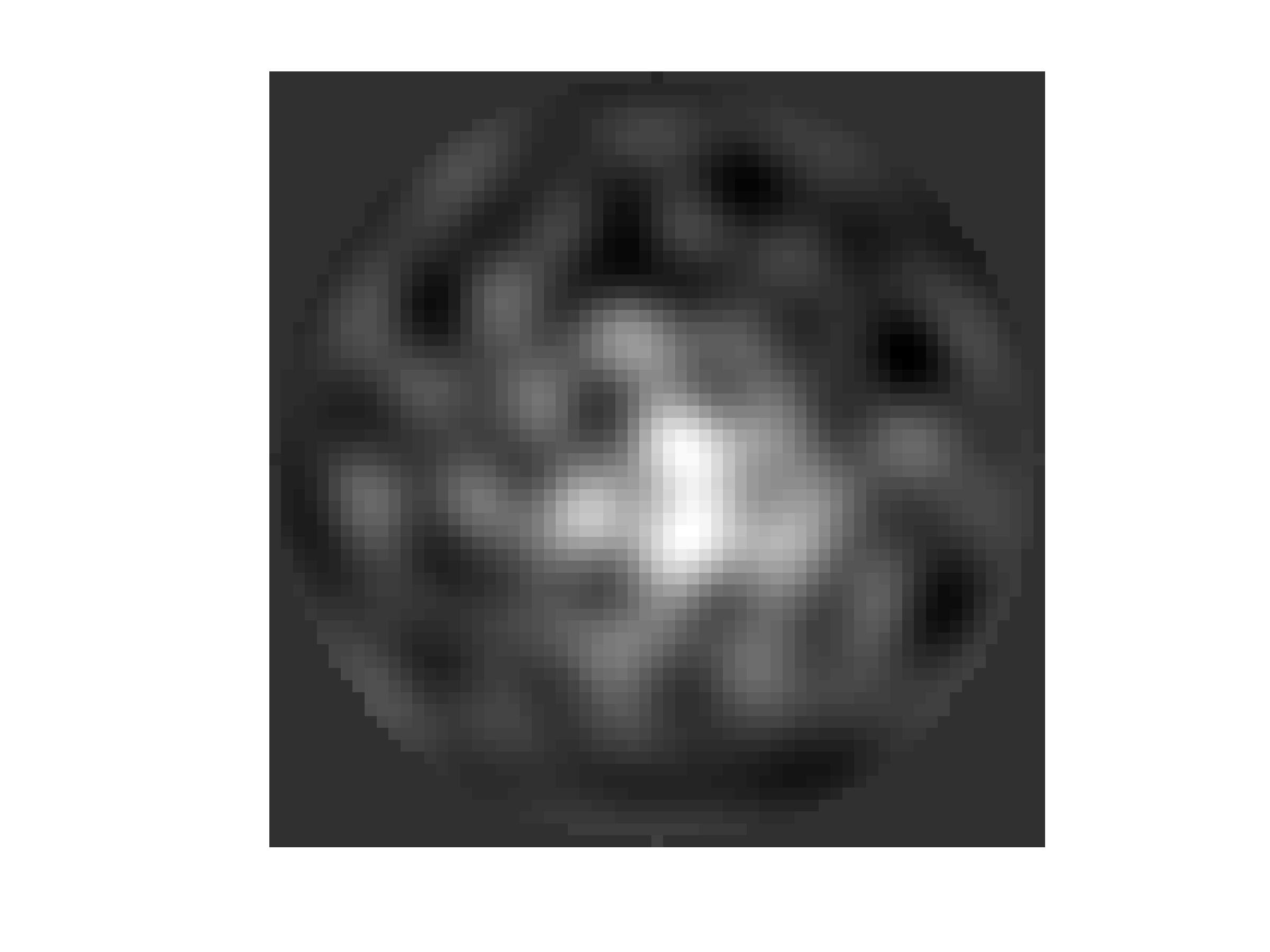}
\caption{Volume 2}
        \end{subfigure}
\begin{subfigure}[b]{0.15\textwidth}
                \centering
                \includegraphics[scale = 0.04]{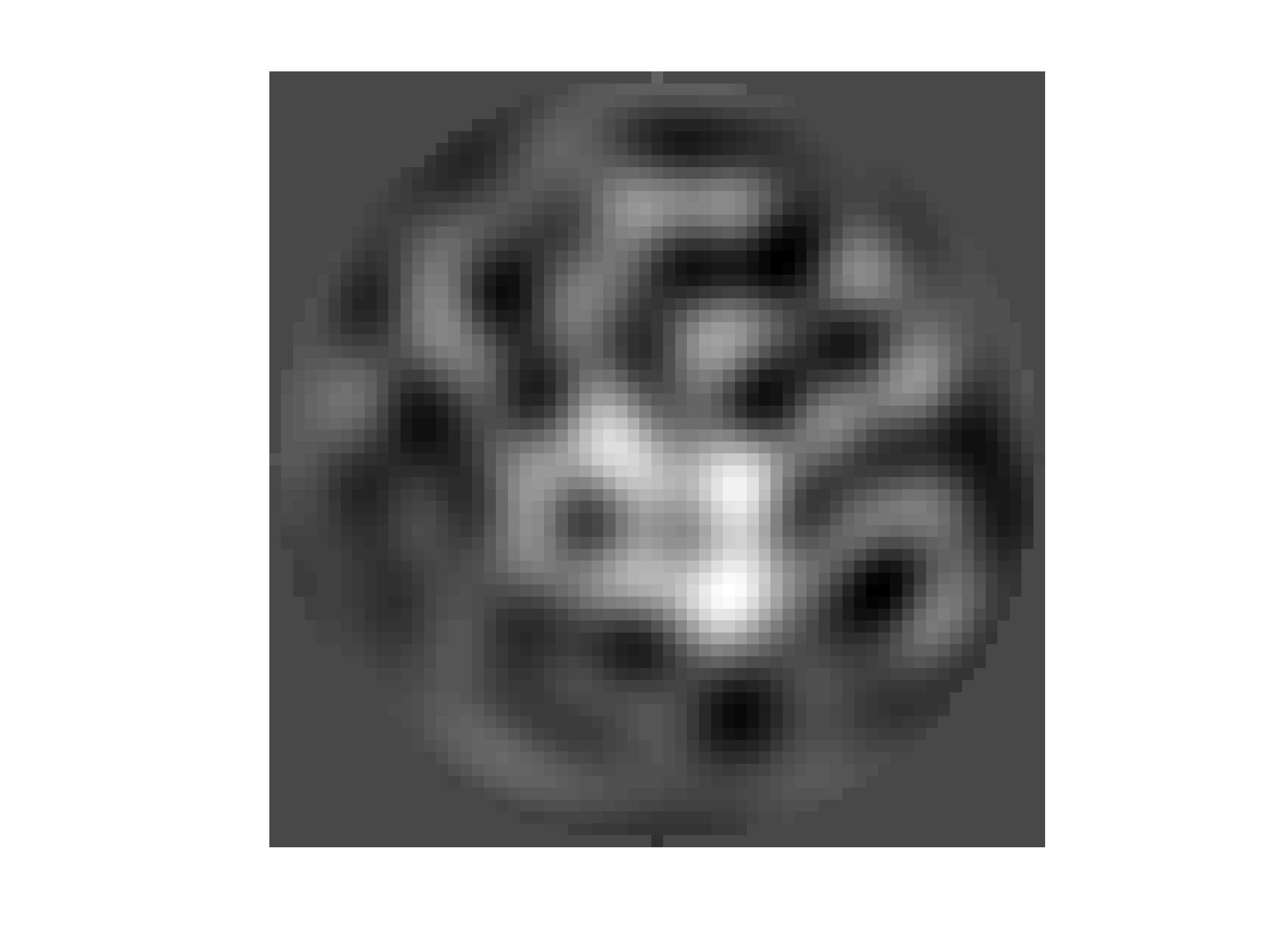}
\caption{Volume 3}
        \end{subfigure}
\caption{Cross-sections of clean and reconstructed objects for the three class experiment. The top row is clean, the second row corresponds to SNR$_\text{het}$ = 0.044 (0.3), the third row to SNR$_\text{het}$ =0.0044 (0.03), and the last row to SNR$_\text{het}$ = 0.0015 (0.01).}
\label{reconstructions_3}
\end{figure}


\begin{figure}[H]
	\begin{subfigure}[b]{0.3\textwidth}
                \centering
                \includegraphics[scale = 0.165]{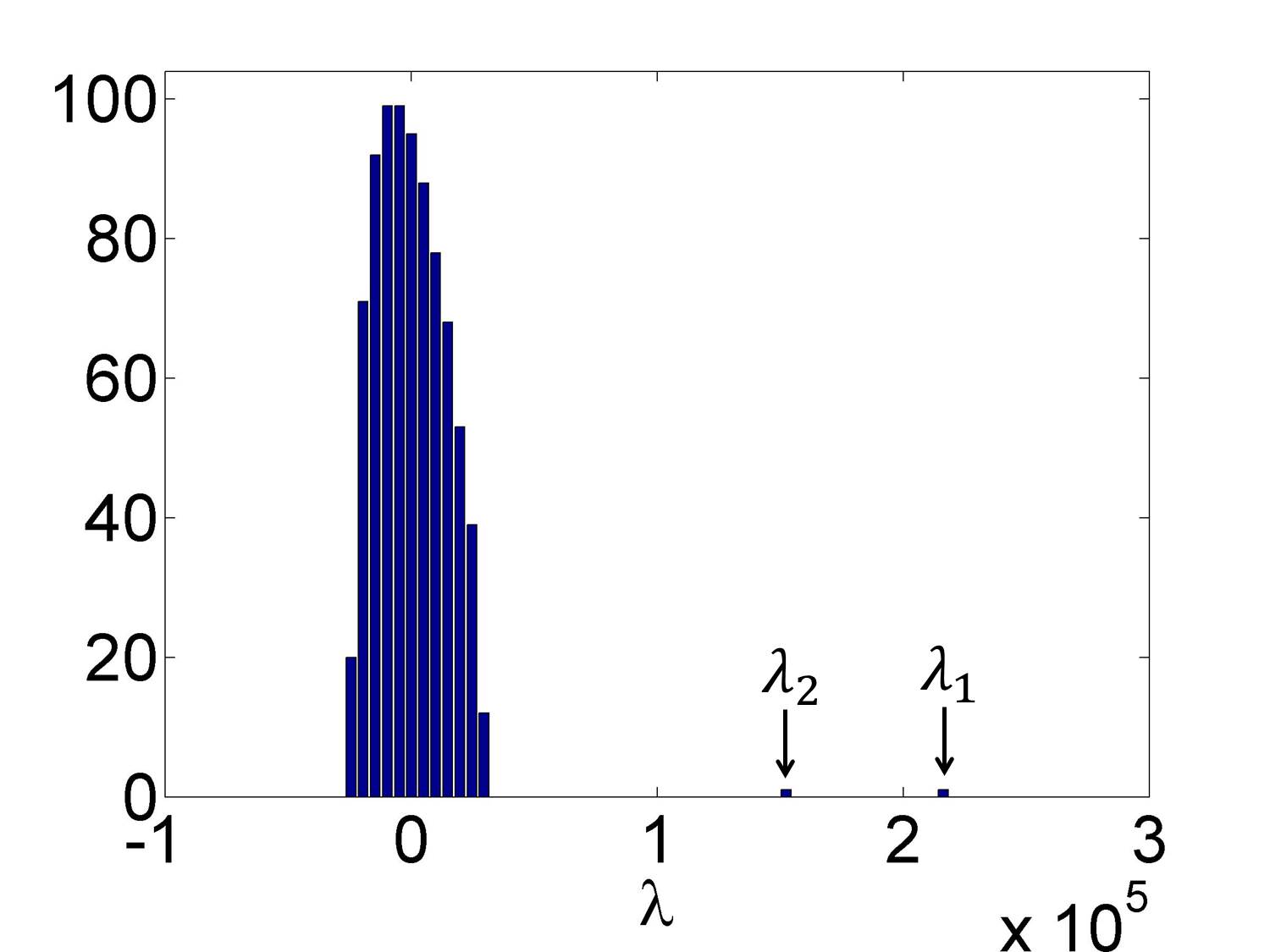}
	\caption{SNR$_{\text{het}}$ = 0.044 (0.3)}
        \end{subfigure}
\begin{subfigure}[b]{0.3\textwidth}
                \centering
                \includegraphics[scale = 0.165]{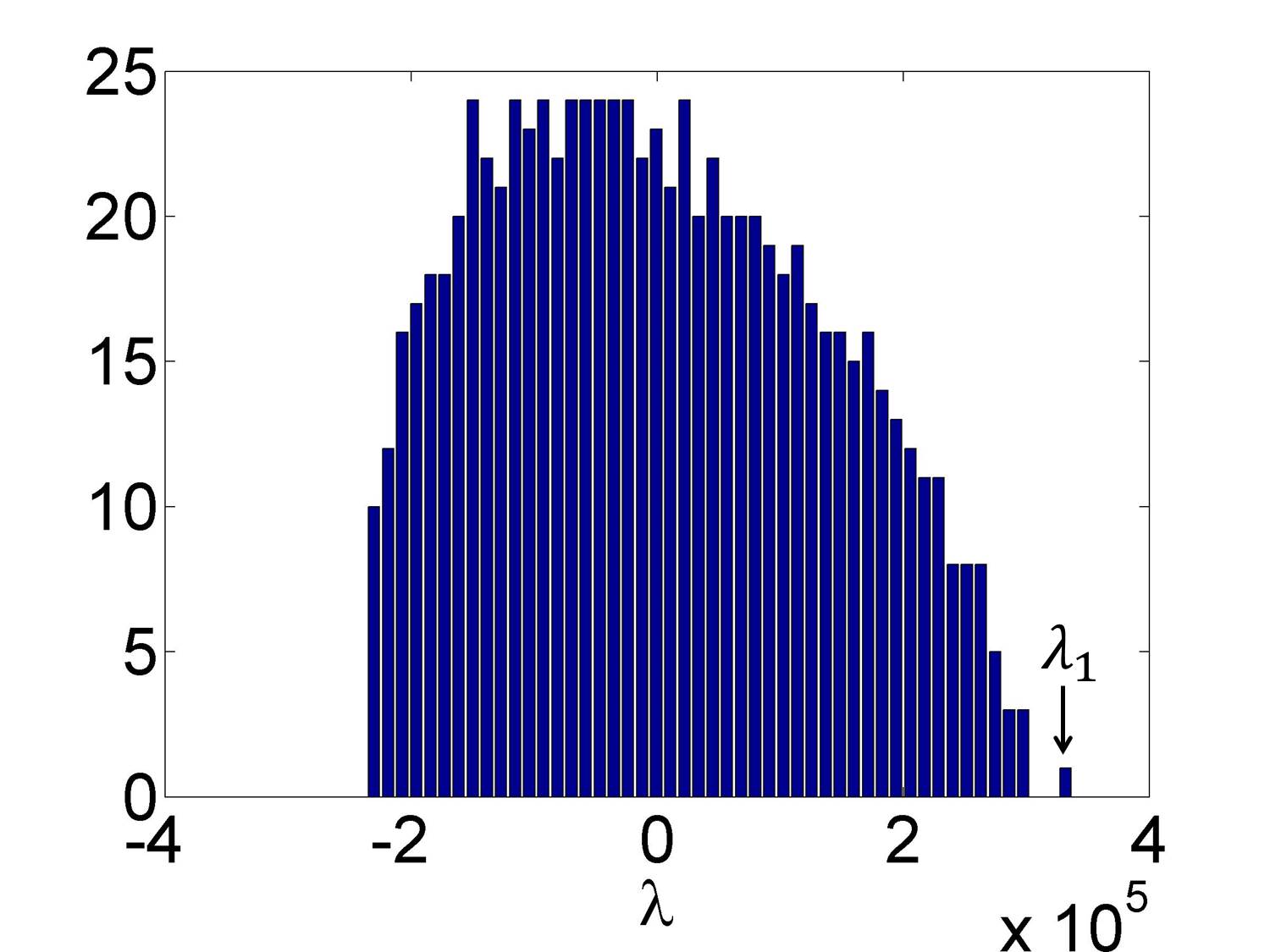}
	\caption{SNR$_{\text{het}}$ = 0.0044 (0.03)}
        \end{subfigure}
\begin{subfigure}[b]{0.3\textwidth}
                \centering
                \includegraphics[scale = 0.067]{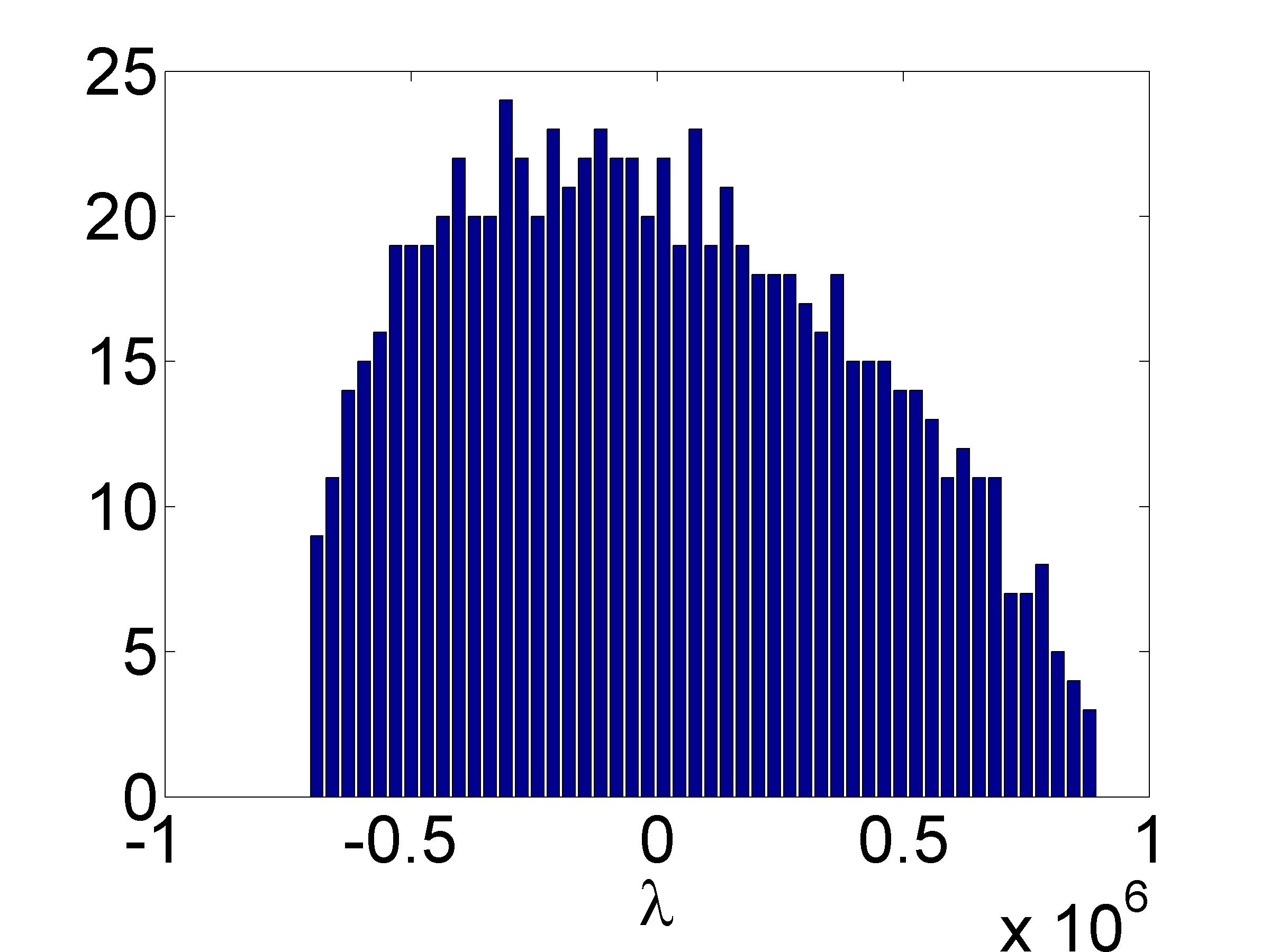}
	\caption{SNR$_{\text{het}}$ =  0.0015 (0.01)}
        \end{subfigure}
\caption{Eigenvalue histograms of reconstructed covariance matrix in the three class case for three SNR values. Note that the noise distribution initially engulfs the second eigenvalue, and eventually the top eigenvalue as well.}
\label{eig_hists_3}
\end{figure}

\begin{figure}[H]
 	\centering
\begin{subfigure}[b]{0.24\textwidth}
                \centering
	\includegraphics[scale = 0.04]{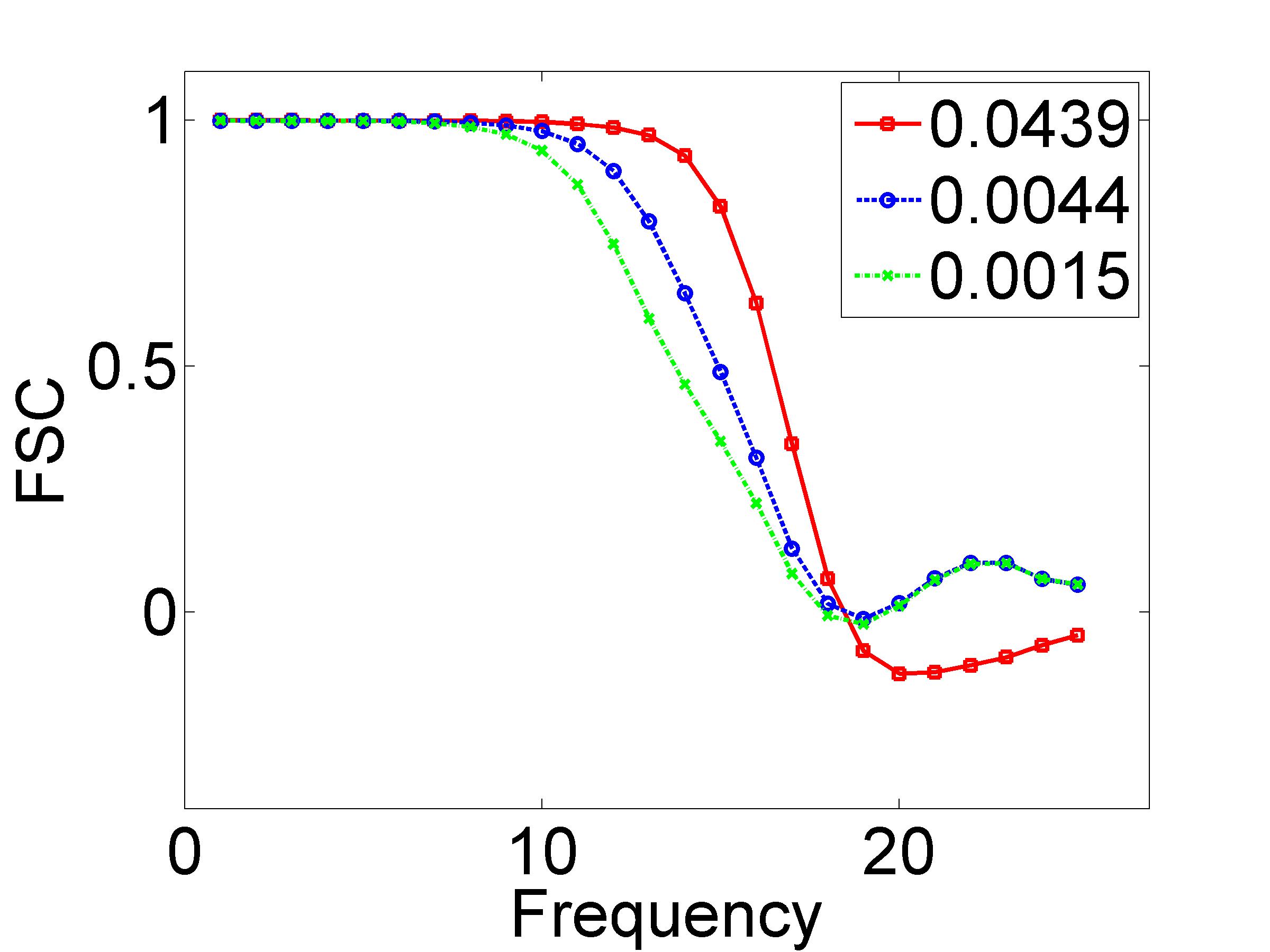}
\caption{Mean}
        \end{subfigure}	
\begin{subfigure}[b]{0.24\textwidth}
	\centering
	\includegraphics[scale = 0.04]{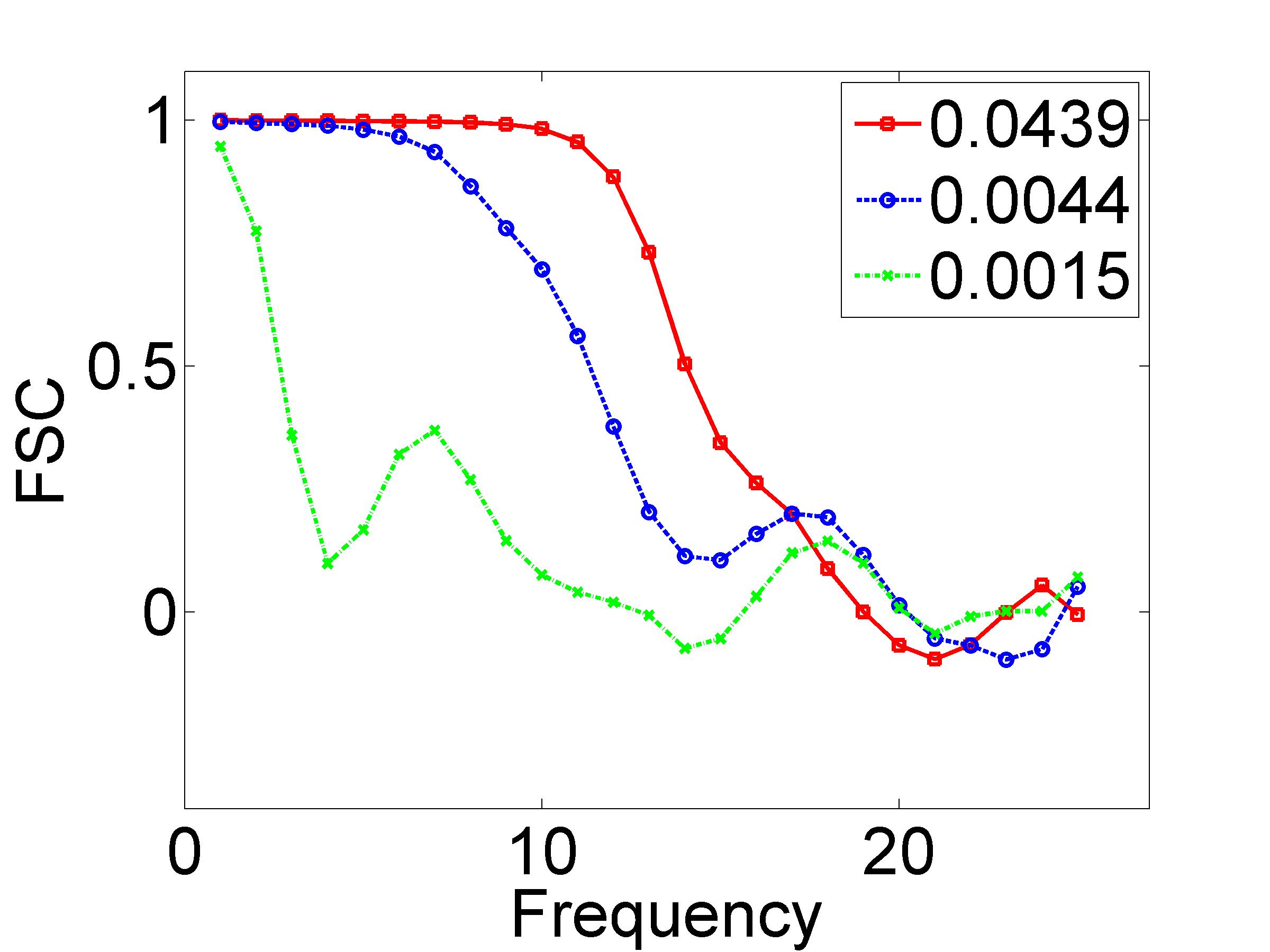}
\caption{Eigenvector 1}
\end{subfigure}
\begin{subfigure}[b]{0.24\textwidth}
	\centering
	\includegraphics[scale = 0.04]{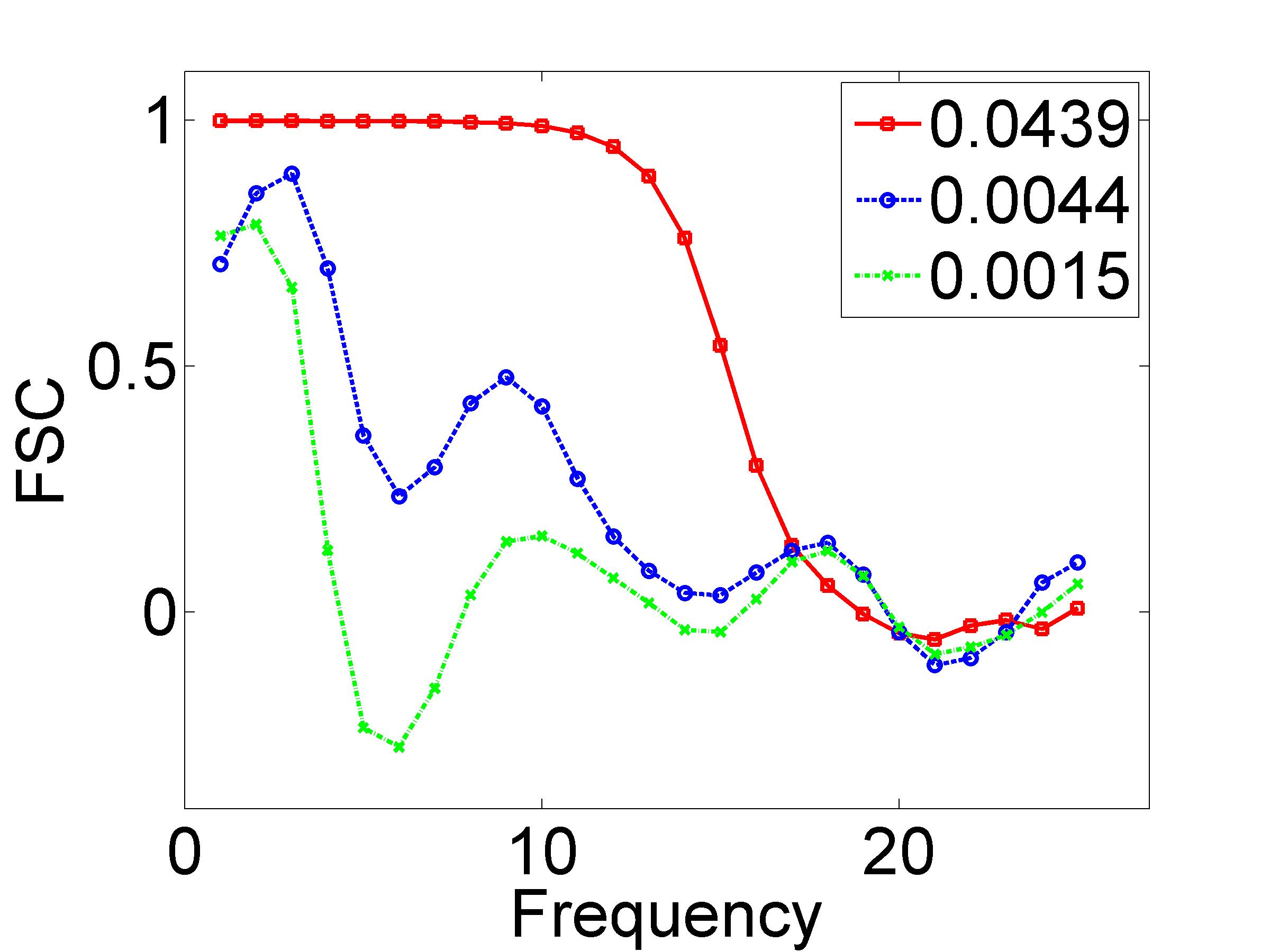}
\caption{Eigenvector 2}
\end{subfigure}
\begin{subfigure}[b]{0.24\textwidth}
	\centering
	\includegraphics[scale = 0.04]{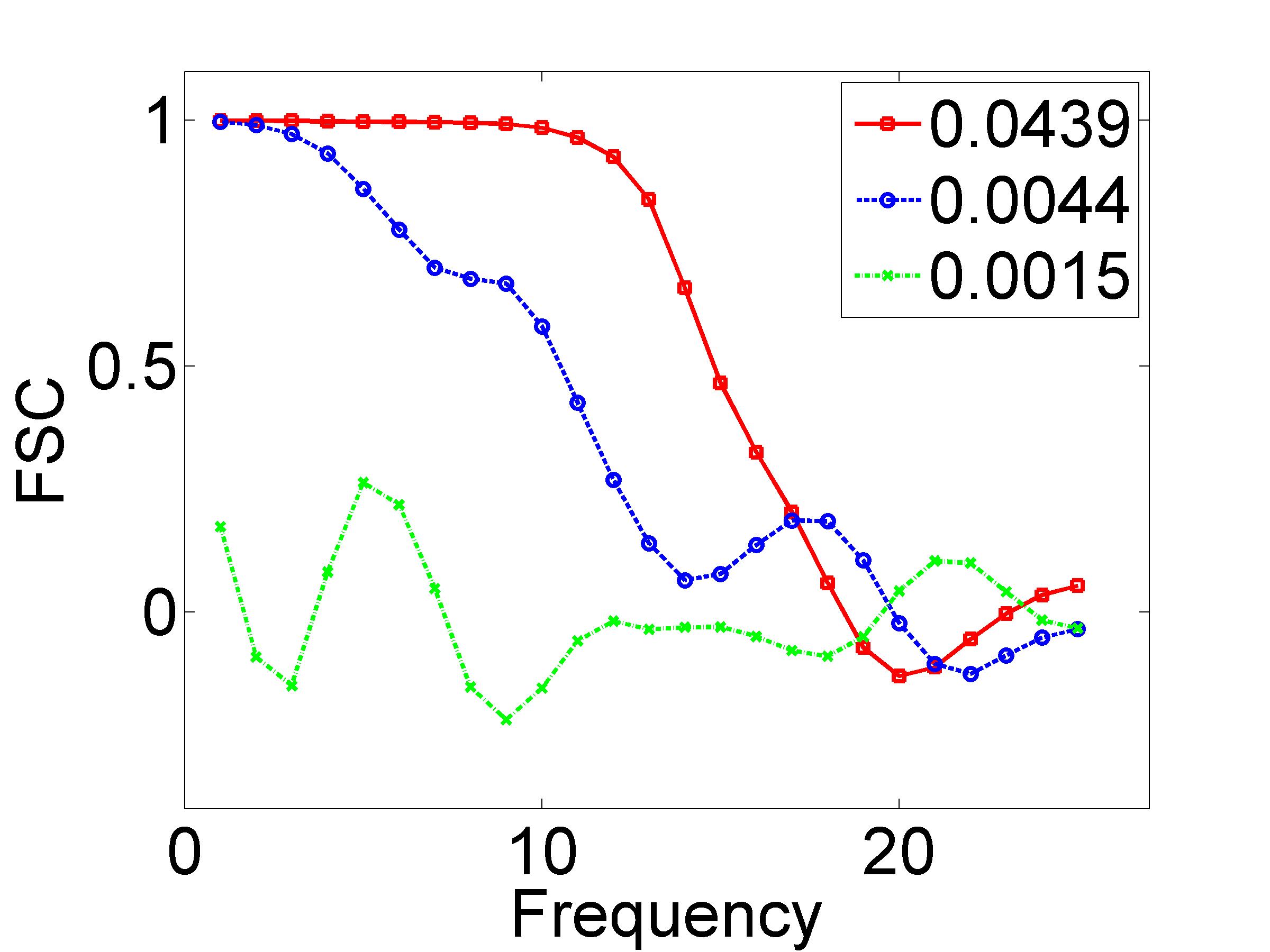}
\caption{Volume 1}
\end{subfigure}
\caption{FSC curves for the mean volume, top eigenvector, and one mean-subtracted volume at the same three SNRs as in Figure \ref{reconstructions_3}. Note that the mean volume is reconstructed successfully for all three SNR levels, and that the second eigenvector is recovered less accurately than the first. }
\label{fig:fsc_3}
\end{figure}

\begin{figure}[H]
 \begin{subfigure}[b]{0.49\textwidth}
                \centering
                \includegraphics[scale = 0.06]{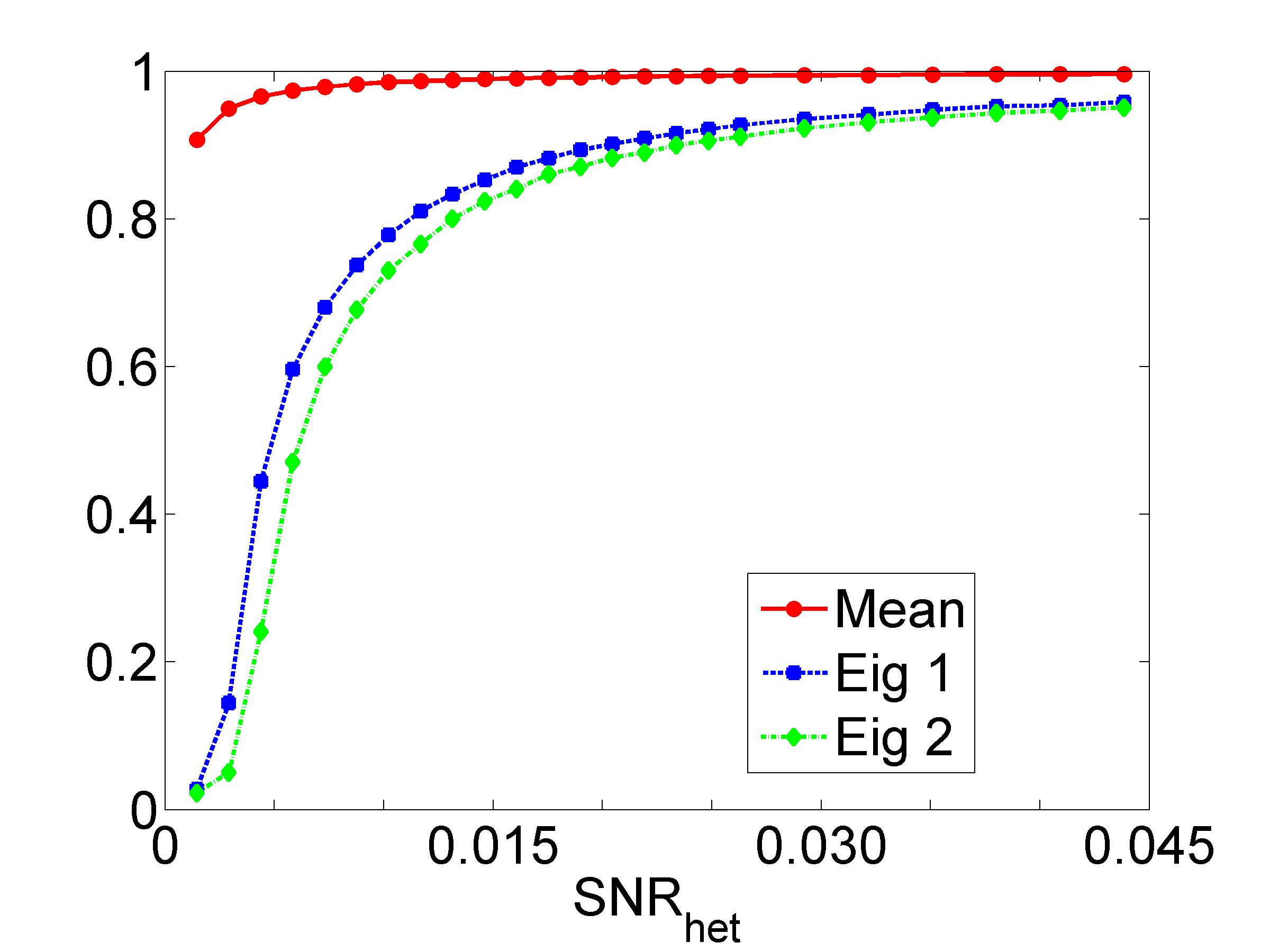}
	\caption{Mean and eigenvector correlations}
        \end{subfigure}
\begin{subfigure}[b]{0.49\textwidth}
	\centering
	\includegraphics[scale = 0.06]{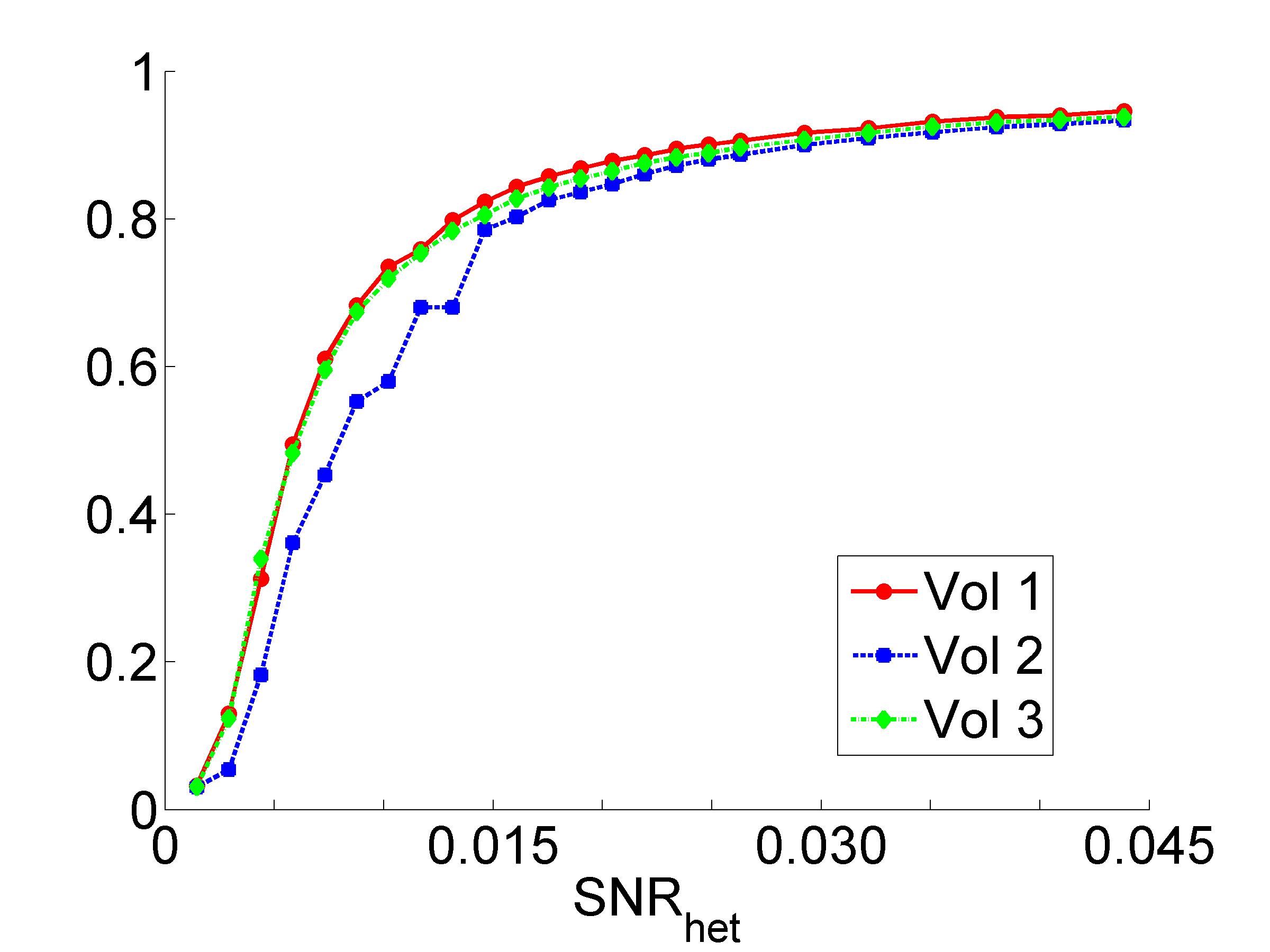}
	\caption{Volume correlations} \label{correlations_c_3}
\end{subfigure}
\caption{Correlations of computed means, eigenvectors, and mean-subtracted volumes with their true values for different SNRs (averaged over 30 experiments). Note that the mean volume is consistently recovered well, whereas recovery of the eigenvectors and volumes exhibits a phase-transition behavior.}
\label{correlations_3}
\end{figure}


\begin{figure}[H]
        \centering
        \begin{subfigure}[b]{0.3\textwidth}
                \centering
                \includegraphics[scale = 0.05]{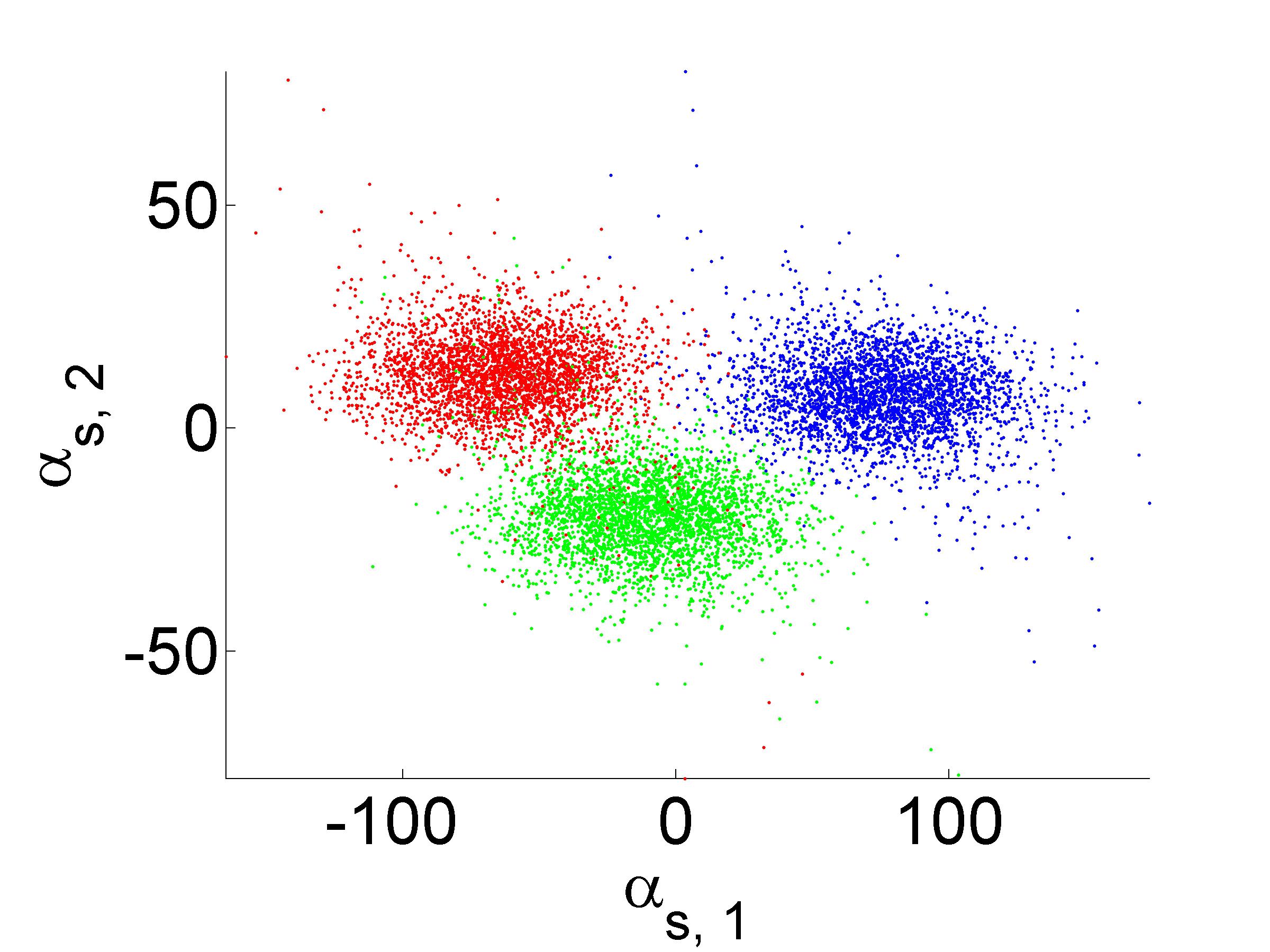}
	\caption{SNR$_{\text{het}}$ = 0.044 (0.3)}
        \end{subfigure}
	\begin{subfigure}[b]{0.3\textwidth}
                \centering
                \includegraphics[scale = 0.05]{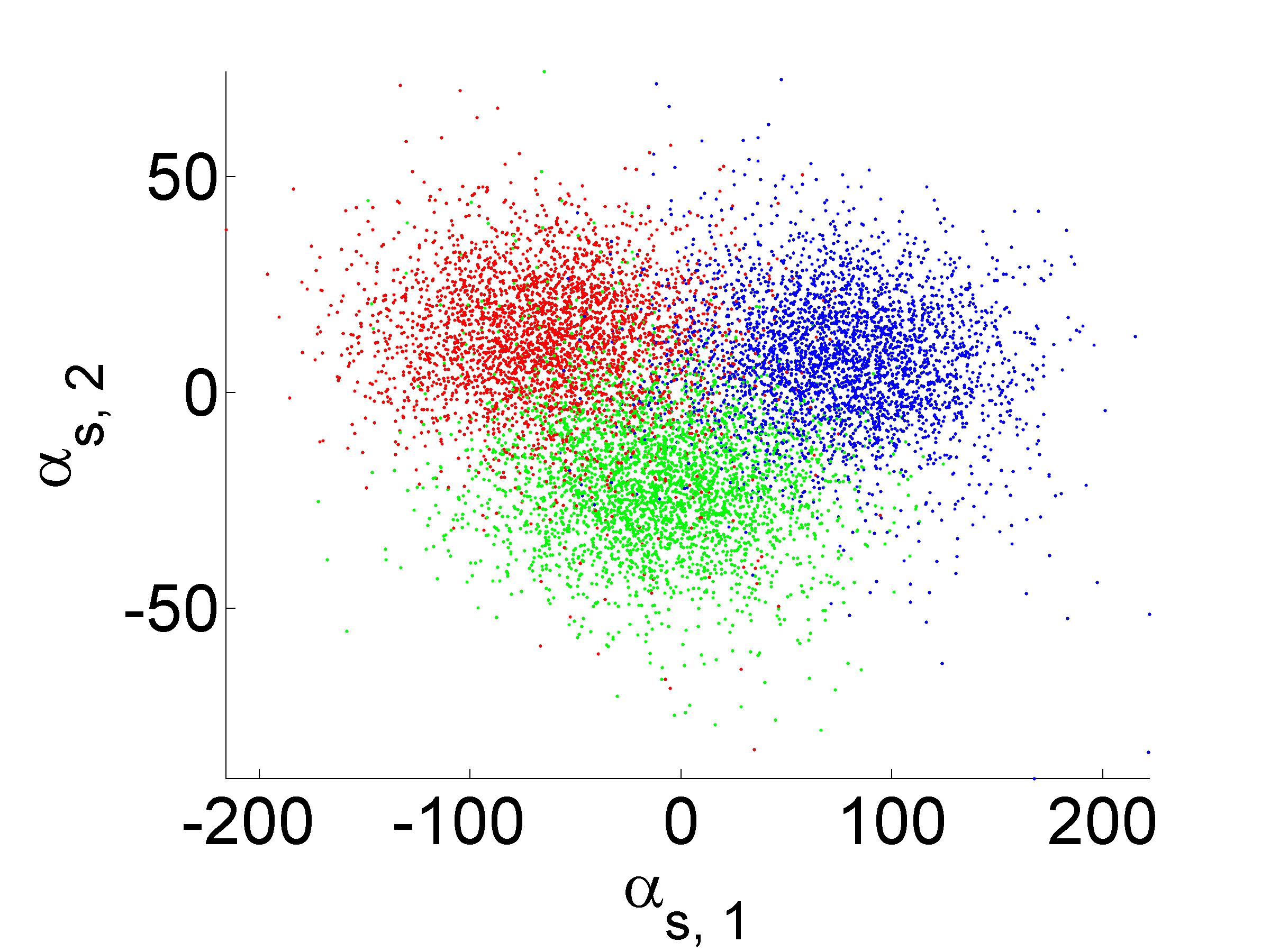}
	\caption{SNR$_{\text{het}}$ = 0.018 (0.12)}
        \end{subfigure}
	\begin{subfigure}[b]{0.3\textwidth}
                \centering
                \includegraphics[scale = 0.05]{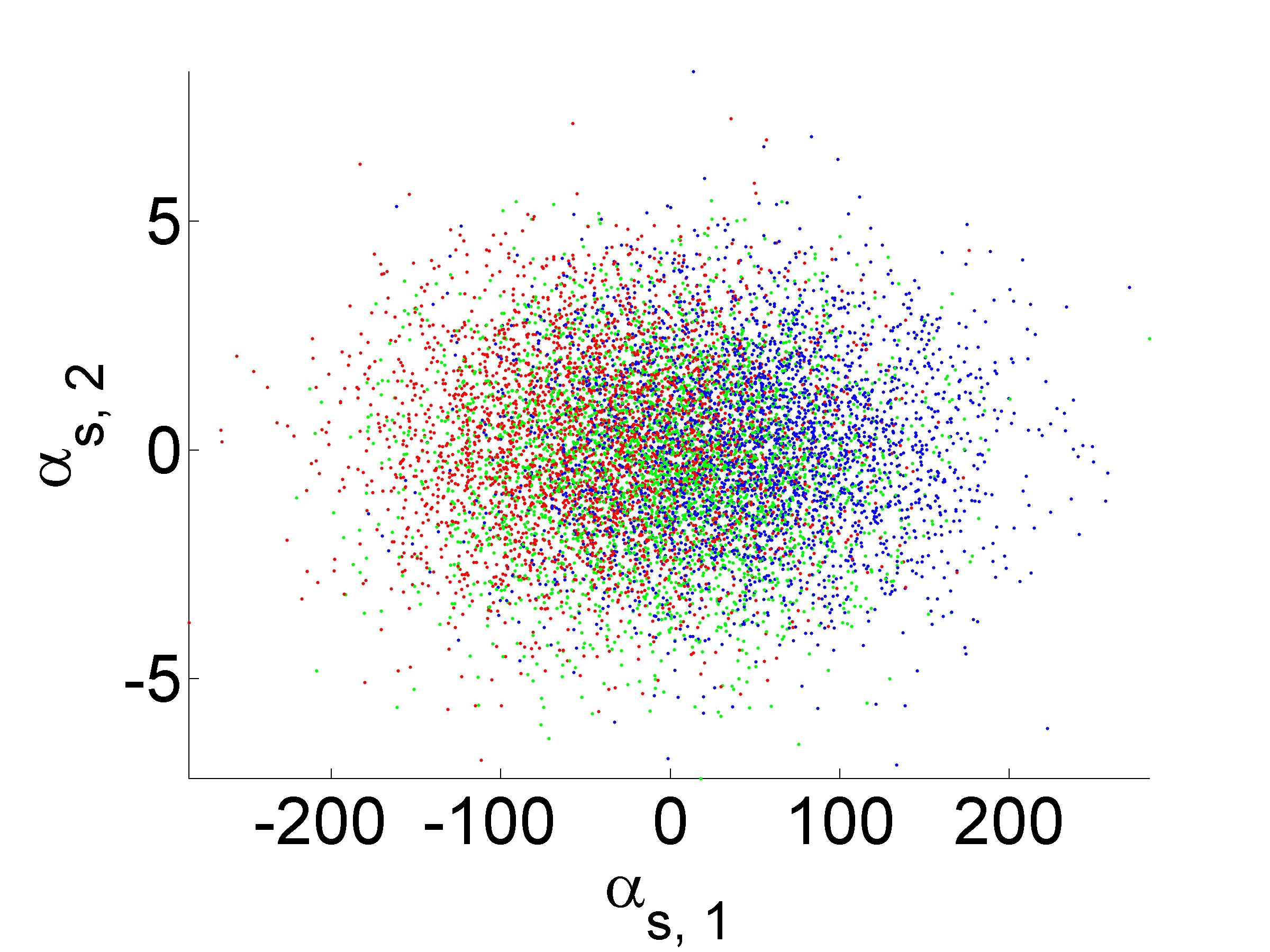}
	\caption{SNR$_{\text{het}}$ = 0.0044 (0.03)}
        \end{subfigure}
\caption{The coordinates $\alpha_{s}$ for the three class case, colored according to true class. The middle scatter plot is near the transition at which the three clusters coalesce.}
\label{coords_3}
\end{figure}

\subsection{Experiment: continuous variation}

In this experiment, we sampled $\scr X_s$ uniformly from the perimeter of the triangle determined by volumes $\scr X^1, \scr X^2, \scr X^3$ (from the three class discrete heterogeneity experiment). This setup is more suitable to model the case when the molecule can vary continuously between each pair $\scr X^i$ and $\scr X^j$. Despite the fact this experiment does not fall under Problem \ref{het_problem}, Figure \ref{eig_hists_cont} shows that we still recover the rank two structure. Indeed, it is clear that all the clean volumes still belong to a subspace of dimension 2. Moreover, we can see the triangular pattern of heterogeneity in the scatter plots of $\alpha_{s}$ (Figure \ref{fig:triangles}). However, note that once the images get moderately noisy, the triangular structure starts getting drowned out. Thus in practice, without any prior assumptions, just looking at the scatter plots of $\alpha_{s}$ will not necessarily reveal the heterogeneity structure in the dataset. To detect continuous variation, a new algorithmic step must be designed to follow covariance matrix estimation. Nevertheless, this experiment shows that by solving the general Problem \ref{main_problem}, we can estimate covariance matrices beyond those considered in the discrete case of the heterogeneity problem.


\begin{figure}[H]
\begin{subfigure}[b]{0.49\textwidth}
                \centering
                \includegraphics[scale = 0.165]{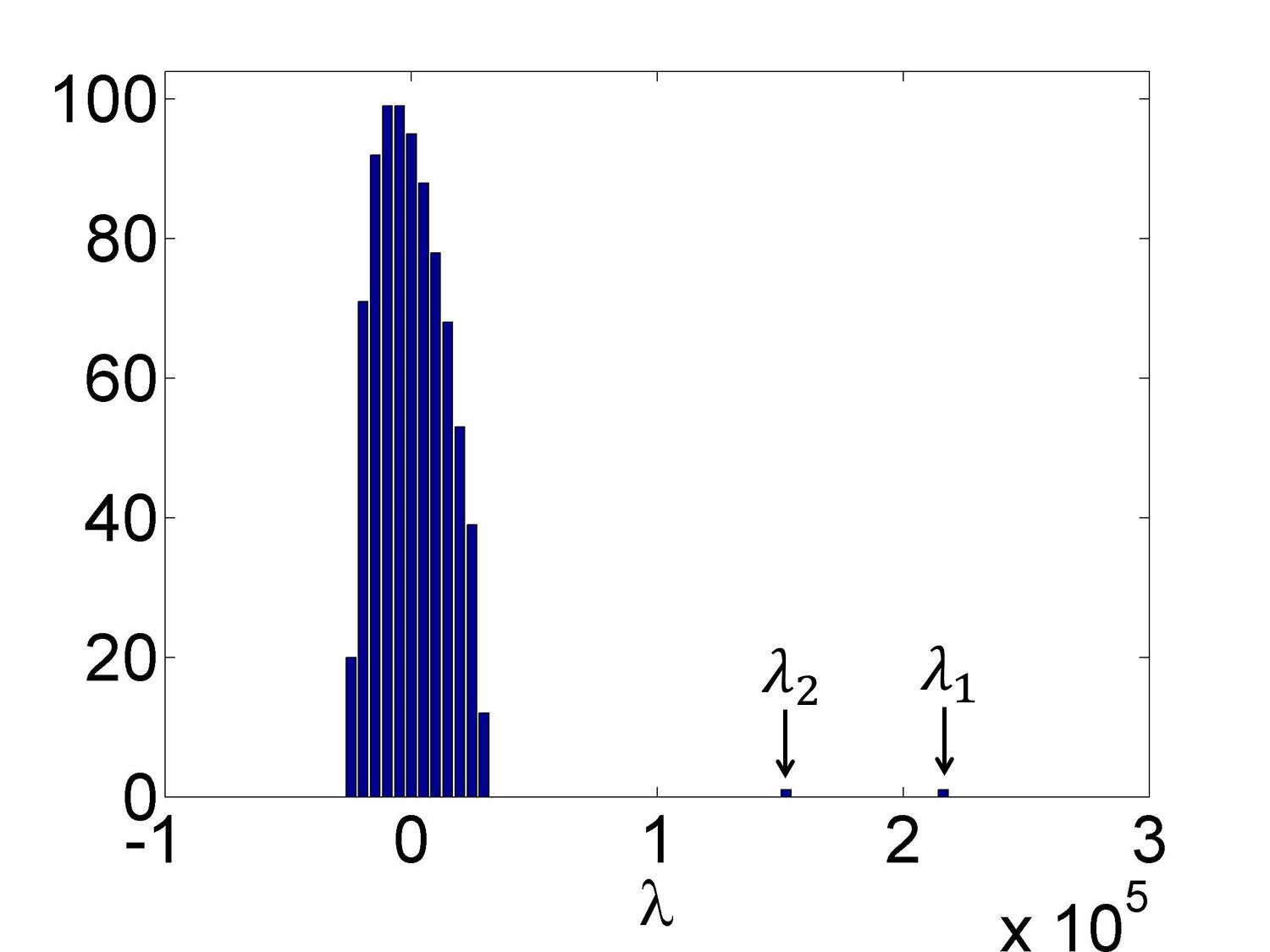}
	\caption{SNR$_{\text{het}}$ = 0.14 (0.97)}
        \end{subfigure}
\begin{subfigure}[b]{0.49\textwidth}
                \centering
                \includegraphics[scale = 0.165]{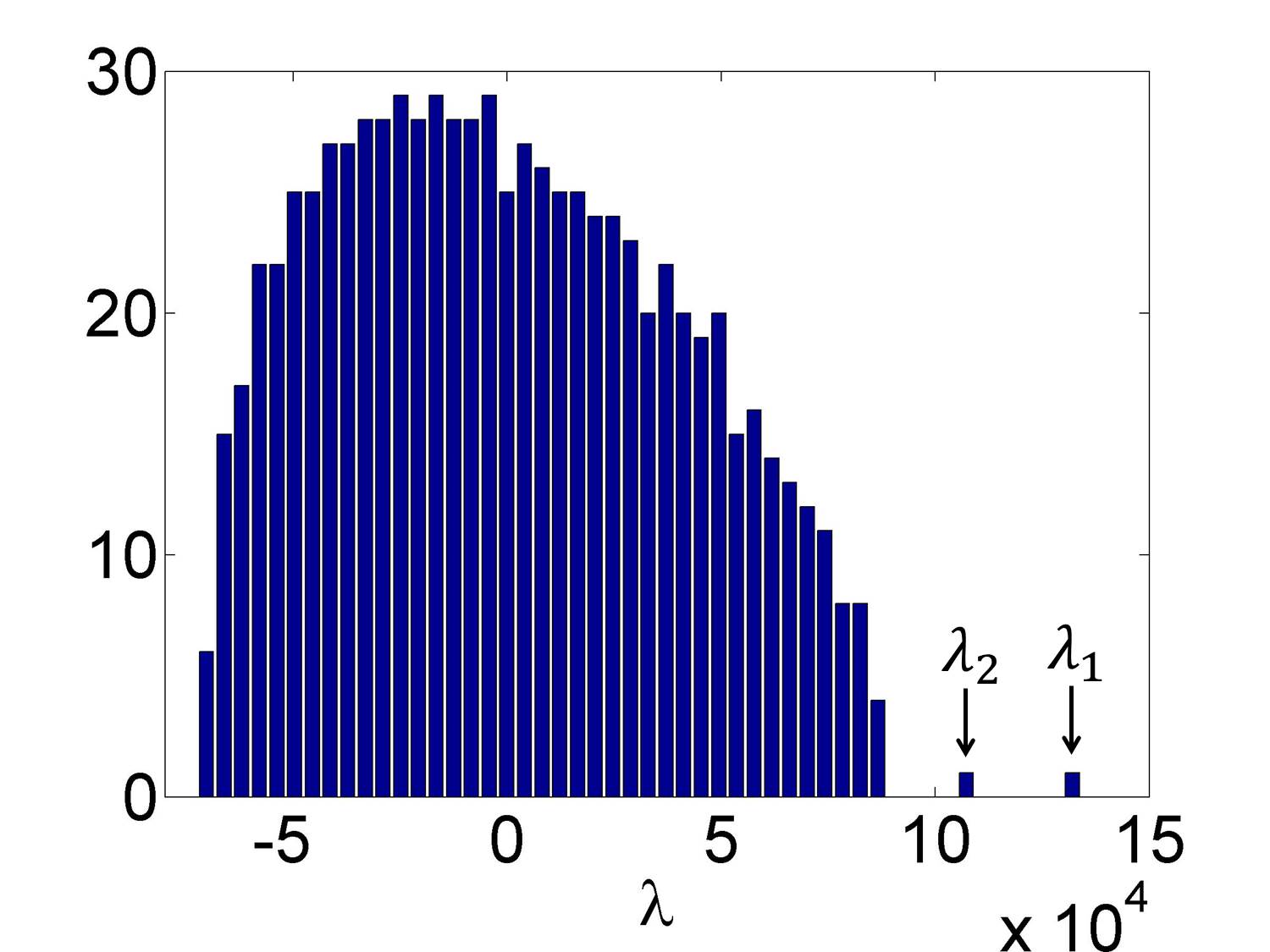}
	\caption{SNR$_{\text{het}}$ = 0.014 (0.1)}
        \end{subfigure}
\caption{Eigenvalue histograms of covariance matrix reconstructed in continuous variation case.}
\label{eig_hists_cont}
\end{figure}

\begin{figure}[h]
 	\centering
\begin{subfigure}[b]{0.3\textwidth}
                \centering
	\includegraphics[scale = 0.2]{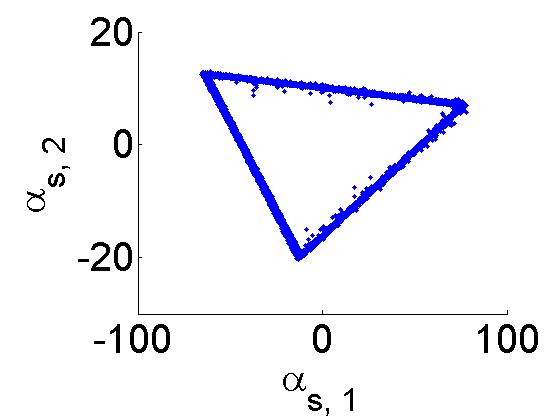}
\caption{Clean images}
        \end{subfigure}	
\begin{subfigure}[b]{0.3\textwidth}
	\centering
	\includegraphics[scale = 0.2]{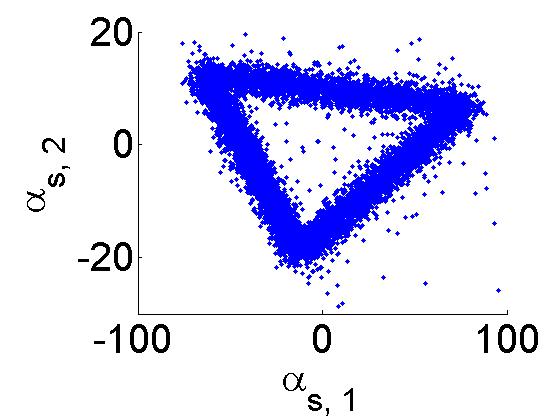}
\caption{SNR$_{\text{het}}$ = 1.4 (9.7)}
\end{subfigure}
\begin{subfigure}[b]{0.3\textwidth}
	\centering
	\includegraphics[scale = 0.2]{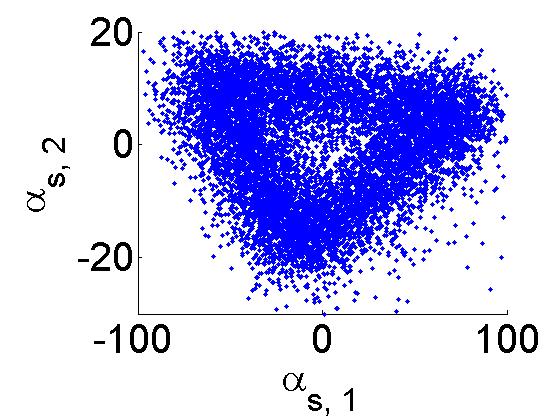}
\caption{SNR$_{\text{het}}$ = 0.14 (0.97)}
\end{subfigure}
\caption{Scatter plots (with some outliers removed) of $\alpha_{s}$ for high SNR values.}
\label{fig:triangles}
\end{figure}

\section{Discussion} \label{future}

In this paper, we proposed a covariance matrix estimator from noisy linearly projected data and proved its consistency. The covariance matrix approach to the cryo-EM heterogeneity problem is essentially a special case of the general statistical problem under consideration, but has its own practical challenges. We overcame these challenges and proposed a methodology to tractably estimate the covariance matrix and reconstruct the molecular volumes. We proved the consistency of our estimator in the cryo-EM case and also began the mathematical investigation of the projection covariance transform. We discovered that inverting the projection covariance transform involves applying the triangular area filter, a generalization of the ramp filter arising in tomography. Finally, we validated our methodology on simulated data, producing accurate reconstructions at low SNR levels. Our implementation of this algorithm is now part of the ASPIRE package at \url{spr.math.princeton.edu}. In what follows, we discuss several directions for future research.

As discussed in Section \ref{rmt}, our statistical framework and estimators have opened many new questions in high-dimensional statistics. While a suite of results are already available for the traditional high-dimensional PCA problem, generalizing these results to the projected data case would require new random matrix analysis. Our numerical experiments in the cryo-EM case have shown many qualitative similarities between the estimated covariance matrix in the cryo-EM case and the sample covariance matrix in the spiked model. There is again a bulk distribution with eigenvalues separated from it. Moreover, there is a phase-transition phenomenon in the cryo-EM case, in which the top eigenvectors of the estimated covariance lose correlation with those of the population covariance once the corresponding eigenvalues are absorbed by the bulk distribution. Answering the questions posed in Section \ref{rmt} would be very useful in quantifying the theoretical limitations of our approach.

As an additional line of further inquiry, note that the optimization problem (\ref{sigopt}) for the covariance matrix is amenable to regularization. If $n \asymp f(p, q)$ is the high-dimensional statistical regime in which the unregularized estimator still carries signal, then of course we need regularization when $n \ll f(p, q)$. Here, $f$ is a function depending on the distribution of the operators $P_s$. Moreover, regularization increases robustness to noise, so in applications like cryo-EM, this could prove useful. Tikhonov regularization does not increase the complexity of our algorithm, but has the potential to make $\hat L_n$ invertible. Under what conditions can we still achieve accurate recovery in a regularized setting? Other regularization schemes can take advantage of a-priori knowledge of $\Sigma_0$, such as using nuclear norm regularization in the case when $\Sigma_0$ is known to be low rank. See \cite{Kuybeda2013116} for an application of nuclear norm minimization in the context of dealing with heterogeneity in cryo-electron tomography. Another special structure $\Sigma_0$ might have is that it is sparse in a certain basis. For example, the localized variability assumption in the case of the heterogeneity problem is such an example; in this case, the covariance matrix is sparse in the real Cartesian basis or a wavelet basis. This sparsity can be encouraged using a matrix 1-norm regularization term. Other methods, such as sparse PCA \cite{sparse} or covariance thresholding \cite{cov_thresh} might be applicable in certain cases when we have sparsity in a given basis.

We developed our algorithm in an idealized environment, assuming that the rotations $\sdr R_s$ (and in-plane translations) are known exactly and correspond to approximately uniformly distributed viewing directions, and that the molecules belong to $\mathscr B$. Moreover, we did not account for the CTF effect of the electron microscope. In practice, of course rotations and translations are estimated with some error. Also, certain molecules might exhibit a preference for a certain orientation, invalidating the uniform rotations assumption. Note that as long as $\hat L_n$ is invertible, our framework produces a valid estimator, but without the uniform rotations assumption, the computationally tractable approach to inverting this matrix proposed in Section \ref{practical} no longer holds. Moreover, molecules might have higher frequencies that those we reconstruct, which could potentially lead to artifacts. Thus, an important direction of future research is to investigate the stability of our algorithm to perturbations from the idealized assumptions we have made. An alternative research direction is to devise numerical schemes to invert $\hat L_n$ without replacing it by $\hat L$, which could allow incorporation of CTF and obviate the need to assume uniform rotations. We proposed one such scheme in Section \ref{comp_challenges_approaches}.

As we discussed in the introduction, our statistical problem (\ref{main_problem}) is actually a special case of the matrix sensing problem. In future work, it would be interesting to test matrix sensing algorithms on our problem. In the cryo-EM case, it would be useful to compare our approach with matrix sensing algorithms. It would also be interesting to explore the applications of our methodology to other tomographic problems involving variability. For example, the field of 4D electron tomography focuses on reconstructing a 3D structure that is a function of time \cite{Kwon25062010}. This 4D reconstruction is essentially a movie of the molecule in action. The methods developed in this paper can in principle be used to estimate the covariance matrix of a molecule varying with time. This is another kind of ``heterogeneity"  that is amenable to the same analysis we used to investigate structural variability in cryo-EM.

\section{Acknowledgements}
E. Katsevich thanks Jane Zhao, Lanhui Wang, and Xiuyuan Cheng (PACM, Princeton University) for their valuable advice on several theoretical and practical issues. Parts of this work have appeared in E. Katsevich's undergraduate Independent Work at Princeton University.

A. Katsevich was partially supported by Award Number DMS-1115615 from NSF.

A. Singer was partially supported by Award Number R01GM090200 from the NIGMS, by Award Number FA9550-12-1-0317 and FA9550-13-1-0076 from AFOSR, and by Award Number LTR DTD 06-05-2012 from the Simons Foundation.

The authors are also indebted to Philippe Rigollet (ORFE, Princeton), as this work benefited from discussions with him regarding the statistical framework. Also, the authors thank Joachim Frank (Columbia University) and Joakim Anden (PACM, Princeton University) for providing helpful comments about their manuscript. They also thank Dr. Frank and Hstau Liao (Columbia University) for allowing them to reproduce Figure 2 from \cite{local} as our Figure \ref{het}. Finally, they thank the editor and the referees for their many helpful comments.

\bibliographystyle{plain}
\bibliography{my_bib}

\appendix
\section{Matrix derivative calculations} \label{appendix_mat_der}
The goal of this appendix is to differentiate the objective functions of (\ref{muopt}) and (\ref{sigopt}) to verify formulas (\ref{mueq}) and (\ref{sigeq}). In order to differentiate with respect to vectors and matrices, we appeal to a few results from \cite{matder}. The results are as follows:
\begin{equation}
\begin{split}
&D_{z^*} (z^H a) = a \\
&D_{z^*} (z^H A z) = Az \\
&D_{Z}(tr(AZ)) = A \\
&D_{Z}(tr(ZAZ^HA)) = AZ^HA.
\end{split}
\label{matders}
\end{equation}
Here, the lowercase letters represent vectors and the uppercase letters represent matrices. Also note that $z^*$ denotes the complex conjugate of $z$. The general term of (\ref{muopt}) is
\begin{equation}
\norm{I_s - P_s\mu}^2 = (I_s^H - \mu^H P_s^H)(I_s - P_s\mu) = \mu^H P_s^H P_s \mu - \mu^H P_s^H I_s - I_s^H P_s \mu + \mathrm{const}.
\end{equation}
We can differentiate this with respect to $\mu^*$ by using the first two formulas of (\ref{matders}). We get
\begin{equation}
D_{\mu^*}\norm{I_s - P_s\mu}^2 = P_s^H P_s \mu - P_s^H I_s.
\end{equation}
Summing in $s$ gives us (\ref{mueq}).

If we let $A_s = (I_s - P_s \mu_n)(I_s - P_s\mu_n)^H - \sigma^2 I$, then the general term of (\ref{sigopt}) is
\begin{equation*}
\begin{split}
&\norm{(I_s - P_s \mu_n)(I_s - P_s\mu_n)^H - (P_s \Sigma P_s^H + \sigma^2 I)}_F^2 \\
&\quad= \norm{A_s - P_s \Sigma P_s^H}_F^2 \\
&\quad= \text{tr}(A_s^H - P_s \Sigma^H P_s^H)(A_s - P_s \Sigma P_s^H) \\
&\quad = \text{tr}(P_s \Sigma^H P_s^H P_s \Sigma P_s^H) - \text{tr}(P_s \Sigma^H P_s^HA_s) - \text{tr}(A_s^H P_s \Sigma P_s^H) + \text{const}, \\
&\quad = \text{tr}(\Sigma P_s^H P_s \Sigma^H P_s^H P_s ) - \text{tr}(P_s^HA_sP_s \Sigma^H ) - \text{tr}(P_s^HA_s^H P_s \Sigma ) + \text{const}.
\end{split}
\end{equation*}
Using the last two formulas of (\ref{matders}), we find that the derivative of this expression with respect to $\Sigma$ is
\begin{equation*}
P_s^H P_s \Sigma^H P_s^H P_s - P_s^H A_s^H P_s.
\end{equation*}
Taking a Hermitian and summing in $s$ gives us (\ref{sigeq}).

\section{Consistency of $\mu_n$ and $\Sigma_n$} \label{analysis_proofs}

In this appendix, we will prove the consistency results about $\mu_n$ and $\Sigma_n$ stated in Section \ref{theoretical}. Recall $\mu_n$ and $\Sigma_n$ are defined nontrivially if $\norm{A_n^{-1}} \leq 2 \norm{A^{-1}}$ and $\norm{L_n^{-1}} \leq 2 \norm{L^{-1}}$. As a necessary step towards our consistency results, we must first prove that the probability of these events tends to 1 as $n \rightarrow \infty$. Such statement follow from a matrix concentration argument based on Bernstein's inequality \cite[Theorem 1.4]{user_friendly}, which we reproduce here for the reader's convenience as a lemma. 

\begin{lemma}{(Matrix Bernstein's Inequality).} \label{matrix_bernstein}
Consider a finite sequence $\rdr Y_s$ of independent, random, self-adjoint matrices with dimension $p$. Assume that each random matrix satisfies
\begin{equation}
\mathbb E[\rdr Y_s] = 0 \quad \text{and} \quad \norm{\rdr Y_s}\leq R \quad a.s. 
\end{equation}
Then, for all $t \geq 0$, 
\begin{equation}
\mathbb P\left\{\norm{\sum_s \rdr Y_s} \geq t\right\} \leq p\cdot \exp\left(\frac{-t^2/2}{\sigma^2 + Rt/3}\right), \quad \text{where } \sigma^2: = \norm{\sum_s \mathbb E(\rdr Y_k^2)}.
\label{tail_bound}
\end{equation}
\end{lemma}
Next, we prove another lemma, which is essentially the Bernstein inequality in a more convenient form.
\begin{lemma} \label{bernstein_corollary}
Let $\rdr Z$ be a symmetric $d \times d$ random matrix, with $\norm{\rdr Z} \leq B$ almost surely. If $\rdr Z_1, \dots, \rdr Z_n$ are i.i.d. samples from $\rdr Z$, then
\begin{equation}
\mathbb P \left \{\norm{\frac{1}{n}\sum_{s =1}^n \rdr Z_s - \mathbb E[\rdr Z]} \geq t\right \} \leq d \exp\left(\frac{-3nt^2}{6B^2 + 4Bt}\right).
\label{general_tail_bound}
\end{equation}
Moreover, 
\begin{equation}
\mathbb E \norm{\frac{1}{n}\sum_{s =1}^n \rdr Z_s - \mathbb E[\rdr Z]} \leq C B \max\left(\sqrt{\frac{\log d}{n}}, \frac{2 \log d}{n}\right),
\label{general_expectation_bound}
\end{equation}
where $C$ is an absolute constant.
\end{lemma}
\begin{proof}
The proof is an application of the matrix Bernstein inequality. Let $\rdr Y_s = \frac1n(\rdr Z_s - \mathbb E \rdr Z)$. Then, note that $\mathbb E[\rdr Y_s] = 0$ and 
\begin{equation}
\norm{\rdr Y_s} \leq \frac1n(\norm{\rdr Z_s} + \mathbb E[\norm{\rdr Z}]) \leq \frac{2B}{n} =: R \quad \text{a.s.}
\end{equation}
Next, we have
\begin{equation}
\mathbb E[\rdr Y_s^2] = \frac{1}{n^2}\mathbb E[\rdr Z_s^2 - \rdr Z_s \mathbb E[\rdr Z] - \mathbb E[\rdr Z] \rdr Z_s + E[\rdr Z]^2] = \frac{1}{n^2}(\mathbb E[\rdr Z_s^2] - \mathbb E[\rdr Z]^2) \preccurlyeq \frac{1}{n^2}\mathbb E[\rdr Z_s^2].
\end{equation}
It follows that
\begin{equation}
\sigma^2 := \norm{\sum_{s = 1}^n \mathbb E[\rdr Y_s^2]} \leq \sum_{s = 1}^n \norm{\mathbb E[\rdr Y_s^2]} \leq \sum_{s = 1}^n \frac{1}{n^2}\norm{\mathbb E[\rdr Z_s^2]} \leq \sum_{s = 1}^n \frac{1}{n^2}\mathbb E[\norm{\rdr Z_s}^2] \leq \frac{B^2}{n}.
\end{equation}
Now, by the matrix Bernstein inequality, we find that
\begin{equation}
\begin{split}
\mathbb P \left \{\norm{\frac{1}{n}\sum_{s =1}^n \rdr Z_s - \mathbb E[\rdr Z]} \geq t\right \} &= \mathbb P \left \{\norm{\sum_{s =1}^n Y_s} \geq t\right\} \\
&\leq d \exp\left(\frac{-t^2/2}{\sigma^2 + Rt/3}\right) \leq d \exp\left(\frac{-3nt^2}{6B^2 + 4Bt}\right).
\end{split}
\end{equation}
This proves (\ref{general_tail_bound}). The bound (\ref{general_expectation_bound}) follows from \cite[Remark 6.5]{user_friendly}.
\end{proof}

\begin{corollary} \label{A_corollary}
Let $\rdr P$ be a random $q \times p$ matrix such that $\norm{\rdr P} \leq B_{P}$ almost surely. Let $\sdr A = \mathbb E[\rdr P^H \rdr P]$ and let $\rdr A_n = \frac{1}{n}\sum_{s = 1}^n \rdr P_s^H \rdr P_s$, where $\rdr P_1, \dots, \rdr P_n$ are i.i.d. samples from $\rdr P$. Then, 
\begin{equation}
\mathbb P \left \{\norm{\rdr A_n - \sdr A} \geq t\right \} \leq p \exp\left(\frac{-3nt^2}{6B_{P}^4 + 4B_{P}^2t}\right).
\label{A_tail_bound}
\end{equation}
Moreover, 
\begin{equation}
\mathbb E\norm{\rdr A_n - A} \leq C B_{P}^2 \max\left(\sqrt{\frac{\log p}{n}}, \frac{2 \log p}{n}\right) = C B_{P}^2 \sqrt{\frac{\log p}{n}},
\label{A_expectation_bound}
\end{equation}
where the last equality holds if $n \geq 4 \log p$.
\end{corollary}
\begin{proof}
These bounds follow by letting $\rdr Z = \rdr P^H \rdr P$ in Lemma \ref{bernstein_corollary} and noting that $\norm{\rdr Z} \leq B_{P}^2$ almost surely. 
\end{proof}

\begin{corollary} \label{L_corollary}
Let $\rdr P$ be a random $q \times p$ matrix such that $\norm{\rdr P} \leq B_{P}$ almost surely. Let $\sdr L \Sigma = \mathbb E[\rdr P^H \rdr P \Sigma \rdr P^H \rdr P]$ and let $\rdr L_n \Sigma = \frac{1}{n}\sum_{s = 1}^n \rdr P_s^H \rdr P_s \Sigma \rdr P_s^H \rdr P_s$, where $\rdr P_1, \dots, \rdr P_n$ are i.i.d. samples from $\rdr P$. Then, 
\begin{equation}
\mathbb P \left \{\norm{\rdr L_n - \sdr L} \geq t\right \} \leq p^2 \exp\left(\frac{-3nt^2}{6 q^4 B_{P}^8 + 4q^2 B_{P}^4t}\right).
\label{L_tail_bound}
\end{equation}
Moreover,
\begin{equation}
\mathbb E \norm{\rdr L_n - \sdr L} \leq C q^2 B_{P}^4 \max\left(\sqrt{\frac{2\log p}{n}}, \frac{4 \log p}{n}\right) = C q^2 B_{P}^4 \sqrt{\frac{2\log p}{n}},
\label{L_expectation_bound}
\end{equation}
where the last equality holds if $n \geq 8 \log p$.

\end{corollary}
\begin{proof}
We wish to apply Lemma \ref{bernstein_corollary} again, this time for $\rdr Z \Sigma = \rdr P^H \rdr P \Sigma \rdr P^H \rdr P$. In this case we must be careful because $\rdr Z$ is an operator on the space of $p \times p$ matrices. We can view it as a $p^2 \times p^2$ matrix if we represent its argument (a $p \times p$ matrix $\Sigma$) as a vector of length $p^2$ (denoted by  $\text{vec}(\Sigma)$). Then, almost surely,
\begin{equation}
\begin{split}
\norm{\rdr Z} = \max_{\norm{\text{vec}(\Sigma)} = 1}\norm{\rdr Z \text{vec}(\Sigma)} &= \max_{\norm{\Sigma}_F = 1}\norm{\rdr Z \Sigma}_F \\
&=  \max_{\norm{\Sigma}_F = 1}\norm{\rdr P^H \rdr P \Sigma \rdr P^H \rdr P}_F \leq \norm{\rdr P}_F^4 \leq q^2 \norm{\rdr P}^4 \leq q^2 B_{P}^4.
\label{L_norm_calculation}
\end{split}
\end{equation}
In the penultimate inequality above we used the fact that $\norm{A}_F \leq \sqrt{\text{rank}(A)}\norm{A}$ for an arbitrary matrix $A$. Now, (\ref{L_tail_bound}) follows from (\ref{general_tail_bound}) by setting $B = q^2 B_{P}^4$ and $d = p^2$.
\end{proof}


%

\begin{proposition} \label{prop_E_n}
Let $\mathcal E_n^A$ be the event that $\norm{A_n^{-1}} \leq 2\norm{A^{-1}}$, and let $\mathcal E_n^L$ be the event that $\norm{L_n^{-1}} \leq 2\norm{L^{-1}}$. Then, 
\begin{equation}
\mathbb P[\mathcal E_n^A] \geq 1 - \alpha_n^A ; \quad  \mathbb P[\mathcal E_n^L] \geq 1 - \alpha_n^L,
\end{equation}
where
\begin{equation}
\alpha_n^A = p \exp\left(\frac{-3n\lambda_{\min}(A)^2/4}{6B_{P}^4 + 2B_{P}^2\lambda_{\min}(A)}\right) \quad \text{and} \quad \alpha_n^L = p^2 \exp\left(\frac{-3n\lambda_{\min}(L)^2/4}{6 q^4 B_{P}^8 + 2q^2 B_{P}^4\lambda_{\min}(L)}\right).
\end{equation}
\end{proposition}
\begin{proof}
Note that $\lambda_{\min}(A_n) \geq \lambda_{\min}(A) - \norm{A_n - A}$. It follows that 
\begin{equation}
\begin{split}
\mathbb P\left[\norm{A_n^{-1}} > 2\norm{A^{-1}}\right] = \mathbb P\left[\lambda_{\min}(A_n) < \frac12 \lambda_{\min}(A)\right] \leq \mathbb P\left[\norm{A_n - A} > \frac12 \lambda_{\min}(A)\right].
\end{split}
\end{equation}
By Corollary \ref{A_corollary}, it follows that
\begin{equation}
\begin{split}
P\left[\norm{A_n - A} > \frac12 \lambda_{\min}(A)\right] \leq p \exp\left(\frac{-3n\lambda_{\min}(A)^2/4}{6B_{P}^4 + 2B_{P}^2\lambda_{\min}(A)}\right) = \alpha_n^A.
\end{split}
\end{equation}
Analogously, Corollary \ref{L_corollary} implies that
\begin{equation}
\begin{split}
P\left[\norm{L_n - L} > \frac12 \lambda_{\min}(L)\right] \leq p^2 \exp\left(\frac{-3n\lambda_{\min}(L)^2/4}{6 q^4 B_{P}^8 + 2q^2 B_{P}^4\lambda_{\min}(L)}\right) = \alpha_n^L.
\end{split}
\end{equation}
\end{proof}

Now, we prove the consistency results, which we restate for convenience. In the following propositions, define 
\begin{equation}
B^2_{I} := \mathbb E[\norm{\rdr I - \rdr P \mu_0}^2].
\label{B_I_def}
\end{equation}
Note that
\begin{equation}
B^2_{I} \leq B_{P}^2 \mathbb E[\norm{\rdr X - \mu_0}^2] + \mathbb E[\norm{\rdr E}]^2.
\label{B_I_ineq}
\end{equation}
Also, recall the following notation introduced in Section \ref{theoretical}:
\begin{equation}
\moment{\rdr V}_m = \mathbb E[\norm{\rdr V - \mathbb E[\rdr V]}^m]^{\frac1m},
\end{equation}
where $\rdr V$ is a random vector. For example, (\ref{B_I_ineq}) can be written as $B^2_{I} \leq B_{P}^2\moment{\rdr X}_2^2 + \moment{\rdr E}_2^2$.

\begin{proposition}\label{mu_appendix}
Suppose $A$ (defined in (\ref{as_convergence2})) is invertible, that $\norm{\rdr P} \leq B_P$ almost surely, and that $\moment{\rdr X}_2, \moment{\rdr E}_2 < \infty$. Then, for fixed $p, q$ we have 
\begin{equation}
\mathbb E\norm{\rdr \mu_n - \mu_0}= O\left(\frac 1 {\sqrt n}\right).
\end{equation}
Hence, under these assumptions, $\mu_n$ is consistent.
\end{proposition}

\begin{proof}
Since $\mathbb P[\norm{\rdr \mu_n - \mu_0} \geq t] \leq t^{-1}\mathbb E[\norm{\rdr \mu_n - \mu_0}]$ by Markov's inequality, it is sufficient to prove that $\mathbb E[\norm{\rdr \mu_n - \mu_0}] \rightarrow 0$ as $n \rightarrow \infty$. Note that by the definition of $\rdr \mu_n$ and Proposition \ref{prop_E_n},
\begin{equation}
\begin{split}
\mathbb E[\norm{\rdr \mu_n - \mu_0}] &= \mathbb P[\mathcal E_n^A] \mathbb E\left[\norm{\rdr \mu_n - \mu_0} | \ \mathcal E_n^A\right] + (1 - \mathbb P[\mathcal E_n^A]) \mathbb E\left[\norm{\rdr \mu_n - \mu_0} | \ \overline{\mathcal E_n^A}\right] \\
&\leq \mathbb P[\mathcal E_n^A]\mathbb E\left[\norm{\rdr A_n^{-1}\rdr b_n - \mu_0} | \ \mathcal E_n^A\right] + \alpha_n^A \norm{\mu_0} \\
&\leq \mathbb P[\mathcal E_n^A]\mathbb E\left[\norm{\rdr A_n^{-1}(\rdr b_n - \rdr A_n\mu)} | \ \mathcal E_n^A\right] + \alpha_n^A \norm{\mu_0} \\
&\leq \mathbb P[\mathcal E_n^A]2 \norm{A^{-1}}\mathbb E\left[\norm{\rdr b_n - \rdr A_n\mu_0} | \ \mathcal E_n^A\right] + \alpha_n^A \norm{\mu_0} \\
&\leq 2 \norm{A^{-1}}\mathbb E\left[\norm{\rdr b_n - \rdr A_n\mu_0}\right] + \alpha_n^A \norm{\mu_0}.
\end{split}
\end{equation}
Since $\rdr b_n - \rdr A_n\mu_0 = \frac{1}{n}\sum_{s = 1}^n \rdr P_s^H (\rdr I_s - \rdr P_s \mu_0)$, where these summands are i.i.d., we find
\begin{equation}
\begin{split}
\mathbb E\left[\norm{\rdr b_n - \rdr A_n\mu_0}\right]^2 \leq \mathbb E\left[\norm{\rdr b_n - \rdr A_n\mu_0}^2\right] &= \frac{1}{n}\mathbb E\left[\norm{\rdr P^H(\rdr I - \rdr P \mu_0)}^2\right] \leq \frac{1}{n}B_{P}^2 B_{I}^2.
\end{split}
\end{equation}
Putting together what we have, we arrive at
\begin{equation}
\mathbb E[\norm{\rdr \mu_n - \mu_0}] \leq \frac{2 \norm{A^{-1}}B_{P}B_{I}}{\sqrt n} + \alpha_n^A \norm{\mu_0}.
\end{equation}
Inspecting this bound reveals that $\mathbb E[\norm{\rdr \mu_n - \mu_0}] \rightarrow 0$ as $n \rightarrow \infty$, as needed. 
\end{proof}

\begin{remark}
Note that with a simple modification to the above argument, we obtain
\begin{equation}
\mathbb P[\mathcal E_n^A] \mathbb E[\norm{\mu_n - \mu_0}^2 | \ \mathcal E_n^A] \leq \frac{4 \norm{A^{-1}}^2}{n}B_{P}^2 B_{I}^2.
\label{useful_later}
\end{equation}
This bound will be useful later.
\end{remark}

Before proving the consistency of $\Sigma_n$, we state a lemma. 

\begin{lemma} \label{rudelson}
Let $\rdr V$ be a random vector on $\bc^p$ with $\mathbb E[\rdr V \rdr V^H] = \Sigma_{\rdr V}$, and let $\rdr V_1, \dots, \rdr V_n$ be i.i.d. samples from $\rdr V$. Then, for some absolute constant $C$,
\begin{equation}
\mathbb E \norm{\frac{1}{n}\sum_{s = 1}^n \rdr V_s \rdr V_s^H - \Sigma_{\rdr V}} \leq C \norm{\Sigma_{\rdr V}}\norm{\Sigma_{\rdr V}^{-1/2}}\frac{\sqrt{\log p}}{\sqrt{n}} \left(\mathbb E\norm{\rdr V}^{\log n}\right)^{1/\log n},
\end{equation}
provided the RHS does not exceed $\norm{\Sigma_{\rdr V}}$. 
\end{lemma}
\begin{proof}
This result is a simple modification of \cite[Theorem 1]{rudelson}.
\end{proof}

\begin{proposition} \label{Sigma_appendix}
Suppose $A$ and $L$ (defined in \ref{as_convergence2}) are invertible, that $\norm{\rdr P} \leq B_P$ almost surely, and that there is a polynomial $Q$ for which
\begin{equation}
\moment{\rdr X}_j, \moment{\rdr E}_j \leq Q(j), \quad j \in \mathbb N.
\label{moment_condition_appendix}
\end{equation}
Then, for fixed $p, q$, we have
\begin{equation}
\mathbb E\norm{\rdr \Sigma_n  - \Sigma_0} = O\left(\frac{Q(\log n)}{\sqrt n}\right).
\end{equation}
Hence, under these assumptions, $\Sigma_n$ is consistent.
\end{proposition}

\begin{proof}
In parallel to the proof of Proposition \ref{mu_appendix}, we will prove that $\mathbb E[\norm{\rdr \Sigma_n - \Sigma_0}] \rightarrow 0$ as $n \rightarrow \infty$. We compute
\begin{equation}
\begin{split}
\mathbb E[\norm{\rdr \Sigma_n - \Sigma_0}] &= \mathbb P[\mathcal E_n^A \cap \mathcal E_n^L] \mathbb E\left[\norm{\rdr \Sigma_n - \Sigma_0} | \ \mathcal E_n^A \cap \mathcal E_n^L\right] + (1 - \mathbb P[\mathcal E_n^A \cap \mathcal E_n^L]) \mathbb E\left[\norm{\rdr \Sigma_n - \Sigma_0} | \ \overline{\mathcal E_n^A \cap \mathcal E_n^L}\right] \\
&\leq \mathbb P[\mathcal E_n^A \cap \mathcal E_n^L]\mathbb E\left[\norm{\rdr L_n^{-1}\rdr B_n - \Sigma_0} | \ \mathcal E_n^A \cap \mathcal E_n^L\right] + (\alpha_n^A + \alpha_n^L) \norm{\Sigma_0} \\
&\leq \mathbb P[\mathcal E_n^A \cap \mathcal E_n^L]\mathbb E\left[\norm{\rdr L_n^{-1}(\rdr B_n - \rdr L_n\Sigma_0)} | \ \mathcal E_n^A \cap \mathcal E_n^L\right] + (\alpha_n^A + \alpha_n^L) \norm{\Sigma_0} \\
&\leq 2 \norm{L^{-1}}\mathbb P[\mathcal E_n^A \cap \mathcal E_n^L]\mathbb E\left[\norm{\rdr B_n - \rdr L_n\Sigma_0} | \ \mathcal E_n^A \cap \mathcal E_n^L\right] + (\alpha_n^A + \alpha_n^L) \norm{\Sigma_0} \\
&\leq 2 \norm{L^{-1}}P[\mathcal E_n^A]\mathbb E\left[\norm{\rdr B_n - \rdr L_n\Sigma_0} | \ \mathcal E_n^A\right] + (\alpha_n^A + \alpha_n^L) \norm{\Sigma_0}.
\label{first_expression}
\end{split}
\end{equation}
Now, we will bound $\mathbb E\left[\norm{\rdr B_n - \rdr L_n\Sigma_0} | \ \mathcal E_n^A\right]$. To do this, we write 
\begin{equation}
\begin{split}
\rdr B_n - \rdr L_n \Sigma_0 &=  \left(\frac{1}{n}\sum_{s = 1}^n \rdr P_s^H (\rdr I_s - \rdr P_s \rdr \mu_n)(\rdr I_s - \rdr P_s \rdr \mu_n)^H \rdr P_s - \frac{1}{n}\sum_{s = 1}^n \rdr P_s^H (\rdr I_s - \rdr P_s \mu_0)(\rdr I_s - \rdr P_s \mu_0)^H \rdr P_s\right) \\
&\quad + \left(\frac{1}{n}\sum_{s = 1}^n \rdr P_s^H (\rdr I_s - \rdr P_s \mu_0)(\rdr I_s - \rdr P_s \mu_0)^H \rdr P_s - (\sigma^2 A + L\Sigma_0)\right) + \sigma^2(A - \rdr A_n)  + (L - \rdr L_n) \Sigma_0 \\
&=: \rdr D_1 + \rdr D_2 + \rdr D_3 + \rdr D_4.
\label{second_expression}
\end{split}
\end{equation}
Let us consider each of these four difference terms in order.
Note that
\begin{equation}
\begin{split}
\mathbb E[\norm{\rdr D_1} | \ \mathcal E_n^A] &\leq B_{P}^2 \frac{1}{n}\sum_{s = 1}^n \mathbb E\left[\norm{(\rdr I_s - \rdr P_s \rdr \mu_n)(\rdr I_s - \rdr P_s \rdr \mu_n)^H - (\rdr I_s - \rdr P_s \mu_0)(\rdr I_s - \rdr P_s \mu_0)^H} | \ \mathcal E_n^A\right].
\end{split}
\end{equation}
Moreover,
\begin{equation}
\begin{split}
&(\rdr I_s - \rdr P_s \rdr \mu_n)(\rdr I_s - \rdr P_s \rdr \mu_n)^H - (\rdr I_s - \rdr P_s \rdr \mu_0)(\rdr I_s - \rdr P_s \rdr \mu_0)^H \\
&\quad = \left\{(\rdr I_s - \rdr P_s \mu_0) + \rdr P_s (\mu_0 - \rdr \mu_n)\right\}\left\{(\rdr I_s - \rdr P_s \mu_0) + \rdr P_s (\mu_0 - \rdr \mu_n)\right\}^H - (\rdr I_s - \rdr P_s \mu_0)(\rdr I_s - \rdr P_s \mu_0)^H \\
&\quad = (\rdr I_s - \rdr P_s \mu_0)(\mu_0 - \rdr \mu_n)^H \rdr P_s^H + \rdr P_s (\mu_0 - \rdr \mu_n)(\rdr I_s - \rdr P_s \mu_0)^H  + \rdr P_s (\mu_0 - \rdr \mu_n)(\mu_0 - \rdr \mu_n)^H \rdr P_s^H
\label{three_terms}
\end{split}
\end{equation}
Using the Cauchy-Schwarz inequality and (\ref{useful_later}), we find
\begin{equation}
\begin{split}
&\mathbb E\left[\norm{(\rdr I_s - \rdr P_s \mu_0) (\mu_0 - \rdr \mu_n)^H \rdr P_s^H} | \ \mathcal E_n^A\right]^2 \\
&\quad\leq B_{P}^2\mathbb E[\norm{\rdr I_s - \rdr P_s \mu_0}\norm{\mu_0 - \rdr \mu_n}| \ \mathcal E_n^A]^2 \\
&\quad\leq B_{P}^2\mathbb E[\norm{\rdr I_s - \rdr P_s \mu_0}^2 | \ \mathcal E_n^A]\mathbb E[\norm{\mu_0 - \rdr \mu_n}^2| \ \mathcal E_n^A] \\
&\quad \leq \frac{4 \norm{A^{-1}}^2}{n\mathbb P[\mathcal E_n^A]^2}B_{P}^4B_{I}^4 \\
\end{split}
\end{equation}
Here, we used (\ref{useful_later}). This bound also holds for the second term in the last line of (\ref{three_terms}). As for the third term,
\begin{equation}
\begin{split}
\mathbb E[\norm{\rdr P_s (\mu_0 - \rdr \mu_n)(\mu_0 - \rdr \mu_n)^H \rdr P_s^H} | \ \mathcal E_n^A] \leq B_{P}^2 \mathbb E[\norm{\mu_0 - \rdr \mu_n}^2 | \ \mathcal E_n^A] \leq \frac{4 \norm{A^{-1}}^2}{n\mathbb P[\mathcal E_n^A]}B_{P}^4 B_{I}^2.
\end{split}
\end{equation}
Putting these bounds together, we arrive at
\begin{equation}
\begin{split}
\mathbb P[\mathcal E_n^A]\mathbb E[\norm{\rdr D_1} | \ \mathcal E_n^A] &\leq \mathbb P[\mathcal E_n^A]B_{P}^2 \left(2\frac{2 \norm{A^{-1}}}{\sqrt{n}\mathbb P[\mathcal E_n^A]}B_{P}^2B_{I}^2 + \frac{4 \norm{A^{-1}}^2}{n\mathbb P[\mathcal E_n^A]}B_{P}^4 B_{I}^2\right) \\
&= \frac{4B_{P}^4 B_{I}^2 \norm{A^{-1}}}{n} \left(\sqrt{n} + \norm{A^{-1}} B_{P}^2\right).
\label{D_1}
\end{split}
\end{equation}

Next, we move on to analyzing $\rdr D_2$. If $\rdr V = \rdr P^H(\rdr I - \rdr P \mu_0)$, note that 
\begin{equation}
\Sigma_{\rdr V} = \mathbb E[\rdr V \rdr V^H] = \mathbb E[\rdr P^H\rdr P(\rdr X - \mu_0)(\rdr X - \mu_0)^H \rdr P^H \rdr P] + \mathbb E[\rdr P^H \rdr E \rdr E^H \rdr P] = L\Sigma_0 + \sigma^2 A. 
\end{equation}
By Lemma (\ref{rudelson}), we find
\begin{equation}
\mathbb P[\mathcal E_n^A]\mathbb E[\norm{\rdr D_2} | \ \mathcal E_n^A] \leq \mathbb E\norm{\frac{1}{n}\sum_{s = 1}^n \rdr V_s \rdr V_s^H - \Sigma_{\rdr V}} \leq C \norm{\Sigma_{\rdr V}}\norm{\Sigma_{\rdr V}^{-1/2}}\frac{\sqrt{\log p}}{\sqrt{n}} \left(\mathbb E\norm{\rdr V}^{\log n}\right)^{\frac{1}{\log n}}
\label{D_2_precursor}
\end{equation}
Since $\Sigma_0 = \mathbb E[(\rdr X - \mu_0)(\rdr X - \mu_0)^H]$, it follows that $\norm{\Sigma_0} \leq \mathbb E[\norm{\rdr X - \mu_0}^2] = \moment{\rdr X}_2^2$. Further, the calculation (\ref{L_norm_calculation}) implies that 
\begin{equation}
\norm{L\Sigma_0} \leq \norm{L\Sigma_0}_F \leq q^4 B_{P}^4 \norm{\Sigma_0}_F \leq q^4 B_{P}^4 \sqrt{\text{rank}(\Sigma_0)} \norm{\Sigma_0} \leq q^4 B_{P}^4 \sqrt{\text{rank}(\Sigma_0)} \moment{\rdr X}_2^2.
\label{D_2}
\end{equation}
Also, it is clear that $\norm{A} \leq B_{P}^2$. Furthermore, Minkowski inequality implies that
\begin{equation}
\begin{split}
\left(\mathbb E\norm{\rdr V}^{\log n}\right)^{\frac{1}{\log n}} &\leq B_{P}\left(B_{P}(\mathbb E[\norm{\rdr X - \mu_0}^{\log n}]^{\frac{1}{\log n}} + \mathbb E[\norm{\rdr E}^{\log n}]^{\frac{1}{\log n}})\right) \\
&= B_{P}\left(B_{P}(\moment{\rdr X}_{\log n} + \moment{\rdr E}_{\log n}\right).
\end{split}
\end{equation}
Hence, (\ref{D_2_precursor}) becomes
\begin{equation}
\begin{split}
&\mathbb P[\mathcal E_n^A]\mathbb E[\norm{\rdr D_2} | \ \mathcal E_n^A] \\
&\quad \leq C B_{P}^3(q^4 B_{P}^2 \sqrt{\text{rank}(\Sigma_0)} \moment{\rdr X}_2^2 + \sigma^2) \norm{(L\Sigma_0 + \sigma^2 A)^{-1/2}}\frac{\sqrt{\log p}}{\sqrt{n}} \left(B_{P}(\moment{\rdr X}_{\log n} + \moment{\rdr E}_{\log n}\right).
\end{split}
\end{equation}

Next, a bound for $\rdr D_3$ follows immediately from (\ref{A_expectation_bound}):
\begin{equation}
\mathbb P[\mathcal E_n^A] \mathbb E[\norm{\rdr D_3} | \ \mathcal E_n^A] \leq \mathbb E[\norm{\rdr D_3}] = \sigma^2\mathbb E[\norm{A - \rdr A_n}]  \leq \sigma^2 C' B_{P}^2 \sqrt{\frac{\log p}{n}}.
\label{D_3}
\end{equation}

Similarly, (\ref{L_expectation_bound}) gives
\begin{equation}
\begin{split}
\mathbb P[\mathcal E_n^A] \mathbb E[\norm{\rdr D_4} | \ \mathcal E_n^A] \leq \mathbb E[\norm{\rdr D_4}] 
&\leq \mathbb E[\norm{L - \rdr L_n}] \norm{\Sigma_0}_F \\
&\leq \sigma^2 C' q^2B_{P}^4 \sqrt{\frac{2\log p}{n}}\sqrt{\text{rank}(\Sigma_0)}\moment{\rdr X}_2^2.
\label{D_4}
\end{split}
\end{equation}

Combining the four bounds (\ref{D_1}), (\ref{D_2}), (\ref{D_3}), (\ref{D_4}) with (\ref{first_expression}) and (\ref{second_expression}), we arrive at
\begin{equation}
\begin{split}
\mathbb E[\norm{\rdr \Sigma_n - \Sigma_0}] &\leq 2\norm{L^{-1}}\left\{\frac{4B_{P}^4 (B_{P}^2\moment{\rdr X}_2^2 + \moment{\rdr E}_2^2) \norm{A^{-1}}}{n} \left(\sqrt{n} + \norm{A^{-1}} B_{P}^2\right) \right. \\
&\quad \quad + C B_{P}^3(q^4 B_{P}^2 \sqrt{\text{rank}(\Sigma_0)} \moment{\rdr X}_2^2 + \sigma^2) \norm{(L\Sigma_0 + \sigma^2 A)^{-1/2}}\frac{\sqrt{\log p}}{\sqrt{n}} \left(B_{P}(\moment{\rdr X}_{\log n} + \moment{\rdr E}_{\log n}\right) \\
&\quad \quad + \sigma^2 C' B_{P}^2 \sqrt{\frac{\log p}{n}} \left. + \sigma^2 C' q^2B_{P}^4 \sqrt{\frac{2\log p}{n}} \sqrt{\text{rank}(\Sigma_0)} \moment{\rdr X}_2^2 \right\} + (\alpha_n^A + \alpha_n^L)\moment{\rdr X}_2^2.
\end{split}
\end{equation}
Fixing all the variables except $n$, we see that the largest term is the one in the second line, and it decays as $Q(\log n)/\sqrt n$ due to the moment growth condition (\ref{moment_condition_appendix}).
\end{proof}

\section{Simplifying (\ref{L_block})} \label{simplifying_integral}
Here, we simplify the expression for an element of $\hat L^{k_1, k_2}$:
\begin{equation}
\begin{split}
\hat L_{i_1i_2, j_1j_2}^{k_1, k_2}  &= \int_{S^2\times S^2}(a_{j_1}^{k_1} \otimes a_{j_2}^{k_2})(\alpha, \beta)\overline{(a_{i_1}^{k_1} \otimes a_{i_2}^{k_2})(\alpha, \beta)} \mathcal K(\alpha, \beta) d\alpha d\beta.
\label{L_block_recall}
\end{split}
\end{equation}
Let $A_{i, j}^k = \overline{a_{i}^{k}} a_{j}^{k}$. Then, (\ref{L_block_recall}) becomes
\begin{equation}
\begin{split}
\hat L^{k_1, k_2}_{i_1i_2, j_1j_2} &= \int_{S^2\times S^2}A_{i_1j_1}^{k_1}(\alpha)\overline{A_{i_2j_2}^{k_2}(\beta)}\mathcal K(\alpha, \beta) d\alpha d\beta.
\end{split}
\end{equation}
Recall from Section \ref{properties} that $a_i^k$ is a spherical harmonic of order up to $k$. It follows that $A_{i_1j_1}^{k_1}$ has a spherical harmonic expansion up to order $2k_1$ (using the formula for the product of two spherical harmonics, which involves the Clebsch-Gordan coefficients). The same holds for $A_{i_2 j_2}^{k_2}$, where the order goes up to $2k_2$.  Let us write $C_{\ell}^m(A_{ij}^{k})$ for the $\ell, m$ coefficient of the spherical harmonic expansion of $A_{ij}^{k}$. Thus, we have
\begin{equation}
A_{i_1j_1}^{k_1}(\alpha) = \sum_{\ell = 0}^{2k_1}\sum_{|m| \leq \ell} C_{\ell, m}(A_{i_1j_1}^{k_1}) Y_{\ell}^m(\alpha), \quad A_{i_2j_2}^{k_2}(\beta) = \sum_{\ell' = 0}^{2k_2}\sum_{|m'| \leq \ell'} C_{\ell', m'}(A_{i_2j_2}^{k_2}) Y_{\ell'}^{m'}(\beta)
\end{equation}
It follows that
\begin{equation}
\begin{split}
\hat L_{i_1i_2, j_1j_2}^{k_1, k_2}  &= \sum_{\ell, m}\sum_{\ell', m'} C_{\ell, m}(A_{i_1j_1}^{k_1}) \overline{ C_{\ell', m'}(A_{i_2j_2}^{k_2}) }\int_{S^2}\int_{S^2} Y_\ell^m(\alpha)\mathcal K(\alpha,\beta) \overline{Y_{\ell'}^{m'}(\beta)}d\alpha d\beta.
\label{soclose}
\end{split}
\end{equation}
Since $\mathcal K(\alpha, \beta)$ depends only on $\alpha \cdot \beta$, by an abuse of notation we can write $\mathcal K(\alpha, \beta) = \mathcal K(\alpha \cdot \beta)$. Thus, the Funk-Hecke theorem applies \cite{natterer}, so we may write 
\begin{equation}
\int_{S^2} Y_\ell^m(\alpha)\mathcal K(\alpha, \beta) d\alpha = c(\ell)Y_\ell^m(\beta),
\end{equation}
where
\begin{equation}
c(\ell) = \frac{2\pi}{P_\ell(1)}\int_{-1}^1 \mathcal K(t)P_{\ell}(t)dt.
\end{equation}
Note that $P_{\ell}$ are the Legendre polynomials. Since $\mathcal K$ is an even function of $t$ and $P_\ell$ has the same parity as $\ell$, it follows that $c(\ell) = 0$ for odd $\ell$. For even $\ell$, we have
\begin{equation}
c(\ell) = 2\int_0^1 \frac{1}{\sqrt{1 - t^2}}P_\ell(t)dt.
\end{equation}
It follows from formula 3 on p. 423 of \cite{prudnikov} that
\begin{equation}
\begin{split}
c(\ell) = 2\int_0^1 \frac{1}{\sqrt{1 - t^2}}P_\ell(t)dt = \pi \left(\frac{\ell!}{2^\ell(\frac \ell 2 !)^2}\right)^2.	
\label{c_ell}
\end{split}
\end{equation}
Using Stirling's formula, we can find that $c(\ell) \sim \ell^{-1}$ for large $\ell$.

Finally, plugging the result of Funk-Hecke into (\ref{soclose}), we obtain
\begin{equation}
\begin{split}
\hat L_{i_1i_2, j_1j_2}^{k_1, k_2}  &= \sum_{\ell, m}\sum_{\ell', m'}c(\ell)C_{\ell, m}(A_{i_1j_1}^{k_1}) \overline{ C_{\ell', m'}(A_{i_2j_2}^{k_2}) }\int_{S^2}Y_\ell^m(\beta)\overline{Y_{\ell'}^{m'}(\beta)}d\beta \\
&= \sum_{\ell,m} c(\ell)C_{\ell, m}(A_{i_1j_1}^{k_1}) \overline{ C_{\ell, m}(A_{i_2j_2}^{k_2})}.
\label{formula2}
\end{split}
\end{equation}
Thus, we have verified (\ref{L_hat_indices}).

\end{document}